\documentclass[10pt]{smfart}

\RequirePackage[T1]{fontenc}
\RequirePackage{amsfonts,latexsym,amssymb}
\RequirePackage[frenchb]{babel}
\addto\extrasfrenchb{\bbl@nonfrenchitemize
\bbl@nonfrenchspacing}

\RequirePackage{smfenum}

\usepackage[matrix,arrow]{xy}
\usepackage{url}
\usepackage{accents}
\usepackage{enumitem}
\usepackage{tikz}
\usepackage{fge}
\newcommand{\moins}{\mathbin{\fgebackslash}}

\newcommand{\subsubsubsection}[1]{{\vskip.1cm\noindent $\diamond$~{\it {#1}}. ---}}

\theoremstyle{plain}
\numberwithin{equation}{section}
\newtheorem{prop}[equation]{Proposition}
\newtheorem{theo}[equation]{Th\'eor\`eme}

\newtheorem{coro}[equation]{Corollaire}
\newtheorem{lemm}[equation]{Lemme}
\theoremstyle{definition}
\theoremstyle{remark}
\newtheorem{defi}[equation]{D\'efinition}
\newtheorem{rema}[equation]{Remarque}
\newtheorem{exem}[equation]{Exemple}

\RequirePackage{mathrsfs}
\let\mathcal\mathscr
\let\cal\mathcal
\let\goth\mathfrak

\def\Q{{\bf Q}} \def\Z{{\bf Z}}
\def\C{{\bf C}} \def\N{{\bf N}} \def\R{{\bf R}}
\def\O{{\cal O}}  
\def\dual{{\boldsymbol *}}
\def\bmu{{\boldsymbol\mu}}

\def\rbp{{\overline\R}_+^\dual}

\def\dbar{\partial^\dual}

\def\ainf{{\bf A}_{{\rm inf}}}
\def\ast{{\bf A}_{\rm st}}
\def\bst{{\bf B}_{{\rm st}}}
\def\Bst{{\mathbb B}_{{\rm st}}}
\def\bcris{{\bf B}_{{\rm cris}}} 
\def\acris{{\bf A}_{{\rm cris}}}
\def\bdr{{\bf B}_{{\rm dR}}}

\def\piqp{{\bf P}^1}
 \def\A{{\bf A}}

\def\hJ{\widehat J}
\def\tJ{\widetilde J}

\def\ocirc#1{\accentset{\circ}{#1}}
\def\wotimes{\,\widehat\otimes\,}

\newcommand{\rig}{\operatorname{rig} }
\newcommand{\eet}{\operatorname{\acute{e}t} }
\newcommand{\proet}{\operatorname{pro\acute{e}t} }
\newcommand{\an}{\operatorname{an} }

\newcommand{\Spf}{\operatorname{Spf} }
\newcommand{\sx}{{\mathcal{X}}}

\def\epsilon{\varepsilon}
\let\emptyset\varnothing

\begin{document}
\title[Cohomologie des courbes $p$-adiques]
{Cohomologie des courbes analytiques $p$-adiques}
\author{Pierre Colmez}
\address{CNRS, IMJ-PRG, Sorbonne Universit\'e, 4 place Jussieu, 75005 Paris, France}
\email{pierre.colmez@imj-prg.fr}
\author{Gabriel Dospinescu}
\address{CNRS, UMPA, \'Ecole Normale Sup\'erieure de Lyon, 46 all\'ee d'Italie, 69007 Lyon, France}
\email{gabriel.dospinescu@ens-lyon.fr}
\author{Wies{\l}awa Nizio{\l}}
\address{CNRS, IMJ-PRG, Sorbonne Universit\'e, 4 place Jussieu, 75005 Paris, France}
\email{wieslawa.niziol@imj-prg.fr}
\dedicatory{\`A la m\'emoire de Robert Coleman et Michel Raynaud}
\begin{abstract}
La cohomologie d'un affino\"{\i}de a des propri\'et\'es peu sympathiques;
on y rem\'edie en g\'en\'eral en rendant l'affino\"{\i}de surconvergent.  Dans cet article,
nous nous int\'eressons \`a la dimension~$1$, et nous calculons, en utilisant des d\'ecompositions
analogues \`a celles des
surfaces de Riemann en pantalons,
les diff\'erentes cohomologies des affino\"{\i}des  
 (pour donner un sens \`a ces d\'ecompositions, nous sommes amen\'es 
\`a modifier l\'eg\`erement la notion de sch\'ema formel $p$-adique, ce qui nous conduit \`a d\'efinir 
la g\'eom\'etrie adoque -- interpolation entre adique et ad hoc). 
Il en r\'esulte
que la cohomologie d'un affino\"{\i}de (de dimension~$1$) n'est pas si pathologique.

Nous en d\'eduisons un calcul des diff\'erentes cohomologies des courbes sans bord (comme le
demi-plan de Drinfeld et ses rev\^etements), et obtenons en particulier
une description de leur cohomologie pro\'etale $p$-adique en termes du complexe de
de Rham et de la cohomologie de Hyodo-Kato, cette derni\`ere ayant des propri\'et\'es
similaires \`a celles de la cohomologie pro\'etale $\ell$-adique, pour $\ell\neq p$.
\end{abstract}
\begin{altabstract}
Cohomology of affinoids does not behave well; often,  this can be remedied by making  affinoids overconvergent. In this paper, we focus on  dimension 1 and  compute, using  analogs of pants decompositions of 
Riemann surfaces,
various cohomologies of affinoids. 
 To give a meaning to these decompositions we modify slightly the notion of $p$-adic
formal scheme, which gives rise to the adoc (an interpolation between adic and ad hoc) geometry. 
It turns out  that cohomology of affinoids (in dimension 1) is not that pathological. 

 From this we deduce a computation of  cohomologies of curves without boundary (like the Drinfeld half-plane and its coverings).  In particular, we obtain a description of their $p$-adic pro-\'etale cohomology in terms of de the Rham complex and the Hyodo-Kato cohomology, the later having properties similar to the ones of $\ell$-adic pro-\'etale cohomology, for $\ell\neq p$.
\end{altabstract}
\setcounter{tocdepth}{2}

\maketitle

{\Small \tableofcontents}

\section*{Introduction}

Soit $p$ un nombre premier et
soit $C$ un corps alg\'ebriquement clos, complet pour une valuation $v_p$ v\'erifiant
$v_p(p)=1$ et\footnote{Cette restriction est due au fait que nous avons choisi
d'exprimer les r\'esultats en termes de l'anneau $\bst$; une formulation (un peu moins
esth\'etique) qui n'utilise
que $\bcris$ serait possible et elle permettrait de supprimer cette restriction.}
 $v_p(C^\dual)=\Q$.  
On note $\O_C$ l'anneau des entiers de $C$, ${\goth m}_C$ l'id\'eal maximal
de $\O_C$, $k_C$ son corps
r\'esiduel, $\O_{\breve C}$ l'anneau $W(k_C)$ et $\breve C$ le sous-corps $W(k_C)[\frac{1}{p}]$ de $C$.

Si $K$ est un sous-corps ferm\'e de $C$, on note $G_K$ le groupe
${\rm Aut}_K(C)$ des automorphismes continus de $C$ laissant fixe $K$.

\Subsection{Cohomologie pro\'etale $p$-adique}\label{TTT1}
Soit $X$ une courbe analytique
d\'efinie sur $C$ (vue, selon le contexte comme une courbe rigide analytique, un espace de Berkovich de dimension~$1$, etc.) 
et, 
sauf mention explicite
du contraire, lisse et
(g\'eom\'etriquement) connexe.
Si $X$ est compacte\footnote{Quel que soit le point de vue (rigide, Berkovich, etc.), 
une courbe compacte est, dans cet article, une courbe propre, et
un affino\"{\i}de est quasi-compact mais pas compact.}, sa cohomologie a de bonnes propri\'et\'es:
les groupes de cohomologies de de Rham $H^1_{\rm dR}(X)$ ou \'etale
$H^1_{\eet}(X,\Q_\ell(1))$, o\`u $\ell$ est un nombre premier,
 sont de dimension finie, ind\'ependante de la cohomologie
consid\'er\'ee (et invariante par extension des scalaires de $C$ \`a un surcorps
avec les m\^emes propri\'et\'es), sur le corps ad\'equat.
De plus, on a des th\'eor\`emes de comparaison reliant ces diff\'erentes
cohomologies.

Il n'en est pas de m\^eme si $X$ est seulement quasi-compacte mais pas compacte
(i.e~un affino\"{\i}de\footnote{Sauf mention
du contraire, un affino\"{\i}de est, dans ce texte, de dimension~$1$, lisse et connexe.}): 
dans ce cas,
$H^1_{\eet}(X,\Q_\ell(1))$ est encore de dimension finie si $\ell\neq p$, mais
$H^1_{\rm dR}(X)$ est un $C$-espace de dimension infinie non s\'epar\'e,
et $H^1_{\eet}(X,\Q_p(1))$ est aussi de dimension infinie (et d\'epend de $C$).
Nous nous proposons de montrer que, malgr\'e ces pathologies apparentes,
la cohomologie des affino\"{\i}des (et, plus g\'en\'eralement, des courbes non compactes)
a des propri\'et\'es raisonnables.

\subsubsection{Courbes quasi-compactes}\label{BAS1}
Si $Y$ est une courbe quasi-compacte,
notons $\O(Y)^{\dual\dual}$ le sous-groupe
des $f\in\O(Y)^\dual$ telles que $f-1$ soit topologiquement nilpotente.
\begin{theo}\label{intro1.1}
Si $Y$ est un affino\"{\i}de,
on a un diagramme commutatif fonctoriel de banachs\footnote{Si $M$ est un $\Z$-module,
on pose $\Q_p\wotimes M:=\Q_p\otimes_{\Z_p}(\varprojlim_n M/p^nM)$.}
$$
\xymatrix@R=.6cm@C=.5cm{
0\ar[r] &\Q_p\wotimes \O(Y)^{\dual\dual}
\ar[r]\ar@{=}[d] & H^1_{\proet}(Y,\Q_p(1))\ar[d] \ar[r] &
(\bst^+\otimes_{\breve C} H^1_{\rm HK}(Y)^{\rm sep})^{N=0,\varphi=p}\ar[r]\ar[d]^-{\theta\otimes\iota_{\rm HK}} & 0\\
0\ar[r]& \Q_p\wotimes \O(Y)^{\dual\dual} \ar[r]^-{\rm dlog} 
& \Omega^1(Y) \ar[r] & H^1_{\rm dR}(Y)^{\rm sep}\ar[r] & 0
} $$
dans lequel la ligne du haut est exacte, celle du bas est un complexe,
et toutes les fl\`eches sont d'image ferm\'ee.
\end{theo}

\begin{rema}\label{basic1}
{\rm (i)} Il y a un \'enonc\'e au niveau entier, cf.~rem.\,\ref{WN}.

{\rm (ii)} Un affino\"{\i}de \'etant quasi-compact, on a $H^1_{\proet}(Y,\Q_\ell(1))=
H^1_{\eet}(Y,\Q_\ell(1))$, pour tout $\ell$ (y compris $\ell=p$), cf.~\cite[cor. 3.17]{RAV}.

{\rm (iii)} 
Le groupe de cohomologie de Hyodo-Kato $H^1_{\rm HK}(Y)^{\rm sep}$ est 
un $\breve C$-espace de dimension finie muni d'actions d'un frobenius semi-lin\'eaire
$\varphi$, d'un op\'erateur de monodromie $N$ v\'erifiant $N\varphi=p\,\varphi N$,
et d'un {\it isomorphisme de Hyodo-Kato}: 
$$\iota_{\rm HK}:C\otimes_{\breve C}H^1_{\rm HK}(Y)^{\rm sep}\cong H^1_{\rm dR}(Y)^{\rm sep},$$
o\`u $H^1_{\rm dR}(Y)^{\rm sep}$ d\'esigne le s\'epar\'e de $H^1_{\rm dR}(Y)$
(i.e. son quotient par l'adh\'erence de $0$):
le groupe $H^1_{\rm dR}(Y)$ est un $C$-espace de dimension infinie, non s\'epar\'e,
mais $H^1_{\rm dR}(Y)^{\rm sep}$ 
est de dimension finie.

{\rm (iv)} 
La preuve du th\'eor\`eme utilise des m\'ethodes syntomiques mais va plus loin:
les m\'ethodes syntomiques~\cite{Ts,CN}
fournissent naturellement un diagramme commutatif \`a lignes exactes\footnote{
$H^1_{\rm HK}(Y)$ est naturellement un quotient $W_1/W_2$ de $\breve C$-banachs et
$\bst^+=\bcris^+[u]$ o\`u $\bcris^+$ est un $\breve C$-banach (et donc
$\bst^+$ est une limite inductive de $\breve C$-banachs). On d\'efinit
$\bst^+\wotimes_{\breve C} H^1_{\rm HK}(Y)$ 
comme $(\bst^+\wotimes_{\breve C}W_1)/(\bst^+\wotimes_{\breve C} W_2)$ o\`u
$\bst^+\wotimes_{\breve C}W_i:=\bst^+\otimes_{\bcris^+}(\bcris^+\wotimes_{\breve C}W_i)$.}:
$$
\xymatrix@R=.6cm@C=.5cm{
0\ar[r] &\O(Y)/C
\ar[r]\ar@{=}[d] & H^1_{\proet}(Y,\Q_p(1))\ar[d] \ar[r] &
(\bst^+\wotimes_{\breve C} H^1_{\rm HK}(Y))^{N=0,\varphi=p}\ar[r]\ar[d]^{\theta\otimes\iota_{\rm HK}} & 0\\
0\ar[r]& \O(Y)/C \ar[r]
& \Omega^1(Y) \ar[r] & H^1_{\rm dR}(Y)\ar[r] & 0
} 
$$
Par rapport au th\'eor\`eme, il y a deux diff\'erences sensibles:

$\bullet$ Le groupe $H^1_{\rm HK}(Y)$ est, comme $H^1_{\rm dR}(Y)$,
 de dimension infinie et non s\'epar\'e
(on a un isomorphisme $\iota_{\rm HK}:C\wotimes_{\breve C}H^1_{\rm HK}(Y)\cong H^1_{\rm dR}(Y)$),
alors que $H^1_{\rm HK}(Y)^{\rm sep}$ est s\'epar\'e, de dimension finie, et
se d\'ecrit simplement (th.\,\ref{expo2}) en termes d'une triangulation de~$Y$. 

$\bullet$ 
La fl\`eche $\O(Y)\to \Q_p\wotimes \O(Y)^{\dual\dual}$ faisant commuter le diagramme
\'evident est $f\mapsto \exp(f)$, mais
l'image de $\O(Y)$ par $f\mapsto \exp(f)$ est $\Q_p\otimes \O(Y)^{\dual\dual}$
qui est dense dans $\Q_p\wotimes \O(Y)^{\dual\dual}$ mais ne lui est pas \'egal
car $\O(Y)^{\dual\dual}$ n'est pas complet pour la topologie $p$-adique: par exemple,
si $Y$ est la boule unit\'e (i.e.~$\O(Y)=C\langle T\rangle$), alors
$\prod_{n\geq 1}(1+p^{1/p^n}T)^{p^n}$ ne converge pas dans $C\langle T\rangle$
(mais converge dans $\O_C[[T]]$).
\end{rema}
\begin{rema}\label{basic42}
Si $Y$ est un affino\"{\i}de de dimension~$d$, et si $r\leq d$,
le r\'esultat ci-dessus sugg\`ere que l'on pourrait peut-\^etre esp\'erer une suite exacte
$$0\to \Q_p\wotimes K^r_{\cal M}(Y)^{++}\to H^r_{\proet}(Y,\Q_p(r))
\to (\bst^+\otimes H^r_{\rm HK}(Y)^{\rm sep})^{N=0,\varphi=p^r}\to 0.$$
(Le groupe $K^r_{\cal M}(Y)^{++}$ est le sous-groupe du groupe
de $K$-th\'eorie de Milnor $K^r_{\cal M}(\O(Y))$ engendr\'e par les symboles
$(f_1,\dots,f_r)$, avec $f_i\in\O(Y)^{\dual\dual}$.
Notons que, puisque l'on ne prend que des symboles d'\'el\'ements de $\O(Y)^{\dual\dual}$,
la relation de Steinberg disparait puisque $1-x\notin \O(Y)^{\dual\dual}$ si $x\in \O(Y)^{\dual\dual}$.)
\end{rema}

Une des raisons de la forme du diagramme ci-dessus est que
$H^1(Y,\O)=0$.  Si $Y$ est compacte, $H^1(Y,\O)\neq 0$ si $Y$ est de genre~$\geq1$,
mais $\Q_p\wotimes \O(Y)^{\dual\dual}=0$ et le diagramme prend la forme 
(classique, c'est un cas particulier du th\'eor\`eme de comparaison de Tsuji~\cite{Ts}
ou m\^eme, dans ce cas, de Kato~\cite{K1}) suivante:
\begin{theo}\label{basic2}
Si $Y$ est une courbe compacte, on a un diagramme commutatif
fonctoriel, dont les lignes sont exactes:
$$
\xymatrix@R=.6cm@C=.5cm{
0\ar[r] &  H^1_{\proet}(Y,\Q_p(1))\ar[d] \ar[r] &
(\bst^+\otimes_{\breve C} H^1_{\rm HK}(Y))^{N=0,\varphi=p}\ar[r]\ar[d]^{\theta\otimes\iota_{\rm HK}} &
H^1(Y,\O)\ar[r]\ar@{=}[d] &0\\
0\ar[r]& \Omega^1(Y) \ar[r] & H^1_{\rm dR}(Y)\ar[r] &
H^1(Y,\O)\ar[r]&0
}
$$
De plus, $H^1_{\proet}(Y,\Q_p(1))$ est un $\Q_p$-espace de dimension finie.
\end{theo}

\subsubsection{Affino\"{\i}des surconvergents}\label{TTT2}
Une mani\`ere standard de rendre la cohomologie de de Rham d'un affino\"{\i}de plus raisonnable
est de le rendre surconvergent: sa cohomologie de de Rham devient de dimension finie
(et topologiquement s\'epar\'ee).
En \'ecrivant un affino\"{\i}de surconvergent $Y^\dagger$ comme une limite projective
d'affino\"{\i}des $Y_\delta$, pour $\delta>0$, 
en utilisant la description du groupe $H^1_{\rm HK}(Y_\delta)^{\rm sep}$ du th.\,\ref{expo2}
ci-dessous, et
en passant \`a la limite dans le th.~\ref{intro1.1},
on obtient le r\'esultat suivant. 
\begin{theo}\label{intro1}
Si $Y^\dagger$ est un affino\"{\i}de surconvergent,
on a le diagramme commutatif fonctoriel d'espaces vectoriels topologiques
suivant:
$$
\xymatrix@R=.6cm@C=.6cm{
0\ar[r] & \O(Y^\dagger)/C\ar[r]^-{\exp}\ar@{=}[d] & H^1_{\proet}(Y^\dagger,\Q_p(1))\ar[d] \ar[r] &
(\bst^+\otimes_{\breve C} H^1_{\rm HK}(Y^\dagger))^{N=0,\varphi=p}\ar[r]\ar[d]^{\theta\otimes\iota_{\rm HK}} & 0\\
0\ar[r]& \O(Y^\dagger)/C \ar[r]^-d & \Omega^1(Y^\dagger) \ar[r] & H^1_{\rm dR}(Y^\dagger)\ar[r] & 0
}
$$
dans lequel les lignes sont exactes et toutes les fl\`eches sont d'image ferm\'ee.
\end{theo}
\begin{rema}\label{intro2}
{\rm (i)} Par d\'efinition, $H^1_{\proet}(Y^\dagger,\Q_p(1))=
\varinjlim_\delta H^1_{\proet}(Y_\delta,\Q_p(1))$, et
$H^1_{\rm HK}(Y^\dagger)=
\varinjlim_\delta H^1_{\rm HK}(Y_\delta)$
est un $\breve C$-espace de dimension finie muni d'actions d'un frobenius semi-lin\'eaire
$\varphi$ et d'un op\'erateur de monodromie $N$,
et d'un isomorphisme 
$\iota_{\rm HK}:C\otimes_{\breve C}H^1_{\rm HK}(Y^\dagger)\cong H^1_{\rm dR}(Y^\dagger)$.

{\rm (ii)} Comme $H^1_{\rm HK}(Y^\dagger)$ est de dimension finie,
$(\bst^+\otimes H^1_{\rm HK}(Y^\dagger))^{\varphi=p,N=0}$ est l'espace des $C$-points
d'un espace de Banach-Colmez~\cite{BC}.
Comme $\O(Y^\dagger)$ est l'espace des sections globales d'un faisceau coh\'erent,
on voit que la cohomologie \'etale g\'eom\'etrique d'un affino\"{\i}de surconvergent,
bien que tr\`es grosse, est compos\'ee d'objets ayant des propri\'et\'es
de finitude raisonnables.

{\rm (iii)} Les $\Q_p$-espaces vectoriels topologiques $(\bst^+\otimes_{\breve C} H^1_{\rm HK}(Y^\dagger))^{N=0,\varphi=p}$ et $H^1_{\rm dR}(Y^\dagger)$ sont des banachs mais tous les
autres espaces non nuls du diagramme sont des limites inductives de banachs.

\end{rema}

\subsubsection{Courbes sans bord}\label{intro3}
Supposons maintenant que
$Y$ n'est pas compacte mais n'a pas de bord quand m\^eme\footnote{Notons
que rendre un affino\"{\i}de surconvergent est une mani\`ere
de supprimer son bord.} (courbe Stein): par exemple,
une courbe ouverte obtenue en retirant un nombre fini de disques ferm\'es
d'une courbe propre, ou un rev\^etement \'etale du demi-plan de Drinfeld.
Dans ce cas, $H^1_{\rm dR}(Y)$ n'est pas forc\'ement de
dimension finie, mais c'est un espace s\'epar\'e, limite projective d\'enombrable
d'espaces de dimension finie.  Si $\ell\neq p$, alors $H^1_{\proet}(Y,\Q_\ell(1))$
est aussi une limite projective d\'enombrable
d'espaces s\'epar\'es de dimension finie, mais ce n'est pas le cas si $\ell=p$.

Une telle courbe est une r\'eunion croissante stricte d'affino\"{\i}des ou, au choix,
d'affino\"{\i}des surconvergents, et on d\'eduit du th.~\ref{intro1.1} (ou du th.~\ref{intro1})
le r\'esultat suivant qui fournit une preuve alternative au th.\,4.12 de~\cite{CDN2}
en dimension~$1$.

\begin{theo}\label{intro4}
Si $Y$ est une courbe non compacte, sans bord,
on a un diagramme commutatif fonctoriel
 de fr\'echets
$$
\xymatrix@R=.6cm@C=.7cm{
0\ar[r] & \O(Y)/C\ar[r]^-{\exp}\ar@{=}[d] & H^1_{\proet}(Y,\Q_p(1))\ar[d] \ar[r] &
(\bst^+\wotimes H^1_{\rm HK}(Y))^{N=0,\varphi=p}\ar[r]\ar[d]^{\theta\otimes\iota_{\rm HK}} & 0\\
0\ar[r]& \O(Y)/C \ar[r]^-d & \Omega^1(Y) \ar[r] & H^1_{\rm dR}(Y)\ar[r] & 0
}
$$
dans lequel les lignes sont exactes et toutes les fl\`eches sont d'image ferm\'ee.
\end{theo}

\begin{rema}\label{intro5}
{\rm (i)}
La principale diff\'erence avec le cas d'un affino\"{\i}de surconvergent
est que $H^1_{\rm HK}(Y)$ et $H^1_{\rm dR}(Y)$ ne sont pas forc\'ement
de dimension finie, mais sont des limites projectives d\'enombrables
d'espaces s\'epar\'es de dimension finie (ce sont donc des fr\'echets,
mais des fr\'echets un peu particuliers), cela explique les produits
tensoriels compl\'et\'es et l'isomorphisme de Hyodo-Kato
$\iota_{\rm HK}:C\wotimes_{\breve C}H^1_{\rm HK}(Y)\cong H^1_{\rm dR}(Y)$
fait aussi intervenir un produit
tensoriel compl\'et\'e.
Les autres espaces sont des duaux de limites inductives compactes de banachs.

{\rm (ii)} Si $Y$ est un affino\"{\i}de surconvergent
ou une courbe sans bord (compacte ou non),
le noyau de $H^1_{\proet}(Y,\Q_p(1))\to \Omega^1(Y)$ est isomorphe \`a
$t\, H^1_{\rm HK}(Y)^{\varphi=1}$, o\`u $t\in(\bcris^+)^{\varphi=p}$ 
est le $2i\pi$ $p$-adique de Fontaine.
\end{rema}

\Subsection{Description combinatoire des diverses cohomologies}\label{intro6}
Expliquons maintenant comment d\'ecrire les objets apparaissant dans les th.\,\ref{intro1.1},
\ref{intro1} et~\ref{intro4} \`a partir de d\'ecoupages en objets plus \'el\'ementaires.
Ces d\'ecoupages sont induits par la stratification naturelle
de la fibre sp\'eciale d'un mod\`ele semi-stable sur $\O_C$ (ouverts de lissit\'e
des composantes irr\'eductibles, et intersections de composantes irr\'eductibles).
Ils fournissent des recouvrements ouverts dont la combinatoire est particuli\`erement simple
(l'intersection de trois ouverts est vide, l'intersection de
deux ouverts est vide ou est un {\og cercle fant\^ome\fg}, 
et chacun des ouverts est affine et a {\og bonne r\'eduction\fg}),
ce qui facilite les calculs \`a la \v{C}ech.
Une des utilisations agr\'eables de ces d\'ecoupages est la trivialisation de 
la construction de l'isomorphisme de Hyodo-Kato
(cette construction est, en g\'en\'eral, assez p\'enible).
\subsubsection{D\'ecoupage en shorts et jambes}\label{intro6.1}
Soit $Y$ une courbe quasi-compacte.  On dispose d'une bijection
entre les triangulations $S$ de $Y$ et les mod\`eles semi-stables
de $Y$ sur $\O_C$. 

Choisissons donc une triangulation $S$ et notons $Y_S$ le mod\`ele 
semi-stable de $Y$ qui lui est associ\'e. On suppose $S$ assez fine
pour que les composantes irr\'eductibles
de la fibre sp\'eciale $Y_S^{\rm sp}$ soient lisses et deux d'entre elles s'intersectent
en au plus un point.
On voit $Y_S^{\rm sp}$ 
comme une {\it courbe propre sur $k_C$ munie d'un ensemble $A$ de points marqu\'es} 
$a=(P_a,\mu(a))$, avec $A=A_c\sqcup (A\moins A_c)$, o\`u:

$\bullet$ $A_c$ est l'ensemble des
points singuliers (intersections de deux composantes irr\'eductibles), 
pour lesquels $\mu(a)\in\Q_+^\dual$, 

$\bullet$ si $a\in A\moins A_c$, alors $\mu(a)=0^+$ (les $P_a$, pour $a\in A\moins A_c$, 
sont les points de $Y_S^{\rm sp}$ qu'il faut enlever pour obtenir
la fibre sp\'eciale au sens usuel). 

Les composantes irr\'eductibles de $Y_S^{\rm sp}$ (qui sont donc
propres et lisses) sont en bijection avec $S$ (on note $ Y^{\rm sp}_s$ la composante
correspondant \`a $s\in S$). 

\smallskip

A partir de ces donn\'ees, on fabrique
un graphe $\Gamma$, dont les sommets sont $S$, les ar\^etes sont $A$, chaque ar\^ete $a$
ayant pour longueur $\mu(a)$ et comme extr\'emit\'es les $s\in S$ tels que $P_a\in Y^{\rm sp}_s$
(et donc $a\in A_c$ a deux extr\'emit\'es alors que $a\in A\moins A_c$ a une seule
extr\'emit\'e).  Le graphe ainsi obtenu est donc le graphe dual de la fibre sp\'eciale
au sens classique auquel on a ajout\'e des ar\^etes de longueur $0^+$ aux sommets
correspondant aux composantes irr\'eductibles non propres, une par point manquant.

Si $s\in S$, le tube $Y_s$ de $ Y^{\rm sp}_s$ priv\'e de ses points marqu\'es
est {\it un short} (i.e.~(le mod\`ele formel d')un affino\"{\i}de
avec bonne r\'eduction), et si $a\in A_c$, le tube $Y_a$ de $P_a$ est {\it une jambe} (i.e.~une couronne
ouverte) de longueur $\mu(a)$ (on a $\O(Y_a)=\O_C[[T_{a,s_1},T_{a,s_2}]]/(T_{a,s_1}T_{a,s_2}-p^{\mu(a)})$,
si $s_1,s_2$ sont les extr\'emit\'es de~$a$).  

Les $Y_i$, pour $i\in I=S\sqcup A_c$,
forment une partition de $Y$, mais si on veut pouvoir reconstruire $Y$, il faut encore
une donn\'ee de recollement de $Y_s$ et $Y_a$ si $s$ est une extr\'emit\'e de $a$.
Le point $P_a$ d\'etermine une valuation de rang~$2$ sur $\O(Y_s)$, et donc
{\it un cercle fant\^ome} $Y_{s,a}$ (le choix d'un param\`etre local
fournit un isomorphisme $\O(Y_{s,a})\cong \O_C[[T_{s,a},T_{s,a}^{-1}\rangle$,
compl\'et\'e de $\O_C[[T_{s,a}]][T_{s,a}^{-1}]$ pour la topologie $p$-adique),
et on aussi un cercle fant\^ome $Y_{a,s}$ correspondant sur $Y_a$, et la donn\'ee
de recollement est un isomorphisme $\iota_{a,s}:Y_{a,s}\cong Y_{s,a}$, i.e.~un isomorphisme
$\O_C[[T_{s,a},T_{s,a}^{-1}\rangle\cong \O_C[[T_{a,s},T_{a,s}^{-1}\rangle$.

L'ensemble des donn\'ees pr\'ec\'edentes 
(i.e.~$\Gamma= (S,A, A_c,\mu)$, $(Y_i)_{i\in I}$, $(\iota_{i,j})_{(i,j)\in I_{2,c}}$,
o\`u $I_{2,c}=\{(a,s),\ 
{\text{$a\in A_c$ et $s$ extr\'emit\'e de $a$}}\}$)
est {\it un patron de courbe}.  Ce qui pr\'ec\`ede 
explique\footnote{Comme un dessin vaut mieux qu'un long discours, 
le lecteur est invit\'e \`a consulter les dessins du chap.~\ref{constr6} pour une repr\'esentation
 {\og physique\fg} des objets ci-dessus.} comment associer
un patron de courbe \`a une courbe munie d'une triangulation assez fine, et qu'on
peut reconstruire $Y$ \`a partir de son patron.  R\'eciproquement, on a
le r\'esultat suivant, analogue adique\footnote{Ou plut\^ot adoque, cf.~(iii) de la rem.\,\ref{constr5}
et~\S\,\ref{adoc1}. 
La g\'eom\'etrie adoque est \`a la g\'eom\'etrie adique ce que le ticket choc est au ticket chic,
\'echo lointain d'une \'epoque \'epique~\url{https://www.youtube.com/watch?v=247w8q1tPsY}.} 
de r\'esultats de Harbater~\cite{H} et Raynaud~\cite{R},
 qui est un peu surprenant au vu de l'abondance de donn\'ees
de recollement possibles.
\begin{theo}\label{expo1}
Si $(\Gamma,(Y_i)_{i\in I}, (\iota_{i,j})_{(i,j)\in I_{2,c}})$ est un patron de courbe,
il existe un unique couple $(Y,S)$, o\`u $Y$ est une courbe quasi-compacte
et $S$ une triangulation de~$Y$, dont ce soit le patron.
\end{theo}

\begin{rema}\label{expo1.1}
(i)
On peut s'amuser \`a varier les longueurs $\mu(a)$ des jambes et multiplier les
$\iota_{i,j}$ par des $\alpha_{i,j}\in\O_C^\dual$ ou, ce qui revient au m\^eme,
fixer les $\iota_{i,j}$ mais remplacer les $Y_a$ par des $Y_a^\alpha$,
avec $\O(Y_a^\alpha)=\O_C[[T_{a,s_1},T_{a,s_2}]]/(T_{a,s_1}T_{a,s_2}-\alpha_a)$ 
et $\alpha_a\in{\goth m}_C\moins\{0\}$
(ou m\^eme $\alpha_a\in{\goth m}_C$ si on se permet des courbes avec des singularit\'es nodales).
Cela fournit (cf.~\no\ref{familles}) 
des familles de courbes param\^etr\'ees par des produits de boules ouvertes.

(ii) On peut aussi, avec les m\^emes techniques, fabriquer une courbe relative
sur ${\rm Spa}(\ainf,\ainf)$ dont la fibre
en $\tilde p=p$ (cf.~\no\ref{Nota1} pour $\tilde p$) 
est $Y_S$
et celle en $\tilde p=0$ est une courbe singuli\`ere sur $\O_{\breve C}$
dont le graphe dual est le m\^eme que celui de sa fibre sp\'eciale (qui est aussi celle de $Y_S$).
\end{rema}

\begin{rema}\label{expo1.2}
{\rm (i)}
En dimension sup\'erieure, si on part d'une vari\'et\'e analytique quasi-compacte $Y$
ayant un mod\`ele semi-stable $Y_S$ sur $\O_C$, assez fin, on peut d\'ecouper cette vari\'et\'e
en prenant les images r\'eciproques des \'el\'ements du d\'ecoupage naturel
de la fibre sp\'eciale.  Chaque pi\`ece est une fibration en affino\"{\i}des
ayant bonne r\'eduction au-dessus d'une polycouronne, 
et ces pi\`eces se recollent le long de fibrations en affino\"{\i}des au-dessus de polycouronnes fant\^omes
pour reconstruire $Y$.  Cela devrait permettre de donner une description de la cohomologie
de $Y$ ne faisant intervenir que la combinatoire du squelette de $Y$ et la cohomologie
d'affino\"{\i}des avec bonne r\'eduction et celle de polycouronnes.

{\rm (ii)} On peut se demander, en l'absence d'un th\'eor\`eme de r\'eduction semi-stable
g\'en\'eral, quelles sont les pi\`eces minimales dont on a besoin pour reconstruire toute
vari\'et\'e quasi-compacte lisse.  En particulier, quel genre de pi\`eces fournit la
th\'eorie des alt\'erations \`a la Hartl~\cite{Hartl} et Temkin~\cite{Tem2}?
\end{rema}

\subsubsection{Applications \`a la cohomologie}
Le d\'ecoupage pr\'ec\'edent d'une courbe en shorts et jambes permet de ramener
l'\'etude de la cohomologie des courbes \`a celle des shorts et des jambes et celle du graphe $\Gamma$.
De mani\`ere g\'en\'erale, si $H^\bullet$ est une cohomologie \`a coefficients
dans un module $\Lambda$, raisonnable (en particulier, $H^0(Z)=\Lambda^{\pi_0(Z)}$ et, si $Z$
est un cercle fant\^ome, on dispose d'une application
r\'esidu $H^1(Y_{i,j})\to\Lambda$ ayant les propri\'et\'es habituelles),
ce d\'ecoupage fournit une filtration naturelle sur $H^1(Y_S)$
dont les quotients successifs\footnote{La notation $M=\big[ \xymatrix@C=.3cm{A\ar@{-}[r]&B\ar@{-}[r]&C}\big]$
signifie que l'on a une filtration $0=M_0\subset M_1\subset M_2\subset M_3=M$
avec $M_1=A$, $M_2/M_1=B$ et $M_3/M_2=C$.} sont:
$$H^1(Y_S)=\big[ \xymatrix@C=.3cm{H^1(\Gamma,\Lambda)
\ar@{-}[r]&\prod_{i\in I}H^1(Y_i)_0\ar@{-}[r]& H^1_c(\Gamma,\Lambda)^\dual}\big],$$
o\`u:

$\bullet$ $H^1(\Gamma,\Lambda)$ et $H^1_c(\Gamma,\Lambda)$ sont les groupes
de cohomologie et de cohomologie \`a support compact de l'espace topologique $\Gamma$
et $H^1_c(\Gamma,\Lambda)^\dual$ est le $\Lambda$-dual de $H^1_c(\Gamma,\Lambda)$;
on a
(cf.~\S\,\ref{TTT16} pour la d\'efinition des fl\`eches correspondantes):
$$H^1(\Gamma,\Lambda)={\rm Coker}(\Lambda^{S}\to\Lambda^{A_c}) 
\quad{\rm et}\quad
H^1_c(\Gamma,\Lambda)^\dual={\rm Ker}(\Lambda^A\to\Lambda^S).$$


$\bullet$ $H^1(Y_i)_0$ est l'ensemble des classes dont tous les r\'esidus 
en les cercles fant\^omes \`a la fronti\`ere de $Y_i$ sont nuls.

$\bullet$ La fl\`eche $H^1(Y_S)\to H^1_c(\Gamma,\Lambda)^\dual$ est celle
envoyant une classe sur la collection de ses r\'esidus.
\begin{rema}
(i) Le sous-groupe $H^1(\Gamma,\Lambda)$ appara\^{\i}t
comme le conoyau de $\prod_{i\in I}H^0(Y_i)\to
\prod_{(i,j)\in I_{2,c}}H^0(Y_{i,j})$; voir le (ii) de la
rem.\,\ref{change} pour le lien entre cette \'ecriture
et celle ci-dessus.

(ii)
On dispose d'un op\'erateur $N_\mu:H^1_c(\Gamma,\Q)^\dual\to H^1(\Gamma,\Q)$
(de monodromie)
qui fait intervenir les longueurs $\mu(a)$ des ar\^etes.
Cela munit $H^1(Y)$, si $\Lambda$ est un $\Q$-module,
d'un op\'erateur de monodromie $N$ (cf.~rem.~\ref{monod}).
\end{rema}

Les cohomologies de de Rham, de Hyodo-Kato, ou pro\'etale $\ell$-adique (pour tout~$\ell$)
sont raisonnables, ce qui conduit au th.\,\ref{expo2} ci-dessous.

On note $\partial Y\subset S$ {\it le bord analytique de $Y$},
i.e.~l'ensemble des $s\in S$ tels que $ Y^{\rm sp}_s$
contienne des points avec $\mu(a)=0^+$ (i.e.~les $s$ tels que la composante
irr\'eductible correspondante de la fibre sp\'eciale classique ne soit pas propre).
On pose
$S_{\rm int}=S\moins \partial Y$, et on note $\Gamma_{\rm int}$ le sous-graphe de $\Gamma$ obtenu
en supprimant les sommets de $\partial Y$ et les ar\^etes ayant une extr\'emit\'e
dans $\partial Y$.
\begin{theo}\label{expo2}
{\rm (i)}
Si $\ell\neq p$, alors $H^1_{\eet}(Y,\Q_\ell(1))$ admet une filtration naturelle
dont les quotients successifs sont
$$H^1_{\eet}(Y,\Q_\ell(1))=\big[
\xymatrix@C=.3cm{
H^1(\Gamma,\Q_\ell(1))\ar@{-}[r]&\prod\limits_{s\in S}H^1_{\eet}( Y^{\rm sp}_s,\Q_\ell(1))
\ar@{-}[r]& H_c^1(\Gamma,\Q_\ell)^\dual}\big].$$

{\rm (ii)} 
$H^1_{\rm HK}(Y)^{\rm sep}$ admet une filtration naturelle
dont les quotients successifs sont
$$H^1_{\rm HK}(Y)^{\rm sep}=\big[
\xymatrix@C=.3cm{
H^1(\Gamma_{\rm int},\breve C)\ar@{-}[r]&\prod\limits_{s\in S_{\rm int}}\hskip-.1cm H^1_{\rm rig}( Y^{\rm sp}_s)
\oplus\prod\limits_{s\in \partial Y}\hskip-.1cm H^1_{\rm rig}( Y^{\rm sp}_s)^{[1]}
\ar@{-}[r]& H_c^1(\Gamma,\breve C)^\dual(-1)}\big],$$
o\`u le $[1]$ en exposant d\'esigne le sous-espace de pente~$1$ pour l'action de $\varphi$
et le twist~$(-1)$ signifie que l'on multiplie l'action naturelle de $\varphi$ par $p$.
\end{theo}
\begin{rema}\label{expo3}

{\rm (i)} 
Les termes des filtrations ci-dessus ne d\'ependent pas de $S$ car
$H^1_{\eet}(\piqp,\Q_\ell(1))=0$ et $H^1_{\rm rig}(\piqp)=0$.  Cela permet,
en passant \`a la limite sur tous les choix possibles, d'en d\'eduire
que tout est fonctoriel.

{\rm (ii)}
Pour $\Q_\ell(1)$, le seul bonus de ce r\'esultat par rapport
\`a ce qui est dit ci-dessus est l'identification (classique)
de $H^1_{\eet}(Y_s,\Q_\ell(1))_0$ avec $H^1_{\eet}( Y^{\rm sp}_s,\Q_\ell(1))$.

{\rm (iii)} Pour $H^1_{\rm HK}$, il faut se fatiguer un peu plus pour arriver au r\'esultat:
l'identification (prop.\,\ref{basic9.11})
de $H^1_{\rm HK}(Y_s)^{\rm sep}_0$ et $H^1_{\rm rig}( Y^{\rm sp}_s)^{[1]}$
utilise la th\'eorie de Cartier.
Le passage de $H^1(\Gamma,\breve C)$ \`a $H^1(\Gamma_{\rm int},\breve C)$
d\'ecoule du calcul (lemme~\ref{basic18}) de l'intersection de $H^1(\Gamma,C)$ et de l'adh\'erence de $0$
dans $H^1_{\rm dR}(Y)$.

{\rm (iv)} Si $Y$ est propre, on a $\partial Y=\emptyset$, et la filtration du (ii)
est \`a la base de la construction de Coleman et Iovita~\cite{CI}.

{\rm (v)}
Le choix de $r\mapsto p^r$ fournit une section de la projection modulo $H^1(\Gamma,\Q_\ell(1))$.
La formule {\og de Picard-Lefschetz\fg} (rem.\,\ref{PL}), qui fait intervenir l'op\'erateur de monodromie~$N$, d\'ecrit
ce qui se passe quand on change $r\mapsto p^r$ .
\end{rema}

En passant \`a la limite, on obtient le r\'esultat suivant:
\begin{theo}\label{expo4}
Soit $Y$ un affino\"{i}de surconvergent ou une courbe sans bord, et soit $S$
une triangulation de $Y$, assez fine.

{\rm (i)}
Si $\ell\neq p$, alors $H^1_{\proet}(Y,\Q_\ell(1))$ admet une filtration naturelle
dont les quotients successifs sont
$$H^1_{\proet}(Y,\Q_\ell(1))=\big[
\xymatrix@C=.3cm{
H^1(\Gamma,\Q_\ell(1))\ar@{-}[r]&\prod_{s\in S}H^1_{\eet}( Y^{\rm sp}_s,\Q_\ell(1))
\ar@{-}[r]& H_c^1(\Gamma,\Q_\ell)^\dual}\big].$$

{\rm (ii)} 
$H^1_{\rm HK}(Y)$ admet une filtration naturelle
dont les quotients successifs sont
$$H^1_{\rm HK}(Y)=\big[
\xymatrix@C=.3cm{
H^1(\Gamma,\breve C)\ar@{-}[r]&\prod_{s\in S}H^1_{\rm rig}( Y^{\rm sp}_s)
\ar@{-}[r]& H_c^1(\Gamma,\breve C)^\dual(-1)}\big].$$
\end{theo}
\begin{rema}\label{intro10.1}
{\rm (i)} 
Pour la m\^eme raison que ci-dessus,
les termes des filtrations ne d\'ependent pas de $S$.

{\rm (ii)} 
Les actions de ${\rm Aut}(Y)$ sur $H^1_{\proet}(Y,\Q_{\ell}(1))$
et $H^1_{\rm HK}(Y)$ fournissent des repr\'esentations
 isomorphes (autant que faire se peut, i.e. apr\`es avoir
choisi un plongement de $\Q_\ell$ dans $C$...).  Si $Y$ est d\'efini sur
une extension finie $K$ de $\Q_p$, les actions du groupe de Weil-Deligne ${\rm WD}_K$ 
de $K$ sont aussi isomorphes.
Voir~\cite{O} pour une autre approche dans le cas propre.

{\rm (iii)} En utilisant les r\'esultats connus~\cite{Carayol2,Faltings2tours,Weinstein} 
sur la cohomologie \'etale $\ell$-adique
de la tour de Drinfeld (en dimension~$1$), le (ii) fournit une preuve du th.\,0.4
de~\cite{CDN} qui donne une description de
l'action de $G\times \check G\times {\rm WD}_F$ sur la cohomologie
de Hyodo-Kato de la tour (dans~\cite{CDN}, cette description est obtenue
en utilisant l'uniformisation de courbes de Shimura, la compatibilit\'e
local-global
et l'isomorphisme entre les tours de Lubin-Tate et de Drinfeld).
Coupl\'e avec le th.\,\ref{intro4}, cela
permet de retrouver le th.\,0.8 de~\cite{CDN}
dont la preuve utilise aussi des m\'ethodes globales.
\end{rema}

\subsubsection{Symboles et int\'egration $p$-adique}\label{TTT5}
Soient $X$ une courbe compacte, $B\hookrightarrow X$ un plongement de la boule
unit\'e ferm\'ee dans $X$ et $Y$ l'affino\"{\i}de de $X$ compl\'ementaire
de l'image de la boule unit\'e ouverte.
On a alors $\O(Y)^\dual/C^\dual=\O(Y)^{\dual\dual}/(1+{\goth m}_C)$,
et $H^1_{\rm HK}(X)=H^1_{\rm HK}(Y)^{\rm sep}$.
(Si on retire $r$ disques, alors $H^1_{\rm HK}(Y)^{\rm sep}/H^1_{\rm HK}(X)\cong \breve C^{r-1}$.)

Soit $J$ la jacobienne de $X$.  
On munit $J(C)$ de la topologie obtenue en voyant $J(C)$ comme l'ensemble
des $C$-points de la vari\'et\'e analytique rigide associ\'ee \`a $J$ ce qui
en fait un groupe de Lie sur $C$; si $J$ est de dimension $g$, son alg\`ebre de Lie est isomorphe \`a $C^g$
et, si $n\in\N$ est suffisamment grand (disons $n\geq n_0$),
l'application exponentielle $\exp_J$ induit un isomorphisme de $(p^n\O_C)^g$ sur un sous-groupe
ouvert $U_n$ de $J(C)$. Les translat\'es des $U_n$, pour $n\geq n_0$,
 forment alors une base de la topologie de $J(C)$.

On d\'eduit de la suite
exacte de Kummer et du calcul~\cite{vdp} de $H^1(Y,\O^\dual)$,
une suite exacte
$$0\to \Q_p\wotimes\O(Y)^\dual\to H^1_{\eet}(Y,\Q_p(1))\to\widehat J\to 0,$$
o\`u $\widehat J$ est {\it le rev\^etement universel du groupe $p$-divisible de $J$},
i.e.~l'ensemble
$$\widehat J=\{(x_n)_{n\in\Z},\ x_n\in J(C),\ p\cdot x_{n+1}=x_n,\ \lim_{n\to -\infty}x_n=0\}.$$
On en d\'eduit,
en comparant la suite exacte ci-dessus et celle du th.~\ref{intro1.1},
le r\'esultat suivant qui est classique
dans le cas de bonne r\'eduction (ou pour les groupes $p$-divisibles~\cite{FF,SW}).

\begin{theo}\label{AF1}
On a un isomorphisme naturel
$$\iota_{\rm st}:\widehat J\overset{\sim}\to (\bst^+\otimes H^1_{\rm HK}(X))^{N=0,\varphi=p}.$$
\end{theo}

\begin{rema}\label{basic3}
Si $X$ est d\'efinie sur une extension finie $K$ de $\Q_p$, et
si $C=\C_p$,
en utilisant des techniques d'int\'egration $p$-adique
sur les courbes~\cite{Cn85,CdS88,Cz-periodes,Cz-inte}, 
on peut donner une description explicite
de cet isomorphisme.
Si $\tJ$ est l'extension universelle de $J$, on
dispose d'une application naturelle $ G_K$-\'equivariante (Lemme~12 ou \S\,B.2 de~\cite{Cz-inte})
$$\iota_{\bdr}:\hJ\to \tJ(\bdr^+)$$
d\'efinie de la mani\`ere suivante: si $x=(x_n)_{n\in\Z}\in\hJ$, on choisit
une suite born\'ee $(\hat x_n)_{n\in\Z}$ de rel\`evements des $x_n$ dans
$\tJ(\bdr^+)$, et on envoie $x$ sur la limite
de $p^n\cdot \hat x_n$, quand $n\to +\infty$, la multiplication par $p^n$
\'etant celle sur $\tJ$.
Alors $$\iota_{\rm st}=\log_{\tJ}\circ\iota_{\bdr},$$ o\`u
$\bst^+\otimes H^1_{\rm HK}(X)$ s'injecte dans $\bdr^+\otimes_K H^1_{\rm dR}(X)$
via $\iota_{\rm HK}$,
$$\log_{\tJ}:\tJ(\bdr^+)\to 
\bdr^+\otimes_K H^1_{\rm dR}(J)^\dual$$
est le logarithme de $\tJ$ \`a valeurs
dans son alg\`ebre de Lie, 
et on a des identifications $H^1_{\rm dR}(J)^\dual\cong H^1_{\rm dR}(X)^\dual\cong H^1_{\rm dR}(X)$,
la seconde
\'etant fournie par le cup-produit dans $H^2_{\rm dR}(X)\cong K$. 
\end{rema}
\subsection{Preuves}
Soient $Y$ une courbe quasi-compacte, $S$ une triangulation assez fine, $Y_S$ le
mod\`ele associ\'e et $(\Gamma,(Y_i)_{i\in I},(\iota_{i,j})_{(i,j)\in I_{2,c}})$
le patron correspondant.

La preuve des th.\,\ref{intro1.1} et~\ref{basic2}
repose sur:

$\bullet$ la d\'efinition d'un groupe de symboles
${\rm Symb}_p(Y)$, :

$\bullet$ la construction de r\'egulateurs \'etale
${\rm Symb}_p(Y)\to H^1_{\eet}(Y,\Q_p(1))$ et syntomique
${\rm Symb}_p(Y)\to H^1_{\rm syn}(Y_S,1)$ qui s'av\`erent \^etre des isomorphismes,

$\bullet$
une description de $H^1_{\rm syn}(Y_S,1)$ en termes du complexe de de Rham et de ses variantes.
\subsubsection{Symboles}
Soit $Y$ une courbe quasi-compacte et soient $C(Y)$ le corps des fonctions rationnelles sur $Y$
et ${\rm Div}(Y)$ le groupe des sommes formelles $\sum_{x\in Y(C)}n_xx$, avec $n_x\in \Z$
pour tout $x\in Y(C)$ (on dit que $\sum_{x\in Y(C)}n_xx$ {\it est \`a support fini} si
$n_x=0$ pour tout $x$ sauf un nombre fini).

Si $\ell$ est un nombre premier, on d\'efinit {\it le groupe de symboles} ${\rm Symb}_\ell(Y)$
comme
$${\rm Symb}_\ell(Y)=
\frac{\{(f_n)_{n\in\N},\ f_n\in C(Y)^\dual,\ {\rm Div}(f_n)\in \ell^n{\rm Div}(Y),
\ f_{n+1}/f_n\in(C(Y)^\dual)^{\ell^n}\}}{\{(h_n^{\ell^n}),\ h_n\in C(Y)^\dual\}}.$$
On a une suite exacte
$$0\to H^1(\Gamma,\Z_\ell(1))\to
{\rm Symb}_\ell(Y)\to {\rm Ker}\big(\prod_{i\in I}{\rm Symb}_\ell(Y^{\rm gen}_i)\to
\prod_{(i,j)\in I_{2,c}}{\rm Symb}_\ell(Y^{\rm gen}_{i,j})\big),$$
o\`u les {\og gen\fg} en exposant signifient {\og fibre g\'en\'erique\fg}
et les groupes ${\rm Symb}_\ell(Y^{\rm gen}_i)$ sont d\'efinis comme ci-dessus en imposant
que le diviseur de $f_n$ soit \`a support fini si $Y_i$ est une jambe (c'est automatique
pour $Y$ ou pour un short, par quasi-compacit\'e).

\subsubsection{R\'egulateur \'etale}
L'application qui, \`a une fonction, associe sa classe de Kummer, fournit 
un r\'egulateur \'etale ${\rm Symb}_\ell(Y)\to H^1_{\eet}(Y,\Z_\ell(1))$ et
la suite exacte de Kummer
$$0\to (\Z/\ell^n)\otimes\O(Y)^\dual\to H^1_{\eet}(Y,(\Z/\ell^n)(1))\to {\rm Pic}(Y)[\ell^n]\to 0$$
permet de montrer (cor.\,\ref{basic12}) que ce r\'egulateur fournit un isomorphisme
$${\rm Symb}_\ell(Y)\overset{\sim}\longrightarrow H^1_{\eet}(Y,\Z_\ell(1)).$$
En utilisant la suite exacte ci-dessus, cela ram\`ene le calcul
de $H^1_{\eet}(Y,\Z_\ell(1))$ \`a celui des groupes
${\rm Symb}_\ell(Y^{\rm gen}_i)$.  Si $\ell\neq p$, ce calcul est trivial si $Y_i$ est
une jambe (lemme~\ref{basic13})
et classique si $Y_i$ est un short (prop.\,\ref{short21}); on en d\'eduit le (i) du th.\,\ref{expo2}
et la formule {\og de Picard-Lefschetz\fg} (rem.\,\ref{PL}, r\'esultats on ne peut plus classiques).

\subsubsection{R\'egulateur syntomique}
Si $\ell=p$, on utilise des m\'ethodes syntomiques pour faire le calcul.
On note ${\rm Syn}(Y_S,1)$  le complexe total associ\'e au complexe double\footnote{Le frobenius $\varphi$ sur les
formes diff\'erentielles est d\'efini de telle sorte que $\varphi\circ d=d\circ\varphi$.}
$$\xymatrix@C=.6cm@R=.5cm{
\prod_{i\in I}F^1\O(\widetilde Y_i)
\ar[r]\ar@<.2cm>[d]^-{1-\frac{\varphi}{p}}&
\prod_{i\in I}\Omega^1(\widetilde Y_i)
\oplus\prod_{(i,j)\in I_{2,c}}F^1\O(\widetilde Y_{i,j})\ar[r]
\ar@<-1cm>[d]^-{1-\frac{\varphi}{p}}\ar@<1.6cm>[d]^-{1-\frac{\varphi}{p}}&
\prod_{(i,j)\in I_{2,c}}\Omega^1(\widetilde Y_{i,j})_{d=0}\ar@<.3cm>[d]^-{1-\frac{\varphi}{p}}\\
\prod_{i\in I}\O(\widetilde Y_i)
\ar[r]& \prod_{i\in I}\Omega^1(\widetilde Y_i)
\oplus\prod_{(i,j)\in I_{2,c}}\O(\widetilde Y_{i,j})\ar[r]&
\prod_{(i,j)\in I_{2,c}}\Omega^1(\widetilde Y_{i,j})_{d=0}}$$
dans lequel, si $Z=Y_i,Y_{i,j}$, $\O(\widetilde Z)$ est une $\acris$-alg\`ebre
munie d'un frobenius $\varphi$ et d'une surjection $\theta:\O(\widetilde Z)\to \O(Z)$
dont la restriction \`a $\acris$ est l'application de Fontaine
$\theta:\acris\to\O_C$, $F^1\O(\widetilde Z)={\rm Ker}\,\theta$,
et $\Omega^1(\tilde Z)$ est le module $\Omega^1_{\O(\widetilde Z)/\acris}$.
Il y a des choix naturels pour les $\O(\widetilde Z)$ qui simplifient grandement
les calculs\footnote{L'anneau $\O(\widetilde Y_{i,j})$ ci-dessus correspond \`a
l'anneau $\O(\widetilde Y_{i,j})^{\rm PD}$ du \no\ref{BAS11}.}: 

$\bullet$ un short
$Y_s$ est obtenu par extension des scalaires
\`a partir d'un sch\'ema formel $\breve Y_s$ lisse sur $\O_{\breve C}$, et on pose
$\O(\widetilde Y_s)=\acris\wotimes_{\O_{\breve C}}\O(\breve Y_s)$ (et on choisit un $\varphi$
sur $\O(\breve Y_s)$); 

$\bullet$ une jambe
$Y_a$ de longueur $r$ v\'erifie $\O(Y_a)=\O_C[[T_1,T_2]]/(T_1T_2-p^r)$, et on pose
$\O(\widetilde Y_a)=\acris[[T_1,T_2]]/(T_1T_2-\tilde p^r)$, o\`u $\tilde p^r\in\acris$
est un teichm\"uller v\'erifiant $\theta(\tilde p^r)=p^r$ 
(et on prend $\varphi$ d\'efini par $\varphi(T_i)=T_i^p$, si $i=1,2$). 

Le complexe ${\rm Syn}(Y_S,1)$ est le complexe double associ\'e au c\^one
$$[\xymatrix{F^1C_{\rm dR}(\widetilde Y_S)\ar[r]^-{1-\varphi/p}&C_{\rm dR}(\widetilde Y_S)}],$$
o\`u 
$C_{\rm dR}(\widetilde Y_S)$ est le complexe
$$\xymatrix@C=.6cm{\prod_{i\in I}\O(\widetilde Y_i)
\ar[r]& \prod_{i\in I}\Omega^1(\widetilde Y_i)
\oplus\prod_{(i,j)\in I_{2,c}}\O(\widetilde Y_{i,j})\ar[r]&
\prod_{(i,j)\in I_{2,c}}\Omega^1(\widetilde Y_{i,j})_{d=0}}$$
qui calcule la cohomologie cristalline logarithmique absolue de $Y_S$.

On note $H^i_{\rm Syn}(Y_S,1)$ les groupes de cohomologie
du complexe ${\rm Syn}(Y_S,1)$.

\vskip.1cm
On d\'efinit un {\it r\'egulateur syntomique} ${\rm Symb}_p(Y)\to H^1_{\rm syn}(Y_S,1)$
en envoyant $(f_n)_{n\in\N}$ sur le cocycle
$\lim_{n\to\infty}\big(\frac{d\tilde f_{n,i}}{\tilde f_{n,i}},\frac{1}{p}\log\frac{\varphi(\tilde f_{n,i})}
{\tilde f_{n,i}^p},\log(\frac{\tilde f_{n,i}}{\tilde f_{n,j}})\big)$,
o\`u $\tilde f_{n,i}\in\O(\widetilde Y_i)$ est un rel\`evement de la restriction
de $f_n$ \`a $Y_i$ (il faut prendre un peu de pr\'ecautions en choisissant
ces rel\`evements car $f_n$ peut avoir des z\'eros et des p\^oles, mais la limite
est holomorphe car les multiplicit\'es de ces z\'eros et p\^oles sont divisibles par $p^n$).
On prouve alors (prop.\,\ref{basic35}) que ce r\'egulateur syntomique induit un isomorphisme
$${\rm Symb}_p(Y)\overset{\sim}{\longrightarrow} H^1_{\rm syn}(Y_S,1).$$
On commence par prouver le r\'esultat pour les $Y_i$ (le cas des shorts (th.\,\ref{short11}) est
nettement plus d\'elicat que celui des jambes (prop.\,\ref{basic34})), et on recolle.

\subsubsection{Cohomologie syntomique et cohomologie de Hyodo-Kato}
Les groupes de cohomologie
de $C_{\rm dR}(Y)=C\otimes_{\acris}C_{\rm dR}(\widetilde Y_S)$
sont les $H^i_{\rm dR}(Y)$,
et les $H^i_{\rm HK}(Y)$ sont, par d\'efinition,
les groupes de cohomologie
de $C_{\rm dR}(\breve Y_S)=\breve C\otimes_{\acris}C_{\rm dR}(\widetilde Y_S)$
(les extensions de scalaires se font via $\theta:\acris\to\O_C$ et $\theta_0:\acris\to\O_{\breve C}$).

Par ailleurs, on peut (lemme~\ref{BAS12.1})
modifier l\'eg\`erement $C_{\rm dR}(\widetilde Y_S)$ pour obtenir
un complexe $\overline C_{\rm dR}(\widetilde Y_S)$ quasi-isomorphe,
de telle sorte que l'inclusion de $\O_{\breve C}$ dans $\acris$
induise un morphisme de complexes 
$\overline C_{\rm dR}(\breve Y_S)\to \Q_p\otimes\overline C_{\rm dR}(\widetilde Y_S)$,
commutant \`a $\varphi$, et section de 
$\Q_p\otimes\overline C_{\rm dR}(\widetilde Y_S)
\to \overline C_{\rm dR}(\breve Y_S)$, ce qui fournit (quasiment gratuitement) des isomorphismes
$$\Q_p\otimes H^1_{\rm dR}(\widetilde Y_S)\cong \bcris^+\wotimes_{\breve C}H^1_{\rm HK}(Y)
\quad{\rm et}\quad
\iota_{\rm HK}:H^1_{\rm dR}(Y)\cong C\wotimes_{\breve C}H^1_{\rm HK}(Y).$$

En jouant avec les diff\'erentes pr\'esentations possibles du c\^one ci-dessus,
on obtient une suite exacte
$$0\to\O(Y)/C\to \Q_p\otimes H^1_{\rm syn}(Y_S,1)\to \Q_p\otimes H^1_{\rm dR}(\widetilde Y_S)^{\varphi=p}\to H^1(Y,\O)\to
\Q_p\otimes H^2_{\rm syn}(Y_S,1)\to\cdot$$
Si $Y$ est compacte, $\O(Y)=C$ et la fl\`eche
$\Q_p\otimes H^1(\widetilde Y_S)^{\varphi=p}\to H^1(Y,\O)$ est surjective, alors
que si $Y$ n'est pas compacte, $H^1(Y,\O)=0$.
Si $Y$ est un affino\"{\i}de, cela fournit donc
une suite exacte
$$0\to \O(Y)/C\to \Q_p\otimes H^1_{\rm syn}(Y_S,1)\to (\bcris^+\wotimes_{\breve C}H^1_{\rm HK}(Y))^{\varphi=p}\to 0.$$
Pour passer de $H^1_{\rm HK}(Y)$ \`a $H^1_{\rm HK}(Y)^{\rm sep}$, il s'agit alors de
comprendre le lien entre $(\Q_p\wotimes\O(Y)^{\dual\dual})/\exp(\O(Y))$ et
l'adh\'erence de $0$ dans $(\bcris^+\wotimes_{\breve C}H^1_{\rm HK}(Y))^{\varphi=p}$.
Localement, ceci entre dans l'isomorphisme
${\rm Symb}_p(Y_i^{\rm gen})\cong H^1_{\rm syn}(Y_i,1)$ (implicitement pour les jambes (prop.\,\ref{basic34})
 et explicitement
pour les shorts (prop.\,\ref{basic29})) et le cas g\'en\'eral s'obtient par recollement (th.\,\ref{basic26}).

\begin{rema}\label{WN}
On dispose de vrais isomorphismes 
$$H^1_{\eet}(Y,\Z_p(1))\overset{\sim}\leftarrow {\rm Symb}_p(Y)\overset{\sim}{\rightarrow}
H^1_{\rm syn}(Y_S,1),$$ ne faisant pas intervenir de d\'enominateurs (au moins si $p\neq 2$).
Les d\'enominateurs apparaissent dans la description de $H^1_{\rm dR}(\widetilde Y_S)$
en termes de $H^1_{\rm HK}(Y)$; ces d\'enominateurs
d\'ependent des longueurs des jambes de $Y_S$ (plus ces jambes sont courtes et plus
les r\'esultats sont impr\'ecis).  Si $Y$ est un affino\"{\i}de, on a une suite
$p^2$-exacte (th.\,\ref{basic26}):
$$0\to \Z_p\wotimes\O(Y)^{\dual\dual}\to H^1_{\rm syn}(Y_S,1)
\to (H^1_{\rm dR}(\widetilde Y_S)^{\rm sep})^{\varphi=p}\to H^1(Y_S,\O),$$
et $H^1(Y_S,\O)$ est tu\'e par une puissance de $p$.
\end{rema}

\Subsubsection{Remerciements}
Nous remercions Bhargav Bhatt, Antoine Ducros, Antoine Chambert-Loir et Peter Scholze pour des conversations
(en pr\'esentiel ou \'electroniques) au sujet de cet article, et le rapporteur
pour sa lecture attentive et ses commentaires judicieux.
W.N. a donn\'e un cours sur des parties de cet article \`a l'automne 2016, \`a l'universit\'e
Fudan de Shanghai; elle voudrait remercier Shanwen Wang pour son invitation et les auditeurs pour leurs
commentaires.

\Subsection{Quelques notations et d\'efinitions}\label{Nota0}
\subsubsection{Anneaux de p\'eriodes}\label{Nota1}
On fixe un morphisme $r\mapsto p^r$ de $(\Q,+)$ dans $(C^\dual,\times)$ prenant la valeur
$p$ pour $r=1$.  
Cela fournit aussi un morphisme $r\mapsto \tilde p^r$
du mono\"{\i}de $\Q_+$ dans $W(\O_C^\flat)=\ainf\subset\acris$, avec $\tilde p^r=[(p^r)^\flat]$
et $(p^r)^\flat=(p^r,p^{r/p},\dots)$. 
On note 
$$\theta_0:\acris\to\O_{\breve C}\quad{\rm et}\quad\theta:\acris\to \O_C$$
les morphismes usuels (on a $\theta(\tilde p^r)=p^r$ si $r\geq 0$, et $\theta_0(\tilde p^r)=0$
si $r>0$).

\subsubsection{S\'eries enti\`eres}\label{Nota2}
Si $A$ est un anneau topologique s\'epar\'e et complet pour la topologie $I$-adique,
o\`u $I$ est un id\'eal de $A$ de type fini et contenant $p$, et si $x=(x_1,\dots,x_d)$, on note:

$\bullet$ $A[[x]]$ l'anneau des s\'eries enti\`eres $\sum_{i\in\N^d}a_ix^i$ avec $a_i\in A$, pour tout $i$,

$\bullet$ $A\langle x\rangle$ le sous-anneau de $A[[x]]$ des
$\sum_{i\in\N^d}a_ix^i$ avec $a_i\to 0$ quand $i\to\infty$,

$\bullet$ $A[[x]]^\dagger$ le sous-anneau de $A[[x]]$ des
$\sum_{i\in\N^d}a_ix^i$ tels qu'il existe $r>0$ tel que $a_i\in p^{\lceil r|i|\rceil}A$, pour tout $i$,

$\bullet$ 
$A[[T,T^{-1}\rangle$ le compl\'et\'e de $A[[T]][T^{-1}]$ pour la topologie $I$-adique.

\subsubsection{Boules, couronnes, cercles fant\^omes}\label{Nota3}
On appelle\footnote{On renvoie au \no\ref{adoc6} pour un point de vue adoque sur ces objets.}
 {\it boule unit\'e ouverte} un sch\'ema formel 
de la forme ${\rm Spf}(\O_C[[T]])$.
Une {\it jambe} $Y$
est un sch\'ema formel de la forme ${\rm Spf}(\O_C[[T_1,T_2]]/(T_1T_2-\alpha))$, avec $\alpha\in{\goth m}_C$.
La longueur\footnote{Elle ne d\'epend pas du choix de l'uniformisante $T_1$ gr\^ace au lemme~\ref{basic32}:
l'inclusion de $\Lambda:=\O_C[[T_1,T_2]]/(T_1T_2-\alpha)$ dans $\O_C[[T_i,T_i^{-1}\rangle$ fournit,
si $u\in (\Lambda[\frac{1}{p}])^\dual$, des couples $(v_1(u),v'_1(u))$ et $(v_2(u),v'_2(u))$,
qui sont
reli\'es par la formule
$v_{2}(u)=v_{1}(u)+v'_{1}(u)v_p(\alpha)$ et $v'_2(u)=-v'_1(u)$, si $u\in \O(Y)$.}, 
de $Y$ est $v_p(\alpha)$ et on l'oriente en choisissant $T_1$ comme param\`etre local
(choisir $T_2$ \`a la place renverse l'orientation).
On appelle {\it cercle fant\^ome} un sch\'ema formel de la forme ${\rm Spf}(\O_C[[T,T^{-1}\rangle)$.

La {\it fronti\`ere de la boule unit\'e} ${\rm Spf}(\O_C[[T]])$ est le cercle
fant\^ome ${\rm Spf}(\O_C[[T,T^{-1}\rangle)$.
La {\it fronti\`ere de la couronne ouverte} ${\rm Spf}(\O_C[[T_1,T_2]]/(T_1T_2-\alpha))$
est constitu\'ee des deux cercles fant\^omes $Y_1={\rm Spf}(\O_C[[T_1,T_1^{-1}\rangle)$
et $Y_2={\rm Spf}(\O_C[[T_2,T_2^{-1}\rangle)$. Dans les deux cas, la fronti\`ere est incluse
dedans. 

Si $Y$ est un sch\'ema formel sur $\O_C$, on pose $\O(Y^{\rm gen})=\O(Y)[\frac{1}{p}]$.
On renvoie \`a l'appendice (en particulier au \S\,\ref{appen21})
pour des consid\'erations sur les fibres g\'en\'eriques.

\subsubsection{R\'esidus}\label{Nota4}
Si $Y$ est un cercle fant\^ome ${\rm Spf}(\O_C[[T,T^{-1}\rangle)$,
on note $\Omega^1(Y)$ le module $\O_C[[T,T^{-1}\rangle\frac{dT}{T}$
des formes diff\'erentielles continues.  
Si $\omega=\sum_{k\in\Z}\alpha_kT^k\frac{dT}{T}$, on d\'efinit\footnote{L'ind\'ependance
de ce r\'esidu par rapport au choix de $T$ r\'esulte du lemme~\ref{basic32}: si 
$u\in(\O_C[[T,T^{-1}\rangle)^\dual$, 
on a ${\rm Res}(\frac{du}{u})=v'(u)$
(ind\'ependant du choix de $T$) et le cas g\'en\'eral s'en d\'eduit par lin\'earit\'e et continuit\'e.}
{\it son r\'esidu}
${\rm Res}_Y(\omega)$ (ou simplement ${\rm Res}(\omega)$) par la formule
${\rm Res}_Y(\omega)=\alpha_0$.

Si $Y$ est une couronne ${\rm Spf}(\O_C[[T_1,T_2]]/(T_1T_2-\alpha))$,
on note $\Omega^1(Y)$ le module $\O_C[[T_1,\frac{\alpha}{T_1}]]\frac{dT_1}{T_1}$
des formes diff\'erentielles continues.
Si $\omega=\sum_{k\in\Z}\alpha_kT_1^k\frac{dT_1}{T_1}$, on d\'efinit {\it son r\'esidu}
${\rm Res}_Y(\omega)$ par la formule
${\rm Res}_Y(\omega)=\alpha_0$.
Notons que ${\rm Res}_Y(\omega)$ d\'epend de l'orientation de $Y$: comme $\frac{dT_1}{T_1}+\frac{dT_2}{T_2}=0$,
changer l'orientation change le signe de ${\rm Res}_Y(\omega)$.
Notons aussi que, si $Y_1,Y_2$ sont les cercles fant\^omes \`a la fronti\`ere de $Y$,
on a ${\rm Res}_Y(\omega)={\rm Res}_{Y_1}(\omega)=-{\rm Res}_{Y_2}(\omega)$.

\section{Cohomologie des graphes}\label{TTT10}
Ce court chapitre contient des d\'efinitions de base concernant les graphes.
Comme nous le verrons, on peut associer \`a une courbe analytique
$X$ (resp.~\`a un mod\`ele semi-stable $X_S$ de $X$)
deux graphes, \`a savoir son squelette analytique $\Gamma^{\rm an}(X)$, cf.~\no\ref{TTT22}, et le graphe
dual $\Gamma(X^{\rm sp})$ de sa fibre sp\'eciale, cf.~\S\,\ref{bTTT25}, 
(resp.~$\Gamma(S)$ et $\Gamma(X_S^{\rm sp})=\Gamma^{\rm ad}(S)$, cf.~\no\ref{GRAB13} et \S\,\ref{TTT25}). 
Si $X$ est une courbe sans bord,
ces deux graphes sont \'egaux mais, si $X$ a un bord, le squelette
$\Gamma^{\rm an}(X)$ est naturellement muni d'une m\'etrique, alors que
$\Gamma(X^{\rm sp})$ est seulement muni d'une semi-m\'etrique et son s\'epar\'e est
$\Gamma^{\rm an}(X)$.  La raison pour laquelle nous nous int\'eressons \`a
ces notions est que la cohomologie de $X$ (resp.~$X_S$) a une d\'ecomposition naturelle
faisant intervenir la cohomologie de $\Gamma(X^{\rm sp})$ (resp.~$\Gamma(X_S^{\rm sp})$).  
\Subsection{Graphes}\label{TTT11}
\subsubsection{Sommets et ar\^etes}\label{TTT12}
Un {\it graphe}\footnote{Tous nos graphes sont suppos\'es
{\it localement finis}, i.e.~tout point de $\Gamma$ 
a un voisinage $U$ tel que $U\cap(\Gamma\moins S)$ n'ait
qu'un nombre fini de composantes connexes.} 
$\Gamma$ est un espace topologique s\'epar\'e muni d'un sous-ensemble discret $S$
(les {\it sommets de $\Gamma$})
 tel que 
les composantes connexes de $\Gamma\moins S$
({\it les ar\^etes de $\Gamma$})
soient des courbes r\'eelles sans bord (i.e.~hom\'eomorphes \`a des segments ouverts ou des cercles).

On note $A$ l'ensemble des ar\^etes.  Si $a\in A$, on note 
$S(a)$ l'ensemble des $s\in S$ dans l'adh\'erence de $a$ (les {\it extr\'emit\'es de $a$}) et
$\overline a$ l'adh\'erence
de $a$ dans $\Gamma$ (r\'eunion de $a$ et $S(a)$). 
Alors $S(a)$ contient au plus deux points.
Remarquons que, si $\Gamma$ est connexe, et si $A$ contient une ar\^ete $a$ sans extr\'emit\'e,
alors $\Gamma=a$ (et $S=\emptyset$ et $A=\{a\}$); dans le cas contraire, toutes les ar\^etes
ont au moins une extr\'emit\'e et, en particulier, ce sont toutes des segments.   

On note $A_c$ l'ensemble des ar\^etes $a$ telles que $\overline a$ soit compacte.

\subsubsection{Orientation}\label{TTT13}
On dit que {\it $\Gamma$ est
orient\'e} si toutes les ar\^etes
sont munies d'une orientation.
Dans ce cas, on peut diff\'erencier les extr\'emit\'es d'une ar\^ete $a$, et on note 
$s_1(a)$ {\it l'origine de $a$} et
$s_2(a)$ {\it le bout de $a$} (s'ils existent: si $\Gamma$ n'est pas un cercle sans sommet,
alors $a$ a une origine et un bout si et seulement si $a\in A_c$, notons que rien
n'emp\^eche alors que $s_2(a)=s_1(a)$;
si $\Gamma$ n'est pas un segment sans sommet, alors
 une ar\^ete n'appartenant pas \`a $A_c$ a soit une origine, soit un bout, mais pas les deux
\`a la fois).

Soit $\Gamma$ un graphe orient\'e.
Si $s\in S$, on pose 
\begin{align*}
A(s)^+=\{a\in A,\ s_1(a)=s\},&\quad A(s)^-=\{a\in A,\ s_2(a)=s\},\\
A(s)=A(s)^+\cup A(s)^-,&\quad A_c(s)=A(s)\cap A_c,
\end{align*} 
et donc $A(s)$ (resp.~$A_c(s)$) 
est l'ensemble des ar\^etes (resp.~ar\^etes relativement compactes)
dont une des extr\'emit\'es est $s$.
Le cardinal de $A(s)$ est {\it la valence de $s$}.

\subsubsection{L'ensemble $\rbp$ des longueurs}\label{TTT14}
On note $\rbp$ l'ensemble
constitu\'e des:

$\bullet$ $r$, pour $r\in \R_+^\dual\sqcup\{\infty,2\infty\}$,

$\bullet$ $r^+$ pour $r\in \R_+\sqcup\{\infty\}$,

$\bullet$ $r^{++}$ pour $r\in\R_+$.

\subsubsection{M\'etrique}\label{TTT15}
On dit que $\Gamma$
 {\it est m\'etris\'e} si, pour toute ar\^ete $a$, son adh\'erence $\overline a$ est munie
d'un hom\'eomorphisme sur un intervalle de $\R$ (pas forc\'ement de longueur finie) ou
un homoth\'etique du cercle unit\'e. On note alors $\mu(a)\in \R_+^\dual\sqcup\{\infty,2\infty\}$ 
la {\it longueur} de l'ar\^ete $a$
(i.e.~la longueur de l'intervalle ou du cercle correspondant), en posant $\mu(a)=\infty$
si l'intervalle de $\R$ est une demi-droite et $\mu(a)=2\infty$ si cet intervalle est $\R$
tout entier.  Un graphe m\'etris\'e est naturellement un espace m\'etrique.

\smallskip
On dit que $\Gamma$
 {\it est semi-m\'etris\'e} si: 

$\bullet$ L'adh\'erence $\overline a$ de tout $a\in A_c$ est munie
d'un hom\'eomorphisme sur un intervalle (compact) de $\R$ ou
un homoth\'etique du cercle unit\'e.

 $\bullet$ 
L'adh\'erence $\overline a$ de tout $a\in A\moins A_c$ est munie
d'une application continue $\overline a\to\R$ induisant un hom\'eomorphisme
d'un 
sous-intervalle $\overline a'$ contenant
l'extr\'emit\'e \'eventuelle de $a$ 
sur un sous-intervalle de $\R$, et contractant les {\it bouts} de $a$, i.e. 
les composantes connexes de $\overline a\moins \overline a'$ (au nombre de $0$, $1$ ou $2$) 
en des points.

On d\'efinit la longueur $\mu(a)\in\rbp$, si $a\in A$, par:

$\bullet$ Si $a\in A_c$, alors $\mu(a)\in \R_+^\dual\sqcup\{\infty,2\infty\}$ est la longueur
de l'intervalle ou du cercle correspondant.

$\bullet$ Si $a\in A\moins A_c$, et si $r\in \R_+^\dual\sqcup\{\infty,2\infty\}$
est la longueur du segment correspondant, alors $\mu(a)=r$ (resp.~$\mu(a)=r^+$, resp.~$\mu(a)=r^{++}$)
si $\overline a\moins\overline a'$ a $0$ (resp.~$1$, resp.~$2$)
composantes connexes.
(Le nombre de $+$ est le nombre de bouts.)

\begin{rema}\label{gam1}
Si $\Gamma$ est semi-m\'etris\'e, on peut fabriquer \`a partir de $\Gamma$
un graphe $\overline\Gamma$ m\'etris\'e, appel\'e {\it le s\'epar\'e de $\Gamma$},
en rempla\c{c}ant $\overline a$
par $\overline a'$ si $a\in A\moins A_c$ 
(i.e.~en contractant tous les bouts d'ar\^etes de longueur nulle)
et en rajoutant les extr\'emit\'es de $\overline a'$ aux sommets
(par exemple, si $a$ a une extr\'emit\'e $s$, et si $\overline a'$ est compact,
alors une des extr\'emit\'es de $\overline a'$ est $s$ et l'autre est un nouveau sommet).

Pour retrouver $\Gamma$ \`a partir de $\overline\Gamma$, il faut rajouter
des ar\^etes de longueur nulle et supprimer les sommets dont la valence
est pass\'ee de $1$ \`a $2$.
\end{rema}

Un graphe semi-m\'etris\'e est dit:

\quad $\bullet$ {\it compact} s'il est m\'etris\'e et l'espace m\'etrique associ\'e est compact,

\quad $\bullet$ {\it quasi-compact} si son s\'epar\'e est compact,

\quad $\bullet$ {\it complet} s'il est m\'etris\'e et l'espace m\'etrique associ\'e est complet.

Il est facile de v\'erifier que:

\quad $\bullet$ $\Gamma$ est compact si et seulement si il est quasi-compact et complet.

\quad $\bullet$ $\Gamma$ est compact si et seulement si $S$ est fini et $A=A_c$.

\quad $\bullet$ Un graphe fini est complet si et seulement si les \'el\'ements de $A\moins A_c$ sont
de longueur $\infty$.

\Subsection{Cohomologie}\label{TTT16}
Soient $\Gamma$ un graphe orient\'e ayant un nombre fini de composantes connexes 
et $L$ un groupe ab\'elien.  
\subsubsection{Cohomologie et cohomologie \`a support compact}\label{TTT17}
On dispose des groupes de cohomologie $H^i(\Gamma,L)$, pour $i=0,1$,
et de cohomologie \`a support compact $H^i_c(\Gamma,L)$, pour $i=0,1$.
Alors $\dim_LH^0(\Gamma,L)$ est le nombre de composantes connexes
de $\Gamma$ et $\dim_LH^0_c(\Gamma,L)$ est le nombre de ses composantes connexes compactes.

Les groupes $H^1(\Gamma,L)$ et $H^1_c(\Gamma,L)$ ont une description combinatoire
qui va nous \^etre utile.
Si $X=A_c,A,S$,  on note $L^X$ l'espace des fonctions $\phi:X\to L$
et $L^{(X)}$ le sous-espace de $L^X$ des fonctions \`a support fini.
On dispose d'applications
$$\partial: L^S\to L^{A_c}\quad{\rm et}\quad \partial:L^{(S)}\to L^{(A_c)},$$
d\'efinies par:
$$\partial\phi(a)=\phi(s_2(a))-\phi(s_1(a)).$$
Alors
$$H^1(\Gamma,L)={\rm Coker}\big[\partial: L^S\to L^{A_c}\big]\quad{\rm et}\quad
H^1_c(\Gamma,L)={\rm Coker}\big[\partial:L^{(S)}\to L^{(A)}\big].$$

\begin{rema}\label{sympat3}
{\rm (o)} Si $\Gamma$ est compact, alors $H^1_c(\Gamma,L)=H^1(\Gamma,L)$.

{\rm (i)} Si $A$ ne contient pas d'ar\^ete sans extr\'emit\'e, alors $S$ rencontre toutes
les composantes connexes de $\Gamma$ et
$H^0(\Gamma,L)={\rm Ker}\big[\partial: L^S\to L^{A_c}\big]$, et si $A$ ne contient pas de cercle,
alors $S$ rencontre
toutes les composantes connexes compactes, et
$H^0_c(\Gamma,L)={\rm Ker}\big[\partial:L^{(S)}\to L^{(A)}\big]$.

{\rm (ii)} Si $S$ est fini, la restriction \`a $A_c$ d'une fonction sur $A$ induit
une suite exacte:
$$0\to H^0_c(\Gamma,L)\to H^0(\Gamma,L)\to L^{A\moins A_c}\to H^1_c(\Gamma,L)\to H^1(\Gamma,L)\to 0.$$
En particulier, si $\Gamma$ est connexe et non compact (et $S$ fini), on a une suite exacte:
$$0\to L\to L^{A\moins A_c}\to H^1_c(\Gamma,L)\to H^1(\Gamma,L)\to 0.$$
\end{rema}

\subsubsection{L'op\'erateur $\dbar$}\label{TTT18}
Comme $\Gamma$ est suppos\'e localement fini,
on dispose d'applications
$$\dbar:L^{A}\to L^{S}\quad{\rm et}\quad \dbar:L^{(A)}\to L^{(S)},$$
d\'efinies par 
$$ \dbar\phi(s)=\sum_{a\in A(s)^+}\phi(a)-\sum_{a\in A(s)^-}\phi(a).$$

Les espaces $L^S$ et $L^A$ sont les duaux de $L^{(S)}$ et $L^{(A)}$,
et $\dbar$ est l'adjoint de $\partial$.  Il s'ensuit que:
$${\rm Ker}\big[\dbar:L^{A}\to L^{S}\big]=H^1_c(\Gamma,L)^\dual.$$

\begin{rema}\label{lapl1}
Si $s\in S$, notons $\dbar_s:L^{A(s)}\to L^{\{s\}}$ la restriction de $\dbar$.
Comme $L^A={\rm Ker}\big[\prod_s L^{A(s)}\to L^{A_c}\big]$ (car $a\in A_c$ appara\^{\i}t
dans exactement deux $A(s)$),
on a aussi
$$H^1_c(\Gamma,L)^\dual=
{\rm Ker}\big[\prod_s{\rm Ker}\big[\dbar_s:L^{A(s)}\to L^{\{s\}}\big]\to L^{A_c}\big].$$
\end{rema}

\subsubsection{Graphes bipartites marqu\'es}
A partir d'un graphe semi-m\'etri\'se $\Gamma=(S,A,\mu)$, on peut fabriquer
un nouveau graphe $\Gamma^{(2)}=(I,I_2,\mu)$ en rajoutant
un sommet sur chaque ar\^ete relativement compacte (ce qui d\'ecoupe l'ar\^ete en deux),
et en donnant au nouveau point une multiplicit\'e \'egale \`a $\mu(a)$;
la structure obtenue est appel\'ee {\it un graphe bipartite marqu\'e}
(bipartite car les sommets sont de deux types possibles: des sommets standard
et des sommets de valence $2$ munis d'une multiplicit\'e).
Les ensembles $I$ des sommets, $I_2$ des ar\^etes et $I_{2,c}$ des
ar\^etes relativement compactes de $\Gamma^{(2)}$ sont donn\'es par 
$$I=A_c\sqcup S, 
\quad I_2=\{(a,s),\ a\in A,\ s\in S(a)\}
\quad I_{2,c}=\{(a,s),\ a\in A_c,\ s\in S(a)\}.$$
On a $I_{2}=I_{2,c}\sqcup (A\moins A_c)$.

\begin{rema}\label{change}
{\rm (o)} La construction ci-dessus peut s'inverser et permet de construire
un graphe semi-m\'etris\'e \`a partir d'un graphe bipartite marqu\'e.

{\rm (i)} 
Le graphe $\Gamma^{(2)}$ poss\`ede une orientation naturelle: toute ar\^ete de $\Gamma^{(2)}$
a exactement une extr\'emit\'e appartenant \`a $S$, et on prend cette extr\'emit\'e
comme origine.

{\rm (ii)}
Le graphe $\Gamma^{(2)}$ est, topologiquement, hom\'eomorphe \`a $\Gamma$,
et donc \`a m\^eme cohomologie.  Il s'ensuit que
les complexes naturels
$K^{I}\to K^{I_{2,c}}$ et $K^{I_2}\to K^{I}$
calculent aussi la cohomologie de $\Gamma$.
\end{rema}

\subsubsection{Monodromie}\label{TTT19}
Si $\Gamma$ est m\'etris\'e ou semi-m\'etris\'e, 
et {\it si $\mu(a)\in\Q$
pour tout $a\in A_c$},
on d\'efinit {\it un op\'erateur de monodromie}, 
$$N_\mu:H^1_c(\Gamma,L)^\dual\to H^1(\Gamma,L)$$ en prenant la compos\'ee
de $\mu:K^A\to K^{A_c}$, qui envoie $\phi$ sur la restriction \`a $A_c$
de $(a\mapsto\mu(a)\phi(a))$,
et de l'application naturelle $K^{A_c}\to H^1(\Gamma,L)$.

\begin{rema}\label{sympat4}
Si $\Gamma$ est compact
et m\'etris\'e, 
alors $N_\mu$
est un isomorphisme.
\end{rema}

\begin{rema}\label{monod}
Soit $M$ un $\Lambda$-module 
admettant $H^1_c(\Gamma,\Lambda)^\dual$ comme quotient et $H^1(\Gamma',\Lambda)$
comme sous-objet, o\`u $\Gamma'$ est un sous-graphe de $\Gamma$.
Alors $M$ est naturellement muni d'un op\'erateur de monodromie $N$ d\'efini
comme la compos\'ee:
$$M\to H^1_c(\Gamma,\Lambda)^\dual\overset{N_\mu}{\longrightarrow}
H^1(\Gamma,\Lambda)\to H^1(\Gamma',\Lambda)\to M,$$
o\`u les fl\`eches autres que $N_\mu$ sont les fl\`eches \'evidentes.
Si l'image de $H^1(\Gamma',\Lambda)$ dans $H^1_c(\Gamma,\Lambda)^\dual$ est nulle,
alors $N$ est nilpotent d'exposant~$2$ (i.e.~$N\circ N=0$).
\end{rema}

\section{Courbes analytiques}\label{TTT20}

Soit $C$ un corps alg\'ebriquement clos, complet pour une valuation $v_p$
(\`a valeurs r\'eelles) v\'erifiant $v_p(p)=1$.
Le but de ce chapitre est de rappeler rapidement les principaux r\'esultats concernant 
la structure des courbes analytiques sur $C$. Nous renvoyons le lecteur \`a \cite{Bal}, \cite{BPR}, \cite{Duc}
pour plus de d\'etails et les preuves des r\'esultats \'enonc\'es ci-dessous. 
   
Une\, {\it courbe analytique} (ou simplement {\it courbe}) est 
un espace rigide analytique s\'epar\'e, purement de dimension $1$ et lisse (sur le corps de base). 
L'espace de Berkovich associ\'e \`a une courbe sur $C$ est une courbe $C$-analytique quasi-lisse au sens de \cite{Duc} (une telle courbe peut donc avoir un bord non vide).
Pour simplifier, nous supposerons sans mention explicite du contraire que les courbes sont connexes. Si $K$ est un sous-corps ferm\'e de $C$, 
une telle courbe $Y$ est dite {\it d\'efinie sur $K$} si elle est obtenue, par extension
des scalaires, \`a partir d'une $K$-courbe $Y_K$.

   \Subsection{Structure globale d'une courbe}\label{Global} 
Une $C$-courbe analytique est soit propre (auquel cas c'est l'espace analytique
associ\'e \`a une courbe projective), soit une r\'eunion croissante d'affino\"ides
de dimension~$1$ et m\^eme une r\'eunion croissante stricte (chaque affino\"{\i}de
est inclus dans l'int\'erieur du suivant) si la courbe est sans bord (et non propre). 

Plus g\'en\'eralement et plus pr\'ecis\'ement, 
soit $K$ un sous-corps ferm\'e de $C$ et soit $X$ un $K$-espace analytique s\'epar\'e, de dimension $1$
(non n\'ecessairement lisse).  
     
$\bullet$ Si $X$ est quasi-compact et irr\'eductible, alors $X$ est  
affino\"ide ou projectif, i.e. l'analytifi\'e d'une courbe alg\'ebrique projective sur $K$ \cite[th.2]{FM}. Du point de vue de la th\'eorie de Berkovich, toute 
courbe analytique compacte (en tant qu'espace topologique) et irr\'eductible $X$ est isomorphe \`a un domaine analytique
d'une courbe projective, et $X$ est affino\"ide si et seulement si 
le bord de $X$ est non vide \cite[th. 6.1.3, cor. 6.1.4]{Duc}. 

$\bullet$ Si $X$ est quasi-compact et purement de dimension $1$, on peut d\'ecrire les affino\"ides de $X$ comme suit. Si $f$ est une fonction m\'eromorphe globale sur $X$, notons $U_f$ l'ensemble des $x\in X$ pour lesquels 
$f$ est analytique au voisinage de $x$ et v\'erifie $|f(x)|\leq 1$. Tout 
affino\"ide $U$ de $X$ est de la forme $U_f$ \cite[th.3]{FM} et, 
si $X=Y^{\rm an}$ pour une courbe alg\'ebrique projective $Y$, $f$ peut \^etre choisie  
rationnelle sur $Y$. Si de plus $Y$ est connexe et r\'eguli\`ere\footnote{Sans cette hypoth\`ese de r\'egularit\'e le r\'esultat qui suit tombe en d\'efaut.} et si 
$f$ est une fonction rationnelle, mais pas r\'eguli\`ere sur $Y$ tout entier, alors 
$U_f$ est un affino\"ide \cite[prop 2, p.168]{FM}. 
     
Supposons de plus que $X$ est connexe. Alors $X$ est {\it paracompact} (i.e. poss\`ede un recouvrement affino\"ide localement fini), en particulier $X$ 
est une r\'eunion d\'enombrable croissante d'affino\"ides, cf. \cite{Liu-vdp} et \cite[th. 4.5.10]{Duc}. Plus pr\'ecis\'ement, le r\'esultat principal de \cite{Liu-vdp} montre que $X$ poss\`ede un mod\`ele formel $\mathfrak{X}$ plat, s\'epar\'e et localement de type fini sur $\O_K$, et $\mathfrak{X}$ poss\`ede un recouvrement localement fini par des ouverts affines. De plus:
     
$\bullet$ 
Si $X$ est lisse, on peut imposer \`a $\mathfrak{X}$ d'\^etre semi-stable, i.e. les singularit\'es de la fibre sp\'eciale sont nodales. 
     
$\bullet$ Si $X$ est lisse, rappelons que $X$ est dit quasi-Stein s'il poss\`ede un recouvrement affinoide admissible croissant $(X_n)_{n\geq 1}$ tel que les fl\`eches de restriction $\O_X(X_{n+1})\to \O_X(X_n)$ aient une image dense pour tout $n$. Si $X_n$ est relativement compact dans $X_{n+1} $ pour tout $n$, l'espace est dit Stein. 
De plus, si 
$X$ est Stein et lisse les composantes irr\'eductibles de la fibre sp\'eciale 
de $\mathfrak{X}$ sont propres.

Enfin, tout affino\"ide connexe, de dimension $1$ et lisse sur $C$ est un domaine 
affino\"ide d'une courbe propre et lisse, 
et s'obtient en retirant une famille finie de disques ouverts, deux 
\`a deux disjoints, de cette courbe \cite{vdp} (voir aussi la prop. \ref{Pconstr9}).

\Subsection{Affino\"{\i}des}\label{constr1}
\subsubsection{Alg\`ebres de Tate}\label{constr2}
Soient $K$ un sous-corps ferm\'e de $C$, d'anneau des entiers $\O_K$,
d'id\'eal maximal ${\goth m}_K$ et de corps r\'esiduel $k_K$. Si $Y$ est un affino\"{\i}de sur $K$, on note: 

$\bullet$ $\O^+(Y)$ le sous-anneau de $\O(Y)$ des fonctions \`a valeurs enti\`eres, 
i.e. de norme spectrale $\leq 1$.

$\bullet$ $\O^{++}(Y)$ l'id\'eal ${\goth m}_C\otimes_{\O_C}\O^+(Y)$ de $\O^+(Y)$, 
i.e. l'id\'eal des fonctions de norme spectrale $<1$.

$\bullet$  
 $\O(Y)^{\dual\dual}$ le sous-groupe $1+\O^{++}(Y)$ de $\O(Y)^\dual$.

Si $Y$ est un affino\"{\i}de r\'eduit sur $K$ et si $K$ est de valuation discr\`ete ou bien alg\'ebriquement clos, alors
$\O^+(Y)$ est une {\it alg\`ebre de Tate} sur $\O_K$, i.e. une alg\`ebre de la forme
$\O_K\langle x_1,\dots,x_n\rangle/I$, o\`u $I$ est un id\'eal de type fini, $\O_K$-satur\'e. Cela r\'esulte du 
th\'eor\`eme de finitude de Grauert-Remmert-Gruson \cite[th. 3.1.17]{Lutk} et du th\'eor\`eme de 
Gruson-Raynaud \cite[th. 3.2.1]{Lutk}.

\begin{lemm}\label{paire10}
Si $f\in\O(Y)^{\dual\dual}$, alors $\log f\in\O(Y)$.
\end{lemm}
\begin{proof}
On peut \'ecrire $f=1+g$ avec $g\in\O^{++}(Y)$,
et alors $\log f=
\sum_{n\geq 1}\frac{(-1)^{n-1}}{n}g^n$, et la s\'erie
converge dans $\O(Y)$.
\end{proof}

\subsubsection{Bord, fronti\`ere, cercles fant\^omes et r\'esidus}\label{constr4}
Soit $Y$ un affino\"{\i}de sur $C$. La $k_C$-alg\`ebre $\O({\cal Y}):=\O^+(Y)/\O^{++}(Y)=\O^+(Y)\otimes_{\O_C} k_C$
est de type fini sur $k_C$ (cela r\'esulte du lemme de normalisation de Noether et de \cite[th. 3.1.17]{Lutk}). On appelle ${\cal Y}:={\rm Spec}(\O({\cal Y}))$ la {\it r\'eduction canonique de $Y$}. L'application de r\'eduction $r: Y\to {\cal Y}$ est surjective et anti-continue\footnote{On consid\`ere ici $Y$ comme espace de Berkovich, pas comme 
espace rigide analytique.} et la fibre de $r$ en tout point maximal de ${\cal Y}$ est un singleton. L'image inverse de l'ensemble des points maximaux (i.e. points g\'en\'eriques des composantes irr\'eductibles) de ${\cal Y}$ est le {\it bord de Shilov de $Y$}, i.e. le plus petit ferm\'e de $Y$ sur lequel toute fonction 
 $|f|$ ($f\in \O(Y)$) atteint son maximum. 
Si $Y$ est r\'eduit et si $U$ est un ouvert affine de 
${\cal Y}$, le tube $]U[=r^{-1}(U)$ de $U$ est un affino\"{\i}de
de $Y$ et sa r\'eduction canonique s'identifie \`a $U$ \cite[lemma 4.8.1]{Fvdp}. 

Soit $Y$ un affino\"{\i}de lisse, de dimension~$1$ sur $C$. 
Alors ${\cal Y}$ est une courbe
affine sur $k_C$ (en g\'en\'eral ni lisse, ni irr\'eductible)
sans multiplicit\'es.  
{\it Le bord $\partial Y$} de $Y$ (vu comme
espace de Berkovich) co\"{\i}ncide avec le bord de Shilov de $Y$ et est en bijection avec l'ensemble
des composantes irr\'eductibles de ${\cal Y}$: si $s\in\partial Y$, la composante
irr\'eductible ${\cal Y}_s$ qui lui correspond est affine et d\'efinit une valuation (de Gauss) $v_s$
sur $\O(Y)$, et donc un point de type 2 de l'espace de Berkovich associ\'e \`a $Y$.

\smallskip
Si $s\in\partial Y$, on associe \`a $s$ un ensemble $A_0(s)$ (qui sera un ensemble d'ar\^etes
de sommet $s$ quand on aura d\'efini le graphe dual de la fibre sp\'eciale de $Y$).
Cet ensemble est en bijection avec les points de $\overline{\cal Y}_s\moins {\cal Y}_s$,
o\`u $\overline{\cal Y}_s$ est la compactifi\'ee de la courbe ${\cal Y}_s$ par des points lisses.
Si $a\in A_0(s)$, on lui associe:

$\bullet$ un point $P_a$ de $\overline{\cal Y}_s\moins {\cal Y}_s$,

$\bullet$ 
une valuation $v_{s,a}$ de rang~$2$ sur $\O(Y)$ par $v_{s,a}(f)=(v_s(f),
v_{P_a}(\alpha^{-1}f))$, o\`u $\alpha\in C$ v\'erifie $v_p(\alpha)=v_s(f)$.

\vskip.1cm
Si $\omega\in\Omega^1(Y)$, et si $a\in A_0(s)$, on d\'efinit
{\it le r\'esidu} ${\rm Res}_a(\omega)$ comme le r\'esidu de la restriction
de $\omega$ \`a $Y_{a,s}$ (i.e. le coefficient $\alpha_0$ dans le d\'eveloppement
$\omega_{|Y_{a,s}}=\sum_{k\in\Z}\alpha_kT^k\frac{dT}{T}$).
On a $\sum_{s\in\partial Y}\sum_{a\in A_0(s)}{\rm Res}_a(\omega)=0$.

\begin{rema}\label{constr5}
{\rm (i)} 
Comme $P_a\notin {\cal Y}_s$, si $a\in A_0(s)$, la valuation $v_{s,a}$ n'est pas totalement positive
sur $\O^+(Y)$ et $v_{s,a}$ est un point (de type 5) du bord 
$$\partial Y^{\rm ad}={\rm Spa}\big(\O(Y),\O_C+\O^{++}(Y)\big)\moins
{\rm Spa}\big(\O(Y),\O^{+}(Y)\big)$$ 
de l'espace adique associ\'e \`a $Y$.

{\rm (ii)}
Si $T\in{\rm Fr}(\O(Y))$ v\'erifie $v_{s,a}(T)=(0,1)$, l'anneau
des entiers du compl\'et\'e
de ${\rm Fr}(\O(Y))$ pour $v_{s,a}$ est
$\O_C[[T,T^{-1}\rangle$.
Le sch\'ema formel $Y_{a,s}$ associ\'e est un 
{\og cercle fant\^ome\fg},
obtenu en {\og retirant\fg} tous les points visibles du cercle
${\rm Sp}(C\langle T,T^{-1}\rangle)$.

On d\'efinit la {\it fronti\`ere} $\partial^{\rm ad}Y$ de $Y$ comme
la r\'eunion\footnote{Que l'on peut voir comme contenue dans la courbe adoque associ\'ee
\`a $Y$, voir~\no\ref{adoc8}.}
des cercles fant\^omes $Y_{s,a}$, 
pour $s\in \partial Y$ et $a\in A_0(s)$.
L'application $v_{s,a}\mapsto Y_{s,a}$
fournit une bijection de $\partial Y^{\rm ad}$
sur l'ensemble de ces cercles fant\^omes,
et la fronti\`ere peut aussi \^etre vue
comme le {\og bord adique\fg} de $Y$.

{\rm (iii)} Il y a (au moins) deux mani\`eres de se repr\'esenter une courbe $p$-adique:
soit comme un graphe (\`a la Berkovich ou \`a la Huber), soit comme une surface de Riemann.
Dans ce texte, nous allons privil\'egier la seconde repr\'esentation et construire
des courbes en recollant des affino\"{\i}des et des couronnes le long de cercles fant\^omes.
Cette op\'eration ne semble pas \^etre licite dans les cadres classiques (rigide, Berkovich, adique,
ou sch\'ema formel); pour lui donner un sens, la solution est de modifier un peu la structure de
sch\'ema formel $p$-adique en munissant la fibre sp\'eciale d'une topologie proche de celle
de Berkovich au lieu de la topologie de Zariski; cela m\`ene \`a la g\'eom\'etrie adoque\footnote{Interpolation
entre adique et ad hoc...}
du \S\,\ref{adoc1}.
La m\^eme op\'eration dans le cadre adique revient \`a recoller des graphes en des points
(de type 5), ce qui peut donner l'impression qu'il y a une unique mani\`ere de faire le
recollement, alors que le point de vue des surfaces de Riemann sugg\`ere
une analogie avec la th\'eorie de Teichm\"uller, qui semble plus proche de la r\'ealit\'e.
\end{rema}

\Subsection{Mod\`eles semi-stables des courbes analytiques}\label{GRAB11}
\cite{Duc,BPR,Tem}.
Soit $X$ une $C$-courbe analytique (lisse et connexe).
\subsubsection{Bord et fronti\`ere}
Si $X$ n'est pas propre, alors $X$ est la r\'eunion croissante d'affino\"{\i}des $Y_n$.
On d\'efinit {\it le bord} $\partial X$ de $X$ comme la limite $\cup_n\big(\cap_{k\geq n}\partial Y_k\big)$
des $\partial Y_n$,
et {la fronti\`ere} $\partial^{\rm ad} X$ de $X$
comme la limite des $\partial^{\rm ad} Y_n$.
(Notons que des points de $\partial Y_n$ peuvent ne pas faire partie du bord de $Y_{n+1}$
et donc pas non plus de $\partial X$.)

\subsubsection{Noeuds}\label{TTT21}
Rappelons que les points de $X$
se r\'epartissent en $4$ types~\cite[3.3.2]{Duc} suivant la forme de leur corps r\'esiduel 
compl\'et\'e.
Si $Y\subset X$, on note $Y_{[i]}$ l'ensemble de ses points de type $i\in\{1,2,3,4\}$. On
note aussi $Y_{[2,3]}=Y_{[2]}\cup Y_{[3]}$, etc. En particulier $X_{[1]}=X(C)$ est l'ensemble des points rigides de 
$X$, et pour tout $x\in X_{[2]}$ le corps r\'esiduel de $C(x)$ 
(corps r\'esiduel compl\'et\'e de $x$) est le corps des fonctions d'une courbe projective lisse sur $k_C$,
appel\'ee {\it courbe r\'esiduelle en $x$}. Le {\it genre} de $x\in X_{[2]}$ est alors le genre de cette 
courbe r\'esiduelle.

On dit que $x$ est {\it un noeud de $X$} s'il ne poss\`ede pas de voisinage qui est une couronne
ouverte.  Cela inclut les points de type $2$ et de genre~$\geq 1$ et les points
du bord $\partial X$ de~$X$.
On note $\Sigma(X)$ l'ensemble des noeuds de $X$.

\subsubsection{Squelette analytique}\label{TTT22}
Soit $X$ une courbe sur 
$C$ (lisse et connexe). Le {\it squelette analytique de $X$} est l'ensemble $\Gamma^{\an}(X)$ des points de $X$ n'ayant pas de voisinage qui est un disque ouvert; c'est un sous-graphe ferm\'e de~$X$, localement
fini et m\'etris\'e, 
compos\'e de points de
type $2$ ou $3$ \cite[5.1.11]{Duc}. L'ensemble des noeuds $\Sigma(X)$
de $X$ est une partie de $\Gamma^{\an}(X)$, que l'on prend comme
ensemble des sommets de $\Gamma^{\an}(X)$.
On note $A^{\rm an}(X)$ (resp.~$A^{\rm an}_c(X)$) l'ensemble des ar\^etes
(resp.~ar\^etes relativement compactes) de $\Gamma^{\rm an}(X)$.

\begin{rema}\label{gam2} 
{\rm (i)} Si $X$ est compacte en tant qu'espace de Berkovich, alors $\Gamma^{\rm an}(X)=\emptyset$ si et seulement si $X\cong \piqp$ \cite[5.4.16]{Duc}. Sans hypoth\`ese de compacit\'e l'\'enonc\'e tombe en d\'efaut, par exemple $\A^1$ et les disques ouverts ont aussi un squelette analytique vide. 
(Sur un corps assez gros, i.e.~sph\'eriquement complet et \`a groupe de valuation $\R$, ce sont les seuls 
exemples mais sur un corps non sph\'eriquement complet, le compl\'ementaire d'un point de type~$4$ dans
$\piqp$ n'est pas de ce type, et il y a m\^eme des exemples de telles courbes qui ne se plongent
pas dans $\piqp$ sur $C$ bien qu'elles se plongent dans $\piqp$ sur un corps assez gros~\cite[prop. 5.5]{LIU}.)

{\rm (ii)} Si $\Gamma^{\rm an}(X)\neq\emptyset$, alors
$\Sigma(X)=\emptyset$ si et seulement si
$X$ est une couronne ouverte 
g\'en\'eralis\'ee\footnote{Une couronne ouverte g\'en\'eralis\'ee est
une courbe qui, apr\`es extension des scalaires \`a un corps assez gros, devient
une couronne ouverte ou un disque ouvert \'epoint\'e ou~${\bf G}_m$.}
auquel cas $\Gamma^{\rm an}(X)$ est un segment ouvert,
ou bien $X$ est une courbe de Tate auquel cas $\Gamma^{\rm an}(X)$ est un cercle.

{\rm (iii)}
Si $\Gamma^{\rm an}(X)\neq\emptyset$ et si $X$ n'est pas une courbe de Tate,
les ar\^etes de $\Gamma^{\rm an}(X)$ correspondent \`a des sous-couronnes
ouvertes g\'en\'eralis\'ees de $X$, la longueur de l'ar\^ete 
\'etant {\it la largeur de la couronne} correspondante.
Plus pr\'ecis\'ement:

\qquad$\diamond$ une ar\^ete de longueur finie correspond
\`a une couronne ouverte 
$\alpha<v_p(z)<\beta$
\linebreak
\phantom{X} \hskip1.3cm
g\'en\'eralis\'ee\footnote{Si
l'ar\^ete appartient \`a $A_c^{\rm an}(X)$, on obtient une vraie couronne et $\beta-\alpha\in\Q$.}
 dont la largeur
est $\beta-\alpha$.

\qquad$\diamond$ une demi-droite correspond \`a un disque ouvert \'epoint\'e 
(de largeur $\infty$
\linebreak
\phantom{X} \hskip1.3cm
 car $\beta=+\infty$),

\qquad$\diamond$ une droite correspond \`a ${\bf G}_m$ (de largeur $2\infty$ car $\alpha=-\infty$ et $\beta=+\infty$).

{\rm (iv)} 
Si $\Gamma^{\rm an}(X)\neq\emptyset$,
alors $\Gamma^{\rm an}(X)$ est compact
si et seulement si $X$ est un affino\"{\i}de ou une courbe propre. En effet,
$\Gamma^{\rm an}(X)$ est un sous-graphe analytiquement admissible \cite[th. 5.1.11]{Duc}, donc admissible de 
$X$ (cf. \cite[1.5.1,5.1.3]{Duc} pour ces notions) et donc la r\'etraction canonique\footnote{Si 
$x\notin  \Gamma^{\rm an}(X)$, $r(x)$ est l'unique point du bord de la composante connexe de 
$x$ dans $X\setminus \Gamma^{\rm an}(X)$.}
$r: X\to \Gamma^{\rm an}(X)$ est continue et compacte \cite[th. 1.5.16]{Duc}. Donc 
$\Gamma^{\rm an}(X)$ est compact si et seulement si $X$ l'est (en tant qu'espace de Berkovich) et on conclut en utilisant les r\'esultats \'enonc\'es dans \ref{Global}. 
\end{rema}

On d\'efinit le {\it squelette adique} $\Gamma^{\rm ad}(X)$ en rajoutant
\`a $\Gamma^{\rm an}(X)$ des ar\^etes non relativement compactes de longueur $0^+$
en les sommets $s$ de $\partial X$, une par cercle fant\^ome de $\partial^{\rm ad}X$
au bord de $s$.

\subsubsection{Triangulations}\label{GRAB13}
Une {\it pseudo-triangulation $S$ de $X$} est 
un sous-ensemble discret et ferm\'e de $X$, 
form\'e de points de $X_{[2]}$, tel que les
composantes connexes de $X\moins S$ soient des disques ouverts ou des couronnes ouvertes g\'en\'eralis\'ees.
Le {\it squelette $\Gamma(S)$ de $S$} est le sous-graphe ferm\'e de $X$
trac\'e sur $X_{[2,3]}$, ayant pour sommets $S$ et pour ar\^etes les squelettes
des composantes connexes de $X\moins S$.
On d\'efinit le {\it squelette adique} $\Gamma^{\rm ad}(S)$ en rajoutant
\`a $\Gamma(S)$ des ar\^etes non relativement compactes de longueur $0^+$
en les sommets $s$ de $\partial X$, une par cercle fant\^ome de $\partial^{\rm ad}X$
\`a la fronti\`ere de $s$.

On dit que $X$ est {\it compacte} (resp.~{\it quasi-compacte}, resp.~{\it compl\`ete}),
si $\Gamma^{\rm ad}(S)$ est un espace m\'etrique
compact (resp.~espace quasi-compact, resp.~espace m\'etrique complet),
pour toute pseudo-triangulation~$S$.

\begin{rema} 
{\rm (i)} Les courbes compactes sont les analytifi\'ees des courbes
alg\'ebriques propres. 

{\rm (ii)} Les courbes connexes quasi-compactes mais non compactes sont les affino\"ides connexes.

En effet, pour les deux assertions il suffit 
de raisonner comme dans la preuve du point iv) de la rem.\,\ref{gam2}, en utilisant la r\'etraction
canonique de $X$ sur $\Gamma(S)$.

{\rm (iii)} Une courbe compacte est compl\`ete, mais on peut fabriquer des courbes
compl\`etes non compactes: par exemple, en retirant un nombre fini de points \`a une courbe compacte ou,
plus g\'en\'eralement, un sous-ensemble compact (comme pour le demi-plan de Drinfeld),
ou en prenant un rev\^etement fini \'etale d'une courbe compl\`ete non compacte.
\end{rema}

Une {\it triangulation} de $X$ est une pseudo-triangulation telle que
les composantes connexes de $X\moins S$ soient relativement compactes dans $X$
(en particulier, ce sont des disques
ou de vraies couronnes et pas des disques \'epoint\'es ou des ${\bf G}_m$).

Si $S$ est une pseudo-triangulation de $X$, alors $S$ contient $\Sigma(X)$.
Si $\Gamma^{\rm an}(X)$ est compact, 
et si $\Sigma(X)\neq\emptyset$,
alors $\Gamma^{\rm an}(X)$ est une triangulation de $X$, plus pr\'ecis\'ement la plus petite triangulation
de $X$ \cite[5.4.12]{Duc}.
Si $A^{\rm an}(X)\neq A^{\rm an}_c(X)$, pour construire une triangulation on a besoin de subdiviser
les ar\^etes non relativement compactes de $\Gamma^{\rm an}(X)$ en une infinit\'e
d'ar\^etes, et donc de rajouter une infinit\'e de sommets \`a $\Sigma(X)$. 

\begin{rema}
{\rm (i)} Un r\'esultat fondamental de la th\'eorie est l'existence de triangulations 
pour toute courbe analytique sur $C$. C'est une cons\'equence du th\'eor\`eme de r\'eduction 
semi-stable (cf. \cite[ch.1]{Bal}), mais peut aussi se d\'emontrer "directement", par une \'etude 
de la structure locale des courbes \cite[th. 5.1.4]{Duc}. Plus pr\'ecis\'ement, tout ensemble ferm\'e 
et discret, constitu\'e de points de type $2$ est contenu dans une triangulation de $X$. 
Si $X$ est compacte (en tant qu'espace de Berkovich), alors 
la r\'eunion des squelettes des triangulations de $X$ est $X_{[23]}$ et 
$X$ est hom\'eomorphe \`a la limite inverse de ces squelettes.

{\rm (ii)} Une triangulation de $X$ permet de d\'ecouper $X$ en shorts et jambes
(cf.~\no\ref{Pgener10} pour une version pr\'ecise), ce qui permet de d\'ecrire $X$ \`a partir
d'objets plus simples.

{\rm (iii)} Si $X$ est partiellement propre, une pseudo-triangulation fournit
un recouvrement de $X$ par des pantalons (cf.~\no\ref{TTT23}), qui permet aussi
de d\'ecrire $X$ \`a partir
d'objets plus simples.

{\rm (iv)} Ces deux descriptions de courbes analytiques joueront un grand r\^ole
dans le calcul de leurs cohomologies. Pour \'eviter de se battre avec la combinatoire
du recouvrement, il est souvent utile de raffiner la pseudo-triangulation.
\end{rema}
On dit qu'{\it une pseudo-triangulation $S$ est fine} si
$\Gamma(S)$ ne contient
pas de boucle avec $1$ ou $2$ sommets.
On peut rendre fine n'importe quelle triangulation $S$ en rajoutant un sommet au milieu
de toutes les ar\^etes de $\Gamma(S)$ (et donc en coupant chaque ar\^ete en deux).

\subsubsection{Triangulations et mod\`eles semi-stables}\label{TTT24}
 Il y a une bijection naturelle (\cite[th. 4.11]{BPR} pour les courbes projectives, \cite[ch. 1]{Bal} et \cite[ch. 6]{Duc}
 dans le cas g\'en\'eral) entre les triangulations de $X$ et les mod\`eles formels ($p$-adiques)
semi-stables\footnote{Au sens large, i.e.~\'etale localement, 
de la forme $\Spf \O_C\{X, Y\} /(XY-a) $, $a\in \O_C\moins\{0\}$.} de $X$ sur $\O_C$: 
si ${\cal X}$ est un tel mod\`ele, l'ensemble $S(\sx)$ des pr\'eimages 
(par la sp\'ecialisation) dans $X$ des points g\'en\'eriques
des composantes irr\'eductibles de {\it la fibre sp\'eciale classique} $k_C\otimes_{\O_C}\sx$ de $\sx$
est une triangulation de $X$. 
Pour aller dans l'autre sens il faut se fatiguer un peu plus
(\cite[th 6.3.15]{Duc} ou \cite[ch1]{Bal}), mais au moins pour une courbe projective $X$ et dans le cas o\`u le
squelette $\Gamma(S)$ de la triangulation a au moins deux ar\^etes la construction est 
assez explicite \cite[th 4.11]{BPR}: si $r: X\to \Gamma(S)$ est la retraction canonique, 
et si $a$ est une ar\^ete, alors $r^{-1}(\overline a)$ 
est un domaine affino\"ide de $X$, et le mod\`ele semi-stable 
correspondant s'obtient en recollant les ${\rm Spf}(\O^+(r^{-1}(\overline a)))$ le long des 
${\rm Spf}(\O^+(r^{-1}(s)))$, $s$ parcourant $S$.

\begin{rema}\label{gam3}
Soient $S$ une triangulation de $X$ et $X_S$ le mod\`ele
semi-stable associ\'e.
 
{\rm (i)}
$\Gamma(S)$ est le graphe dual de $k_C\otimes_{\O_C} X_S$:
les sommets de $\Gamma(S)$ correspondent aux composantes
irr\'eductibles de $k_C\otimes_{\O_C} X_S$ et les ar\^etes aux points singuliers.
Si ${\cal Y}_s$ est la composante irr\'eductible correspondant \`a $s\in S$, on note 
$Z_s$ l'image inverse de ${\cal Y}_s$ par l'application de sp\'ecialisation et
$Y_s$ l'image inverse de l'ouvert de lissit\'e de ${\cal Y}_s$ (i.e. ${\cal Y}_s$ priv\'e des points communs avec les autres composantes de $k_C\otimes_{\O_C} X_S$).  
Si $P_a$ est le point singulier correspondant \`a l'ar\^ete $a$, on note $Y_a$
l'image inverse de $P_a$; c'est une couronne ouverte dont on note $\mu(a)$ la largeur. 

{\rm (ii)}
$S$ contient $\partial X$ puisqu'elle contient $\Sigma(X)$;
par la bijection ci-dessus entre \'el\'ements de $S$ et composantes irr\'eductibles,
les points de $\partial X$
correspondent aux composantes irr\'eductibles de $k_C\otimes_{\O_C} X_S$ qui ne sont
pas propres.  
\end{rema}

\Subsection{Fibre sp\'eciale d'une courbe quasi-compacte}\label{TTT25}
Soit $X$ une courbe quasi-compacte (lisse et connexe) sur $C$, 
et soient $S$ une triangulation de $X$ et $X_S$
le mod\`ele semi-stable sur $\O_C$ associ\'e.  La courbe $k_C\otimes_{\O_C} X_S$
est un invariant un peu trop grossier:
on veut garder la trace de 
la largeur
des couronnes correspondant aux points singuliers, ce qui am\`ene naturellement \`a la notion
de courbe marqu\'ee ci-dessous.
\subsubsection{Courbes marqu\'ees}\label{TTT26}
Soit $X$ une courbe sur $k_C$.
Un {\it point marqu\'e} sur $X$ est un couple
$(P,\mu(P))$, o\`u:

\quad $\bullet$ $P\in X(k_C)$, 

\quad $\bullet$ {\it la multiplicit\'e $\mu(P)$}
de $P$ est un \'el\'ement de $\Q_+^\dual\sqcup\{0^+\}$.

\smallskip
Une {\it courbe marqu\'ee} $(X,A,\mu)$ est une courbe $X$ munie d'un ensemble $A$ de points marqu\'es
(et $\mu$ est la fonction associant \`a un point marqu\'e sa multiplicit\'e).
Une courbe marqu\'ee $(X,A,\mu)$ est {\it semi-stable} si:

$\bullet$ $X$ est \`a singularit\'es nodales,

$\bullet$ les composantes irr\'eductibles de $X$ sont propres et ne comportent qu'un
nombre fini de points marqu\'es,

$\bullet$ les points singuliers de $X$ sont marqu\'es et ont une multiplicit\'e~$\in\Q_+^\dual$,
les points marqu\'es non singuliers ont multiplicit\'e $0^+$.

Elle est {\it stable} si elle est semi-stable et si, de plus, aucune composante
connexe de $X$ n'est un $\piqp$ avec $0$, $1$ ou $2$ points marqu\'es.

\begin{rema}\label{gam4}
{\rm (i)} On peut consid\'erer une courbe lisse $X$ comme une courbe marqu\'ee stable
en lui associant la courbe marqu\'ee stable
$(\overline X,\overline X\moins X,0^+)$, o\`u $\overline X$ est la compactification lisse
de $X$ et tous les \'el\'ements de $\overline X\moins X$ sont de multiplicit\'e~$0^+$.

{\rm (ii)} Plus g\'en\'eralement, une courbe $X$ \`a singularit\'es nodales, dont les
points singuliers sont munis d'une multiplicit\'e~$\in\Q_+^\dual$, peut \^etre consid\'er\'ee
comme une courbe marqu\'ee semi-stable: on consid\`ere la compactifi\'ee $\overline X$
de $X$ obtenue en rajoutant des points lisses, et on d\'efinit l'ensemble des
points marqu\'es comme la r\'eunion des points singuliers de $X$ (avec la multiplicit\'e donn\'ee)
et des points de $\overline X\moins X$, de multiplicit\'e~$0^+$.

{\rm (iii)} Si $X$ est une courbe analytique et si $S$ est une triangulation de
$X$, la courbe $k_C\otimes_{\O_C}X_S$ 
est une courbe \`a singularit\'es nodales dont les points singuliers
sont naturellement munis d'une multiplicit\'e rationnelle, \`a savoir la largeur de la couronne
correspondante. On transforme cette courbe en courbe marqu\'ee semi-stable en utilisant
le~(ii).
\end{rema}

Si $X$ est une courbe analytique et si $S$ est une triangulation fine de $X$,
la {\it fibre sp\'eciale} $X_S^{\rm sp}$ de $X_S$ est la courbe marqu\'ee semi-stable obtenue
\`a partir de $k_C\otimes_{\O_C}X_S$ par le proc\'ed\'e du (iii) de la rem.\,\ref{gam4}.

$\bullet$ Ses composantes irr\'eductibles sont en bijection avec les \'el\'ements de $S$;
on note $Y_s^{\rm sp}$ la composante irr\'eductible associ\'ee \`a $s$ (c'est, par construction,
une courbe propre sur $k_C$ et on note $\ocirc Y_s^{\rm sp}$ le compl\'ementaire
dans $Y_s^{\rm sp}$ des points marqu\'es).

$\bullet$ Ses points singuliers sont en bijection avec les ar\^etes de $\Gamma(S)$;
on note $P_a$ le point correspondant \`a $a$; c'est un point marqu\'e
de $X_S^{\rm sp}$ de multiplicit\'e $\mu(a)$ (largeur de la couronne correspondante). 

$\bullet$ Les autres points marqu\'es de $X_S^{\rm sp}$ sont de multiplicit\'e $0^+$
et leur ensemble est le compl\'ementaire de $k_C\otimes_{\O_C}X_S$
dans $X_S^{\rm sp}$.

\subsubsection{Graphe dual d'une courbe marqu\'ee}\label{TTT27}
Si $X$ est une courbe marqu\'ee, on lui associe un graphe semi-m\'etris\'e $\Gamma(X)=(S,A,\mu)$,
{\it le graphe dual de $X$},
d\'efini de la mani\`ere suivante:

$\bullet$ $S$ est en bijection avec l'ensemble des composantes irr\'eductibles de $X$.

$\bullet$ $A$ est en bijection avec les points marqu\'es de $X$,
la longueur
d'une ar\^ete $a$ \'etant la multiplicit\'e du point $P_a$ correspondant\footnote{Qui peut
\^etre $0^+$
ce qui fait que $\Gamma(X)$ peut \^etre seulement semi-m\'etris\'e et pas m\'etris\'e.}.

$\bullet$ Les ar\^etes relativement compactes de $\Gamma(X)$ sont en bijection avec les points marqu\'es
singuliers de $X$, les extr\'emit\'es de l'ar\^ete $a_P$ correspondant \`a $P$ \'etant les
sommets correspondant 
aux composantes irr\'eductibles contenant\footnote{Ces deux extr\'emit\'es peuvent \^etre \'egales
si des composantes irr\'eductibles de $X$ s'autointersectent.} $P$.

$\bullet$ Les ar\^etes non relativement compactes 
de $\Gamma(X)$ sont en bijection avec les points marqu\'es de multiplicit\'e $0^+$,
l'extr\'emit\'e de l'ar\^ete $a_P$ correspondant \`a $P$ \'etant 
le sommet correspondant \`a la
composante irr\'eductible contenant $P$.

\begin{rema}\label{TTT27.1}
Si $X$ est une courbe analytique et si $S$ est une triangulation fine de $X$,
alors 
$$\Gamma(X_S^{\rm sp})=\Gamma^{\rm ad}(S).$$
\end{rema}

\subsection{M\'etrique canonique et valuation}
L'espace $X_{[2,3]}$ des points de type $2$ ou $3$ est naturellement muni
d'une m\'etrique $dt$ qui permet de m\'etriser tout sous-graphe: la
distance associ\'ee est d\'efinie par
$d(r,s)=\inf\big|\int_r^sdt\big|$, le minimum \'etant pris sur tous les
chemins joignant $r$ \`a $s$.
Si $C\subset X$ est une couronne ferm\'ee (i.e.~$C\cong\{\alpha\leq v_p(z)\leq \beta\}$),
et si $r$ et $s$ sont les valuations de Gauss des cercles $v_p(z)=\alpha$
et $v_p(z)=\beta$ aux extr\'emit\'es de la couronne, alors $d(r,s)$
est la largueur $\beta-\alpha$ de la couronne.
Cette distance envoie les points de type $1$ \`a l'infini (i.e.~elle les transforme
en pointes, comme les \'el\'ements de $\R$ pour la m\'etrique de Poincar\'e sur le 
demi-plan de Poincar\'e).

Si $f\in C(X)^\dual$,
si $r,s\in X_{[2,3]}$ 
et si $v_r,v_s\in X_{[2,3]}$ sont
les valuations correspondantes sur $C(X)$, alors
$$v_s(f)-v_r(f)=\int_r^s{\rm Res}_t\big(\tfrac{df}{f}\big)\,dt,$$
l'int\'egrale \'etant prise sur n'importe quel chemin allant de $r$ \`a $s$: un tel chemin est une
r\'eunion finie de segments ferm\'es et la fonction ${\rm Res}\big(\frac{df}{f}\big)$
est constante sur l'int\'erieur d'un tel segment;
cet int\'erieur 
correspond \`a une couronne ouverte de $X$, ses points sont
les points de Gauss de cercles virtuels, et le sens de parcours du segment
pour aller de $r$ \`a $s$ d\'etermine une orientation de la couronne, ce qui
fixe le signe du r\'esidu de $\frac{df}{f}$.

Si $X$ est un affino\"ide, et $r\in X$ est un point de type~$2$,
on pose $d(r,\partial X)=\inf_{s\in\partial X}d(r,s)$, 
si $X$ est la r\'eunion croissante d'affino\"{\i}des $X_n$, on pose
$d(r,\partial X)=\lim_{n\in\N} d(r,\partial X_n)$, et si $X$ est propre, on pose
$d(r,\partial X)=+\infty$.
Si $Z\subset X$ est un affino\"{\i}de, on pose $d(Z,\partial X)=\inf_{r\in Z_{[2]}}d(r,\partial X)$. 
\begin{prop}\label{metr1}
Soient $X$ une courbe 
et $Z\subset X$ un affino\"{\i}de connexe.
Si $f\in\O^+(X)$,
on a $v_Z(f-f(x_0))\geq d(Z,\partial X)$,
pour tout $x_0\in Z(C)$.
\end{prop}
\begin{proof}
Soit $g=f-f(x_0)$. Il n'y a rien \`a prouver si $g=0$ ni si $X$ est propre (car alors $g=0$).
On peut donc supposer $X$ affino\"{\i}de (et passer \`a la limite pour le cas g\'en\'eral) et $g\neq 0$.
Soit $X_g$ l'ensemble des z\'eros de $g$; c'est un sous-ensemble discret
de $X$.  Le squelette analytique $\Gamma$ de $X\moins X_g$ est obtenu en rajoutant
des demi-droites \`a $\Gamma^{\rm an}(X)$, une demi-droite par \'el\'ement de $X_g$.
La fonction $t\mapsto {\rm Res}_t\frac{dg}{g}$ est constante sur chaque
ar\^ete orient\'ee de $\Gamma$.  

Soit $r\in \Gamma\cap Z$ r\'ealisant le minimum de $s\mapsto v_s(g)$, et tel qu'il existe
une ar\^ete de bout $r$ sur laquelle ${\rm Res}\frac{dg}{g}\leq -1$: un tel point existe
car $v_s(g)$ tend vers $+\infty$ en l'infini de $\Gamma\cap Z$ (en particulier,
sur la branche correspondant \`a $x_0$) et est croissante sur
les ar\^etes sur lesquelles ${\rm Res}\frac{dg}{g}\geq 0$.

Soit $\gamma$ un chemin maximal, trac\'e sur $\Gamma$,
partant de $r$ et le long duquel ${\rm Res}\frac{dg}{g}\leq -1$.
Alors $\gamma$ aboutit au bord de $\Gamma^{\rm an}(X)$: en effet, il ne peut
pas aboutir au bout d'une des demi-droites correspondant aux z\'eros de $g$
puisque ${\rm Res}\frac{dg}{g}>0$ le long de ces demi-droites, et s'il aboutit
en un point int\'erieur, ce point est forc\'ement un noeud~$r'$ (puisque
${\rm Res}\frac{dg}{g}$ est constant sur les ar\^etes).  Or la somme
des r\'esidus de $\frac{dg}{g}$ le long des ar\^etes arrivant en $r'$ est nulle,
et comme il y en a une pour lequel ce r\'esidu est~$\leq -1$, cela implique
qu'il y en a une $a$ pour lequel ce r\'esidu est~$\geq 1$.  Mais, cela fait
que le r\'esidu sur $a$, consid\'er\'ee comme ar\^ete sortante, est~$\leq -1$,
ce qui prouve que $\gamma$ n'est pas maximal puisqu'on peut le prolonger par
l'ar\^ete~$a$. 

La longueur ${\rm lg}(\gamma)$ de $\gamma$ 
est donc~$\geq d(r,\partial X)\geq d(Z,\partial X)$.
Par ailleurs, si $s\in \gamma$, alors $v_r(g)-v_s(g)=
\int_{r}^s-{\rm Res}\frac{dg}{g}\,dv$ (l'int\'egrale \'etant le long de $\gamma$).
 En passant \`a la limite, en en utilisant
le fait que $v_s(g)\geq 0$ puisque $g\in\O^+(X)$ et que
$-{\rm Res}\frac{dg}{g}\geq 1$, on obtient la minoration
$v_Z(g)=v_r(g)\geq {\rm lg}(\gamma)\geq d(Z,\partial X)$.  
\end{proof}

\begin{coro}\label{metr2}
Si $Y$ est compl\`ete, toute fonction holomorphe born\'ee
est constante.
\end{coro}
\begin{proof}
Cela r\'esulte de ce que $d(Z,\partial Y)=+\infty$ si $Z$ est un affino\"ide connexe de $Y$
(car $Y$ est suppos\'ee compl\`ete),
et donc $v_Z(f-f(x_0))=+\infty$, si $x_0\in Z(C)$ (prop.~\ref{metr1}): 
autrement dit, $f$ est constante sur $Z$, pour tout $Z$.

\end{proof}
\begin{rema}
Cela fournit une preuve alternative de l'\'enonc\'e de la rem.\,A.2 de \cite{CDN}.
\end{rema}

\section{Construction de courbes analytiques}\label{constr6}
Dans ce chapitre, on explique comment construire des courbes
en recollant des shorts (ou plus g\'en\'eralement des affino\"{\i}des)
et des jambes (ou des boules ouvertes) le long de cercles fant\^omes.
Pour donner un sens \`a cette construction nous allons avoir besoin de raffiner 
la structure de sch\'ema formel associ\'e \`a une alg\`ebre de Tate, ce qui conduit
\`a la g\'eom\'etrie {\it adoque}.

\Subsection{G\'eom\'etrie adoque}\label{adoc1}
\subsubsection{Affines adoques}\label{adoc2}
Soit $A=\O_C\langle T_1,\dots,T_n\rangle/I$ 
une $\O_C$-alg\`ebre de Tate r\'eduite\footnote{On peut
aussi remplacer $\O_C$ par un \'epaississement pour avoir une th\'eorie {\og en famille\fg}; c'est
ce qui est fait au paragraphe suivant.}, telle que $\overline A=k_C\otimes_{\O_C}A$
soit aussi r\'eduite (typiquement, $A=\O^+(Y)$, o\`u $Y$ est un affino\"{\i}de r\'eduit sur $C$).  
Le sch\'ema formel $p$-adique associ\'e $Y={\rm Spf}(A)$
est un espace annel\'e d'espace topologique sous-jacent la fibre sp\'eciale 
$Y^{\rm sp}={\rm Spec}(\overline A)$,
munie de la topologie de Zariski, le faisceau structural \'etant,
si $A$ est normal,
$U\mapsto \O^+(]U[)$ o\`u $]U[$ d\'esigne le tube de $U$ sur la fibre g\'en\'erique
(c'est un affino\"{\i}de sur $C$): si $U$ est l'ouvert $f\neq 0$ de $Y^{\rm sp}$,
et si $\tilde f$ est un rel\`evement de $f$ dans $A$, alors $\O_Y(U)=A\langle X\rangle/(1-X\tilde f)$.

Nous allons associer \`a $A$ un {\it sch\'ema adoque} ${\rm Sp}^{\rm ado}(A)$ d\'efini
en rajoutant des points et des ouverts
\`a l'espace topologique ${\rm Spf}(A)$ et en modifiant le faisceau structural en cons\'equence.
Avant de donner la d\'efinition de cet espace, commen\c{c}ons par examiner ${\rm Spf}(A)$
du point de vue {\og points classiques\fg}.

Si $K$ est un corps alg\'ebriquement clos,
complet pour une valuation r\'eelle, muni d'un morphisme
continu $\O_C\to K$, alors $X(K)={\rm Hom}(A,K)$ est un sous-espace
ferm\'e de $K^{n}$ (les $s:A\to K$ que l'on consid\`ere
sont les morphismes continus de $\O_C$-alg\`ebres).  Comme $\O_C$ est de caract\'eristique
mixte, on peut prendre $K=C$ ou $K=C^\flat$ (ou n'importe quel corps v\'erifiant les conditions
ci-dessus et contenant un de ces corps).
Alors $X(C)$ est born\'e dans $C^n$ tandis que $X(C^\flat)$ n'est pas born\'e
dans $(C^\flat)^n$, et $X(C^\flat)$ a beaucoup plus d'ouverts naturels
que les compl\'ementaires de sous-vari\'et\'es alg\'ebriques ferm\'ees
(par exemple, l'intersection avec la boule ouverte unit\'e de $(C^\flat)^n$).

\vskip.1cm
\noindent $\bullet$ {\it Les points de ${\rm Sp}^{\rm ado}(A)$}.---
Les points de 
l'espace topologique sont ceux de l'espace
de Berkovich sur $k_C$ associ\'e \`a $k_C\otimes_{\O_C} A$ (notons que cet anneau est discret), 
i.e.~les valuations $v$ sur $k_C\otimes_{\O_C} A$
(i.e. $v(xy)=v(x)+v(y)$ et $v(x+y)\geq\inf(v(x),v(y))$) 
\`a valeurs dans $\Z\sqcup\{\infty\}$, d'image $\{0,\infty\}$ ou 
contenant\footnote{Normaliser l'image \'evite de travailler \`a \'equivalence pr\`es.} $1$ ou $-1$.

Si $v$ est une telle valuation, $I_v=\{a\in \overline A,\ v(a)=\infty\}$ est un id\'eal
premier, et $v$ s'\'etend en une valuation de ${\rm Fr}(\overline A/I_v)$;  on note
$K_v$ le compl\'et\'e de ${\rm Fr}(\overline A/I_v)$ pour $v$.  Alors $v$ est obtenue
en composant $\overline A\to K_v$ avec la valuation $v$ sur $K_v$.

Les valuations \`a valeurs dans $\{0,\infty\}$ correspondent aux \'el\'ements de ${\rm Spec}(\overline A)$.
Si $v$ est \`a valeurs dans $\{0,\infty\}$, la valuation induite sur $K_v$ est 
{\it la valuation triviale}
(i.e. $v(0)=\infty$ et $v(x)=0$ si $x\neq 0$).

\vskip.1cm
\noindent $\bullet$
{\it La topologie de ${\rm Sp}^{\rm ado}(A)$}.---
Si $f\in \overline A$, on note $U_f$ l'ouvert de Zariski 
$$U_f=\{v,\ v(f)<\infty\}.$$
Si $g_1,\dots,g_d\in\overline A$, on note $U_{f,g_1,\dots,g_d}$
l'ensemble 
$$U_{f,g_1,\dots,g_d}=\{v\in U_f,\ v(g_i)>0,\,1\leq i\leq d\}.$$
On munit l'espace ci-dessus de la topologie engendr\'ee par les
$U_{f,g_1,\dots,g_d}$; un ouvert de cette forme est dit {\it standard};
comme $$U_{f_1,g_1,\dots,g_r}\cap U_{f_2,g_{r+1},\dots,g_{r+s}}=
U_{f_1f_2,g_1,\dots,g_{r+s}},$$ on n'a que les r\'eunions quelconques \`a rajouter.
L'espace obtenu est compact.

\vskip.1cm
\noindent $\bullet$
{\it Le faisceau structural}.---
Si $\tilde f,\tilde g_1,\dots,\tilde g_d\in A$ sont des rel\`evements de $f,g_1,\dots,g_d$,
on pose 
\begin{align*}
\overline A_{f,g_1,\dots,g_d}=&\ \overline A[[T_1,\dots,T_d]][X]/(1-Xf,T_1-g_1,\dots,T_d-g_d),\\
A_{f,g_1,\dots,g_d}=&\ A[[T_1,\dots,T_d]]\langle X\rangle/(1-X\tilde f,T_1-\tilde g_1,\dots,T_d-\tilde g_d).
\end{align*}
(Que $A_{f,g_1,\dots,g_d}$ ne d\'epende pas des choix de
$\tilde f,\tilde g_1,\dots,\tilde g_d$ r\'esulte de ce que, si $h_1,h_2\in A$ ont m\^eme
image dans $\overline A$, alors il existe $r>0$ tel que $h_1-h_2\in p^rA$.)

Notons que, si $r>0$ est assez petit, on a
$$A_{f,g_1,\dots,g_d}/p^r\cong (\O_C/p^r)\wotimes_{k_C}\overline A_{f,g_1,\dots,g_d}$$
(Cela r\'esulte de ce que les g\'en\'erateurs de l'id\'eal d\'efinissant $A$
et les coefficients de $f$ et des $g_i$ appartiennent \`a $\O_{\breve C}+p^r\O_C$
si $r$ est assez petit car il n'y en a qu'un nombre fini qui ne sont pas divisibles par $p$.) 

\begin{rema}\label{adoc3}
La topologie de ${\rm Sp}^{\rm ado}(A)$ est moins fine que celle de l'espace de Berkovich correspondant
(en dimension~$1$, qui est le seul cas qui va nous int\'eresser, les deux topologies 
co\"{\i}ncident si on reste sur $k_C$ mais pas si on \'etend les scalaires 
\`a un corps muni d'une valuation non triviale)
et, si $U=U_{f,g_1,\dots,g_r}$ est un ouvert standard, alors $\overline A_{f,g_1,\dots,g_r}$
est inclus dans l'anneau des fonctions analytiques sur l'espace de Berkovich associ\'e \`a $U$.
\end{rema}

On note $\overline \O$ et $\O$ les faisceaux associ\'es aux pr\'efaisceaux
$$U_{f,g_1,\dots,g_d}\mapsto \overline A_{f,g_1,\dots,g_d}
\quad{\rm et}\quad
U_{f,g_1,\dots,g_d}\mapsto A_{f,g_1,\dots,g_d}.$$

\begin{lemm}\label{adoc4}
$\overline\O(U_{f,g_1,\dots,g_d})= \overline A_{f,g_1,\dots,g_d}$ et
$\O(U_{f,g_1,\dots,g_d})= A_{f,g_1,\dots,g_d}$.
\end{lemm}
\begin{proof}
Notons juste $U$ l'ouvert standard $U_{f,g_1,\dots,g_d}$ et $\overline A$ l'anneau
$\overline A_{f,g_1,\dots,f_d}$.
Soit $F\in \overline\O(U)$.  Il existe donc, pour tout $v\in U$, un ouvert standard $U(v)=U_{f_v,g_{v,1},\dots,g_{v,d_v}}$
tel que $F_{|U(v)}\in \overline A_{f_v,g_{v,1},\dots,g_{v,d_v}}$.

Soient $\eta_i$, $i\in I$, les points g\'en\'eriques des composantes irr\'eductibles
de $\overline X={\rm Spec}(\overline A)$.  Alors $U'=\cup_i U(\eta_i)$ est un ouvert de Zariski
de $\overline X$ et $f_{|U'}$ appartient \`a l'anneau total des fractions
${\rm Fr}(\overline A)$ de $\overline A$.  Par ailleurs $f$ est une fonction analytique
sur $\overline X^{\rm an}$ (espace de Berkovich associ\'e \`a $\overline X$); on en d\'eduit que
les p\^oles de $F$, vu comme \'el\'ement de ${\rm Fr}(\overline A)$, ne sont qu'apparents,
et donc que $F$ se prolonge en une fonction sur $X$, i.e. $F\in\overline A$. 

Cela prouve le premier \'enonc\'e.  Pour en d\'eduire le second, partons de
$F\in\O(U)$; il existe alors un recouvrement fini de $U$ par des
ouverts standard $U_i=U_{f_i,g_{i,1},\dots,g_{i,d_i}}$ tel que
$F_{|U_i}\in A_i=A_{f_i,g_{i,1},\dots,g_{i,d_i}}$.
Soit $A_{i,j}=A_{f_if_j,g_{i,1},\dots,g_{i,d_i},g_{j,1},\dots,g_{j,d_j}}$
et soient 
$$B={\rm Ker}\big[\oplus_iA_i\to\oplus_{i,j}A_{i,j}\big],\quad
\overline B={\rm Ker}\big[\oplus_i\overline A_i\to\oplus_{i,j}\overline A_{i,j}\big]$$
(La fl\`eche $A_i\oplus A_j\to A_{i,j}$ 
\'etant $(f_i,f_j)\mapsto {f_i}_{|U_i\cap U_j}-{f_j}_{|U_i\cap U_j}$.)
La fl\`eche naturelle $A\to\oplus_i A_i$ (resp.~$\overline A\to\oplus_i\overline A_i$)
induit une injection $A\hookrightarrow B$ (resp.~$\overline A\hookrightarrow \overline B$).
On cherche \`a prouver que la premi\`ere injection est un isomorphisme sachant que
la seconde en est un.  

Il suffit de prouver l'\'enonc\'e modulo~$p^r$, avec $r>0$ assez petit, car tout est complet
pour la topologie $p$-adique. Le r\'esultat est donc une cons\'equence des
isomorphismes $A/p^r\cong(\O_C/p^r)\otimes_{k_C}\overline A$
et $B/p^r\cong(\O_C/p^r)\otimes_{k_C}\overline B$.
\end{proof}

\begin{defi}\label{adoc5}
{\rm (i)}
Un espace annel\'e de la forme ${\rm Sp}^{\rm ado}(A)$, avec $A$ une $\O_C$-alg\`ebre de Tate,
ou un de ses ouverts standard est un {\it affine adoque}.  Un {\it sch\'ema adoque}
est un espace annel\'e localement isomorphe \`a un affine adoque.

{\rm (ii)} Si $Y$ est un affine adoque, on pose $\O(Y^{\rm gen})=\O(Y)[\frac{1}{p}]$.
On renvoie au \S\,\ref{appen21} pour des consid\'erations sur {\it la fibre g\'en\'erique
d'un sch\'ema adoque}.
\end{defi}

\subsubsection{Boules ouvertes, jambes, cercles fant\^omes}\label{adoc6}
En dimension~$1$, en prenant le voisinage infinit\'esimal de points bien choisis,
on obtient les affines adoques suivants:

\vskip.1cm
\noindent$\bullet$ {\it Cercle fant\^ome} 

C'est l'affine adoque $Y$ de fonctions globales 
$A=\O_C[[ T,T^{-1}\rangle$. Alors 
$\overline A=k_C((T))$, et $Y$ ne poss\`ede qu'un point (la valuation ${\rm ord}_T$),
et donc aussi un unique ouvert non vide $\{{\rm ord}_T\}$, et on a $\O(\{{\rm ord}_T\})=
\O_C[[ T,T^{-1}\rangle$.

Les points classiques de $Y$ sont $Y(C)=\emptyset$ et $Y(C^\flat)={\goth m}_{C^\flat}\moins\{0\}$.
Autrement dit, $Y$ n'a pas de points en caract\'eristique $0$ (d'o\`u son aspect fant\^ome),
mais en caract\'eristique~$p$ on obtient un ouvert de la droite affine analytique (la boule ouverte
priv\'ee de son centre).

\vskip.1cm
\noindent$\bullet$ {\it Boule ouverte adoque}

C'est l'espace l'affine adoque $Y$ de fonctions globales $A=\O_C[[ T]] $.
Alors $\overline A=k_C[[T]]$, et $Y$ ne poss\`ede que deux points (le point classique $T=0$
et la valuation ${\rm ord}_T$),
et deux ouverts non vides $\{{\rm ord}_T\}$ et $Y$. On a $\O(\{{\rm ord}_T\})=
\O_C[[ T,T^{-1}\rangle$ et $\O_Y(Y)=\O_C[[ T]] $.
En particulier, $Y$ contient le cercle fant\^ome
${\rm Sp}^{\rm ado}(\O_C[[ T,T^{-1}\rangle)$ comme sous-sch\'ema adoque ouvert.

Les points classiques de $Y$ sont $Y(C)={\goth m}_C$ et $Y(C^\flat)={\goth m}_{C^\flat}$.
Autrement dit, $Y$ est la boule ouverte unit\'e en caract\'eristique~$0$ et en caract\'eristique~$p$.

\vskip.1cm
\noindent$\bullet$ {\it Jambe de longueur~$r$} 

C'est l'affine adoque $Y$ de fonctions globales $A=\O_C[[T_1,T_2]] /(T_1T_2-p^r)$.
Alors
$\overline A=k_C[[T_1,T_2]]/(T_1T_2)$, et $Y$ poss\`ede trois points: le point classique
d'\'equation $T_1=T_2=0\in k_C$ et les valuations ${\rm ord}_{T_1}$ 
(qui se factorise \`a travers $\overline A/T_2$) 
et ${\rm ord}_{T_2}$ (qui se factorise \`a travers $\overline A/T_1$).  

Les ouverts non vides
sont $Y$, $\{{\rm ord}_{T_1}\}$ et $\{{\rm ord}_{T_2}\}$; notons que
le seul ouvert contenant $0$ est $Y$.

Le faisceau structural est donn\'e par
$\O_Y(Y)=\O_C[[ T_1,T_2]] /(T_1T_2-p^r)$, $\O_Y(\{{\rm ord}_{T_i}\})=\O_C[[T_i,T_i^{-1}\rangle$.
Il en r\'esulte que $Y$ contient, comme sous-espaces adoques ouverts, les deux cercles fant\^omes
$Y_i={\rm Sp}^{\rm ado}(\O_C[[ T_i,T_i^{-1}\rangle)$, pour $i=1,2$.

En ce qui concerne les points classiques, en caract\'eristique $0$, on obtient la
couronne ouverte $0<v_p(T_1)<r$ (de longueur $r$) et, en caract\'eristique $p$, deux boules
unit\'e ouvertes recoll\'ees en leurs centres.

\Subsubsection{La boule unit\'e ferm\'ee}\label{adoc7}

C'est l'affine adoque associ\'ee \`a 
$A=\O_C\langle X\rangle$ (et donc $\overline A=k_C[X]$).

\vskip.1cm
\noindent$\bullet$ {\it Les points}.--- L'espace ${\rm Sp}^{\rm ado}(A)$ a trois types de points:

$\diamond$ Le {\og point g\'en\'erique\fg}, i.e. la valuation triviale sur $A$ ($v(f)=0$ si $f\neq 0$, $v(0)=\infty$).

$\diamond$ Si $a\in k_C$, le point classique correspondant \`a $a$, i.e. la valuation d\'efinie
par $v_a(f)=v_{\rm triv}(f(a))$, o\`u $v_{\rm triv}$ est la valuation triviale sur $k_C$.

$\diamond$ Si $a\in \piqp(k_C)$, la valuation $f\mapsto {\rm ord}_a(f)$, o\`u ${\rm ord}_a(f)$
est l'ordre du z\'ero de $f$ en~$a$ (si $a=\infty$, alors ${\rm ord}_a(f)=-\deg f$).

Du point de vue points classiques, on a $Y(C)=\O_C$ et $Y(C^\flat)=C^\flat$. Autrement dit,
en caract\'eristique~$0$ on obtient la boule unit\'e ferm\'ee et en caract\'eristique~$p$,
la droite affine.

\vskip.1cm
\noindent$\bullet$ {\it La topologie}.---
Une base d'ouverts de la topologie adoque est donn\'ee par:

$\diamond$
les compl\'ementaires de sous-ensembles finis de points classiques.. 

$\diamond$ les $V(a)=\{v_a,{\rm ord}_a\}$, pour $a\in k_C$, (voisinage infinit\'esimal du point $a$,
d\'efini par $\{v,\ v(T-a)>0\}$),

$\diamond$ les $\ocirc{V}(a)=\{{\rm ord}_a\}$, pour $a\in \piqp(k_C)$ (si $a\in k_C$, c'est l'intersection
de $V(a)$ et de l'ouvert $U_{T- a}$, et si $a=\infty$, c'est l'ouvert de $U_{T\neq 0}$
d\'efini par $v(T^{-1})>0$),

\begin{rema}
(i) Le {\og point g\'en\'erique\fg} est ferm\'e mais tr\`es gros: le compl\'ementaire
de tout ouvert le contenant est inclus dans la r\'eunion de $\ocirc{V}(\infty)$
et d'un nombre fini de~$V(a)$.

(ii)
Par comparaison, une base d'ouverts adiques est donn\'ee par:

$\diamond$ les compl\'ementaires d'un nombre fini de points classiques,

$\diamond$ le compl\'ementaire de $V(a)$ si $a\in k_C$ (d\'efini par $v(T-a)\leq v(1)$)
ou de $V(\{{\rm ord}_\infty\})$ (d\'efini par $v(T)\geq v(1)$).

Pour la topologie adique, le point g\'en\'erique est dense (contrairement
\`a la topologie adoque), et ${\rm ord}_a$ est dense
dans $V(a)$ (comme pour la topologie adoque).

(iii)
Le sch\'ema formel associ\'e \`a $A$ a pour espace topologique
sous-jacent le sch\'ema ${\rm Spec}(k_C[T])$; il n'a pour points que le point g\'en\'erique
et les points classiques, et pour ouverts  non vides uniquement les compl\'ementaires
de sous-ensembles finis de points classiques (et le point g\'en\'erique est dense).
\end{rema}

\noindent $\bullet$ {\it Le faisceau structural adoque}

--- Si $a\in k_C$, alors $\O_Y(V(a))=\O_C[[T-a]]$ (i.e.~$V(a)$ est une boule unit\'e ouverte).

--- Si $a\in \piqp(k_C)$, alors $\ocirc V(a)$ est un cercle fant\^ome (on a
$\O_Y(V(a))=\O_C[[ T-a,(T-a)^{-1}\rangle$, si $a\in k_C$, et $\O_Y(\ocirc{V}(\infty))=
\O_C[[ T^{-1},T\rangle$). Remarquons que le cercle fant\^ome $\ocirc{V}(\infty)$
a un statut sp\'ecial car il est 
aussi ferm\'e contrairement \`a tous les autres (l'adh\'erence de $\ocirc{V}(a)$ est $V(a)$, si $a\in k_C$).

--- Si $a_1,\dots,a_k\in k_C$, alors $\O_Y(Y\moins\{a_1,\dots,a_k\})=
\O_C\big\langle T,\frac{1}{T-a_1},\dots,\frac{1}{T-a_k}\big\rangle$.
(comme pour le sch\'ema formel correspondant).

\begin{rema}
{\rm (i)}
 $U=\sqcup_{a\in k_C} V(a)$ est un ouvert strict de $B$ (son compl\'ementaire
est constitu\'e du point g\'en\'erique et de ${\rm ord}_\infty$). 
On a $\O_Y(U)=\prod_{a\in k_C}\O_C[[T-[a]]]$ (une sorte de compl\'et\'e profini
de $\O_C\langle T\rangle$).

{\rm (ii)} Si $U$ est un ouvert de Zariski (i.e. le compl\'ementaire d'un ensemble fini $A$
de points classiques), alors $U$ contient les cercles fant\^omes $\ocirc V(a)$ correspondant
aux $a\in A$ (ainsi que la cercle fant\^ome $\ocirc V(\infty)$).
\end{rema}

\Subsubsection{Affino\"{\i}des}\label{adoc8}
La boule unit\'e adoque est un cas particulier de $\O_C$-short, cf.~\no\ref{adoc14}.
C'est aussi 
un cas particulier d'affine adoque $Y^{\rm ado}$
associ\'e \`a un affino\"{\i}de~$Y$ (i.e.~associ\'e \`a $\O^+(Y)$) de dimension~$1$.
L'espace topologique $Y^{\rm ado}$ est la r\'eduction canonique de $Y$ du \no\ref{constr4}, vue
comme espace de Berkovich au lieu de sch\'ema (ce qui rajoute des points et des ouverts).

Comme la boule unit\'e, $Y^{\rm ado}$ a trois types de points: les points g\'en\'eriques
des composantes irr\'eductibles de la r\'eduction canonique, les points classiques,
et des valuations de rang~$1$ pour tout point des compactifi\'ees lisses des composantes
irr\'eductible.

Une base d'ouverts de la topologie est constitu\'ee des ouverts de Zariski
(compl\'ementaire d'ensembles finis
de points classiques), des voisinages infinit\'esimaux des points classiques
(le voisinage infinit\'esimal d'un point singulier peut \^etre tr\`es compliqu\'e),
et des cercles fant\^omes associ\'es aux valuations de rang~$1$.

\begin{rema}
La fronti\`ere $\partial^{\rm ad}Y$ de $Y$ du (ii) de la rem.~\ref{constr5}
est un ouvert de $Y^{\rm ado}$ constitu\'e de cercles fant\^omes.
\end{rema}

\Subsection{$R$-alg\`ebres de Tate}\label{adoc9}
\subsubsection{\'Epaississements de $\O_C$}\label{adoc10}
Un {\it \'epaississement $R$ de $\O_C$} est un anneau muni d'un morphisme
surjectif $\theta_R:R\to\O_C$ tel que, si $r\in\Q_+^\dual$, et si $I_r=\theta_R^{-1}(p^r\O_C)$,
alors $R$ est s\'epar\'e et complet pour la topologie $I_r$-adique.  En particulier,
$R$ est local, de corps r\'esiduel $k_C$; on note ${\goth m}_R$ son id\'eal maximal.
Des exemples naturels d'\'epaississements de $\O_C$ sont $\O_C$, $\ainf$ ou $\O_C[[T_1,\dots,T_d]]$.

Si $R$ est un \'epaississement de $\O_C$, on munit $R$ de la topologie 
$I_r$-adique (qui ne d\'epend pas du choix de $r\in\Q_+^\dual$) et on note ${\rm Sp}(R)$ l'ensemble
des id\'eaux premiers ferm\'es ${\goth p}$ de $R$, tels que ${\goth p}[\frac{1}{p}]$
soit un id\'eal maximal de $R[\frac{1}{p}]$.  

-- Si $R=\O_C[[T_1,\dots,T_d]]$,
alors ${\rm Sp}(R)$ n'est autre que l'ensemble ${\goth m}_R^d$ 
des $C$-points de la boule
unit\'e ouverte de dimension~$d$.  

-- Le cas $R=\ainf$ est plus amusant (cf. \cite[cor. 3.3]{FFpreface} pour ce r\'esultat de Fargues et Fontaine): 
dans ce cas ${\rm Sp}(R)$ a un \'el\'ement un peu \'etonnant, \`a savoir
${\rm Ker}\,\theta_0$ pour lequel $R/{\rm Ker}\,\theta_0=\O_{\breve C}$;
tous les autres \'el\'ements de ${\rm Sp}(\ainf)$ sont de la forme
$(p-[a])$, avec $a\in{\goth m}_{C^\flat}$ (non uniquement d\'etermin\'e),
et $\ainf/(p-[a])$ est l'anneau des entiers d'un corps $C_a$ alg\'ebriquement clos,
complet pour~$v_p$ et dont le bascul\'e $C_a^\flat$ est $C^\flat$.  
(Si $a=p^\flat$ -- et donc $[a]=\tilde p$ -- 
ce quotient n'est autre que $\O_C$.) 

On \'ecrit $\tilde p=0$ pour le quotient $\ainf/{\rm Ker}\,\theta_0$
et $p=[a]$ pour $\ainf/(p-[a])$ (et $p=\tilde p$, si $a=p^\flat$, et donc $\ainf/(p-[a])=\O_C$).

\subsubsection{$R$-alg\`ebres de Tate}\label{adoc11}
Soit $R$ un \'epaississement de $\O_C$.
Un $R$-module $M$ est dit {\it orthonormalisable} s'il est isomorphe
au module $\ell^\infty_0(I,R)$ des suites $(x_i)_{i\in I}$ tendant vers~$0$
en $\infty$: un tel isomorphisme fournit 
une famille $(e_i)_{i\in I}$ d'\'el\'ements de $M$ telle que tout
\'el\'ement $x$ de $M$ puisse s'\'ecrire, de mani\`ere unique, sous
la forme $\sum_{i\in I}x_ie_i$ avec $x_i\to 0$ quand $i\to\infty$;
une telle famille est appel\'ee {\it une base orthonormale de $M$ sur $R$}.

Une {\it $R$-alg\`ebre de Tate} $B$ est une $R$-alg\`ebre orthonormalisable, 
quotient d'un
$R\langle x_1,\dots,x_n\rangle$.
On pose $\overline B=k_C\otimes_R B$
(c'est un quotient de $k_C[x_1,\dots,x_n]$).

Soit $B$ une $R$-alg\`ebre de Tate de la forme\footnote{En particulier $R$ est suppos\'ee
contenir $\O_K$; par exemple, si $R=\ainf$, cela impose $K\subset\breve{C}$.}
 $R\wotimes_{\O_K}B_K$, o\`u
$B_K$ est une $\O_K$-alg\`ebre de Tate, et $K$ est un sous-corps complet de $C$, 
{\it de valuation discr\`ete}. 
\begin{lemm}\label{PP1}
Soient $r>0$ et $R^{(r)}=\O_K+I_r$. Si $(e_i)_{i\in I}$
est une famille d'\'el\'ements de $R^{(r)}\wotimes_{\O_K}B_K$,
les conditions suivantes sont \'equivalentes:

{\rm (i)} $(e_i)_{i\in I}$ est une base orthonormale de $B$ sur $R$,

{\rm (ii)} $(\overline e_i)_{i\in I}$ est une base alg\'ebrique de $\overline B$ sur $k_C$.
\end{lemm}
\begin{proof}
L'implication (i)$\Rightarrow$(ii) est imm\'ediate; prouvons la r\'eciproque.
Soit $(f_i)_{i\in I}$ une famille d'\'el\'ements de $B_K$ telle que $f_i$ ait
pour image $\overline e_i$ dans $\overline B$.  Les $f_i$ forment une
base orthonormale de $B_K$ sur $\O_K$ 
(on est
dans le cas de valuation discr\`ete o\`u l'\'equivalence entre (i) et (ii) est parfaitement
classique),
et donc aussi de $B$ sur $R$.

Soit $M$ la matrice des $e_i$ dans la base des $f_i$.  Par construction, $M$ est
\`a coefficients dans $R^{(r)}$ et congrue \`a $1$ modulo ${\goth m}_R$, et donc
$M=1-N$, avec $N$ \`a coefficients dans $I_r$.  Il s'ensuit que $M$ est inversible,
d'inverse $1+N+N^2+N^3+\cdots$ (la s\'erie converge car $N^k$ est \`a coefficients
dans $I_r^k$ et $R$ est s\'epar\'e et complet pour la topologie $I_r$-adique).
Les $e_i$, pour $i\in I$, forment donc une base orthonormale de $B$ sur $R$.
\end{proof} 

\begin{lemm}\label{PP2}
Soit $B$ de la forme $R\wotimes_{\O_K}B_K$ comme ci-dessus, et 
soit $A$ une sous-$R$-alg\`ebre ferm\'ee de $B$ v\'erifiant:

$\bullet$ $\overline A$ est de type fini sur $k_C$,

$\bullet$ il existe $r>0$ et une famille $(e_i)_{i\in I}$
d'\'el\'ements de $R^{(r)}\wotimes_{\O_K}B_K$ dont les
images dans $B/A$ forment une base orthonormale de $B/A$.

Alors $A$ est une $R$-alg\`ebre de Tate.
\end{lemm}
\begin{proof}
Il existe $\overline z_1,\dots,\overline z_s\in\overline A$ engendrant $\overline A$ sur
$k_C$, et donc $J\subset \N^s$ tel que les $\overline{\bf z}^{\bf j}=
\overline z_1^{j_1}\cdots\overline z_s^{j_s}$, pour ${\bf j}=(j_1,\dots,j_s)\in J$,
forment une base de $\overline A$ sur $k_C$.

Choisissons des rel\`evements $z_1,\dots,z_s$ de $\overline z_1,\dots,\overline z_s$ dans $A$.
Quitte \`a diminuer $r$, on peut supposer que les $z_i$ appartiennent \`a 
$R^{(r)}\wotimes_{\O_K}B_K$.  Les ${\bf z}^{\bf j}$, pour ${\bf j}\in J$
appartiennent aussi \`a $R^{(r)}\wotimes_{\O_K}B_K$, et il r\'esulte
du lemme~\ref{PP1} que les ${\bf z}^{\bf j}$, pour ${\bf j}\in J$, et les
$e_i$, pour $i\in I$, forment une base orthonormale de $B$ sur $R$.
On en d\'eduit que les ${\bf z}^{\bf j}$, pour ${\bf j}\in J$,
forment une base orthonormale de $A$ sur $R$, et donc que $A$ est orthonormalisable
et quotient de $R\langle z_1,\dots,z_s\rangle$.

Cela permet de conclure.
\end{proof}

\subsubsection{Rel\`evement de morphismes}\label{adoc12}
Si $K$ est un sous-corps ferm\'e de $C$ et si $A$ est une alg\`ebre de Tate 
sur $\O_K$,
posons $\overline A=A/{\goth m}_KA$.  Rappelons que, si $A$ est sans $\O_K$-torsion, alors $A$ est 
{\it formellement lisse} sur $\O_K$ si et seulement si
$\overline A$ est lisse sur~$k_K$.
\begin{prop}\label{QQ1}
{\rm (\cite[th.\,A-1]{Cn85})}
Soient $A, B$ des alg\`ebres de Tate sur $\O_K$, 
avec $A$
formellement lisse sur $\O_K$.
Si $\overline s:\overline A\to\overline B$ est un morphisme 
de $\O_K$-alg\`ebres,
il existe $s:A\to B$ relevant $\overline s$.
\end{prop}

\begin{coro}\label{QQ2}
{\rm (i)} Si $A$ est une alg\`ebre de Tate
formellement lisse sur $\O_{\breve C}$, il existe
$\varphi:A\to A$ relevant $x\mapsto x^p$ sur $\overline A$ {\rm (un tel $\varphi$
est un {\it frobenius de $A$})}.

{\rm (ii)} Si $B$ est une alg\`ebre de Tate 
formellement lisse sur $\O_C$, il existe une alg\`ebre de Tate
$\breve B$ sur $\O_{\breve C}$ telle que l'on ait $B\cong\O_C\wotimes_{\O_{\breve C}}
\breve B$.
\end{coro}
\begin{proof}
Le (i) est une application directe de la prop.~\ref{QQ1} (avec $B=A$).
Pour prouver le (ii), on part d'un rel\`evement formellement lisse $A$ de $\overline B$
sur $\O_{\breve C}$ (il en existe par la propri\'ete de rel\`evement infinit\'esimal 
des morphismes lisses),
 et on observe que 
c'est une alg\`ebre de Tate sur $\O_{\breve C}$. Alors 
$A':=\O_C\wotimes_{\O_{\breve C}} {A}$ est une alg\`ebre de Tate formellement lisse sur 
$\O_C$ et $A'/{\goth m}A'\cong \overline{B}$. La prop.~\ref{QQ1} fournit un rel\`evement $s: A'\to B$ de l'identit\'e de $\overline{B}$. Il reste \`a v\'erifier que $s$ est un isomorphisme.

Soit $r>0$. Le morphisme induit $s_r: A'/p^r\to B/p^r$ est un morphisme d'alg\`ebres de type
fini sur $\O_C/p^r$. Il s'ensuit que, pour $t < r$ assez petit, $s_t$ est un isomorphisme
(on a $A'/p^t\cong (\O_C/p^t)\otimes\overline{B}$ et $B/p^t\cong (\O_C/p^t)\otimes\overline{B}$,
si $t$ est assez petit, et
$s_t$ est l'identit\'e puisque $s$ est un rel\`evement de l'identit\'e). On conclut en utilisant
la compl\'etude de $A'$ et $B$ pour la topologie $p$-adique.
\end{proof}

\begin{rema}\label{QQ3}
 Il r\'esulte de la preuve du (ii) qu'une alg\`ebre de Tate 
formellement lisse sur $\O_C$ est, \`a isomorphisme pr\`es,
compl\`etement d\'etermin\'ee par $\overline B$.
\end{rema}

\begin{rema}\label{QQ4}
Une {\it $\dagger$-alg\`ebre} sur $\O_K$ est
une $\O_K$-alg\`ebre plate $A$ de la forme
$\O_K[x_1,\dots,x_n]^\dagger/I$, o\`u $I$ est un id\'eal de type fini.
On dit que $A$ est {\it formellement lisse} sur $\O_K$ si
$\overline A=A/{\goth m}_KA$ est lisse sur $k_K$.
La prop.~\ref{QQ1} s'\'etend \`a ce cadre: si $A, B$ sont
des $\dagger$-alg\`ebres sur $\O_K$, 
avec $A$
formellement lisse sur $\O_K$, et si
$\overline s:\overline A\to\overline B$ est un morphisme de $k_K$-alg\`ebres,
il existe $s:A\to B$ relevant $\overline s$.
Cela permet de construire des
rel\`evements de Frobenius surconvergents.
\end{rema}

\Subsection{Shorts, jambes et cercles fant\^omes}\label{adoc13}

\subsubsection{Shorts}\label{adoc14}
{\it Un short} $Y$ est le sch\'ema adoque associ\'e \`a un sch\'ema formel 
affine, lisse sur $\O_{C}$,
dont la fibre g\'en\'erique $Y_{C}$
est un affino\"{\i}de connexe, de dimension~$1$.
On a $\O(Y)=\O^+(Y_{C})$ et
$\O(Y)$ est un quotient formellement lisse 
de $\O_{C}\langle x_1,\dots,x_n\rangle$, pour un certain~$n$.

\begin{rema}\label{tetrapil21}
{\rm (i)} D'apr\`es le (ii) du cor.~\ref{QQ2}, un short $Y$ admet un mod\`ele
$\breve Y$ sur $\O_{\breve C}$, i.e.~$\O(Y)=
R\wotimes_{\O_{\breve C}}\O({\breve Y})$, et $\breve Y$ est unique \`a isomorphisme pr\`es.

{\rm (ii)}
Si $R$ est un \'epaississement de $\O_C$, on d\'efinit {\it un $R$-short}
comme l'extension des scalaires \`a $R$ d'un short $\breve Y$ sur $\O_{\breve C}$.
Si $Y$ est un short, on peut le plonger dans un $R$-short $Y_R$ pour tout \'epaississement $R$ de $\O_C$: il
suffit de choisir un mod\`ele $\breve Y$ de $Y$
sur $\O_{\breve Y}$, et de poser $\O(Y_R)=R\wotimes_{\O_{\breve C}}\O({\breve Y})$.

{\rm (iii)} 
En particulier, on peut plonger
un short $Y$ dans un $\ainf$-short $\widetilde Y$.
La fibre en $p=\tilde p$ de $\widetilde Y$ n'est autre que $Y$, tandis que celle
en $\tilde p=0$ est $\breve Y$.

Un choix de frobenius $\varphi$ sur $\breve Y$ en induit un sur $\widetilde Y$
puisqu'on dispose d'un frobenius sur $\ainf$.  Alors $\breve Y$ est stable par le
frobenius agissant sur $\widetilde Y$, et la restriction de ce frobenius \`a $\breve Y$
est celui dont on est parti.
\end{rema}
\vskip.2cm
\begin{center}
\begin{tikzpicture}
[scale=.8]
\draw[very thick]  (0,-2) ..controls +(2.5,0) and +(2.5,0).. (0,8);
\draw[very thick]  (0,-2) ..controls +(-2.5,0) and +(-2.5,0).. (0,8);
\draw[very thick]  (-1,1) ..controls +(0.3,-.7) and +(-0.3,-.7).. (1,1);
\draw[very thick]  (-.75,.75) ..controls +(.2,.3) and +(-.2,.3).. (.75,.75);
\draw[very thick]  (-1,3) ..controls +(0.3,-.7) and +(-0.3,-.7).. (1,3);
\draw[very thick]  (-.75,2.75) ..controls +(.2,.3) and +(-.2,.3).. (.75,2.75);
\draw[very thick]  (-1,5) ..controls +(0.3,-.7) and +(-0.3,-.7).. (1,5);
\draw[very thick]  (-.75,4.75) ..controls +(.2,.3) and +(-.2,.3).. (.75,4.75);
\draw[very thick, dotted, color=red, fill=gray!20] (1,2) circle (.27 and .3);
\draw[red](.2,2) node{$]P_{a_3}[$};
\draw[very thick, dotted, color=red, fill=gray!20] (1,4.1) circle (.27 and .3);
\draw[red](.2,4.1) node{$]P_{a_2}[$};
\draw[very thick, dotted, color=red, fill=gray!20] (-.3,6.8) circle (.3);
\draw[red](.5,6.8) node{$]P_{a_1}[$};
\draw[very thick,color=blue] (.1,-1.5) circle (.35);
\draw[color=blue] (-.05,-1.4) circle (.12);
\draw[color=blue] (.2,-1.6) circle (.12);
\draw[very thick,color=blue] (.5,-.3) circle (.35);
\draw[very thick,color=blue] (-.8,.1) circle (.35);
\draw[color=blue] (-.95,.2) circle (.12);
\draw[color=blue] (-.7,0) circle (.12);
\draw[very thick,color=blue] (1.35,1) circle (.2 and .3);
\draw[very thick,color=blue] (.1,1.4) circle (.3);
\draw[very thick,color=blue] (-.8,1.7) circle (.3);
\draw[very thick,color=blue] (-1.2,3.4) circle (.25 and .3);
\draw[very thick,color=blue] (-.7,4.05) circle (.3);
\draw[very thick,color=blue] (.4,6) circle (.3);
\draw[very thick,color=blue] (.1,7.5) circle (.3);
\draw[very thick]  (8,-2) ..controls +(.5,2) and +(.5,-2).. (8,3);
\draw[very thick]  (8,3) ..controls +(-.5,2) and +(-.5,-2).. (8,8);
\draw[ball color=red] (7.75,6.8) circle (.12); \draw[red] (8.8,6.8)node{$(P_{a_1},0^+)$};
\draw[ball color=red] (7.75,4.1) circle (.12); \draw[red] (8.8,4.1)node{$(P_{a_2},0^+)$};
\draw[ball color=red] (8.2,2) circle (.12); \draw[red] (9.25,2)node{$(P_{a_3},0^+)$};
\draw[ball color=blue] (8.1,-1.5) circle (.12);
\draw[ball color=blue] (8.33,-.3) circle (.12);
\draw[ball color=blue] (8.35,.2) circle (.12);
\draw[ball color=blue] (8.35,1) circle (.12);
\draw[ball color=blue] (8.3,1.4) circle (.12);
\draw[ball color=blue] (8.25,1.7) circle (.12);
\draw[ball color=blue] (7.9,3.4) circle (.12);
\draw[ball color=blue] (7.83,3.75) circle (.12);
\draw[ball color=blue] (7.65,6) circle (.12);
\draw[ball color=blue] (7.9,7.5) circle (.12);
\draw[color=blue] (.15,7.78) -- (7.9,7.5); \draw[color=blue] (.15,7.22) -- (7.9,7.5);
\draw[color=blue] (.45,6.28) -- (7.65,6); \draw[color=blue] (.45,5.72) -- (7.65,6);
\draw[color=blue] (-1.15,3.68) -- (7.9,3.4); \draw[color=blue] (-1.15,3.12) -- (7.9,3.4);
\draw[color=blue] (1.4,1.28) -- (8.35,1); \draw[color=blue] (1.4,.72) -- (8.35,1);
\draw[color=blue] (.16,-1.18) -- (8.1,-1.5); \draw[color=blue] (.16,-1.82) -- (8.1,-1.5);
\draw[dotted, color=red] (1.05,4.38)--(7.8,4.1); \draw[dotted, color=red] (1.05,3.82)--(7.8,4.1);

\draw[thick,color=blue,{->}] (0,-4) -- +(-90:.6); \draw[thick,color=blue,{->}] (0,-4) -- +(-70:.6);
\draw[thick,color=red,{->}] (0,-4) -- +(-50:.6); \draw[thick,color=blue,{->}] (0,-4) -- +(-30:.6);
\draw[thick,color=blue,{->}] (0,-4) -- +(-10:.6); \draw[thick,color=blue,{->}] (0,-4) -- +(10:.6);
\draw[thick,color=blue,{->}] (0,-4) -- +(30:.6); \draw[thick,color=blue,{->}] (0,-4) -- +(50:.6);
\draw[thick,color=blue,{->}] (0,-4) -- +(70:.6); \draw[thick,color=blue,{->}] (0,-4) -- +(90:.6);
\draw[thick,color=blue,{->}] (0,-4) -- +(110:.6); \draw[thick,color=blue,{->}] (0,-4) -- +(130:.6);
\draw[thick,color=blue,{->}] (0,-4) -- +(150:.6); \draw[thick,color=blue,{->}] (0,-4) -- +(170:.6);
\draw[thick,color=red,{->}] (0,-4) -- +(190:.6); \draw[thick,color=blue,{->}] (0,-4) -- +(210:.6);
\draw[red](0,-4)+(190:.8) node{$a_1$};
\draw[red](0,-4)+(250:.8) node{$a_2$};
\draw[red](0,-4)+(-50:.8) node{$a_3$};
\draw[thick,color=blue,{->}] (0,-4) -- +(230:.6); \draw[thick,color=red,{->}] (0,-4) -- +(250:.6);
\draw[ball color=black] (0,-4) circle (.2); 
\draw[color=blue] (0,-4)+(90:.6) -- (0,-2.8);
\draw[color=blue] (0,-2.8)--(.15,-2.2); \draw[color=blue] (0,-2.8)--(.35,-2.2);
\draw[color=blue] (.15,-2.2)--(-.1,-1.4); 
\draw[color=blue] (.15,-2.2)--(0,-1.45); 
\draw[color=blue] (.35,-2.2)--(.2,-1.6);
\draw[color=blue] (.35,-2.2)--(.25,-1.65);
\draw[color=blue] (.35,-2.2)--(.15,-1.55);
\draw[color=blue] (0,-4)+(110:.6) -- (-1,-2);
\draw[color=blue] (-1,-2)--(-1.2,-.5); \draw[color=blue] (-1,-2)--(-1,-.5);
\draw[color=blue] (-1.2,-.5)--(-.95,.25); 
\draw[color=blue] (-1.2,-.5)--(-.90,.2); 
\draw[color=blue] (-1,-.5)--(-.75,0);
\draw[color=blue] (-1,-.5)--(-.65,.05);

\draw(-1.8,-4)node{$\Gamma^{\rm ad}(Y)$};
\draw (3.5,-5)node{Un short $Y$};
\draw (9,-2)node{$Y^{\rm sp}$};

\end{tikzpicture}
\end{center}

{\Small Ce dessin repr\'esente un short $Y$, avec l'ensemble de ses points classiques (la surface de Riemann),
sa fibre sp\'eciale $Y^{\rm sp}$ (la courbe) et son squelette (le point h\'eriss\'e de fl\`eches).
Il est obtenu en retirant \`a une courbe de genre $3$ avec bonne r\'eduction
les tubes $]P_{a_1}[$, $]P_{a_2}[$, $]P_{a_3}[$ des points marqu\'es $P_{a_1}$, $P_{a_2}$, $P_{a_3}$
de $Y^{\rm sp}$.

$\bullet$ Le squelette (adique) $\Gamma^{\rm ad}(Y)$ 
est constitu\'e du point noir (point de Gauss de $Y$) 
et des trois fl\`eches rouges $a_1$, $a_2$, $a_3$
(correspondant aux trois points marqu\'es de $Y^{\rm sp}$).  
Le squelette analytique $\Gamma^{\rm an}(Y)$ est juste le point noir.

$\bullet$ Les disques bleus de la surface de Riemann sont les points classiques des
tubes des points non marqu\'es de $Y^{\rm sp}$
(ils sont cens\'es recouvrir la surface priv\'ee de $]P_{a_1}[$, $]P_{a_2}[$, $]P_{a_3}[$); on a dessin\'e
les sp\'ecialisations de cinq de ces disques).
Les cercles rouges en pointill\'es sont les cercles fant\^omes \`a la fronti\`ere
du short.

$\bullet$ L'espace de Berkovich est l'arbre issu du point noir dont les branches aboutissent aux points
classiques
 (auxquelles il faut rajouter des branches s'arr\^etant en chemin, correspondant
aux points de type 4, puisqu'on n'a pas suppos\'e $C$ sph\'eriquement complet); il ne voit pas
les fl\`eches bleues.  

$\bullet$ Les fl\`eches bleues sont en bijection
avec les points non marqu\'es de $Y^{\rm sp}$; ce sont les points adiques de type 5 dans l'adh\'erence
du point de Gauss;
chacune est 
la racine du sous-arbre des branches aboutissant
dans le tube correspondant (on a dessin\'e une petite partie de deux de ces arbres). 

}

\subsubsection{$R$-Jambes}\label{adoc15}
Si $R$ est un \'epaississement de $\O_C$,
{\it une $R$-jambe} $Y$ est le sch\'ema adoque associ\'e \`a un anneau
de la forme $R[[T_0,T_1]]/(T_0T_1-\alpha)$ (la topologie est la $(I_r, T_0, T_1)$-adique), 
avec $\alpha\in{\goth m}_R$.
(Si $\alpha=0$, la jambe est singuli\`ere et sa longueur est $+\infty$.) 
Si $R=\O_C$, {\it la longueur de $Y$} est $v_p(\alpha)$.
Une $\O_C$-jambe est dite {\it normalis\'ee}, si $\alpha=p^r$, avec $r\in\Q_+^\dual$;
sa longueur est alors $r$.

\vskip.3cm
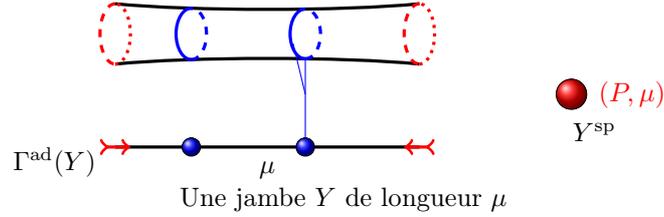
\begin{center}
\begin{tikzpicture}

\draw[very thick]  (0,.4) ..controls +(1,-.1) and +(-1,-.1).. (4,.4);
\draw[very thick]  (0,-.4) ..controls +(1,.1) and +(-1,.1).. (4,-.4);

\draw[very thick, dotted, color=red] (0,-.4) arc (-90: 90 : .2 and .4);
\draw[very thick, dashed, color=red] (0,.4) arc (90: 270 : .2 and .4);
\draw[very thick, dotted, color=red] (4,-.4) arc (-90: 90 : .2 and .4);
\draw[very thick, dashed, color=red] (4,.4) arc (90: 270 : .2 and .4);
\draw[very thick, dashed, color=blue] (2.5,-.33) arc (-90: 90 : .2 and .33);
\draw[very thick, color=blue] (2.5,.33) arc (90: 270 : .2 and .33);
\draw[very thick, dashed, color=blue] (1,-.34) arc (-90: 90 : .2 and .34);
\draw[very thick, color=blue] (1,.34) arc (90: 270 : .2 and .34);

\draw[blue] (2.5,-.33) -- (2.5,-1.4);
\draw[blue] (2.34,-.13) -- (2.5,-.8);

\draw[very thick]  (0,-1.5) -- (4,-1.5);
\draw[very thick, red, {>->}]  (-.2,-1.5) -- (.2,-1.5); 
\draw[very thick, red, {>->}]  (4.2,-1.5) -- (3.8,-1.5); 
\draw[ball color=blue] (2.5,-1.5) circle(.12);
\draw[ball color=blue] (1,-1.5) circle(.12);
\draw (2,-1.8) node{$\mu$};
\draw (-.8,-1.7) node{$\Gamma^{\rm ad}(Y)$};
\draw (6.3,-1.3) node{$Y^{\rm sp}$};

\draw[ball color=red] (6,-.8) circle (.2); \draw[red] (6.8,-.8)node{$(P,\mu)$};

\draw (3,-2.2) node{Une jambe $Y$ de longueur $\mu$};

\end{tikzpicture}
\end{center}

{\Small 
Ce dessin repr\'esente une jambe de longueur $\mu$. Sa fibre sp\'eciale est un point marqu\'e
de multiplicit\'e $\mu$, et son squelette est un segment ouvert de longueur $\mu$.
Le cyclindre repr\'esente l'ensemble des points classiques de $Y$; chaque point de type $2$
du squelette est la base d'un arbre aboutissant sur un cercle du cyclindre. Les deux fl\`eches
rouges du squelette correspondent aux deux cercles fant\^omes (repr\'esenta\'es par des cercles
rouges en pointill\'es) \`a la fronti\`ere.

}

\vskip.2cm
Si $Y$ est une $\O_C$-jambe normalis\'ee 
(avec $\O(Y)=\O_C[[T_0,T_1]]/(T_0T_1-p^r)$), on peut la plonger dans la
$\ainf$-jambe $\widetilde Y$ d\'efinie par $\O(\widetilde Y)=\ainf[[T_0,T_1]]/(T_0T_1-\tilde p^r)$.
On note $\breve Y$ la fibre de $\widetilde Y$ en $\tilde p=0$:
on a $\O(\breve Y)=\O_{\breve C}[[T_0,T_1]]/(T_0T_1)$, et donc $\breve Y$
est singuli\`ere, et obtenue en recollant deux boules ouvertes sur $\O_{\breve C}$
en leurs centres.
On munit $\breve Y\subset\widetilde Y$ de frobenius compatibles $\varphi$
par la formule $\varphi(T_i)=T_i^p$, si $i=0,1$.

\subsubsection{$R$-Cercles fant\^omes}\label{adoc16}
{\it Un $R$-cercle fant\^ome} $Y$ est le sch\'ema adoque associ\'e \`a un anneau
de la forme
$R[[T,T^{-1}\rangle$ (compl\'et\'e de $R[[T]][T^{-1}]$ pour la topologie $I_r$-adique).
 Un {\it param\`etre local de $Y$} est un \'el\'ement $z$
de $\O(Y)=R[[T,T^{-1}\rangle$ dont l'image $\overline z$ dans $k_C((T))$
v\'erifie $v_T(\overline z)=1$ (donc, en particulier, $T$ est
un param\`etre local de $Y$); on a alors un isomorphisme
$R[[z,z^{-1}\rangle\overset{\sim}{\to}\O(Y)$.

Si $Y$ est un $\O_C$-cercle fant\^ome avec $\O(Y)=\O_C[[T,T^{-1}\rangle$, on le plonge dans
le $\ainf$-cercle fant\^ome $\widetilde Y$ d\'efini par
$\O(\widetilde Y)=\ainf[[T,T^{-1}\rangle$.  On note $\breve Y$ la fibre
de $\widetilde Y$ en $\tilde p=0$ (et donc $\O(\breve Y)=\O_{\breve C}[[T,T^{-1}\rangle$).
On munit $\breve Y\subset\widetilde Y$ de frobenius compatibles $\varphi$
par la formule $\varphi(T)=T^p$.

\subsubsection{Partie polaire en un cercle fant\^ome de la fronti\`ere}\label{adoc17}
Soit $Y$ un $R$-short, et soit 
$\O^{++}(Y)={\rm Ker}(\O(Y)\to k_C\otimes_{R}\O(Y))$.
Soit $P\in \breve Y^{\rm sp}$
de multiplicit\'e $0^+$, et soit $z\in\O^+(\breve Y)$ v\'erifiant $v_{{\breve Y},P}(z)=(0,-1)$
(un tel $z$ n'existe pas forc\'ement s'il n'y a qu'un seul point de multiplicit\'e $0^+$
sur $Y^{\rm sp}$, mais existe toujours si le nombre de tels points est~$\geq 2$ gr\^ace
au th\'eor\`eme de Riemann-Roch).
Alors $z^{-1}$ est un param\`etre du cercle fant\^ome 
correspondant \`a $P$
\`a la fronti\`ere
de $\breve Y$ .
Soit $T$ un param\`etre du $R$-cercle fant\^ome d\'efini par ce cercle fant\^ome;
alors $z^{-1}\in R[[T,T^{-1}\rangle$ et
on a un isomorphisme
$\iota:R[[z^{-1},z\rangle\overset{\sim}{\to} R[[T,T^{-1}\rangle$.
\begin{lemm}\label{PP3}
Les $\iota(z^n)$, pour $n\in\N$,
 forment une base orthonormale de $R[[T,T^{-1}\rangle/TR[[T]]$ sur $R$.
En particulier, $\iota$ induit des surjections
$$\O^+(Y)\to R[[T,T^{-1}\rangle/TR[[T]]
\quad{\rm et}\quad
\O^{++}(Y)\to {\goth m}_R[[T,T^{-1}\rangle/T{\goth m}_R[[T]]$$
\end{lemm}
\begin{proof}
Il existe $\alpha\in k_C^\dual$ tel que 
$\iota(z)\equiv[\alpha]T^{-1}$ mod~${\goth m}_R\langle T^{-1}\rangle+R[[T]]$
et, quitte \`a changer $T$ en $[\alpha^{-1}]T$, on peut supposer $\alpha=1$.
Il existe alors $r>0$ tel que
$\iota(z)-T^{-1}\in I_r\O(Y)+R[[T]]$.
Alors $\iota(z^k)-T^{-k}\in I_r\O(Y)+T^{-(k-1)}R[[T]]$, pour tout $k\geq 0$.
On en d\'eduit que tout \'el\'ement $f$ de $(R/I_r)[[T,T^{-1}\rangle$ 
peut s'\'ecrire, de mani\`ere
unique, sous la forme $f^++f^-$, avec $f^+\in T(R/I_r)[[T]]$ et 
$f^-\in (R/I_r)[\iota(z)]$.
On prouve,
par r\'ecurrence, le m\^eme r\'esultat modulo $I_r^n$ et, par passage
\`a la limite, on en tire
que tout \'el\'ement $f$ de $R[[T,T^{-1}\rangle$ peut s'\'ecrire, de mani\`ere
unique, sous la forme $f^++f^-$, avec $f^+\in TR[[T]]$ 
et $f^-\in R\langle \iota(z)\rangle$.

Le r\'esultat s'en d\'eduit.
\end{proof}
\Subsection{Compactification d'un affino\"{\i}de}\label{Pconstr7}
\begin{prop} {\rm (Van der Put~\cite{vdp})}\label{Pconstr9}
Soit $Y_C$ un affino\"{\i}de sur $C$.  Alors il existe une courbe alg\'ebrique propre $X_C$
contenant $Y_C$ et telle que $X_C\moins Y_C$ soit une r\'eunion finie de disques ouverts.
\end{prop}
\begin{proof}
Soit $Y$ le sch\'ema adoque associ\'e au sch\'ema formel ${\rm Spf}(\O^+(Y_C))$,
et soit $X$ le sch\'ema adoque obtenu en recollant des boules ouvertes le long des cercles fant\^omes
\`a la fronti\`ere de $Y$.  L'espace topologique de $X$ est la compactifi\'ee $\overline{\cal Y}$
de la fibre sp\'eciale ${\cal Y}$ de ${\rm Spf}(\O^+(Y_C))$.  Il s'agit de prouver que
si on restreint le faisceau structural de $X$ aux ouverts de Zariski,
on obtient des $\O_C$-alg\`ebres de Tate: cela implique que cette restriction
est un sch\'ema formel $p$-adique\footnote{Une manifestation d'un principe GaGa (G\'eom\'etrie adoque
et G\'eom\'etrie analytique).} qui est un mod\`ele d'une courbe $X_C$, propre puisque sa
fibre sp\'eciale est propre, et qui est l'espace rigide associ\'e \`a une courbe alg\'ebrique
d'apr\`es GAGA rigide.

Si $U$
est un ouvert (de Zariski) assez petit de $\overline{\cal Y}$, il y a deux cas possibles:

$\bullet$  $U\subset {\cal Y}$
auquel cas $]U[$
est un ouvert affino\"{\i}de de $Y_C$, et on a $\O_X(U)=\O^+(]U[)$, ce qui
fait qu'il n'y a rien \`a prouver.

$\bullet$  $U\moins(U\cap{\cal Y})=\{P\}$, auquel cas $U\moins\{P\}$ est un
ouvert de ${\cal Y}$ que l'on peut supposer lisse et connexe (et donc $]U\moins\{P\}[$
est un short et on note $C_P$ le cercle fant\^ome \`a la fronti\`ere de $]U\moins\{P\}[$ 
correspondant \`a $P$; on recolle $]U\moins\{P\}[$ et le disque ouvert $D_P$
avec $\O(D_P)=\O_C[[T]]$ le long de $C_P$): 
$]U\moins\{P\}[$ est obtenu, par extension des scalaires, \`a partir d'un short $\breve Y_P$
d\'efini sur $\O_{\breve C}$ (i.e.~$\O^+(]U\moins\{P\}[)=\O_C\wotimes_{\O_{\breve C}}\O(\breve Y_P)$);
on choisit 
un param\`etre
$z\in\O^+(\breve Y_P)$ du cercle fant\^ome correspondant \`a $P$, et un isomorphisme
$\iota:\O_\C[[z,z^{-1}\rangle\overset{\sim}{\to} \O_C[[T,T^{-1}\rangle$.
La propri\'et\'e de faisceau fournit alors la formule:
\begin{align*}
\O_Y(U)=&\ {\rm Ker}\big[\O^+(]U\moins\{P\}[)\oplus \O_C[[T]]\to \O_C[[T,T^{-1}\rangle\big]
\\=&\ 
{\rm Ker}\big[\O^+(]U\moins\{P\}[)\to \frac{\O_C[[T,T^{-1}\rangle}{\O_C[[T]]}\big],
\end{align*}
o\`u l'application $\O^+(]U\moins\{P\}[)\to \O_C[[T,T^{-1}\rangle$ est la restriction de
$\iota$.
Pour simplifier les notations, posons $A=\O_Y(U)$ et $B=\O^+(]U\moins\{P\}[)$.
L'intersection $A^{++}$ de $A$ et $B^{++}$ est le noyau
$$A^{++}={\rm Ker}\big[B^{++}\to \frac{{\goth m}_C[[T,T^{-1}\rangle}{{\goth m}_C[[T]]}\big].$$
D'apr\`es le lemme~\ref{PP3}, on a des suites exactes
\begin{align*}
0\to A\to B\to \frac{\O_C[[T,T^{-1}\rangle}{\O_C[[T]]}\to 0,\quad
0\to A^{++}\to B^{++}\to \frac{{\goth m}_C[[T,T^{-1}\rangle}{{\goth m}_C[[T]]}\to 0.
\end{align*}
On en d\'eduit une suite exacte
$$0\to\overline A\to \overline B\to \frac{k_C((T))}{k_C[[T]]}\to 0.$$
Il en r\'esulte que $\overline A=\O_{\cal Y}(U)$.
En particulier, $\overline A$
est de type fini sur $k_C$.

Par ailleurs, d'apr\`es le lemme~\ref{PP3}, les $z^{-n}$, pour $n\geq 1$, forment une
base de Banach de $B/A$ sur $\O_C$.  Or les $z^{-n}$ appartiennent \`a $\breve B=\O(\breve Y_P)$
et on est dans les conditions d'application du lemme~\ref{PP2} ce qui
permet d'en d\'eduire que $A$ est une $\O_C$-alg\`ebre de Tate, ce que l'on voulait.
\end{proof}

\Subsection{Patron d'une courbe}\label{Pconstr10}
\subsubsection{D\'efinition}
Un {\it patron de courbe} $(\Gamma,(Y_i)_{i\in I},(\iota_{i,j})_{(i,j)\in I_2})$ est la donn\'ee de:

$\bullet$ Un graphe bipartite marqu\'e fini\footnote{Cette restriction n'est pas essentielle mais
nous n'aurons besoin que de patrons finis dans ce qui suit.}
$\Gamma=(I,I_2,\mu)$, $I=A_c\sqcup S$, $\mu(a)\in\Q_{>0}$ si $a\in A_c$, 
sans boucle ayant $1$ ou $2$ sommets.

$\bullet$ Pour tout $a\in A_c$, une jambe $Y_a$ de longueur $\mu(a)$
avec $$\O(Y_a)=\O_C[[T_{a,s_1},T_{a,s_2}]]/(T_{a,s_1}T_{a,s_2}-p^{\mu(a)}),$$ 
o\`u $s_1,s_2$ sont
les extr\'emit\'es de $a$, et on note  
$Y_{a,s_i}$, pour $i=1,2$,
 le cercle fant\^ome avec $\O(Y_{a,s_i})=\O_C[[T_{a,s_i},T_{a,s_i}^{-1}\rangle$.

$\bullet$ Pour tout $s\in S$, un short $Y_s$ dont la fronti\`ere $\partial^{\rm ad}Y_s$
est constitu\'ee de cercles fant\^omes $Y_{s,a}$ index\'es par $a\in A(s)$,

$\bullet$ Pour tout couple $(a,s)$, avec $a\in A_c$ et $s\in S(a)$, un isomorphisme
$\iota_{a,s}:Y_{s,a}\overset{\sim}{\to} Y_{a,s}$ de cercles fant\^omes.

\begin{rema}\label{Pgener}
On peut g\'en\'eraliser la notion de patron en ajoutant un sous-ensemble
$B$ de $A\moins A_c$, et pour tout $b\in B$:

$\bullet$ une boule ouverte $Y_b$, avec
$\O(Y_b)=\O_C[[T_b]]$, 

$\bullet$ un cercle fant\^ome $Y_{b,s}$, o\`u $S(b)=\{s\}$,
avec $\O(Y_{b,s})=\O_C[[T_{b,s},T_{b,s}^{-1}\rangle$,

$\bullet$ un isomorphisme $\iota_{b,s}:Y_{s,b}\overset{\sim}{\to} Y_{b,s}$.
\end{rema}

\subsubsection{Patrons d'une courbe quasi-compacte}\label{Pgener10}
On peut associer un patron de courbe \`a une courbe quasi-compacte $Y$ sur $C$,
munie d'une triangulation~$S$.
Soit $Y_S$ le $\O_C$-mod\`ele semi-stable associ\'e \`a $S$.
On suppose $S$
suffisamment fine pour 
que la fibre sp\'eciale $Y_S^{\rm sp}$ ait au moins deux composantes
irr\'eductibles, que ces composantes irr\'eductibles
soient lisses et que deux d'entre elles s'intersectent en au plus un point.
Le patron de $Y$ associ\'e \`a $S$ est obtenu de la mani\`ere suivante.

$\bullet$ On note 
$\Gamma=(I,I_2,\mu)$, avec $I=A_c\sqcup S$, le graphe bipartite marqu\'e associ\'e
au graphe dual de la fibre sp\'eciale $Y_S^{\rm sp}$ de $Y_S$ (cf.~\no\ref{TTT27}).
(on dispose donc de courbes propres $Y_s^{\rm sp}$, pour $s\in S$,
d'un ouvert $\ocirc{Y}_s^{\rm sp}$ de $Y_s^{\rm sp}$, et de points marqu\'es $P_a$
de multiplicit\'e $\mu(a)$, avec $\mu(a)\in\Q_+^\dual$ si $a\in A_c$, et $\mu(a)=0^+$
si $a\in I_2\moins I_{2,c}$).

$\bullet$ Si $s\in S$, on note $A(s)$ l'ensemble des ar\^etes ayant $s$
pour extr\'emit\'e (les $P_a$, pour $a\in A(s)$, 
sont les points marqu\'es de $Y^{\rm sp}_s$),
et $A_c(s)= A(s)\cap A_c$ 
(les $P_a$, pour $a\in A_c(s)$, sont les points singuliers de $Y^{\rm sp}_s$).

$\bullet$ 
Si $s\in S$, on munit d'une structure de $\O_C$-short le sch\'ema formel $Y_s$
 d\'efini par $\O(Y_s)=\O^+(]\ocirc{Y}^{\rm sp}_s[)$, o\`u
$]\ocirc{Y}^{\rm sp}_s[$ est le tube de $\ocirc{Y}^{\rm sp}_s$ dans
$Y$.
Si $a\in A_c$, on note $Y_a$ la $\O_C$-jambe d\'efinie par
$\O(Y_a)=\O^+(]P_a[)$, o\`u $]P_a[$ est 
le tube de $P_a$ dans $Y$; elle est de longueur $\mu(a)$.

(Les $Y_i$, pour $i\in I$,
forment un recouvrement adoque de $Y$.) 

$\bullet$ Si $a\in A_c$, 
on note $S(a)$ l'ensemble des extr\'emit\'es de $a$,
i.e.~l'ensemble des $s\in S$ tels que $P_a\in Y^{\rm sp}_s$.
Alors $S(a)$ a deux \'el\'ements $s_1=s_1(a)$ et $s_2=s_2(a)$ dont
l'ordre est d\'etermin\'e par l'orientation choisie de $\Gamma$.
On choisit une fonction $T_{a,s_1}$ sur un ouvert de
$Y_{s_1}\sqcup Y_a\sqcup Y_{s_2}$ dont la restriction
\`a $Y_a$ induit un isomorphisme de $Y_a$ sur la couronne
$0<v_p(T_{a,s_1})<\mu(a)$, et telle que $v_{s_1}(T_{a,s_1})=0$.
On pose $T_{a,s_2}=p^{\mu(a)}/T_{a,s_1}$, donc $v_{s_2}(T_{a,s_2})=0$ et 
$$\O(Y_a)=\O_C[[T_{a,s_1},T_{a,s_2}]]/(T_{a,s_1}T_{a,s_2}-p^{\mu(a)}).$$

$\bullet$ Si $(a,s)\in I_{2,c}$,
 $Y_a$ intersecte $Y_{s}$ le long d'un cercle fant\^ome $Y_{a,s}$
et comme $S$ est fine, 
$Y_i\cap Y_j\neq\emptyset$ si et seulement si 
$(i,j)\in I_{2,c}$ (si $i\neq j$ bien s\^ur).

\Subsection{Construction de courbes \`a partir d'un patron}\label{Pconstr11}
La proc\'edure fournissant un patron
d'une courbe quasi-compacte \`a partir d'une triangulation peut s'inverser, ce qui
permet de construire des courbes (et m\^eme des familles de courbes) \`a partir
d'un patron.

\subsubsection{Patrons de $R$-courbes}
Un {\it patron de $R$-courbe} $(\Gamma,(Y_i)_{i\in I},(\iota_{i,j})_{(i,j)\in I_{2,c}})$ est la donn\'ee de:

$\bullet$ Un graphe bipartite marqu\'e fini $\Gamma=(I,I_2,m)$, avec $I=A_c\sqcup S$,
et $m:A_c\to {\goth m}_R$,
sans boucle ayant $1$ ou $2$ sommets.

$\bullet$ Pour tout $a\in A_c$, une $R$-jambe $Y_a$ 
avec $$\O(Y_a)=R[[T_{a,s_1},T_{a,s_2}]]/(T_{a,s_1}T_{a,s_2}-m(a)),$$ 
o\`u $s_1,s_2$ sont
les extr\'emit\'es de $a$, et on note  
$Y_{a,s_i}$, pour $i=1,2$ le 
$R$-cercle fant\^ome avec $\O(Y_{a,s_i})=R[[T_{a,s_i},T_{a,s_i}^{-1}\rangle$.

$\bullet$ Pour tout $s\in S$, un $R$-short $Y_s$ avec $\O(Y_s)\cong R\wotimes_{\O_{\breve C}}\O(\breve Y_s)$,
o\`u $\breve Y_s$ est un short d\'efini sur $\O_{\breve C}$,
dont la fronti\`ere $\partial^{\rm ad}\breve Y_s$
est constitu\'ee de cercles fant\^omes $\breve Y_{s,a}$ index\'es par $a\in A(s)$,
et on d\'efinit le $R$-cercle fant\^ome $Y_{s,a}$ par
 $\O(Y_{s,a})=R\wotimes_{\O_{\breve C}}\O(\breve Y_{s,a})$.

$\bullet$ Pour tout couple $(i,j)\in I_{2,c}$, un isomorphisme
$\iota_{i,j}:Y_{j,i}\overset{\sim}{\to} Y_{i,j}$ de $R$-cercles fant\^omes.

\subsubsection{Construction de $R$-courbes}
A partir d'un patron de $R$-courbe, on construit une $R$-courbe adoque $Y^{\rm ado}$ 
en recollant
les $Y_s$ et les $Y_a$ via les $\iota_{a,s}$.  
Si le patron est le patron d'une $\O_C$-courbe $Y$, alors $Y^{\rm ado}$ est
le sch\'ema adoque associ\'e \`a $Y$.  En g\'en\'eral, on a
le r\'esultat suivant qui est une manifestation d'un principe GaGa.
\begin{theo}\label{PP4}
$Y^{\rm ado}$ est
le sch\'ema adoque associ\'e \`a une $R$-courbe $Y$
{\rm (i.e.~un $R$-sch\'ema formel de dimension relative~$1$)}.
\end{theo}
\begin{proof}
Notons que, si $s\in S$, alors $k_C\otimes_R\O(Y_s)=k_C\otimes_{\O_{\breve C}}\O(\breve Y_s)$,
et donc $Y_s$ et $\breve Y_s$ ont m\^eme fibre sp\'eciale $Y_s^{\rm sp}$ 
qui est, par d\'efinition, une courbe projective lisse munie de points marqu\'es
de multiplicit\'e $0^+$ dont l'int\'erieur $\ocirc Y_s^{\rm sp}$ est le compl\'ementaire
de ces points marqu\'es (et $\O(\ocirc Y_s^{\rm sp})=k_C\otimes_R\O(Y_s)$).
Si $s\in S$ et si $a\in A(s)$, on note $P_{s,a}$ le point de $Y^{\rm sp}_s$
correspondant \`a $Y_{s,a}$.  Alors $Y^{\rm sp}$ (resp.~$\ocirc Y^{\rm sp}$) est obtenue \`a partir
de la r\'eunion disjointe des $Y^{\rm sp}_s$ (resp.~$\ocirc Y^{\rm sp}_s$), pour $s\in S$,
en identifiant les points $P_{s_1(a),a}$ et $P_{s_2(a),a}$, pour $a\in A_c$.
Le point $P_a=P_{s_1(a),a}=P_{s_2(a),a}$ est alors un point marqu\'e singulier
de $Y^{\rm sp}$ (appartenant \`a $\ocirc Y^{\rm sp}$),
 et on d\'efinit sa multiplicit\'e comme \'etant $m(a)$.

Ceci fournit une description de l'espace topologique sous-jacent \`a $Y^{\rm ado}$.
Pour d\'emontrer le th\'eor\`eme, il reste \`a prouver que la restriction 
$\O_Y$, aux ouverts de Zariski,
 du faisceau structural $\O_{Y^{\rm ado}}$
de $Y^{\rm ado}$ 
produit des $R$-alg\`ebres de Tate, et il
suffit de v\'erifier ceci pour un voisinage assez petit $U$ de tout point ferm\'e $P$
de $\ocirc Y^{\rm sp}$. 
Il y a deux cas \`a consid\'erer pour $P$.

\quad $\bullet$
Si $P$ n'est pas un point marqu\'e, alors $P$ appartient \`a une unique
composante $Y^{\rm sp}_s$ et on prend pour $U$ un ouvert 
de $\ocirc{Y}^{\rm sp}_s$; son tube $]U[$ dans (la fibre g\'en\'erique de) $\breve Y_s$
est un sous-affino\"{\i}de de $\breve Y_s$, et on a
$\O_Y(U)=R\wotimes_{\O_{\breve C}}\O^+(]U[)$ qui est bien une $R$-alg\`ebre de Tate.

\quad $\bullet$
Si $P=P_a$, avec $a\in A_c$, et si $S(a)=\{s_0,s_1\}$, alors
$P_a$ est le point d'intersection de $Y^{\rm sp}_{s_0}$ et $Y^{\rm sp}_{s_1}$.
On prend pour $U$ un ouvert de la forme $U_0\cup U_1$, o\`u $U_i$
est un ouvert de $Y^{\rm sp}_{s_i}$ contenant $P_a$ et suffisamment petit pour qu'il existe
$Z_i\in\O(U_i)$ ayant un z\'ero simple en $P_a$.
La propri\'et\'e de faisceau d\'ecrit $\O_Y(U)$ comme le noyau
$$\O_Y(U)={\rm Ker}\big[ R\wotimes_{\O_{\breve C}}(\O^+(]U_0[)\times\O^+(]U_1[))\oplus\O(Y_a)\to
\O(Y_{a,s_0})\times\O(Y_{a,s_1})\big]\,,$$
o\`u $]U_i[$ est le tube de $U_i$ dans $\breve Y_{s_i}$,
les fl\`eches de $\O(Y_a)$ dans les $\O(Y_{a,s_i})$ sont induites par les inclusions
des $Y_{a,s_i}$ dans $Y_a$, et la fl\`eche de $\O^+(]U_i[)$ dans $\O(Y_{a,s_i})$
par $Y_{a,s_i}\cong Y_{s_i,a}\subset ]U_i[$, si $i=0,1$.

Posons $A=\O_Y(U)$, $\breve B=\O^+(]U_0[)\times\O^+(]U_1[)$
et $B=R\wotimes_{\O_{\breve C}}\breve B$.  Commen\c{c}ons par prouver que la suite
$$0\to A\to B\to\frac{R[[T_0,T_0^{-1}\rangle\oplus R[[T_1,T_1^{-1}\rangle}{R[[T_0,T_1]]/(T_0T_1-m(a))}
\to 0$$
est exacte.  Elle est exacte \`a gauche et au milieu par d\'efinition de $A$.
Pour prouver la surjectivit\'e de la fl\`eche de droite, il suffit de le faire modulo $m(a)$ car
tout est s\'epar\'e et complet pour la topologie $m(a)$-adique.
Mais alors, si $R_a=R/m(a)$, l'application
 $B/m(a)\to \frac{ R_a [[T_0,T_0^{-1}\rangle\oplus  R_a [[T_1,T_1^{-1}\rangle}{ R_a [[T_0,T_1]]/(T_0T_1)}$
se factorise \`a travers
$\frac{ R_a [[T_0,T_0^{-1}\rangle}{T_0 R_a [[T_0]]}\oplus
\frac{ R_a [[T_1,T_1^{-1}\rangle}{T_1 R_a [[T_1]]}$.
Comme $B=B_0\oplus B_1$, avec $B_i=R\wotimes_{\O_{\breve C}}\breve B_i$ et 
$\breve B_i=\O^+(]U_i[)$, l'application ci-dessus est la somme
des $B_i/m(a)\to \frac{ R_a [[T_i,T_i^{-1}\rangle}{T_i R_a [[T_i]]}$,
ce qui permet d'utiliser le lemme~\ref{PP3} pour prouver sa surjectivit\'e.

Soient $B^{++}={\goth m}_R\wotimes_{\O_{\breve C}}\breve B$ et $A^{++}=A\cap B^{++}$.
Le lemme~\ref{PP3} permet aussi de prouver que la suite
$$0\to A^{++}\to B^{++}\to \frac{{\goth m}_R[[T_0,T_0^{-1}\rangle\oplus {\goth m}_R[[T_1,T_1^{-1}\rangle}{{\goth m}_R[[T_0,T_1]]/(T_0T_1-m(a))}
\to 0$$
est exacte. 
Si $\overline B=B/B^{++}$ et $\overline A=A/A^{++}$, on a donc une suite exacte
$$0\to \overline A\to \overline B\to \frac{k_C((T_0))\oplus k_C((T_1))}{k_C[[T_0,T_1]]/(T_0T_1)}\to 0.$$
Il en r\'esulte que $\overline A=\O_{Y^{\rm sp}}(U)$;
en particulier, $\overline A$
est de type fini sur $k_C$.

Par ailleurs, si $z_i\in \O^+(]U_i[)$ est tel que $z_i^{-1}$ est un param\`etre du cercle fant\^ome
$Y_{s_i,a}$, on d\'eduit du lemme~\ref{PP3}, que $1$ et les $z_i^{-k}$, pour $i=0,1$ et $k\geq 1$,
forment une base orthonormale de $B/A$. Comme cette base est constitu\'ee d'\'el\'ements
de $\breve B$, le lemme~\ref{PP2} permet d'en d\'eduire que $A$ est une $R$-alg\`ebre de Tate,
ce que l'on voulait.
\end{proof}
\begin{rema}\label{Pgener1}
On peut adapter la construction ci-dessus
\`a un patron g\'en\'eralis\'e
comme dans la rem.\,\ref{Pgener}; on recolle des $R$-boules ouvertes $Y_b$ (avec $\O(Y_b)=R[[T_b]]$)
le long des $R$-cercles fant\^omes $Y_{s,b}$, pour $b\in B$.
La courbe ainsi obtenue est propre si et seulement si
$B=A\moins A_c$.
\end{rema}

\begin{center}
\begin{tikzpicture}
[scale=.8]
\draw[very thick] (0,4) arc (50: 310 : 1.4 and 3);
\draw[very thick] (5,4) arc (130: -130 : 1.4 and 3);

\draw[very thick]  (0,.2) ..controls +(1,-.1) and +(-1,-.1).. (5,.2);
\draw[very thick]  (0,-.6) ..controls +(1,.1) and +(-1,.1).. (5,-.6);
\draw[very thick, dotted, red] (0,-.6) arc (-90: 90 : .2 and .4); \draw[very thick, dashed, red] (0,.2) arc (90: 270 : .2 and .4);
\draw[very thick, dotted, red] (5,-.6) arc (-90: 90 : .2 and .4); \draw[very thick, dashed, red] (5,.2) arc (90: 270 : .2 and .4);

\draw[very thick]  (0,4) ..controls +(1,-.1) and +(-1,-.1).. (1.8,4);
\draw[very thick]  (0,3.2) ..controls +(1,.1) and +(-1,.1).. (1.8,3.2);
\draw[very thick, dotted, red] (0,3.2) arc (-90: 90 : .2 and .4); \draw[very thick, dashed, red] (0,4) arc (90: 270 : .2 and .4);
\draw[very thick, dotted, red] (1.8,3.2) arc (-90: 90 : .2 and .4); \draw[very thick, dashed, red] (1.8,4) arc (90: 270 : .2 and .4);
\draw[very thick, dotted, red] (3.2,3.2) arc (-90: 90 : .2 and .4); \draw[very thick, dashed, red] (3.2,4) arc (90: 270 : .2 and .4);
\draw[very thick, dotted, red] (5,3.2) arc (-90: 90 : .2 and .4); \draw[very thick, dashed, red] (5,4) arc (90: 270 : .2 and .4);
\draw[very thick]  (5,4) ..controls +(-1,-.1) and +(1,-.1).. (3.2,4);
\draw[very thick]  (5,3.2) ..controls +(-1,.1) and +(1,.1).. (3.2,3.2);

\draw[very thick] (1.8,3.2) arc (-150:-30:.8); \draw[very thick] (1.8,4) arc (150:30:.8);

\draw[very thick]  (5,.2) ..controls +(.5,0) and +(.5,0).. (5,3.2);
\draw[very thick]  (0,.2) ..controls +(-.5,0) and +(-.5,0).. (0,3.2);

\draw[thick, blue] (-1,-.85) circle(.25);
\draw[thick, blue] (6,-.85) circle(.25);
\draw[blue] (-1,-3) -- (-1,-.7);
\draw[blue] (-1,-1.5) -- (-.9,-.9);
\draw[ball color=blue] (9.5,-.85) circle (.12);
\draw[ball color=blue] (11.5,-.85) circle (.12);
\draw[blue](6,-.6)--(9.5,-.85);
\draw[blue](6,-1.1)--(9.5,-.85);

\draw[thick, blue] (-1,1.75) circle(.25);
\draw[ball color=blue] (8.85,1.75) circle (.12);
\draw[blue](-1,2)--(8.85,1.75);
\draw[blue](-1,1.5)--(8.85,1.75);

\draw[very thick]  (-1.5,-.2) ..controls +(0.15,-.35) and +(-0.15,-.35).. (-.5,-.2);
\draw[very thick]  (-1.375,-.325) ..controls +(.1,.15) and +(-.1,.15).. (-.625,-.325);

\draw[very thick]  (-1.7,1) ..controls +(0.15,-.35) and +(-0.15,-.35).. (-.7,1);
\draw[very thick]  (-1.575,.875) ..controls +(.1,.15) and +(-.1,.15).. (-.825,.875);
\draw[very thick]  (-1.7,2.5) ..controls +(0.15,-.35) and +(-0.15,-.35).. (-.7,2.5);
\draw[very thick]  (-1.575,2.375) ..controls +(.1,.15) and +(-.1,.15).. (-.825,2.375);

\draw[very thick]  (-1.5,3.8) ..controls +(0.15,-.35) and +(-0.15,-.35).. (-.5,3.8);
\draw[very thick]  (-1.375,3.675) ..controls +(.1,.15) and +(-.1,.15).. (-.625,3.675);

\draw[very thick]  (5.5,3.8) ..controls +(0.15,-.35) and +(-0.15,-.35).. (6.5,3.8);
\draw[very thick]  (5.625,3.675) ..controls +(.1,.15) and +(-.1,.15).. (6.375,3.675);

\draw[very thick]  (5.7,1.6) ..controls +(0.15,-.35) and +(-0.15,-.35).. (6.7,1.6);
\draw[very thick]  (5.825,1.475) ..controls +(.1,.15) and +(-.1,.15).. (6.575,1.475);

\draw[very thick]  (5.5,-.2) ..controls +(0.15,-.35) and +(-0.15,-.35).. (6.5,-.2);
\draw[very thick]  (5.625,-.325) ..controls +(.1,.15) and +(-.1,.15).. (6.375,-.325);

\draw[very thick, dotted, color=red, fill=gray!20] (6.7,2.8) circle (.27 and .3);
\draw[thick, dotted, red] (6.7,3.1) -- (12.5,2.6);
\draw[thick, dotted, red] (6.7,2.5) -- (12.5,2.6);
\draw[very thick, dotted, color=red, fill=gray!20] (6.7,.8) circle (.27 and .3);

\draw(2.5,0.45) node{$Y_{a_3}$};
\draw(1,4.3) node{$Y_{a_1}$};
\draw(4,4.3) node{$Y_{a_2}$};
\draw(2.5,4.7) node{$Y_{s_2}$};
\draw(-1.9,4.5) node{$Y_{s_1}$};
\draw(6.9,4.5) node{$Y_{s_3}$};

\draw(2.5,-3.4) node{$\Gamma^{\rm ad}(Y)$};
\draw[ball color=black] (-1,-3) circle (.2); 
\draw[ball color=black] (6,-3) circle (.2); 
\draw[ball color=black] (2.5,-2) circle (.2); 
\draw[very thick] (-1,-3)--(6,-3);
\draw[very thick] (-1,-3)--(2.5,-2);
\draw[very thick] (2.5,-2)--(6,-3);
\draw (.75,-2.2) node{$a_1$};
\draw (4.25,-2.2) node{$a_2$};
\draw (2.5,-2.75) node{$a_3$};
\draw (6.2,-2.2) node{$a_4$};
\draw (6.9,-2.7) node{$a_5$};
\draw (-1,-3.4) node{$s_1$};
\draw (2.5,-1.6) node{$s_2$};
\draw (6,-3.4) node{$s_3$};
\draw[ball color=blue] (3.5,-3) circle (.12);
\draw[very thick, dashed, blue] (3.5,-.55) arc (-90: 90 : .2 and .35); \draw[very thick, blue] (3.5,.15) arc (90: 270 : .2 and .35);
\draw[blue] (3.5,-3) -- (3.35,0);
\draw[thick,color=red,{->}] (6,-3) -- +(80:.7); \draw[thick,color=red,{->}] (6,-3) -- +(20:.7);
\draw[thick,color=red,{->}] (6,-3) -- (5.5,-2.86); \draw[thick,color=red,{->}] (6,-3) -- (5.5,-3);
\draw[thick,color=red,{->}] (-1,-3) -- (-.5,-2.86); \draw[thick,color=red,{->}] (-1,-3) -- (-.5,-3);
\draw[thick,color=red,{->}] (2.5,-2) -- (2,-2.14); \draw[thick,color=red,{->}] (2.5,-2) -- (3,-2.14);

\draw[very thick]  (8,4.5) ..controls +(.5,-3) and +(-2,.5).. (12,-1);
\draw[very thick]  (13,4.5) ..controls +(-.5,-3) and +(2,.5).. (9,-1);
\draw[very thick]  (7.5,3.5) -- (13.5,3.5);
\draw[ball color=red] (10.5,-.24) circle (.12);
\draw (11.6,-.24) node{$(P_{a_3},\mu_3)$};
\draw[ball color=red] (8.2,3.5) circle (.12);
\draw (9.2,3.2) node{$(P_{a_1},\mu_1)$};
\draw[ball color=red] (12.8,3.5) circle (.12);
\draw (13.7,3.2) node{$(P_{a_2},\mu_2)$};
\draw[ball color=red] (12.5,2.6) circle (.12);
\draw (13.5,2.4) node{$(P_{a_4},0^+)$};
\draw[ball color=red] (11.55,.8) circle (.12);
\draw (12.5,.6) node{$(P_{a_5},0^+)$};
\draw(10.5,3.9) node{$Y_{s_2}^{\rm sp}$};
\draw(8.45,4.3) node{$Y_{s_1}^{\rm sp}$};
\draw(12.55,4.3) node{$Y_{s_3}^{\rm sp}$};
\draw(10.5,-1.2) node{$Y^{\rm sp}$};

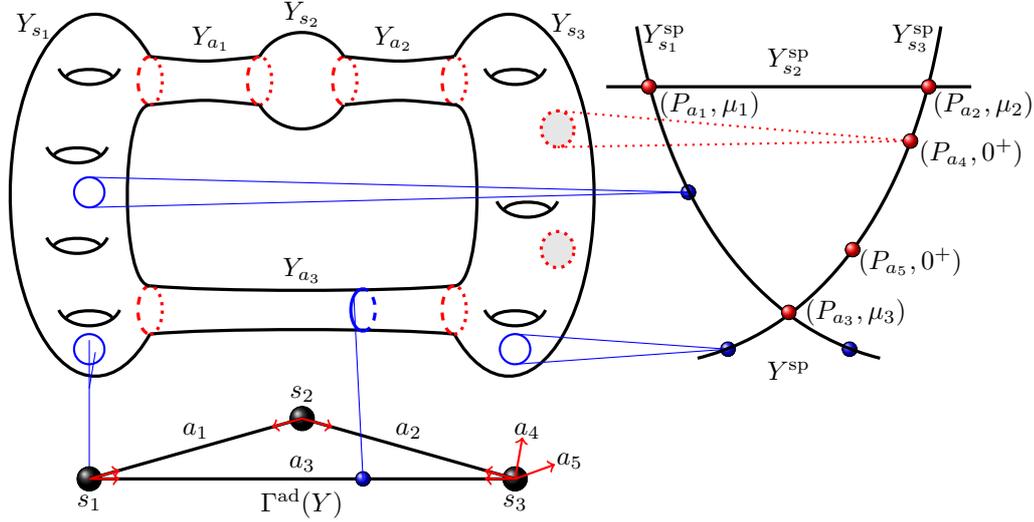
\draw (5.5,-4.4) node{Un affino\"{\i}de $Y$ construit \`a partir de 3 shorts et 3 jambes};

\end{tikzpicture}
\end{center}

{\Small
Ce dessin repr\'esent un affino\"{\i}de $Y$ obtenu en recollant, le long de cercles
fant\^omes (en pointill\'es rouges), des shorts $Y_{s_1}$, $Y_{s_2}$
et $Y_{s_3}$, de genres respectifs $4$, $0$ et $3$, et des jambes $Y_{a_1}$, $Y_{a_2}$
et $Y_{a_3}$, de longueurs respectives $\mu_1$, $\mu_2$, $\mu_3$.
La fibre sp\'eciale $Y^{\rm sp}$ a cinq points marqu\'es: trois correspondant aux sp\'ecialisations des jambes
et deux aux tubes enlev\'es sur $Y_{s_3}$.  Le squelette $\Gamma^{\rm an}(Y)$ est constitu\'e
du triangle de sommets $s_1$, $s_2$ et $s_3$ (les longueurs des ar\^etes $a_1$, $a_2$ et $a_3$
sont $\mu_1$, $\mu_2$ et $\mu_3$) et on obtient $\Gamma^{\rm ad}(Y)$ en rajoutant
les fl\`eches $a_4$ et $a_5$.

Notons que $Y_{s_2}^{\rm sp}$ est un ${\bf P}^1$ avec deux points marqu\'es.  Il peut donc se contracter
et les deux points fusionner en un point de multiplicit\'e $\mu_1+\mu_2$. Cela comprime la sph\`ere
$Y_{s_2}$ en un cercle et les deux jambes $Y_{a_1}$ et $Y_{a_2}$ se recollent le long de ce cercle
pour former une jambe de longueur $\mu_1+\mu_2$. Enfin, le sommet $s_2$ dispara\^{\i}t et les ar\^etes
$a_1$ et $a_2$ fusionnent en une ar\^ete de longueur $\mu_1+\mu_2$.

}

\subsubsection{Cas particuliers}\label{familles}
Le th.~\ref{PP4} admet un certain nombre de sp\'ecialisations 
sympathiques.

\smallskip
\noindent $\bullet$ {\it Le cas $R=\O_C$}.

Par sp\'ecialisation
au cas $R=\O_C$ et $m(a)=p^{\mu(a)}$, on 
tire le r\'esultat suivant (l'unicit\'e est imm\'ediate; l'existence est fournie
par le th.~\ref{PP4}).
\begin{theo}\label{Pconstr12}
Si $(\Gamma,(Y_i)_{i\in I},(\iota_{i,j})_{(i,j)\in I_{2,c}})$ est un patron de courbe,
il existe un unique couple $(Y,S)$, o\`u $Y$ est une courbe quasi-compacte
et $S$ une triangulation de~$Y$, dont ce soit le patron.
\end{theo}

\noindent $\bullet$ {\it Le cas $R=\O_C[[T_a,\ a\in A_c]]$}.

Soit $(\Gamma,(Y_i)_i,(\iota_{i,j})_{(i,j)\in I_{2,c}})$ un patron de courbe,
avec $\Gamma=(A,S,\mu)$, et soit $Y$ la $\O_C$-courbe qui lui correspond.
On fabrique un patron de $R$-courbe en changeant 
$a\mapsto\mu(a)$ en $a\mapsto T_a$.
La $R$-courbe associ\'ee peut \^etre vue comme une famille de courbes param\'etr\'ees par
la boule unit\'e ouverte $\ocirc B^{A_c}$.  La fibre en $(p^{\mu(a)})_{a\in A_c}$,
 est $Y$;
toutes les courbes de la famille ont la m\^eme
fibre sp\'eciale (au marquage pr\`es: $P_a$ devient de multiplicit\'e $v_p(z_a)$
sur la fibre en $(z_a)_{a\in A_c}$); 
la fibre en $(0)_{a\in A_c}$,
est une courbe singuli\`ere ayant m\^eme graphe dual que la fibre sp\'eciale.

\smallskip
\noindent $\bullet$ {\it Le cas $R=\ainf$}.

Soit $(\Gamma,(Y_i)_i,(\iota_{i,j})_{(i,j)\in I_{2,c}})$ un patron de courbe,
avec $\Gamma=(I,I_2,\mu)$, et soit $Y$ la $\O_C$-courbe qui lui correspond.
On peut fabriquer un patron de $\ainf$-courbe
$(\Gamma,(\widetilde Y_i)_i,(\tilde\iota_{i,j})_{(i,j)\in I_{2,c}})$ 
en choisissant des mod\`eles $\breve Y_s$ sur $\O_{\breve C}$ des $Y_s$, pour $s\in S$,
et en posant:

$\diamond$ $\O(\widetilde Y_s):=\ainf\wotimes_{\O_{\breve C}}\O(\breve Y_s)$, si $s\in S$,

$\diamond$ $\O(\widetilde Y_a):=\ainf[[T_{a,s_1},T_{a,s_2}]]/(T_{a,s_1}T_{a,s_2}-\tilde p^{\mu(a)})$,
si $a\in A_c$,

\noindent et en choisissant:

$\diamond$ $\tilde\iota_{a,s}:\ainf[[T_{s,a},T_{s,a}^{-1}\rangle\overset{\sim}{\to}
\ainf[[T_{a,s},T_{a,s}^{-1}\rangle$ un rel\`evement de $\iota_{a,s}$,
si $(a,s)\in I_{2,c}$ (un tel rel\`evement est obtenu en choisissant
un rel\`evement 
$\tilde\iota_{a,s}(T_{s,a})\in
\ainf[[T_{a,s},T_{a,s}^{-1}\rangle$ de
$\iota_{a,s}(T_{s,a})\in \O_C[[T_{a,s},T_{a,s}^{-1}\rangle$).

Le patron ainsi d\'efini d\'epend du choix des $\breve Y_s$ et des $\tilde\iota_{i,j}$;
la $\ainf$-courbe $\widetilde Y$ qui lui correspond peut \^etre 
vue comme une famille de courbes ayant toutes
la m\^eme fibre sp\'eciale,
param\'etr\'ees par les id\'eaux premiers ferm\'es de $\ainf$:

-- la fibre en $p=\tilde p$ est $Y$;

-- la fibre en $p=[a]$ avec $a\in {\goth m}_{C^\flat}\moins\{0\}$
est un mod\`ele sur $\O_{C_a}$ d'une courbe lisse d\'efinie sur $C_a$;

-- la fibre $\breve Y$ en $\tilde p=0$ est un mod\`ele 
sur $\O_{\breve C}$
d'une courbe singuli\`ere
d\'efinie sur $\breve C$,
ayant m\^eme graphe dual que la fibre sp\'eciale
de $Y$; 

-- la fibre en $p=0$ un mod\`ele sur $\O_{C^\flat}$
d'une courbe lisse sur $C^\flat$.

On peut donc, en particulier, voir $\widetilde Y$ comme une famille 
de courbes interpolant entre
$Y$ et $\breve Y$.

\subsection{Cohomologie de de Rham adoque}\label{tetrapil1}
Soit $(\Gamma,(Y_i)_i,(\iota_{i,j})_{(i,j)\in I_{2,c}})$ un patron de courbe,
avec $\Gamma=(I,I_2,\mu)$, et soit $Y$ la $\O_C$-courbe qui lui correspond.
On d\'ecompose $I$ sous la forme $A_c\sqcup S$ habituelle. Si $a\in A_c$
a pour extr\'emit\'es $s_1,s_2$, on note
$U_a^{\rm sp}$ l'ouvert de $Y^{\rm sp}$
constitu\'e de $P_a$ et des lieux de lissit\'e
de $\ocirc Y_{s_1}^{\rm sp}$ et $\ocirc Y_{s_2}^{\rm sp}$.
Le tube $U_a$ de $U^{\rm sp}_a$ s'obtient en recollant $Y_{s_1}$ et $Y_{s_2}$
\`a $Y_a$ le long de $Y_{a,s_1}$ et $Y_{a,s_2}$.

Les $U_a$ forment un recouvrement de $Y$, et on peut calculer
la cohomologie de de Rham de $Y$ \`a la \v{C}ech en utilisant ce recouvrement.
Si $a_1,a_2\in A(s)$, on a $U_{a_1}\cap U_{a_2}=\emptyset$ sauf s'il
existe $s\in S$ tel que
$a_1,a_2\in A(s)$, auquel cas
$U_{a_1,a_2}:=U_{a_1}\cap U_{a_2}=Y_s$.
On note $C^\bullet_{\rm dR}(Y)$ le complexe
$$C^\bullet_{\rm dR}(Y):=\big[\prod_{a\in A_c}\Omega^\bullet(U_a)\longrightarrow
\prod_{a_1,a_2}\Omega^\bullet(U_{a_1,a_2})\big]$$
Ses groupes d'hypercohomologie sont les groupes $H^i_{\rm dR}(Y)$ de cohomologie de de Rham (logarithmique)
de $Y$.
Un $1$-cocycle est de la forme 
\begin{align*}
((\omega_a)_a,(f_{a_1,a_2})_{a_1,a_2}), 
\quad\omega_a\in\Omega^1(U_a),\  f_{a_1,a_2}\in\O(U_{a_1,a_2}),\\
df_{a_1,a_2}=\omega_{a_2}-\omega_{a_1},\quad f_{a_1,a_2}+f_{a_2,a_3}+f_{a_3,a_1}=0,
{\text{ si $a_1,a_2,a_3\in A(s)$.}}
\end{align*}
Un $1$-cobord est de la forme $((df_a)_a,(f_{a_2}-f_{a_1})_{a_1,a_2})$.

De m\^eme, on d\'efinit les $H^i_{\rm dR}(Y^{\rm ado})$ comme
les groupes d'hypercohomologie du complexe
$$C^\bullet_{\rm dR}(Y^{\rm ado}):=\big[\prod_{i\in I}\Omega^\bullet(Y_i)\longrightarrow
\prod_{(i,j)\in I_{2,c}}\Omega^\bullet(Y_{i,j})\big]$$
On choisit $a(s)\in A(s)$, pour tout $s\in S$. Cela fournit une application
$Z^1_{\rm dR}(Y)\to Z^1_{\rm dR}(Y^{\rm ado})$ envoyant
$((\omega_a)_a,(f_{a_1,a_2})_{a_1,a_2})$ sur $((\eta_i)_i,(g_{i,j})_{i,j})$, avec
$\eta_a=\omega_a$ si $a\in A_c$, $\eta_s=\omega_{a(s)}$ si $s\in S$
et $g_{a,s}=f_{a,a(s)}$ si $(a,s)\in I_{2,c}$.
\begin{prop}\label{tetrapil2}
L'application $Z^1_{\rm dR}(Y)\to Z^1_{\rm dR}(Y^{\rm ado})$
ci-dessus induit un isomorphisme
$$H^1_{\rm dR}(Y)\overset{\sim}{\to} H^1_{\rm dR}(Y^{\rm ado})$$
qui ne d\'epend pas du choix des $a(s)$.
\end{prop}
\begin{proof}
Si on change $a(s)$ en $a'(s)$, cela modifie le cocycle $((\eta_i)_i,(g_{i,j})_{i,j})$
par le bord de $((f_{a(s),a'(s)})_s,(0)_a)$, ce qui montre que le morphisme
de la proposition ne d\'epend pas du choix des $a(s)$.

Prouvons son injectivit\'e. Si $\eta_i=dg_i$ et $g_{i,j}=g_j-g_i$, et si $a\in A_c$
a pour extr\'emit\'es $s_1$ et $s_2$, alors $g_a$
est la restriction \`a $Y_a$ de la fonction $f_a$ sur $U_a$ valant $g_{s_i}-f_{a(s_i),a}$
sur $Y_{s_i}$ (ces fonctions se recollent sur les $Y_{a,s_i}$ car $f_{a,a(s)}=g_{a,s}=g_s-g_a$). 
Alors $((\omega_a)_a,(f_{a_1,a_2})_{a_1,a_2})$ est le bord
de $(f_a)_a$, ce qui prouve l'injectivit\'e.

Passons \`a la surjectivit\'e.
Compte-tenu de l'injectivit\'e et de ce que tout est (d\'eriv\'e) complet pour la topologie $p$-adique,
il suffit de v\'erifier le r\'esultat mod $p^r$, avec $r>0$. On choisit $r$ assez petit pour que
$Y/p^r\cong (\O_C/p^r)\otimes \ocirc Y^{\rm sp}$. Auquel cas, comme $H^2$ est sans $p$-torsion,
on est ramen\'e au m\^eme \'enonc\'e
pour $\ocirc Y^{\rm sp}$, o\`u les deux membres calculent la cohomologie
de de Rham de $\ocirc Y^{\rm sp}$ vu comme espace de Berkovich et sont donc \'egaux.
\end{proof}

\begin{rema}\label{tetrapil3}
Les m\^emes techniques permettent de prouver que, si on compactifie un short~$Y$
en une $\O_C$-courbe propre $X$
en recollant des disques $D_i$ le long des cercles fant\^omes $C_i$ \`a la fronti\`ere, alors
on peut calculer la cohomologie de de Rham de $X$ en utilisant le recouvrement par $Y$ et les $D_i$.
On en d\'eduit une suite exacte
$$0\to H^1_{\rm dR}(X)\to H^1_{\rm dR}(Y)\oplus \prod_i H^1_{\rm dR}(D_i)
\to \prod_i H^1_{\rm dR}(C_i)\to H^2_{\rm dR}(X)$$
\end{rema}

\section{Cohomologie des boules, des jambes et des cercles fant\^omes}
Dans ce chapitre, on d\'efinit les symboles $\ell$-adiques des jambes et des cercles fant\^omes. Si $\ell\neq p$,
ces groupes ont une expression simple (et classique), et on relie (prop.\,\ref{basic34})
les symboles $p$-adiques \`a la
cohomologie syntomique en se ramenant
aux cas des boules unit\'es ouverte et ferm\'ee. Le lien entre cette cohomologie syntomique et la cohomologie
\'etale $p$-adique est explor\'e dans l'appendice.
\Subsection{Symboles $\ell$-adiques}
\subsubsection{Fonctions inversibles}
Le lemme~\ref{basic32} et le cor.\,\ref{basic33} sont parfaitement classiques,
mais la preuve que nous en donnons est une bonne introduction aux m\'ethodes
utilis\'ees dans l'article.

On munit $\bcris^+[[T]]$ du frobenius $\varphi(\sum_{n\geq 0}a_nT^n)=\sum_{n\geq 0}\varphi(a_n)T^{pn}$.
\begin{lemm}\label{basic31}
Soit $\tilde u\in 1+T\bcris^+[[T]]$.  Les conditions suivantes sont \'equivalentes:

{\rm a)} $\tilde u\in 1+T\acris[[T]]$,

{\rm b)} $\frac{\varphi(\tilde u)}{\tilde u^p}\in 1+pT\acris[[T]]$.

{\rm c)} $\log\frac{\varphi(\tilde u)}{\tilde u^p}\in pT\acris[[T]]$.
\end{lemm}
\begin{proof}
L'implication a)$\Rightarrow$b) est imm\'ediate ainsi que l'\'equivalence de b) et c).
Il ne reste donc que b)$\Rightarrow$a) \`a prouver.
On peut \'ecrire $\tilde u$ sous la forme $\tilde u=\prod_{n\geq 1}(1+u_nT^n)$,
et on montre, par r\'ecurrence que $u_n\in \acris$: si $$\tilde u_n=
\tilde u\prod_{i\leq n-1}(1+u_iT)^{-1},$$ l'hypoth\`ese de r\'ecurrence
permet de montrer que $\frac{\varphi(\tilde u_n)}{\tilde u_n^p}\in 1+pT\acris[[T]]$.
Or le terme de degr\'e $n$ est juste $-pu_n$, et donc $u_n\in\acris$.
\end{proof}

\begin{lemm}\label{basic32}
{\rm (i)}
Tout \'el\'ement $u$ de
 $(\O_C[[T,T^{-1}\rangle[\frac{1}{p}])^\dual$
peut s'\'ecrire, de mani\`ere unique, sous la forme
$u=cT^ku_+u_-$, avec $k\in\Z$, $c\in C^\dual$, $u_+\in 1+T\O_C[[T]]$
et $u_-\in 1+T^{-1}{\goth m}_C\langle T^{-1}\rangle$.

{\rm (ii)} $(v(u),v'(u)):=(v_p(c),k)$ ne d\'epend pas du choix
de l'uniformisante~$T$.
\end{lemm}
\begin{proof}
L'unicit\'e d'une telle \'ecriture est claire: si
$aT^ku_+u_-=bT^\ell v_+v_-$ avec $k\geq \ell$, alors
$\frac{v_-}{u_-}=\frac{a}{b}T^{k-\ell}\frac{u_+}{v_+}$, d'o\`u l'on d\'eduit
que $v_-=u_-$ (car le membre droite appartient \`a $\O_C[[T]][\frac{1}{p}]$ et celui de
gauche \`a $1+T^{-1}{\goth m}_C\langle T^{-1}\rangle$ et l'intersection de
ces deux ensembles est r\'eduite \`a $\{1\}$),
puis que $k=\ell$ et $a=b$ (regarder le terme constant du membre de droite), et donc $u_+=v_+$ aussi.

Prouvons l'existence.  En utilisant la th\'eorie des polygones de Newton, on prouve que
$u=\sum_{n\in\Z}c_nT^n\in \O_C[[T,T^{-1}\rangle[\frac{1}{p}]$ est inversible
si et seulement si $\inf_{n\in\Z}v_p(c_n)$ est atteint.
Si on note $k$ le minimum des $n$ pour lesquels $v_p(c_n)$ est minimum, 
on peut \'ecrire $u$ sous la forme $c_kT^ku_0$, avec $u_0=\sum_{n\in\Z}a_nT^n$, 
avec $a_n\in\O_C$ pour tout $n$, $a_0=1$, $a_n\in{\goth m}_C$ si $n<0$ et
$a_n\to 0$ quand $n\to -\infty$.

Relevons $u_0$ en $\tilde u_0\in\acris[[T,T^{-1}\rangle$ (par exemple
en posant $\tilde u_0=\sum_{n\in\Z}[a_n^\flat]T^n$).  Alors
$\tilde u_0$ est inversible et donc 
on peut \'ecrire $\frac{d\tilde u_0}{\tilde u_0}=a\frac{dT}{T}+f_+\frac{dT}{T}+f_-\frac{dT}{T}$,
avec $a\in\acris$, $f_+\in T\acris[[T]]$ et $f_-\in T^{-1}\acris\langle T^{-1}\rangle$.

On a $\frac{\varphi(\tilde u_0)}{\tilde u_0^p}\in 1+p\acris[[T,T^{-1}\rangle$
et on peut donc \'ecrire $\log\frac{\varphi(\tilde u_0)}{\tilde u_0^p}=p(b+g_++g_-)$,
avec $b\in\acris$, $g_+\in T\acris[[T]]$ et $g_-\in T^{-1}\acris\langle T^{-1}\rangle$,
et on a $d g_++d g_-=(\frac{\varphi}{p}-1)(a+f_++f_-)\frac{dT}{T}$ 
et comme $\varphi(T)$ respecte la d\'ecomposition en puissances positives et n\'egatives de $T$,
on en d\'eduit que $dg_+=(\frac{\varphi}{p}-1)f_+\frac{dT}{T}$.

Posons $\tilde u_+=\exp(\int f_+\frac{dT}{T})\in 1+T\bcris^+[[T]]$. 
On a $\log\frac{\varphi(\tilde u_+)}{(\tilde u_+)^p}=pg_+\in p\acris[[T]]$.
Il en r\'esulte, gr\^ace au lemme~\ref{basic31}, que $\tilde u_+\in 1+\acris[[T]]$,
et donc, si $u_+=\theta(\tilde u_+)$, on a $u_+\in 1+T\O_C[[T]]$.
De plus, par construction, $\frac{du_0}{u_0}-\frac{d u_+}{u_+}=
\theta(a+f_-)\frac{dT}{T}\in \O_C\langle T^{-1}\rangle\frac{dT}{T}$,
et donc $v=\frac{u_0}{u_+}\in\O_C\langle T^{-1}\rangle$
Si $v=\sum_{n\leq 0}b_nT^n$,
il suffit alors de poser $c=c_kb_0$ et $u_-=\frac{v}{b_0}$ pour avoir
la factorisation de $u$ cherch\'ee.

Pour prouver le (ii), posons $\Lambda=\O_C[[T,T^{-1}\rangle$ et $I={\goth m}_C[[T,T^{-1}\rangle$.
Si $S\in\Lambda$ est tel que $\iota_S:\O_C[[S,S^{-1}\rangle\to\Lambda$ est un isomorphisme,
alors l'image de ${\goth m}_C[[S,S^{-1}\rangle$ est encore $I$ et 
$\iota_S$ induit un isomorphisme $k_C((S))\overset{\sim}{\to} k_C((T))$ par passage au quotient
modulo $I$.
Maintenant,
$r=v_p(c)$ est caract\'eris\'e par l'appartenance
de $p^{-r}u$ et $p^ru^{-1}$ \`a $\Lambda$ et $k$ est la valuation $T$-adique
de l'image de $p^{-r}u$ dans $k_C((T))$; c'est donc aussi la valuation $S$-adique
de $p^{-r}u$ dans $k_C((S))$ au vu de l'isomorphisme ci-dessus.
Il s'ensuit que $(r,k)$ ne d\'epend que de $u$ et pas du choix de $T$, comme voulu.
\end{proof}

On rappelle que si $Y$ est une jambe ou un cercle fant\^ome on d\'efinit 
$\O(Y^{\rm gen})$ comme \'etant $\O(Y)[\frac{1}{p}]$; on a donc 
$\O(Y^{\rm gen})^{\dual\dual}=\O(Y)^{\dual\dual}$.

\begin{coro}\label{basic33}
Si $Y$ est une jambe, avec $\O(Y)=\O_C[[T_1,T_2]]/(T_1T_2-p^r)$, et
si $u\in \O(Y^{\rm gen})^\dual$, on peut \'ecrire $u$, de mani\`ere unique, sous la forme
$u=cT_1^ku_+u_-$, avec $k\in\Z$, $c\in C^\dual$, $u_+\in 1+T_1\O_C[[T_1]]$
et $u_-\in 1+T_2\O_C[[T_2]]$.

\end{coro}
\begin{proof}
On a $\O(Y^{\rm gen})^\dual\subset(\O_C[[T_1,T_1^{-1}\rangle[\frac{1}{p}])^\dual$,
ce qui fournit une factorisation de $u$ sous la forme
$u=cT_1^ku_+u_-$ avec $k\in\Z$, $c\in C^\dual$, $u_+\in 1+T_1\O_C[[T_1]]$
et $u_-\in 1+T_1^{-1}{\goth m}_C\langle T_1^{-1}\rangle$.
Mais $$u_-=c^{-1}T_1^{-k}(u^+)^{-1}u\Rightarrow u_-\in \O(Y^{\rm gen})^\dual\cap 
\big(1+T_1^{-1}\O_C\langle T_1^{-1}\rangle\big)=1+T_2\O_C[[T_2]],$$
ce qui permet de conclure.
\end{proof}

\subsubsection{Symboles}\label{BAS6.6}
Si $Y$ est un cercle fant\^ome une jambe ou la boule unit\'e ferm\'ee, on pose
$${\rm Symb}_\ell(Y^{\rm gen})=\Z_\ell\wotimes\O(Y^{\rm gen})^\dual.$$
Si $Y$ est une jambe ou la boule unit\'e ferm\'ee, 
on d\'efinit $A_{\ell,n}(Y^{\rm gen})$ comme le sous-groupe de $C(Y^{\rm gen})^\dual$
 des $f$ dont le diviseur est \`a support fini (ceci est automatique dans le cas
de la boule) et divisible par $\ell^n$.
\begin{lemm}\label{basic13}
On a un isomorphisme naturel
$${\rm Symb}_\ell(Y^{\rm gen})
\overset{\sim}{\to}
\varprojlim_n\big(A_{\ell,n}(Y^{\rm gen})/(C(Y^{\rm gen})^\dual)^{\ell^n}\big)$$
\end{lemm}
\begin{proof}
L'injectivit\'e est une cons\'equence de ce que, si $f\in\O(Y^{\rm gen})^\dual$
et si $f=g^{\ell^n}$ avec $g\in C(Y^{\rm gen})^\dual$, alors $g\in\O(Y^{\rm gen})^\dual$
puisque $\O(Y^{\rm gen})$ est normal (ou bien en remarquant que $g$ n'a ni z\'ero ni p\^ole sur $Y$). 
Pour la surjectivit\'e soit $f=(f_n)_{n\in\N}$ un \'el\'ement du membre de droite. Puisque 
tout diviseur de support fini
est principal,
il existe $h_n\in C(Y^{\rm gen})^\dual$ tel que $u_n:=f_nh_n^{-\ell^n}\in\O(Y^{\rm gen})^\dual$. Or
$f_{n+1}/f_n\in (C(Y^{\rm gen})^\dual)^{\ell^n}$, donc $u_{n+1}/u_n\in (C(Y^{\rm gen})^\dual)^{\ell^n}$ et en utilisant encore la normalit\'e de $\O(Y^{\rm gen})$ on obtient $u_{n+1}/u_n\in (\O(Y^{\rm gen})^\dual)^{\ell^n}$. Ainsi 
$(u_n)_n$ induit un \'el\'ement de $\Z_\ell\wotimes\O(Y^{\rm gen})^\dual$ s'envoyant sur $(f_n)_n$, ce qui permet de conclure.
\end{proof}

\begin{prop}\label{jambe11}
Si $Y$ est une jambe ou un cercle fant\^ome, et si $\ell\neq p$, 
l'application r\'esidu\footnote{Obtenue, par continuit\'e,
\`a partir de $f\mapsto{\rm Res}_Y\frac{df}{f}$.} induit un isomorphisme
$${\rm Res}:{\rm Symb}_\ell(Y^{\rm gen})\overset{\sim}{\longrightarrow}\Z_\ell.$$
\end{prop}
\begin{proof}
Il ressort du lemme~\ref{basic32} et du cor.~\ref{basic33} (et de ce que $C^\dual$
est $\ell$-divisible) que l'on a des isomorphismes
$$\Z_\ell\wotimes\O(Y^{\rm gen})^\dual\cong
\begin{cases}
\Z_\ell\oplus(\Z_\ell\wotimes\O(\ocirc B)^\dual)\oplus (\Z_\ell\wotimes\O(\ocirc B)^\dual)
&{\text{si $Y$ est une jambe,}}\\
\Z_\ell\oplus(\Z_\ell\wotimes\O(\ocirc B)^\dual)\oplus (\Z_\ell\wotimes\O(B)^{\dual\dual})
&{\text{si $Y$ est un cercle fant\^ome}}
\end{cases}$$
On conclut en remarquant que $1+X\O_C[[X]]$ et $1+X{\goth m}_C\langle X\rangle$
sont tous les deux $\ell$-divisibles si $\ell\neq p$ car la s\'erie
$\sum_{k\geq 0}\binom{1/\ell}{k}(x-1)^k$ converge dans les deux cas.
\end{proof}

\Subsection{R\'egulateur syntomique}\label{BAS19.1}
\subsubsection{Le r\'egulateur}\label{BAS19.2}
Si $Y$ est une jambe (avec $\O(Y)=\O_C[[T_1,T_2]]/(T_1T_2-p^r)$),
on pose $\O(\widetilde Y)=\acris[[T_1,T_2]]/(T_1T_2-\tilde p^r)$ comme au \no\ref{adoc15}.
Si $Y$ est
un cercle fant\^ome (avec $\O(Y)=\O_C[[T,T^{-1}\rangle$),
on pose $\O(\widetilde Y)=\acris[[T,T^{-1}\rangle$.
Si $Y$ est la boule unit\'e ouverte $\ocirc B$ (et donc $\O(Y)=\O_C[[T]]$),
on pose $\O(\widetilde Y)=\acris[[T]]$, et si $Y$ est la boule unit\'e
ferm\'ee $B$ (i.e.~$\O(Y)=\O_C\langle T\rangle$), on pose
$\O(\widetilde Y)=\acris\langle T\rangle$.

Dans tous les cas:

$\bullet$ On dispose d'une surjection $\O(\widetilde Y)\to\O(Y)$ et on note
$F^1\O(\widetilde Y)$ le noyau.

$\bullet$
 On munit $\O(\widetilde Y)$ du frobenius $\varphi$, agissant
comme on pense sur $\acris$, et par $T_i\mapsto T_i^p$ ou $T\mapsto T^p$ suivant les cas. 

$\bullet$ On note $\Omega^1(\widetilde Y)$ le module des 
diff\'erentielles continues de $\O(\widetilde Y)$ relativement \`a $\acris$. On a donc 
 $\Omega^1(\widetilde Y)=\O(\widetilde Y) \frac{dZ}{Z}$ avec $Z=T_1$ (resp. $Z=T$) si $Y$ est une jambe
 (resp. un cercle fant\^ome), et on a $\Omega^1(\widetilde Y)=\O(\widetilde Y) dT$ si 
 $Y=B$ ou $Y=\ocirc B$. 

\smallskip
Si $Y=\ocirc{B}, B$, une jambe ou un cercle fant\^ome,
les {\it groupes de cohomologie syntomique $H^i_{\rm syn}(Y,1)$},
sont les groupes de cohomologie du complexe (notons que
$\varphi(F^1\O(\widetilde Y))\subset p \O(\widetilde Y)$ et
$\varphi(\Omega^1(\widetilde Y))\subset p\Omega^1(\widetilde Y)$, ce qui donne un sens \`a $\frac{\varphi}{p}$):
$${\rm Syn}(Y,1):=\xymatrix@C=1.2cm{F^1\O(\widetilde Y)\ar[r]^-{d,1-\frac{\varphi}{p}}
&\Omega^1(\widetilde Y)\oplus \O(\widetilde Y)\ar[r]^-{(1-\frac{\varphi}{p})-d}
&\Omega^1(\widetilde Y)}.$$
Remarquons que, dans tous les cas, un \'el\'ement de $H^0$ est une constante (car tu\'e par $d$)
appartenant \`a $(F^1\acris)^{\varphi=p}$, et donc
$$H^0_{\rm syn}(Y,1)=\Z_pt.$$
Comme $\Z_p\wotimes C^\dual=0$, on peut \'ecrire $u\in \Z_p\wotimes \O(Y^{\rm gen})^\dual$
sous la forme $\prod_{n\geq 0}(T^{a_n}u_n^{p^n})$, avec $a_n\in\Z$ (et $a_n=0$ si $Y=\ocirc B,B$)
et $u_n\in\O(Y)^{\dual}$. 
Si on choisit un rel\`evement $\tilde u_n$ de $u_n$ dans $\O(\widetilde Y)^*$,
et si on pose $a=\sum_{n\geq 0}p^na_n$ et
$$x=a\tfrac{dT_1}{T_1}+\sum_{n\geq 0}p^n\tfrac{d\tilde u_n}{\tilde u_n}
\quad{\rm et}\quad
 y=\tfrac{1}{p}\sum_{n\geq 0}p^n\log\tfrac{\varphi(\tilde u_n)}{\tilde u_n^p},$$
alors $(x,y)$ est un $1$-cocycle de ${\rm Syn}(Y,1)$.
 Comme $\tilde u_n$ est bien d\'etermin\'e \`a un \'el\'ement de $1+F^1\O(\widetilde Y)$ pr\`es, la classe de $(x,y)$ dans $H^1_{\rm syn}(Y,1)$ ne d\'epend que de $u$, et pas du choix
des $\tilde u_n$ (car $(\tfrac{d\tilde v}{\tilde v}, \frac{1}{p}\log\tfrac{\tilde v^p}{\varphi(\tilde v)})$
est le bord de $\log\tilde v$ qui appartient \`a $F^1\O(\widetilde Y)$ si $v\in 1+F^1\O(\widetilde Y)$ 
puisque $F^1\acris$ admet des puissances divis\'ees).
Elle d\'efinit donc une application 
$$\delta_Y:{\rm Symb}_p(Y^{\rm gen})\cong \Z_p\wotimes\O(Y^{\rm gen})^\dual\to H^1_{\rm syn}(Y,1).$$

\subsubsection{La boule ouverte}\label{BAS19}
Dans ce cas, 
$1+T\O_C[[T]]$ est complet pour la topologie $p$-adique, et
l'application naturelle $1+T\O_C[[T]]\to \Z_p\wotimes\O(\ocirc{B}^{\rm gen})^\dual$
est un isomorphisme.
On peut donc d\'efinir $\delta_{\ocirc{B}}$ sans passage \`a la limite:
si $u\in 1+T\O_C[[T]]$, on choisit un rel\`evement $\tilde u$
de $u$ dans $1+T\acris[[T]]$ (i.e $\theta(\tilde u)=u$), et alors
$\big(\frac{d\tilde u}{\tilde u},\frac{1}{p}\log\frac{\tilde u^p}{\varphi(\tilde u)}\big)$
d\'efinit un $1$-cocycle de ${\rm Syn}({\ocirc{B}},1)$: on a $\frac{\tilde u^p}{\varphi(\tilde u)}\in 1+pT\acris[[T]]$,
et donc $\frac{1}{p}\log\frac{\tilde u^p}{\varphi(\tilde u)}\in\acris[[T]]$.  Ce cocycle est
le cobord de $\log\tilde u$ qui n'appartient \`a $F^1\acris[[T]]$ que si $u=1$.
On a donc construit une injection naturelle
$$\delta_{\ocirc{B}}:1+T\O_C[[T]]\hookrightarrow H^1_{\rm syn}({\ocirc{B}},1).$$
\begin{prop} \label{basic30}
On a 
$H^0_{{\rm syn}}({\ocirc{B}},1)=\Z_pt$, $H^2_{\rm syn}({\ocirc{B}},1)=0$
et $\delta_{\ocirc{B}}$ induit 
un isomorphisme
$$\delta_{\ocirc{B}}:1+T\O_C[[T]]\overset{\sim}\to H^1_{\rm syn}({\ocirc{B}},1).$$
\end{prop}

\begin{proof}
La nullit\'e de $H^2$ r\'esulte de ce que $1-T^{p-1}\varphi$ est surjectif
sur $\acris[[T]]$ (et m\^eme bijectif, d'inverse $1+T^{p-1}\varphi+(T^{p-1}\varphi)^2+\cdots$), 
ce qui implique que 
$1-\frac{\varphi}{p}: \Omega^1(\tilde{\ocirc{B}})\to  \Omega^1(\tilde{\ocirc{B}})$ est surjectif. 

Maintenant, on a 
$${\rm Syn}({\ocirc{B}},1)={\rm Syn}({\ocirc{B}},1)^+\oplus C_0^\bullet,$$ o\`u
$${\rm Syn}({\ocirc{B}},1)^+
:=\xymatrix@C=1.2cm{F^1\O(\tilde{\ocirc{B}})_0\ar[r]^-{d,1-\frac{\varphi}{p}}
&\Omega^1(\tilde{\ocirc{B}})\oplus \O(\tilde{\ocirc{B}})_0\ar[r]^-{(1-\frac{\varphi}{p})-d}
&\Omega^1(\tilde{\ocirc{B}})}$$
avec $\O(\tilde{\ocirc{B}})_0=T\acris[[T]]$,
et $$C_0^\bullet=
[\xymatrix@C=.6cm{F^1\acris\ar[r]^-{1-\frac{\varphi}{p}}&\acris\ar[r]&0}].$$
On a $H^0(C_0^\bullet)=\Z_pt$, $H^1(C_0^\bullet)=0$ et
$H^0({\rm Syn}({\ocirc{B}},1)^+)=0$.  Pour conclure, il suffit donc
de prouver que 
l'injection $\delta_{\ocirc{B}}:1+T\O_C[[T]]\hookrightarrow 
H^1({\rm Syn}({\ocirc{B}},1)^+)$
est une surjection.
Soit donc $(x,y)$ un $1$-cocycle de ${\rm Syn}({\ocirc{B}},1)^+$,
avec $x=\big(\sum_{n\geq 0}x_n T^n\big)dT$, $y=\sum_{n\geq 1}y_nT^n$.
Soit $f=\sum_{n\geq 1}\frac{x_{n-1}}{n}T^n$ de telle sorte que
$f\in T\bcris^+[[T]]$ v\'erifie $df=x$, et soit $\tilde u=\exp f\in1+T\bcris^+[[T]]$.
Il s'agit de prouver qu'en fait $\tilde u\in 1+T\acris[[T]]$ car alors 
$(x,y)=\delta_{\ocirc{B}}(\theta(\tilde u))$. 

Or $(1-\tfrac{\varphi}{p})f=y$ car les deux membres ont m\^eme terme constant (i.e.~$0$)
et m\^eme diff\'erentielle puisque
$$d\big((1-\tfrac{\varphi}{p})f\big)=(1-\tfrac{\varphi}{p})\cdot df=
(1-\tfrac{\varphi}{p})x=dy,$$
et donc $\varphi(f)-pf=py\in pT\acris[[T]]$ et
$\frac{\varphi(\tilde u)}{\tilde u^p}\in 1+pT\acris[[T]]$,
ce qui permet d'utiliser le lemme~\ref{basic31} pour conclure.
\end{proof}

\subsubsection{La boule ferm\'ee}
Soit $B$ la boule ferm\'ee.  Alors 
$$\Z_p\wotimes\O(B^{\rm gen})^{\dual}=\Z_p\wotimes(1+T{\goth m}_C\langle T\rangle).$$
Soit ${\goth m}_{\rm cris}=\{x\in\acris,\ \theta(x)\in{\goth m}_C\}$.
\begin{prop}\label{boule1}
$H^2_{\rm syn}(B,1)=\Z_pt$, $H^2_{\rm syn}(B,1)=0$ et
$\delta_B$ induit un isomorphisme
$$\delta_B:\Z_p\wotimes\O(B^{\rm gen})^{\dual\dual}\cong H^1_{\rm syn}(B,1),$$
si $p>2$ {\rm(}si $p=2$, $\delta_B$ est presque un isomorphisme\footnote{Cela veut dire
que $\delta_B$ est injectif et son conoyau est tu\'e par $2$.}{\rm )}.
\end{prop}
\begin{proof}
Le calcul de $H^0$ a d\'ej\`a \'et\'e fait.  Pour prouver la nullit\'e de $H^2$, il suffit de montrer que $T^kdT\in (1-\varphi/p)\Omega^1(\widetilde B)+d\O(\widetilde B)$ pour tout 
$k$. En \'ecrivant $k+1=p^in$ avec $n$ premier \`a $p$ et en posant $f_n=\frac{T^n}{n}$, on a 
$$T^kdT=df_n+(1-\varphi/p)\omega, \quad \omega:=-\sum_{j=0}^{i-1} (\varphi/p)^j(df_n).$$

L'injectivit\'e de $\delta_B$ se d\'emontre comme pour la boule ouverte. La surjectivit\'e est nettement plus d\'elicate et demande quelques 
pr\'eliminaires, qui seront aussi utilis\'es dans le chapitre \ref{short1}.

\begin{lemm}\label{trick}
 Soient $(a_n)_{n\geq 0}$ et $(b_n)_{n\geq 0}$ deux suites dans $\acris$, qui tendent vers $0$ et telles que 
 $b_n=\varphi^n(a_0)+p\varphi^{n-1}(a_1)+...+p^na_n$ pour tout $n$. Alors $\theta_0(b_n)/p^{n+1}\in \O_{\breve C}$ pour tout $n$ et 
 $\lim_{n\to\infty} \theta_0(b_n)/p^{n+1}=0$.
 \end{lemm}

\begin{proof} On a $\acris=\O_{\breve C}\oplus{\rm Ker}\,\theta_0$ et cette d\'ecomposition est $\varphi$-\'equivariante.
En appliquant $\theta_0$ \`a l'\'egalit\'e 
$$b_{n+s}=\varphi^s(b_n)+p^{n+1} \varphi^{s-1}(a_{n+1})+...+p^{n+s} a_{n+s}$$
et en passant \`a la limite pour $s\to\infty$, on obtient 
$$v_p(\theta_0(b_{n}))\geq \inf_{k\geq n+1} (k+v_p(\theta_0(a_{k}))),$$
 ce qui permet de conclure puisque $\lim_{k\to\infty} v_p(\theta_0(a_k))=\infty$.
\end{proof}

\begin{lemm}\label{AH}
 Si $p>2$ et $a\in \ker(\theta_0)$, il existe $z\in TF^1\acris\langle T\rangle$ tel que 
 $$\exp\big(z+\sum_{\ell\geq 0}p^{-\ell}\varphi^\ell(a)T^{p^\ell}\big)\in
1+T{\goth m}_{\rm cris}\langle T\rangle.$$ 
Si $a\in p^N\ker(\theta_0)$ on peut choisir un tel $z$ dans $p^NTF^1\acris\langle T\rangle$. 
\end{lemm}

\begin{proof} Puisque $\acris=\ainf[\frac{\tilde p^k}{k!}]^\wedge$, il existe une suite $x_k\in \ainf$ tendant vers $0$ telle que 
$a=\sum_{k\geq 0} x_k\frac{\tilde p^k}{k!}$. Notons $a_1=\sum_{k=0}^{p-1} x_k\frac{\tilde p^k}{k!}$ et 
$a_2=a-a_1$. Puisque $\theta_0(a)=0$, on a $x_0\in  W({\goth m}_{C^\flat})$ et donc $a_1\in  W({\goth m}_{C^\flat})$ aussi. 
En \'ecrivant $a_1=\sum_{k\geq 0} p^k[b_k]$ avec $b_k\in {\goth m}_{C^\flat}$, on obtient 
$$\exp\big(\sum_{\ell\geq 0}p^{-\ell}\varphi^\ell(a_1)T^{p^\ell}\big)=\prod_{k\geq 0} \exp_{\rm AH}([b_k]T)^{p^k}\in 1+T{\goth m}_{\rm cris}\langle T\rangle,$$
puisque l'exponentielle d'Artin-Hasse
$$\exp_{\rm AH}(X)=\exp(\sum_{i\geq 0}p^{-i}X^{p^i})$$ appartient \`a $1+X\Z_p[[X]]$. 
  
   Il suffit donc de d\'emontrer le r\'esultat pour $a_2$. La relation $\varphi(\tilde{p})=\tilde{p}^p$ et l'in\'egalit\'e 
   $v_p(\frac{(p^\ell k)!}{k!})=k\frac{p^\ell-1}{p-1}\geq p^\ell$, valable pour $k\geq p$
et $\ell\geq 1$, montrent que la s\'erie $\sum_{\ell\geq 1}p^{-\ell}\varphi^\ell(a_2)T^{p^\ell}$ converge dans 
$p\acris\langle T\rangle$ et, puisque $p>2$, son exponentielle est dans $1+T{\goth m}_{\rm cris}\langle T\rangle$. Il reste donc \`a montrer l'existence de $z\in TF^1\acris\langle T\rangle$ tel que $\exp(z+a_2T)\in  1+T{\goth m}_{\rm cris}\langle T\rangle$. Il suffit de poser 
$z=-\sum_{k\geq p} x_k\frac{\tilde{p}^k-p^k}{k!}T$ et de noter que $z+a_2T=\sum_{k\geq p} x_k\frac{p^k}{k!}T\in p\acris\langle T\rangle$.
\end{proof}

\begin{rema}
Pour $p=2$, la preuve ci-dessus tombe en d\'efaut mais reste valable si on remplace $a$ par $2a$.
\end{rema}

Revenons \`a la preuve de la surjectivit\'e. Soit $(x,y)$ un $1$-cocycle de ${\rm Syn}(B,1)^+$,
avec $x=\big(\sum_{n\geq 0}x_n T^n\big)dT$ et $y=\sum_{n\geq 1}y_nT^n$.
En d\'ecomposant $n$ sous la forme $n=p^ij$, avec $(j,p)=1$,
on peut \'ecrire $$x=\sum_{(j,p)=1} jp^{c(j)}b_j,\quad y=\sum_{(j,p)=1} p^{c(j)}a_j$$
avec  $c(j)\to+\infty$
quand $j\to\infty$ et $$b_j=\sum_{i\geq 0}b_{i,j}T^{jp^i-1}dT\in \acris\langle T\rangle dT,\quad a_j=\sum_{i\geq 0}a_{i,j}T^{jp^i}\in \acris\langle T\rangle.$$

  Soit $$f_j=\sum_{i\geq 0}\tfrac{b_{i,j}}{p^i}T^{jp^i}.$$ Il suffit de montrer l'existence de 
  $z_j\in TF^1\acris\langle T\rangle $ qui tendent vers $0$ et tels que $\tilde u_j=\exp (z_j+ f_j)\in1+T\bcris^+[[T]]$
appartienne \`a $1+T{\goth m}_{\rm cris}\langle T\rangle$, car alors
$(x,y)=\delta_B(\prod_j\theta(\tilde u_j)^{p^{c(j)}})$ dans $H^1_{\rm syn}(B,1)$.

La relation $(1-\frac{\varphi}{p})x=dy$ se traduit par
$b_{0,j}=a_{0,j}$ et $b_{i,j}=\varphi(b_{i-1,j})+p^{i}a_{i,j}$ si $i\geq 1$.
On a alors $$b_{i,j}=\varphi^{i}(a_{0,j})+p\varphi^{i-1}(a_{1,j})+\cdots+p^{i}a_{i,j}.$$
Cela montre d'une part que 
$$f_j=\sum_{k\geq 0}\big(\sum_{\ell\geq 0}p^{-\ell}\varphi^\ell(a_{k,j})T^{jp^{k+\ell}}\big)$$
et d'autre part 
(lemme \ref{trick}) que $\sum_{i\geq 1}\theta_0(b_{i,j}) p^{-i}T^{jp^i}$ converge dans $p\O_{\breve C}\langle T\rangle$, donc 
 son exponentielle appartient \`a $1+T{\goth m}_{\rm cris}\langle T\rangle$
(si $p=2$, il faut multiplier $x$ et $y$ par $2$ pour assurer cette appartenance). 
En utilisant la d\'ecomposition $\varphi$-\'equivariante $\acris=\O_{\breve C}\oplus{\rm Ker}\,\theta_0$ on peut ainsi se ramener 
 au cas $a_{i,j}\in \ker(\theta_0)$ pour tous 
 $i,j$. Comme $a_{k,j}\to 0$ dans ${\rm Ker}\,\theta_0$ (i.e.~$p$-adiquement), il suffit d'appliquer le lemme \ref{AH} pour conclure. 
\end{proof}

\subsubsection{Jambes et cercles fant\^omes}
Soit $Y$ une jambe 
ou un cercle fant\^ome 
(i.e.~$\O(Y)=\O_C[[T_1,T_2]]/(T_1T_2-p^r)$ ou
$\O_C[[T_1,T_1^{-1}\rangle$).

\begin{prop}\label{basic34}
$H^0_{\rm syn}(Y,1)=\Z_pt$ et
$\delta_Y: {\rm Symb}_p(Y^{\rm gen})\to H^1_{\rm syn}(Y,1)$
est un isomorphisme, si $p>2$ {\rm(}et presque un isomorphisme si $p=2${\rm)}.
\end{prop}
\begin{proof}
Il ressort du lemme~\ref{basic32} et du cor.~\ref{basic33} (et de ce que $C^\dual$
est $p$-divisible) que l'on a des isomorphismes
$$\Z_p\wotimes\O(Y^{\rm gen})^\dual\cong
\begin{cases}
\Z_p\oplus(\Z_p\wotimes\O(\ocirc B)^\dual)\oplus (\Z_p\wotimes\O(\ocirc B)^\dual)
&{\text{si $Y$ est une jambe,}}\\
\Z_p\oplus(\Z_p\wotimes\O(\ocirc B)^\dual)\oplus (\Z_p\wotimes\O(B)^{\dual\dual})
&{\text{si $Y$ est un cercle fant\^ome.}}
\end{cases}$$
Le complexe ${\rm Syn}(Y,1)$ se d\'ecompose pareillement (en regardant le signe
des puissances de $T$) sous la forme
$${\rm Syn}(Y,1)\cong
\begin{cases}
{\rm Syn}(\ocirc B,1)^+\oplus {\rm Syn}(\ocirc B,1)^+\oplus C_1^\bullet
&{\text{si $Y$ est une jambe,}}\\
{\rm Syn}(\ocirc B,1)^+\oplus {\rm Syn}(B,1)^+\oplus C_1^\bullet
&{\text{si $Y$ est un cercle fant\^ome,}}
\end{cases}$$
o\`u $C_1^\bullet$ est le complexe
$$ [\xymatrix@C=1.2cm{F^1\acris
\ar[r]^-{(0,1-\frac{\varphi}{p})}&\acris\tfrac{dT_1}{T_1}\oplus\acris\ar[r]
^-{(1-\frac{\varphi}{p})-(0)}&\acris\tfrac{dT_1}{T_1}}].$$
On a $H^0(C_1^\bullet)=\Z_pt$ et $H^1(C_1^\bullet)=\Z_p\frac{dT_1}{T_1}$ (car
$\varphi(\frac{dT_1}{T_1})=p\frac{dT_1}{T_1}$).
Cela permet de d\'eduire le r\'esultat de ceux pour les boules
(prop.~\ref{basic30} pour la boule ouverte et prop.~\ref{boule1} pour la boule ferm\'ee).
\end{proof}

\begin{rema}
Si $Y$ est la boule ouverte  ou une jambe,
la relation entre les $H^i_{\rm syn}(Y,1)$ et les $H^i_{\eet}(Y,\Z_p(1))$ est analys\'ee
dans l'appendice.
\end{rema}

\subsection{D\'eg\'en\'erescence de couronnes}\label{jambe13.1}
Les r\'esultats de ce \S~et du suivant seront utilis\'es dans le chap.\,\ref{BAS4}
 pour les preuves
de la prop.\,\ref{dege1} et celle du th.\,\ref{basic21} (via le lemme~\ref{BAS12.1}).

Soit $Y$ une jambe, avec $\O(Y)=\O_C[[T_1,T_2]]/(T_1T_2-p^r)$
et soient $C_1,C_2$ les cercles fant\^omes aux extr\'emit\'es de $Y$,
avec $\O(C_i)=\O_C[[T_i,T_i^{-1}\rangle$ si $i=1,2$.

\subsubsection{D\'eg\'en\'erescence g\'eom\'etrique} \label{jambe13.2}
Soit $Y^\infty$ la couronne singuli\`ere $\O(Y^\infty)=\O_C[[T_1,T_2]]/(T_1T_2)$.
Les cercles fant\^omes aux extr\'emit\'es de $Y^\infty$ sont encore $C_1$ et $C_2$.
Si $Z=Y,Y^\infty,C_i$, on d\'efinit $\Omega^1(Z)$ comme le module $\O(Z)\frac{dT_1}{T_1}$
(avec la relation $\frac{dT_2}{T_2}=-\frac{dT_1}{T_1}$).
Soit $\Omega^\bullet(Z)$ le complexe
$$\Omega^\bullet(Z)=\big(\O(Z)\to\Omega^1(Z)\big).$$

Notons que tout \'el\'ement de $\O(Y^\infty)$ ou $\O(Y)$ peut s'\'ecrire,
de mani\`ere unique, sous la forme $a_0+\sum_{n\geq 1}a_nT_1^n+\sum_{n\geq 1}b_nT_2^n$.
De m\^eme, tout \'el\'ement de $\Omega^1(Y^\infty)$ ou $\Omega^1(Y)$ peut s'\'ecrire,
de mani\`ere unique, sous la forme 
$\big(a_0+\sum_{n\geq 1}a_nT_1^n+\sum_{n\geq 1}b_nT_2^n)\frac{dT_1}{T_1}$.
Cela fournit un isomorphisme naturel de complexes $\Omega^\bullet(Y^\infty)\cong \Omega^\bullet(Y)$.

Le probl\`eme que nous allons rencontrer est que cet isomorphisme ne commute pas
aux fl\`eches naturelles vers $\Omega^\bullet(C_i)$.  Pour rem\'edier \`a ce probl\`eme, on
introduit le quotient $\overline \Omega^\bullet(C_i)$ de $\Omega^\bullet(C_i)$ 
d\'efini par:
$$\overline \Omega^\bullet(C_i)=\Omega^\bullet(C_i)/\big(T_i^{-1}\O_C\langle T_i^{-1}\rangle
\to d(T_i^{-1}\O_C\langle T_i^{-1}\rangle)\big).$$
(Le complexe $T_i^{-1}\O_C\langle T_i^{-1}\rangle
\to d(T_i^{-1}\O_C\langle T_i^{-1}\rangle)$ est acyclique et donc $\overline \Omega^\bullet(C_i)$
est quasi-isomorphe \`a $\Omega^\bullet(C_i)$.)
\begin{lemm}\label{jambe1}
Si $p^Nr\geq 1$, 
le diagramme suivant commute \`a $p^{N-1}$ pr\`es:
$$\xymatrix@R=.5cm{\Omega^\bullet(Y^\infty)\ar[r]\ar[d]& \overline \Omega^\bullet(C_i)\ar@{=}[d]\\
\Omega^\bullet(Y)\ar[r]& \overline \Omega^\bullet(C_i)}$$
\end{lemm}
\begin{proof}
On peut supposer que $i=1$.  Alors $T_1^n\in\O(Y^\infty)$, pour $n\geq 0$, s'envoie
sur $T_1^n$, que ce soit par $\O(Y^\infty)\to \O(C_1)$ ou par $\O(Y^\infty)\to\O(Y)\to\O(C_1)$.
Quant \`a $T_2^n$, pour $n\geq 1$, il s'envoie sur $0$ par $\O(Y^\infty)\to \O(C_1)$
et sur $p^{nr}T_1^{-n}$ par $\O(Y^\infty)\to\O(Y)\to\O(C_1)$, et donc sur $0$ si on
quotiente par $T_1^{-1}\O_C\langle T_1^{-1}\rangle$.

De m\^eme, $T_1^n\frac{dT_1}{T_1}$, pour $n\geq 0$, s'envoie sur $T_1^n\frac{dT_1}{T_1}$
par les deux chemins.  Par contre, $T_2^n\frac{dT_1}{T_1}$, pour $n\geq 1$, s'envoie
sur $0$ par $\Omega^1(Y^\infty)\to \Omega^1(C_1)$ et sur
$p^{nr}T_1^{-n}\frac{dT_1}{T_1}$ par $\Omega^1(Y^\infty)\to\Omega^1(Y)\to\Omega^1(C_1)$.
Or $p^{nr}T_1^{-n}\frac{dT_1}{T_1}=-d(\frac{p^{nr}}{n}T_1^{-n})$,
et on v\'erifie (il suffit de consid\'erer $n=p^k$)
que $p^{N-1}\frac{p^{nr}}{n}\in \Z_p$, pour tout $n\geq 1$, si $p^Nr\geq 1$,
et donc $p^{N-1}T_2^n\frac{dT_1}{T_1}$ s'envoie sur $0$ si on quotiente par
$d(T_1^{-1}\O_C\langle T_1^{-1}\rangle)$.
\end{proof}

\subsubsection{D\'eg\'en\'erescence arithm\'etique}\label{jambe13.3}
On pose 
$$\O(\widetilde Y)=\acris[[T_1,T_2]]/(T_1T_2-\tilde p^r)
\quad{\rm et}\quad
\O(\breve Y)=\O_{\breve C}[[T_1,T_2]]/(T_1T_2).$$
On a $$\O(Y)=\O_C\otimes_{\acris}\O(\widetilde Y)
\quad{\rm et}\quad
\O(\breve Y)=\O_{\breve C}\otimes_{\acris}\O(\widetilde Y).$$

Si $h$ est\footnote{Nous n'aurons besoin que de $h=0$ dans cet article sauf dans
le \no\ref{BAS12} pour passer du complexe d\'efinissant la cohomologie de Hyodo-Kato
\`a un complexe quasi-isomorphe sur lequel l'isomorphisme de Hyodo-Kato devient transparent.}
 un entier~$\geq 0$,
et si $i=1,2$, on d\'efinit $\O(\widetilde C_i)^{{\rm PD}_h}$ et $\O(\breve C_i)^{{\rm PD}_h}$
comme les compl\'et\'es $p$-adiques\footnote{${\rm PD}_h$ r\'ef\`ere \`a "puissances divis\'ees
partielles de niveau $h$"; si $h=0$, on obtient les puissances divis\'ees standard.
Le $U$ et ses puissances divis\'ees n'ont d'int\'er\^et 
que dans une situation o\`u le cercle fant\^ome est l'intersection
d'une jambe $Y$ et d'un short $Y'$, et on veut mettre un frobenius sur le complexe de de Rham,
ce qui demande de travailler dans $\widetilde Y\times\widetilde Y'$ et de prendre
un ${\rm PD}_h$-voisinage de l'intersection; le lemme de Poincar\'e permet de montrer que rajouter $U$
ne change pas la cohomologie de de Rham de $\widetilde C_i$, \`a $p^h$ pr\`es.}
\begin{align*}
\O(\widetilde C_i)^{{\rm PD}_h}&=\big(\acris[[T_i,T_i^{-1}\rangle[\tfrac{U^k}{[k/p^h]!},\ k\in\N]\big)^{\wedge{\text-}p}
\\
\O(\breve C_i)^{{\rm PD}_h}&=\big(\O_{\breve C}[[T_i,T_i^{-1}\rangle[\tfrac{U^k}{[k/p^h]!},\ k\in\N]\big)^{\wedge{\text-}p}
\end{align*}

On d\'efinit les modules $\Omega^1(\widetilde C_i)^{{\rm PD}_h}$ et $\Omega^1(\breve C_i)^{{\rm PD}_h}$
comme les modules des diff\'erentielles continues sur $\acris$ et $\O_{\breve C}$
respectivement, et on d\'efinit les complexes
$$
\Omega^{\tau\leq 1}(\widetilde C_i)=\big(\O(\widetilde C_i)^{{\rm PD}_h}\to \Omega^1(\widetilde C_i)^{{\rm PD}_h}_{d=0}\big),
\quad
\Omega^{\tau\leq 1}(\breve C_i)=\big(\O(\breve C_i)^{{\rm PD}_h}\to \Omega^1(\breve C_i)^{{\rm PD}_h}_{d=0}\big),
$$
Ces complexes calculent la cohomologie de de Rham de $\widetilde C_i$ et $\breve C_i$ (\`a $p^h$ pr\`es), 
mais
comme ci-dessus, nous aurons besoin d'un quotient quasi-isomorphe.
Si $Z=\widetilde C_i,\breve C_i$, et si $\Lambda_Z=\acris$ si $Z=\widetilde C_i$ et
$\Lambda_Z=\O_{\breve C}$ si $Z=\breve C_i$, on pose:
$$\overline \Omega^{\tau\leq 1}(Z)= \Omega^{\tau\leq 1}(Z)/\big(T_i^{-1}\Lambda_Z\langle T_i^{-1}\rangle\to
d\big(T_i^{-1}\Lambda_Z\langle T_i^{-1}\rangle\big)\big).$$
\begin{lemm}\label{jambe2}
Si $p^Nr\geq 1$, 
le diagramme suivant commute \`a $p^N$ pr\`es:
$$\xymatrix@R=.5cm{\Omega^\bullet(\breve Y)\ar[r]\ar[d]& \overline \Omega^{\tau\leq 1}(\breve C_i)\ar[d]\\
\Omega^\bullet(\widetilde Y)\ar[r]& \overline \Omega^{\tau\leq 1}(\widetilde C_i)}$$
\end{lemm}
\begin{proof}
La preuve est la m\^eme que celle du lemme~\ref{jambe1} \`a la petite diff\'erence pr\`es
qu'il faut utiliser l'appartenance de $p^N\frac{\tilde p^{nr}}{n}$ \`a $\acris$, si $p^Nr\geq 1$
(cela r\'esulte de ce que $\tilde p^p\in p\acris$).
\end{proof}

\section{Cohomologie des shorts}\label{short1}
Dans ce chapitre, le plus technique de l'article, on calcule la cohomologie
des affino\"{\i}des ayant bonne r\'eduction: on exprime la cohomologie \'etale $\ell$-adique
en termes de symboles (cor.\,\ref{basic12}), et on en d\'eduit, si $\ell\neq p$, une expression
de cette cohomologie \'etale en termes de la jacobienne de la r\'eduction 
(prop.\,\ref{short21}, r\'esultat parfaitement classique). Dans le cas $\ell=p$, on relie
les symboles \`a la cohomologie syntomique 
via un r\'egulateur syntomique (\no\ref{BAS18.1}) et on prouve un d\'evissage
de la cohomologie \'etale faisant intervenir la cohomologie de de Rham s\'epar\'ee
(cor.\,\ref{shortetale}). Enfin, dans le cas des petits shorts, on prouve
que le r\'egulateur syntomique est un isomorphisme
(th.\,\ref{short11}), ce qui nous permettra
d'\'etendre les r\'esultats au cas de mauvaise r\'eduction dans le chapitre suivant.

\Subsection{Cohomologie \'etale des courbes quasi-compactes}\label{BAS5}
Soit $Y$ une courbe quasi-compacte d\'efinie sur $C$. Rappelons que cela signifie que 
$Y$ est lisse, irr\'eductible, et que $Y$ est propre ou affino\"{\i}de. 
Soit $C(Y)$ le corps des fonctions m\'eromorphes sur $Y$.
\subsubsection{Groupe de Picard}\label{SSS12}
Soit ${\rm Div}(Y)$ le $\Z$-module libre de base~$Y(C)$.
Soit ${\rm Pic}(Y)$ le quotient de ${\rm Div}(Y)$ par le sous-$\Z$-module
des ${\rm Div}(f)$, pour $f\in {\rm Frac}(\O(Y))$.
Si $D\in{\rm Div}(Y)$, et si les $U_i$ forment un recouvrement
de $Y$ par des affino\"ides, tels que $D={\rm Div}(f_i)$ sur $U_i$, alors
$f_{i,j}=f_i/f_j\in\O(U_i\cap U_j)^\dual$ et les $f_{i,j}$ d\'efinissent
un \'el\'ement $[D]$ de $H^1(Y,\O^\dual)$ qui ne d\'epend que de $D$,
et l'application $D\mapsto[D]$ induit un isomorphisme de
${\rm Pic}(Y)$ sur $H^1(Y,\O^\dual)$.

Soit $X$ une compactification de $Y$ par recollement de disques ouverts $D_i$
le long des \'el\'ements $C_0,\dots,C_r$ de $\partial^{\rm ad} Y$ (si $Y$ est propre,
alors $\partial^{\rm ad} Y=\emptyset$ et $X=Y$, bien s\^ur).
Alors $X$ est l'analytification d'une courbe alg\'ebrique propre, encore not\'ee $X$.

Choisissons $P_0\in X(C)$ (si $Y$ n'est pas compact, on prend pour $P_0$ le centre
de~$D_0$). 
Soient $J$ la jacobienne de la courbe alg\'ebrique $X$
(le groupe $J(C)$ est alors naturellement un groupe de Lie sur $C$), 
$\iota:X\to J$ le plongement
envoyant $P$ sur la classe de $P-P_0$, et $H$ le sous-groupe de $J(C)$ engendr\'e par les
$\iota(D_i)$. 

Si $Y$ est propre, l'application $\sum_Pn_PP\mapsto\big(\sum_P n_P\iota(P),\sum_Pn_P\big)$
induit un isomorphisme ${\rm Pic}(Y)\cong (J(C)\times\Z)$; si $Y$ n'est pas propre, on
a le r\'esultat suivant:
\begin{prop}\label{VDP}
{\rm (van der Put~\cite{vdp})}
Si $Y$ n'est pas propre, l'application $\sum_Pn_PP\mapsto\sum_P n_P\iota(P)$
induit un isomorphisme $J(C)/H\cong {\rm Pic}(Y)$.
\end{prop}

\begin{rema}\label{inte3}
Supposons $Y$ non propre (et donc $Y$ est un affino\"{\i}de). 

{\rm (i)}
$H$ est un sous-groupe ouvert de $J(C)$, cf. \cite{vdp}. 

{\rm (ii)} Si $C=\C_p$, alors $J(C)/H$ est un groupe de torsion (cf. \cite[th.4.1]{RLC} ou \cite[th.3]{Cz-inte} pour une preuve plus \'el\'ementaire); il en est donc de
m\^eme de ${\rm Pic}(Y)$. 

{\rm (iii)} ${\rm Pic}(Y)$ est $p$-divisible (puisque $J(C)$ l'est). Puisque $H^2(Y, {\bf G}_m)=0$ 
(cf. \cite[lemma\,6.1.2]{Berk}), la suite de Kummer montre que $H^2(Y, \bmu_p)=0$ 
(voir aussi \cite[cor.\,6.1.3]{Berk}). 
\end{rema}
\subsubsection{La suite de Kummer}\label{SSS15}
Si $\ell$ est un nombre premier, et
si $M=\Z/\ell^n,\Z_\ell,\Q_\ell$, les groupes de cohomologie \'etale
$H^i_{\eet}(Y,M(1))$ sont, par d\'efinition, reli\'es par: 
$$H^i_{\eet}(Y,\Q_\ell(1))=\Q_\ell\otimes H^i_{\eet}(Y,\Z_\ell(1)),
\quad H^i_{\eet}(Y,\Z_\ell(1))=\varprojlim H^i_{\eet}(Y,\Z/\ell^n(1)).$$

La suite exacte $0\to\Z/\ell^n(1)\to{\bf G}_m\to{\bf G}_m\to 0$ induit
la suite exacte de Kummer:
$$0\to (\Z/\ell^n)\otimes\O(Y)^\dual\to H^1_{\eet}(Y,\Z/\ell^n(1))\to {\rm Pic}(Y)[\ell^n]\to 0.$$

En prenant une limite projective sur $n$, puis en inversant $\ell$,
on obtient les suites exactes:
\begin{align*}
0\to \Z_\ell\widehat\otimes\O(Y)^\dual\to H^1_{\eet}(Y,\Z_\ell(1))\to T_\ell({\rm Pic}(Y))\to 0,\\
0\to \Q_\ell\widehat\otimes \O(Y)^\dual\to H^1_{\eet}(Y,\Q_\ell(1))\to V_\ell({\rm Pic}(Y))\to 0.
\end{align*}
\begin{rema}\label{bana1}
Les $\Z_\ell$-modules $\Z_\ell\widehat\otimes\O(Y)^\dual$ et
$T_\ell({\rm Pic}(Y))$ sont sans torsion et complets pour la topologie $\ell$-adique; il en est donc de m\^eme
de $H^1_{\eet}(Y,\Z_\ell(1))$.
Il s'ensuit que $H^1_{\eet}(Y,\Q_\ell(1))$ est un $\Q_\ell$-banach dont la boule unit\'e
est $H^1_{\eet}(Y,\Z_\ell(1))$.
\end{rema}

\subsubsection{Description de la cohomologie \'etale en termes de symboles}\label{BAS6}
Soit $A_{\ell,n}(Y)$ le groupe 
$$A_{\ell,n}(Y)=\{f\in C(Y)^\dual,\
{\rm Div}(f)\in  \ell^n{\rm Div}(Y)\}.$$
Si $f\in A_{\ell,n}(Y)$, on associe \`a $f$ sa classe de Kummer $c_n(f)$ dans
$H^1_{\eet}(Y,\Z/\ell^n(1))$ (comme ${\rm Div}(f)\in  \ell^n{\rm Div}(Y)$, l'alg\`ebre
obtenue en rajoutant $f^{1/\ell^n}$ est \'etale sur $\O(Y)$).
\begin{lemm}\label{inte10}
$c_n$ induit un isomorphisme: $$A_{\ell,n}(Y)/(C(Y)^\dual)^{\ell^n}\cong H^1_{\eet}(Y,\Z/\ell^n(1)).$$
\end{lemm}
\begin{proof}
L'application $f\mapsto \ell^{-n}{\rm Div}(f)$ induit une
suite exacte
$$0\to \O(Y)^\dual/(\O(Y)^\dual)^{\ell^n}\to A_{\ell,n}(Y)/(C(Y)^\dual)^{\ell^n}\to {\rm Pic}(Y)[\ell^n]\to 0.$$
On en d\'eduit le diagramme commutatif suivant dans lequel les lignes sont exactes
$$\xymatrix@R=.5cm@C=.6cm{
0\ar[r]&\O(Y)^\dual/(\O(Y)^\dual)^{\ell^n}\ar@{=}[d]\ar[r]& A_{\ell,n}(Y)/(C(Y)^\dual)^{\ell^n}\ar[r]\ar[d]^{c_n}
&{\rm Pic}(Y)[\ell^n]\ar[d]^{D\mapsto[D]}\ar[r]&0\\
0\ar[r]&\O(Y)^\dual/(\O(Y)^\dual)^{\ell^n}\ar[r]& H^1_{\eet}(Y,\Z/\ell^n(1))\ar[r]
& H^1(Y, {\bf G}_m)[\ell^n]\ar[r]&0}$$
et on conclut gr\^ace au lemme des 5.
\end{proof}

Posons
\begin{align*}
A_{\ell,\infty}(Y)&= \{(u_n,v_n)_{n\in\N},\ u_n\in A_{\ell,n}(Y),\ u_{n+1}=u_nv_n^{\ell^n},\ v_n\in C(Y)^\dual\}\\
A_{\ell,\infty}(Y)^{\rm triv}&= \{(f_n^{\ell^n}, f_{n+1}^\ell/f_n)_{n\in\N},\ f_n\in C(Y)^\dual\}
\subset A_{\ell,\infty}(Y),
\end{align*}
et d\'efinissons le groupe des symboles $\ell$-adiques par
$${\rm Symb}_\ell(Y)=A_{\ell,\infty}(Y)/A_{\ell,\infty}(Y)^{\rm triv}.$$
Remarquons que $\O(Y)^{\dual\dual}\hookrightarrow {\rm Symb}_\ell(Y)$ (par $u\mapsto(u_n,v_n)_n$,
avec $u_n=u$ et $v_n=1$ pour tout $n$).

\begin{coro}\label{basic12}
On a un isomorphisme naturel
$${\rm Symb}_\ell(Y)\cong H^1_{\eet}(Y,\Z_\ell(1)).$$
\end{coro}

\begin{proof}
Il suffit de prouver la bijectivit\'e de l'application naturelle 
${\rm Symb}_\ell(Y)\to \varprojlim_{n} A_{\ell,n}(Y)/(C(Y)^\dual)^{\ell^n}$. 
La surjectivit\'e est \'evidente, et si $(u_n, v_n)_{n\in \N}\in  A_{\ell,\infty}(Y)$ 
a une image nulle, on peut \'ecrire $u_n=f_n^{\ell^n}$ avec 
$f_n\in C(Y)^\dual$, et donc $f_{n+1}^\ell/f_n=\zeta_n v_n$ avec $\zeta_n\in\mu_{\ell^n}$. Il existe 
$\eta_n\in \mu_{\ell^n}$ tels que $\eta_{n+1}^{\ell}\zeta_n=\eta_n$ pour tout $n$, et en posant $g_n=\eta_n f_n$,
 on a $u_n=g_n^{\ell^n}$ et $v_n=g_{n+1}^{\ell}/g_n$.
\end{proof}

\subsubsection{Le cas des shorts}
Soit $X$ une courbe propre et lisse sur $\O_C$, et soit $J:={\rm Pic}^0(X)$ la jacobienne
de $X$. Alors $J$ est un sch\'ema ab\'elien sur $\O_C$ et on dispose d'un morphisme
de $\O_C$-sch\'ema $X\times X\to J$ envoyant $(P,Q)$ sur la classe du diviseur $Q-P$
(on note $(P,Q)\mapsto Q-P$ ce morphisme ainsi que les applications qui s'en d\'eduisent
sur les fibres g\'en\'eriques et sp\'eciales).

L'injection naturelle $J(\O_C)\to J(C)$ est un
isomorphisme, ce qui fournit une fl\`eche de r\'eduction $J(C)\to J(k_C)$ que l'on note $P\mapsto\bar P$
(on note de m\^eme la r\'eduction $X(C)\to X(k_C)$)
et dont on note $J(C)_0$ le noyau (c'est un sous-groupe de Lie de $J(C)$ car $\overline{Q-P}=\bar{Q}-\bar{P}$
et $J(C)_0$ est ouvert dans $J(C)$).

\begin{lemm}\label{BS1}
Si $P\in X(k_C)$, le sous-groupe de $J(C)$ engendr\'e par
les $Q_1-Q_2$, pour $Q_1,Q_2\in ]P[$, est $J(C)_0$.
\end{lemm}
\begin{proof}
Soit $g$ le genre de $X$.  Il n'y a rien \`a prouver si $g=0$; on peut donc supposer
$g\geq 1$.  
Si $Q_1,Q_2\in]P[$, on a $\bar Q_1=\bar Q_2=P$ et donc $Q_1-Q_2\in J(C)_0$, ce qui prouve une
des deux inclusions.

Pour prouver l'autre,
fixons $\tilde P\in ]P[$ et soit $d\geq 2g+1$.  Alors $\Sigma_d:X^d\to J$
d\'efinie par $\Sigma_d(Q_1,\dots,Q_d)=\sum_{i=1}^d(Q_i-\tilde P)$ fait de $X^d/S_d$
une fibration en ${\bf P}^{d-g}$ au-dessus de~$J$.
Il existe un ouvert de Zariski $U$ de $J$, contenant la section nulle $0\in J(\O_C)$,
tel que cette fibration soit triviale sur $U$, ce qui fournit une section
$s:U\to X^d/S_d$
prenant la valeur $(\tilde P,\dots,\tilde P)$ en~$0$. 
Alors $U(\O_C)$ contient $J(C)_0$ et $s(J(C)_0)$ est contenue dans $]P[^d$, 
ce qui montre que tout \'el\'ement de
$J(C)_0$ est la somme d'au plus $d$ \'el\'ements de la forme $Q-\tilde P$, avec $Q\in]P[$.
D'o\`u le r\'esultat.
\end{proof}

Soit $Y$ un $\O_C$-short obtenu en enlevant \`a $X$ les tubes de points $Q_0,\dots,Q_r\in X(k_C)$, 
deux \`a deux distincts, et soit $Y^{\rm gen}$ l'affino\"{\i}de qui
est la fibre g\'en\'erique de $Y$. 
Les $Q_i$ sont en bijection naturelle avec les \'el\'ements de
$\partial^{\rm ad}Y$. 

Soit $H(k_C)$
le sous-groupe de $J(k_C)$ engendr\'e par les $Q_i-Q_j$.
Il r\'esulte du lemme~\ref{BS1} et de la prop.~\ref{VDP} que:
\begin{coro}\label{BS2}
${\rm Pic}(Y^{\rm gen})=J(k_C)/H(k_C)$.
\end{coro}

\begin{rema}\label{Gerritzen} 
En combinant le cor.\,\ref{BS2} et le calcul standard du groupe de Picard de 
$X^{\rm sp}\moins \{Q_0,...,Q_r\}$ on en d\'eduit un isomorphisme 
$${\rm Pic}(Y^{\rm gen})\cong {\rm Pic}(\bar{Y}),$$ 
o\`u $\bar{Y}={\rm Spec}(\O(Y)^+\otimes_{\O_C} k_C)$ est la r\'eduction canonique de 
$Y$. Ce r\'esultat est un cas particulier d'un th\'eor\`eme de Gerritzen \cite{G} et Heinrich-van der Put \cite{HvdP} dans le cas d'une valuation discr\`ete.
\end{rema}

\begin{rema}\label{BS2.1}
{\rm (i)} On a un isomorphisme\footnote{Notons que $\O(Y^{\rm gen})^{\dual\dual}=\O(Y)^{\dual\dual}$.}
$$\O(Y^{\rm gen})^\dual/C^\dual\O(Y)^{\dual\dual}\overset{\sim}{\to}
 \O(Y^{\rm sp})^\dual/k_C^\dual,$$
la fl\`eche envoyant $u\in \O(Y^{\rm gen})^\dual$ sur la r\'eduction $\overline u$ de $p^{-r}u$,
avec $r=v_Y(u)$. 

{\rm (ii)} On a une suite exacte
$$0\to \O(Y^{\rm sp})^\dual/k_C^\dual\to {\rm Ker}\big(\dbar:\Z^{\partial^{\rm ad}Y}\to\Z\big)
\to H(k_C)\to 0,$$
o\`u $\dbar:\Z^{\partial^{\rm ad}Y}\to\Z$ envoie $(n_i)_i$ sur $\sum_i n_i$,
l'application de $\O(Y^{\rm sp})^\dual/k_C^\dual$ dans $\Z^{\partial^{\rm ad}}$
est $f\mapsto (v_{Q_i}(f))_i$, et celle de
${\rm Ker}\big(\dbar:\Z^{\partial^{\rm ad}Y}\to\Z\big)$ dans $H(k_C)$
est $(n_i)_i\mapsto \sum_i n_iQ_i$.
\end{rema}
On dispose d'une application r\'esidu ${\rm Symb}_\ell(Y^{\rm gen})\to \Z_\ell^{\partial^{\rm ad}Y}$
d\'efinie comme la compos\'ee de la restriction ${\rm Symb}_\ell(Y^{\rm gen})\to
{\rm Symb}_\ell(\partial^{\rm ad}Y)$ et de l'application r\'esidu sur chacun des disque
fant\^omes de $\partial^{\rm ad}Y$. En combinant avec l'isomorphisme
du cor.~\ref{basic12}, cela fournit une application r\'esidu
$H^1_{\eet}(Y^{\rm gen},\Z_\ell(1))\to \Z_\ell^{\partial^{\rm ad}Y}$.
\begin{prop}\label{short21}
{\rm (i)}
Si $\ell\neq p$, l'application r\'esidu induit une suite exacte
$$0\to T_\ell J(k_C)\to H^1_{\eet}(Y^{\rm gen},\Z_\ell(1))\to 
{\rm Ker}\big(\dbar:\Z_\ell^{\partial^{\rm ad}Y}\to\Z_\ell\big)\to 0.$$

{\rm (ii)}
Si $\ell= p$, l'application r\'esidu induit une suite exacte
$$0\to T_p J(k_C)\to H^1_{\eet}(Y^{\rm gen},\Z_p(1))/(\Z_p\wotimes\O(Y)^{\dual\dual})\to 
{\rm Ker}\big(\dbar:\Z_p^{\partial^{\rm ad}Y}\to\Z_p\big)\to 0.$$
\end{prop}
\begin{proof}
On a une suite exacte (car la multiplication par $\ell^n$ est
surjective sur $J(k_C)$)
\begin{align*}
0\to T_\ell J(k_C)\to T_\ell(J(k_C)/H(k_C))\to \Z_\ell\otimes H(k_C)\to 0
\end{align*}
Le diagramme suivant est alors commutatif \`a lignes et colonnes exactes:

$$\xymatrix@R=.4cm@C=.4cm{
&&0\ar[d]&0\ar[d]\\
&&{\rm Ker}\ar[d]\ar[r]^-{\sim}& T_\ell J(k_C)\ar[d]\\
0\ar[r]&\Z_\ell\otimes\O(Y^{\rm sp})^\dual\ar[r]\ar@{=}[d]
&\frac{H^1_{\eet}(Y^{\rm gen},\Z_\ell(1))}{\Z_\ell\wotimes\O(Y)^{\dual\dual}}\ar[d]\ar[r] 
&T_\ell(J(k_C)/H(k_C))\ar[d]\ar[r]&0\\
0\ar[r]&\Z_\ell\otimes\O(Y^{\rm sp})^\dual\ar[r]&
{\rm Ker}\big(\dbar:\Z_\ell^{\partial^{\rm ad}Y}\to\Z_\ell\big)\ar[r]\ar[d]
&\Z_\ell\otimes H(k_C)\ar[r]\ar[d]&0\\
&&0&0}
$$

$\bullet$
Pour prouver la commutativit\'e, fixons $P_0\in]Q_0[$ et notons $\iota:X\to J$
l'application $P\mapsto P-P_0$ (c'est la restriction \`a $\{P_0\}\times X$ de l'application
consid\'er\'ee plus haut) que l'on \'etend par additivit\'e \`a ${\rm Div}(X)$. 
On note $\overline\iota:{\rm Div}(X)\to J(k_C)$ la compos\'ee de $\iota$ et de la r\'eduction.

On utilise la description de $H^1_{\eet}(Y^{\rm gen},\Z_\ell(1)$
en termes de symboles (cor.\,\ref{basic12}). Soit donc
$(u_n,v_n)_n\in A_{\ell,\infty}(Y)$.
La fl\`eche de $H^1_{\eet}(Y^{\rm gen},\Z_\ell(1))\to \Z_\ell\otimes H(k_C)$ passant
par $T_\ell(J(k_C)/H(k_C))$ envoie $(u_n,v_n)_n$ sur $\lim_{n\to\infty}\overline{\iota}({\rm Div}(u_n))$,
o\`u ${\rm Div}(u_n)$ est le diviseur de $u_n$ sur $Y$ 
($\overline{\iota}({\rm Div}(u_n))$ est aussi l'image $\iota({\rm Div}(\overline u_n))$
dans $J(k_C)$ du diviseur de $\overline u_n$, vue comme fonction sur $Y^{\rm sp}$).
L'autre fl\`eche envoie $(u_n,v_n)_n$ sur $-\lim_{n\to\infty}\sum_iv_{Q_i}(\overline u_n)\,\iota(Q_i)$.
La diff\'erence entre les termes g\'en\'eraux
des deux suites est $\iota({\rm Div}(\overline u_n))$ o\`u,
cette fois-ci, $\overline u_n$ est vue comme une fonction sur $X^{\rm sp}$;
cette diff\'erence est donc nulle dans $J(k_C)$ puisque
c'est l'image d'un diviseur principal sur~$X^{\rm sp}$.

$\bullet$ La premi\`ere ligne est obtenue en quotientant les deux premiers termes
de la suite de Kummer par $\Z_\ell\wotimes\O(Y)^{\dual\dual}$ et en utilisant le 
(i) de la rem.~\ref{BS2.1}.

$\bullet$ La seconde ligne est obtenue en tensorisant par $\Z_\ell$ la suite du (ii) de
la rem.~\ref{BS2.1}.

$\bullet$ La seconde colonne verticale est la suite exacte ci-dessus.

Le r\'esultat s'en d\'eduit en utilisant
le fait que $\Z_\ell\wotimes\O(Y)^{\dual\dual}=0$ si $\ell\neq p$.
\end{proof}

\begin{rema}
{\rm (i)} $\O(Y)^{\dual\dual}$ n'est pas $p$-adiquement complet et il s'injecte strictement
dans $\Z_p\wotimes\O(Y)^{\dual\dual}$.

{\rm (ii)}
Si $\ell\neq p$, alors 
${\rm rg}_{\Z_\ell}
T_\ell J(k_C)=2g$, mais si $\ell=p$, alors
${\rm rg}_{\Z_p}
T_p J(k_C)$ est \'egal \`a la dimension du sous-espace de pente~$0$ du module
de Dieudonn\'e de $J(k_C)$ et peut prendre les valeurs $0,1,\dots,g$.
\end{rema}

\Subsection{Cohomologie syntomique}\label{BAS18}
\subsubsection{Lien avec le complexe de de Rham}
Soit $Y$ un short sur $\O_C$, et soit $\breve Y$ un mod\`ele sur $\O_{\breve C}$.
On choisit un frobenius $\varphi$ sur $\O(\breve Y)$, et on note encore $\varphi$
son extension \`a
$$\O(\widetilde Y):=\acris\wotimes_{\O_{\breve C}}\O(\breve Y)$$
Les $H^i_{\rm syn}(Y,1)$,
sont les groupes de cohomologie du complexe
$${\rm Syn}(Y,1):=\xymatrix@C=1.2cm{F^1\O(\widetilde Y)\ar[r]^-{d,1-\frac{\varphi}{p}}
&\Omega^1(\widetilde Y)\oplus \O(\widetilde Y)\ar[r]^-{(1-\frac{\varphi}{p})-d}
&\Omega^1(\widetilde Y)},$$
o\`u $F^1\O(\widetilde Y)={\rm Ker}(\O(\widetilde Y)\to\O(Y))$.

On note $C_{\rm dR}^\bullet(\widetilde Y)$ le complexe 
$\xymatrix{\O(\widetilde Y)\ar[r]^-d
&\Omega^1(\widetilde Y)}$, et $H^i_{\rm dR}(\widetilde Y)$ ses groupes
de cohomologie. 
Comme on ne peut pas diviser $\varphi$ par $p$ sur tout $\acris$, on introduit
$\O(\widetilde Y)'=\O(\widetilde Y)+\big(1-\frac{\varphi}{p}\big)\O(\widetilde Y)$
(alors $M:=\O(\widetilde Y)'/\O(\widetilde Y)$ est tu\'e par $p$),
et $C_{\rm dR}^\bullet(\widetilde Y)'=
\big(\xymatrix{\O(\widetilde Y)'\ar[r]^-d
&\Omega^1(\widetilde Y)}\big)$.  On a alors une suite exacte
$$0\to {\rm Syn}(Y,1)\to [\xymatrix{C_{\rm dR}^\bullet(\widetilde Y)
\ar[r]^-{1-\frac{\varphi}{p}} & C_{\rm dR}^\bullet(\widetilde Y)'}]
\to (\O(Y)\to M\to 0)\to 0.$$
En passant \`a la suite de cohomologie,
et en utilisant le fait que $1-\frac{\varphi}{p}$ est une surjection
de $H^0(C_{\rm dR}^\bullet(\widetilde Y))=\acris$ sur 
$H^0(C_{\rm dR}^\bullet(\widetilde Y)')$,
cela fournit des suites exactes
\begin{align}\label{BAS18.0}
0\to \Z_pt\to \acris^{\varphi=p}\to {\rm Ker}[\O(Y)\to M]\to
H^1_{\rm syn}(Y)\to H^1_{\rm dR}(\widetilde Y)^{\varphi=p}\to 0,\notag\\
0\to {\rm Ker}[\O(Y)/\O_C\to M]\to
H^1_{\rm syn}(Y)\to H^1_{\rm dR}(\widetilde Y)^{\varphi=p}\to 0.
\end{align}
Remarquons que,
$M$ \'etant de $p$-torsion, ${\rm Ker}[\O(Y)/\O_C\to M]$ est
$p$-isomorphe \`a $\O(Y)/\O_C$.

\subsubsection{R\'egulateur syntomique}\label{BAS18.1}
On rappelle que $$A_{p,n}(Y^{\rm gen})=\{f\in C(Y^{\rm gen})^\dual,\ {\rm Div}(f)\in p^n{\rm Div}(Y^{\rm gen})\}.$$
\begin{lemm}\label{BAS18.2}
Soit $u\in A_{p,n}(Y^{\rm gen})$.

{\rm (i)}
Il existe $u_0\in \breve C(\breve Y)^\dual$, $v\in\O(Y)^\dual$, tels que
$u=u_0 v a^{p^n}$, avec $a\in C(Y^{\rm gen})^\dual$.

{\rm (ii)} Il existe $u_1\in \breve C(\breve Y)^\dual$ tel que 
$u_0=u_1b^{p^n}$, avec $b\in \breve C(\breve Y)^\dual$,
et tel que les diviseurs de $u_0$ et $u_1$ sur la fibre sp\'eciale
soient \'etrangers.
\end{lemm}
\begin{proof}
On fixe une compactification $\breve X$ de $\breve Y$, et on note $J$ la jacobienne de $\breve X$.
On note $P_0,\dots,P_r$ les \'el\'ements de $\breve X(k_C)\moins \breve Y(k_C)$
et on fixe un rel\`evement $\breve P_i\in \breve X(\breve C)$ de $P_i$.
On envoie $\breve X$ dans $J$ par $P\mapsto (P-\breve P_0)$. Le cor.\,\ref{BS2} montre que 
${\rm Pic}(Y^{\rm gen})$ est le quotient de $J(k_C)$ par le sous-groupe
engendr\'e par $P_1,\dots,P_r$.  

Soit $u\in A_{p,n}(Y^{\rm gen})$; \'ecrivons ${\rm Div}(u)$ sous la forme $p^nD$, avec $D\in{\rm Div}(Y^{\rm gen})$.
Soit $\overline D\in {\rm Div}(\breve Y(k_C))$ l'image de $D$, et soit $\overline D_0\in
{\rm Div}(\breve Y(k_C))$ ayant m\^eme image que $D$ dans $J(k_C)$ (on peut choisir $\overline D_0$
\'etranger \`a tout ensemble fini $S$ donn\'e; en particulier, on peut trouver deux tels choix
\'etrangers l'un \`a l'autre). Choisissons un rel\`evement $\breve D_0$ de
$\overline D_0$ dans ${\rm Div}(\breve Y(\breve C))$.
Par construction $D-\breve D_0$ a pour image $0$ dans $J(k_C)$ et (lemme~\ref{BS1}) on peut trouver
$Q_1,\dots,Q_d\in ]P_0[$ tels que $D-\breve D_0-\sum_{i=1}^d (Q_i-\breve P_0)=0$ dans $J(C)$.
Il existe donc $N\in\Z$ tel que $D-\breve D_0-\sum_{i=1}^d Q_i-N \breve P_0={\rm Div}(a)$,
avec 
$a\in C(Y^{\rm gen})^\dual$.  Alors $ua^{-p^n}$ a comme
diviseur $p^n\breve D_0$ sur $Y$. Il s'ensuit que $p^nD_0=n_1P_1+\cdots+n_rP_r$ dans
$J(k_C)$, et il existe $R_1,\dots,R_d\in ]P_0[(\breve C)$ tels que
$$p^nD_0-(n_1\breve P_1+\cdots+n_r\breve P_r)-(R_1+\cdots+R_d)+N\breve P_0={\rm Div}(u_0),$$ 
avec
$u_0\in \breve C(\breve X)^\dual$.
Mais alors $v:=ua^{-p^n}u_0^{-1}\in\O(Y^{\rm gen})^\dual$ et, quitte \`a multiplier $a$ par $\alpha\in C^\dual$,
on peut s'arranger pour que $v\in \O(Y)^\dual$.  

Ceci d\'emontre le (i).  Pour prouver le (ii), on part
de $\overline D_1$ \'etranger \`a $\overline D_0$, ayant m\^eme image que $\overline D_0$
dans $J(k_C)$, et on rel\`eve $\overline D_1$ en $\breve D_1$ dans ${\rm Div}(\breve Y)$.
Il existe alors $Q_1,\dots,Q_d\in ]P_0[(\breve C)$ et $N\in\Z$, tels que
$\breve D_0-\breve D_1-(Q_1+\cdots+Q_d)+N \breve P_0={\rm Div}(b)$, avec
$b\in C(Y^{\rm gen})^\dual$. Mais alors $u_1=u_0b^{-p^n}$ a pour diviseur
$p^n\breve D_1$ sur $Y$.  Ceci permet de conclure.
\end{proof}

On rappelle que
$${\rm Symb}_p(Y^{\rm gen})=\{u=(u_n)_{n\in\N}, u_n\in A_{p,n}(Y^{\rm gen}),\ u_{n+1}/u_n=v_n^{p^n}\}/
\{(a_n^{p^n})_{n\in\N}\}.$$
Soit $u=(u_n)_{n\in\N}\in{\rm Symb}_p(Y^{\rm gen})$.  On \'ecrit
$u_n$ sous la forme $u_{n,0} a_n^{p^n}v_n$ comme ci-dessus,
et on choisit des rel\`evements $\tilde v_n\in\O(\widetilde Y)$ de $v_n$
et $\tilde a_n\in {\rm Fr}(\O(\widetilde Y))$ de $a_n$, 
et on pose $\tilde u_n=u_{n,0}\tilde a_n^{p^n}\tilde v_n$.
On pose $\delta_n(u_n)=\big(\tfrac{d\tilde u_n}{\tilde u_n},
\tfrac{1}{p}\log\frac{\varphi(\tilde u_n)}{\tilde u_n^p}\big)$
modulo~$p^n$.
\begin{lemm}\label{basic28}
{\rm (i)} $\delta_n(u_n)$ est un $1$-cocycle de ${\rm Syn}(Y,1)$ modulo~$p^n$.

{\rm (ii)} Les $\delta_n(u_n)$ convergent vers un $1$-cocycle $\delta_Y(u)$
de ${\rm Syn}(Y,1)$ dont la classe dans $H^1_{\rm syn}(Y,1)$ ne d\'epend
que de $u$.

\end{lemm}
\begin{proof}
Soient $x_n=\tfrac{d \tilde u_n}{\tilde u_n}$ 
et $y_n=\tfrac{1}{p}\log\frac{\varphi(\tilde u_n)}{\tilde u_n^p}$
de telle sorte que $\delta_n(u_n)=(x_n,y_n)$ modulo~$p^n$.
On a $(1-\frac{\varphi}{p})x_n=dy_n$, et donc cette
relation est a fortiori v\'erifi\'ee modulo~$p^n$.

De plus, si on prend deux \'ecritures $u_n=u_{n,0} a_n^{p^n}v_n=u_{n,1} (a_nb_n)^{p^n}v_n$,
avec $u_{n,0}$ et $u_{n,1}$ de diviseurs \'etrangers, on a
$x_n=\tfrac{du_{n,0}}{u_{n,0}}+\tfrac{d\tilde v_n}{v_n}=\tfrac{du_{n,1}}{u_{n,1}}+\tfrac{d\tilde v_n}{v_n}$
modulo $p^n$, et comme les diviseurs de $u_{n,0}$ et $u_{n,1}$ sont \'etrangers, cela
montre que $x_n$ est holomorphe (i.e.~$x_n\in\Omega^1(\widetilde Y)$ modulo~$p^n$).
Le m\^eme argument montre que $y_n\in\O(\widetilde Y)$ modulo~$p^n$, ce qui prouve le (i).

La convergence de la suite de terme g\'en\'eral $\delta_n(u_n)$ vient de
ce que $\delta_{n+1}(u_{n+1})=\delta_n(u_n)$ modulo~$p^n$ car $u_{n+1}/u_n$
est une puissance $p^n$-i\`eme.  Si on note $\delta(u)=(x,y)$ la limite,
alors $(1-\frac{\varphi}{p})x=dy$ par passage \`a la limite,
et donc $\delta(u)$ est un $1$-cocycle 
de ${\rm Syn}(Y,1)$.  Ce $1$-cocycle est uniquement d\'etermin\'e au choix pr\`es
des $\tilde v_n$, et donc est bien d\'etermin\'e \`a addition pr\`es du cobord limite
des $\log(\tilde v_n/\tilde v'_n)$ si $\tilde v_n$ et $\tilde v'_n$ sont deux
rel\`evements de $v_n$. La classe de $\delta(u)$ dans $H^1_{\rm syn}(Y,1)$
ne d\'epend donc que de $u$, ce qui prouve le (ii).
\end{proof}
Le {\it r\'egulateur syntomique} est l'application
$$\delta_Y:{\rm Symb}_p(Y^{\rm gen})\to H^1_{\rm syn}(Y,1)$$
dont l'existence d\'ecoule du lemme~\ref{basic28}.
On peut composer $\delta_Y$
avec la fl\`eche $H^1_{\rm syn}(Y,1)\to H^1_{\rm dR}(\widetilde Y)^{\varphi=p}$
de (\ref{BAS18.0}).
\begin{lemm}\label{short10}
Si $u\in\O(Y)^{\dual\dual}$ l'image de $\delta_Y(u)$ dans
$H^1_{\rm dR}(\widetilde Y)^{\varphi=p}$ est de torsion.
\end{lemm}
\begin{proof}
Si $\tilde u\in \O(\widetilde Y)^{\dual\dual}$ v\'erifie $\theta(\tilde u)=u$,
alors $\log\tilde u$ converge dans $\O(\widetilde Y)[\frac{1}{p}]$ et
$\delta_Y(u)$ est le bord de $\log\tilde u$ et donc est tu\'e par $p^N$ si
$p^N\log\tilde u\in \O(\widetilde Y)$.
\end{proof}
\subsection{Symboles et formes diff\'erentielles}\label{gabr1}
     Soient $\breve Y$ un short sur $\O_{\breve C}$, $Y=\breve Y\otimes_{\O_{\breve C}} \O_C$ et $Y^{\rm gen}$ la fibre g\'en\'erique de $Y$. Soit
     $\bar{Y}={\rm Spec}(\O(Y)\otimes_{\O_C} k_C)$ la fibre sp\'eciale (classique) de $Y$, un $k_C$-sch\'ema lisse irr\'eductible s'identifiant \`a la fibre sp\'eciale de $\breve Y$. On note $\Omega^1_{\breve Y}=\Omega^1_{\breve Y/W(k_C)}$ et 
    $\Omega^1_{\bar{Y}}=\Omega^1_{\bar Y/k_C}$ les faisceaux des diff\'erentielles de $\breve Y$ et $\bar Y$ (sur ${\bar Y}_{\eet}$).

\subsubsection{Formes diff\'erentielles et op\'erateur de Cartier}
On choisit un rel\`evement 
     $\varphi: \breve Y\to \breve Y$ du frobenius absolu de $\bar Y$. Il induit, par fonctorialit\'e, un 
     morphisme $\varphi: \Omega^1_{\breve Y}\to \Omega^1_{\breve Y}$, dont l'image 
est contenue dans $p \Omega^1_{\breve Y}$. Comme $\Omega^1_{\breve Y}$ est sans 
$p$-torsion, cela permet de d\'efinir un endomorphisme $F=\varphi/p: \Omega^1_{\breve Y}\to \Omega^1_{\breve Y}$
    et un calcul imm\'ediat montre que la compos\'ee de la r\'eduction $\bar{F}: \Omega^1_{\bar Y}\to \Omega^1_{\bar Y} $ modulo $p$ de $F$ et de la projection naturelle $\Omega^1_{\bar Y}\to \Omega^1_{\bar Y}/d\O_{\bar Y}$ est l'op\'erateur de Cartier 
  $C^{-1}: \Omega^1_{\bar Y}\to \Omega^1_{\bar Y}/d \O_{\bar Y}$, qui est un isomorphisme puisque $\bar{Y}$ est lisse. 
  On note $C: \Omega^1_{\bar Y}/d\O_{\bar Y}\to \Omega^1_{\bar Y}$ l'inverse de $C^{-1}$, c'est donc un endomorphisme de  
  $\Omega^1_{\bar Y}$ nul sur $d\O_{\bar Y}$. 
  
  Consid\'erons la suite $(B_n)_{n\geq 1}$ de sous-faisceaux de $\Omega^1_{\bar Y}$ d\'efinie par  
  $B_1=d\O_{\bar Y}$ et $B_{n+1}=B_n+\bar{F}^n(B_1)$
   et sa limite $B_{\infty}=\cup_{n} B_n=\sum_{n\geq 0} \bar{F}^n(B_1)$. Le lemme ci-dessous est la combinaison d'un cas particulier d'un 
   r\'esultat de Raynaud et de la th\'eorie classique de l'op\'erateur de Cartier;
c'est la cl\'e des r\'esultats de ce paragraphe. 
      
\begin{lemm}\label{IlRay}
On a une d\'ecomposition en somme directe de faisceaux 
$$\Omega^1_{\bar Y}=B_{\infty}\oplus \big(k_C\otimes_{\mathbb{F}_p} (\Omega^1_{\bar Y})^{C=1}\big),$$
ainsi qu'un isomorphisme de faisceaux
$$d\log: \O_{\bar{Y}}^\dual/(\O_{\bar{Y}}^\dual)^p\cong (\Omega^1_{\bar Y})^{C=1}.$$
   \end{lemm}
   
   \begin{proof} Par d\'efinition (et la discussion ci-dessus) $B_{\infty}$ est le sous-faisceau de $C$-torsion de $\Omega^1_{\bar Y}$. Le premier point d\'ecoule alors du fait que toute forme diff\'erentielle $\omega\in \Omega^1_{\bar Y}(U)$ ($U$ \'etant un ouvert de $\bar Y$) est 
$C$-finie\footnote{i.e. l'espace engendr\'e par les $C^n\omega$ pour $n\geq 0$ est de dimension finie sur $k_C$. Le point cl\'e est que 
$C$ diminue l'ordre des p\^oles des formes diff\'erentielles.}, donc l'espace engendr\'e par les $C^n\omega$ se d\'ecompose en une partie de $C$-torsion et une partie sur laquelle $C$ est inversible, et cette derni\`ere partie est engendr\'ee par les points fixes sous $C$ (car $k_C$ est alg\'ebriquement clos). 
Voir \cite[\S\,2.5]{IldRW} pour plus de d\'etails. Le second point 
   est un th\'eor\`eme classique de Cartier (cf. \cite[th. 2.1.17]{IldRW}).
\end{proof}
   
Notons simplement $H^1_{\bar Y}$ le faisceau $\Omega^1_{\breve Y}/d\O_{\breve Y}$ et $H^1_{\bar Y}[p^{\infty}]$
son sous-faisceau de $p$-torsion. 
      
\begin{lemm}\label{Ray}
La compos\'ee $H^1_{\bar Y}=\Omega^1_{\breve Y}/d\O_{\breve Y}
\to \Omega^1_{\bar Y}/d\O_{\bar Y}=\Omega^1_{\bar Y}/B_1$ induit un isomorphisme de faisceaux
$$ H^1_{\bar Y}[p^{\infty}]/pH^1_{\bar Y}[p^{\infty}]\cong B_{\infty}/B_1.$$
      \end{lemm}

\begin{proof} 
Montrons d'abord qu'une section locale $f\in \O_{\breve Y}$ v\'erifie $df\in p^n\Omega^1_{\breve Y}$ si et seulement si l'on peut \'ecrire (localement) $f=\sum_{k=0}^n p^k F^{n-k}(g_k)$ avec $g_k\in \O_{\breve Y}$. C'est un cas particulier de \cite[prop.\,2.3.13]{IldRW}, 
mais la preuve \'etant facile, nous allons la donner pour le confort du lecteur. 
Un sens est \'evident, puisque $d\circ F=p\,F\circ d$. 
Pour l'autre sens, on raisonne par r\'ecurrence.
Supposons donc que l'on dispose d'une \'ecriture $f=\sum_{k=0}^{n-1} p^k F^{n-1-k}(g_k)$. Puisque $d\circ F=p\,F\circ d$ et $\Omega^1_{\breve Y}$ est sans $p$-torsion, la divisibilit\'e de $df$ par $p^n$ se traduit par 
$\sum_{k=0}^{n-1} \bar{F}^{n-1-k}(d\bar{g}_k)=0$ dans $\Omega^1_{\bar Y}$. En utilisant la compatibilit\'e entre 
$\bar{F}$ et $C^{-1}$, on voit que cette relation force $d\bar{g}_k=0$ pour tout $k$, donc (encore par la th\'eorie de Cartier) on peut \'ecrire $g_k=F(a_k)+db_k+pc_k$. En ins\'erant ces relations dans l'\'ecriture de $f$, on obtient une repr\'esentation $f=\sum_{k=0}^n p^k F^{n-k}(h_k)$, ce qui permet de conclure. 
      
Revenons \`a la preuve du lemme. Si la classe de $\omega\in \Omega^1_{\breve Y}$ dans $H^1_{\bar Y}$ est tu\'ee par 
$p^n$, alors $p^n\omega=df$ pour un $f\in \O_{\breve Y}$. 
En \'ecrivant (localement) $f=\sum_{k=0}^n p^k F^{n-k}(g_k)$, on voit comme ci-dessus que $\omega=\sum_{k=0}^n F^{n-k}(dg_k)$ et donc $\overline{\omega}=\sum_{k=0}^n \bar{F}^{n-k} (d\bar{g}_k)\in B_{\infty}$, ce qui montre que la r\'eduction mod $p$ induit bien un morphisme $H^1_{\bar Y}[p^{\infty}]/pH^1_{\bar Y}[p^{\infty}]\to B_{\infty}/B_1$. Son injectivit\'e \'etant imm\'ediate, montrons la surjectivit\'e. Soit $x\in B_{\infty}$, alors il existe 
        $n$ et $f_i\in \Omega_{\bar Y}$ tels que $x=\sum_{k=0}^n \bar{F}^{n-k}(df_k)$. On rel\`eve 
        $f_k$ en $\hat{f}_k\in \Omega_{\breve Y}$. Le m\^eme calcul que ci-dessus montre que la classe de 
        $\omega:=\sum_{k=0}^n p^k F^{n-k}(d\hat{f}_k)$ dans $H^1_{\bar Y}$ est de torsion et se r\'eduit sur la classe de $x$. 
            \end{proof}

\subsubsection{Le s\'epar\'e de la cohomologie de de Rham}

On d\'efinit les $H^i_{\rm dR}(Y)$ comme les groupes de cohomologie du complexe
$\O(Y)\overset{d}{\longrightarrow}\Omega^1(Y)$ munis de la topologie quotient.
On note $H^1_{\rm dR}(Y)_{\rm tors}$ l'adh\'erence, dans $H^1_{\rm dR}(Y)$,
du sous-groupe de torsion (en particulier $H^1_{\rm dR}(Y)_{\rm tors}$ contient
l'adh\'erence de $0$ qui est loin d'\^etre nulle).
On note $H^1_{\rm dR}(Y)^{\rm sep}$ le quotient 
$H^1_{\rm dR}(Y)/H^1_{\rm dR}(Y)_{\rm tors}$. On v\'erifie sans mal que 
$H^1_{\rm dR}(Y)^{\rm sep}$ est sans torsion et complet pour la topologie $p$-adique. 
\begin{prop}\label{sansk0}
On a un isomorphisme naturel
$$H^1_{\rm dR}(\breve Y)^{\rm sep}/pH^1_{\rm dR}(\breve Y)^{\rm sep}\cong k_C\otimes_{\mathbf{F}_p} 
\Omega^1({\bar Y})^{C=1}$$ et une suite exacte
$$0\to \O(\bar Y)^\dual/(\O(\bar Y)^\dual)^p\to H^1_{\rm dR}(\breve Y)^{\rm sep}/pH^1_{\rm dR}(\breve Y)^{\rm sep}\to 
k_C\otimes_{\mathbf{F}_p} {\rm Pic}(\bar{Y})[p]\to 0.$$
\end{prop}
\begin{proof} 
Soient $M=H^0(\breve Y, H^1_{\bar Y})=\Omega^1({\breve Y})/d\O(\breve Y)$ (l'\'egalit\'e d\'ecoule du caract\`ere affine de $\bar Y$) et $\mathcal{F}=H^1_{\bar Y}[p^{\infty}]$. Puisque  
  $H^1_{\rm dR}(\breve Y)^{\rm sep}$ est le quotient de
  $M$ par l'adh\'erence de $M[p^{\infty}]=H^0(\breve Y, \mathcal{F})$, 
on a une suite exacte $$M[p^{\infty}]/p\to M/p\to H^1_{\rm dR}(\breve Y)^{\rm sep}/p\to 0.
$$
  D'autre part comme $\breve Y$ est affine et $\mathcal{F}/p\mathcal{F}$ est limite inductive de faisceaux coh\'erents (lemme 
  \ref{Ray}), on a $H^1(\breve Y, \mathcal{F}/p\mathcal{F})=0$ et par d\'evissage $H^1(\breve Y, \mathcal{F})=0$, donc 
  $M[p^{\infty}]/p\cong H^0(\breve Y, \mathcal{F})/p\cong H^0(\breve Y, \mathcal{F}/p\mathcal{F})\cong H^0(\bar Y, B_{\infty}/B_1)$, le dernier isomorphisme \'etant une cons\'equence du lemme \ref{Ray}. On obtient donc une suite exacte 
  $$ H^0(\bar Y, B_{\infty}/B_1)\to \Omega^1({\bar Y})/d\O(\bar Y)\to H^1_{\rm dR}(\breve Y)^{\rm sep}/p\to 0.$$
  Puisque $\breve Y$ est affine, on obtient $H^1_{\rm dR}(\breve Y)^{\rm sep}/p\cong H^0(\bar Y, \Omega^1_{\bar Y}/B_{\infty})$, qui s'identifie (par le lemme \ref{IlRay}) \`a $H^0(\bar Y, k_C\otimes_{\mathbf{F}_p} (\Omega^1_{\bar Y})^{C=1})=k_C\otimes_{\mathbf{F}_p} \Omega^1({\bar Y})^{C=1}$. Cela montre le premier point. Le second s'en d\'eduit, en utilisant encore une fois 
   le lemme 
  \ref{IlRay} et la suite exacte \'evidente 
  $$0\to \O(\bar Y)^\dual/(\O(\bar Y)^\dual)^p\to H^0(\bar Y, \O_{\bar Y}^\dual/(\O_{\bar Y}^\dual)^p)\to {\rm Pic}(\bar Y)[p]\to 0.\qedhere$$
    \end{proof}
   
\begin{coro}\label{sansk}
{\rm (i)} On a un isomorphisme naturel   
$$(H^1_{\rm dR}(\breve Y)^{\rm sep})^{\varphi=p}/p \cong \Omega^1({\bar Y})^{C=1}$$ 
et une suite exacte 
$$0\to \O(\bar Y)^\dual/(\O(\bar Y)^\dual)^p\to (H^1_{\rm dR}(\breve Y)^{\rm sep})^{\varphi=p}/p\to 
{\rm Pic}(\bar{Y})[p]\to 0.$$

{\rm (ii)} La fl\`eche naturelle
$(H^1_{\rm dR}(\breve Y)^{\rm sep})^{\varphi=p}\to
(\acris\otimes_{W(k_C)} H^1_{\rm dR}(\breve Y)^{\rm sep})^{\varphi=p}$ est 
un isomorphisme de $\Z_p$-modules libres de rang fini.
\end{coro}
\begin{proof}
      On d\'eduit directement de la proposition ci-dessus et du fait que 
      $H^1_{\rm dR}(\breve Y)^{\rm sep}$ est sans $p$-torsion et complet 
      pour la topologie $p$-adique
       que $H^1_{\rm dR}(\breve Y)^{\rm sep}$ est un 
      $W(k_C)$-module libre de type fini. Puisque $d\circ F=p\,F\circ d$, $F$ induit un endomorphisme de 
      $H^1_{\rm dR}(\breve Y)=\Omega^1(\breve Y)/d\O(\breve Y)$, qui induit \`a son tour un endomorphisme semi-lin\'eaire $F$ de 
  $H^1_{\rm dR}(\breve Y)^{\rm sep}$. L'isomorphisme $H^1_{\rm dR}(\breve Y)^{\rm sep}/p\cong k_C\otimes_{\mathbf{F}_p} 
      \Omega^1({\bar Y})^{C=1}$ 
   fourni par la proposition ci-dessus est compatible avec l'action de $F$ sur $H^1_{\rm dR}(\breve Y)^{\rm sep}$ et le frobenius de $k_C$. On en d\'eduit que $H^1_{\rm dR}(\breve Y)^{\rm sep}/p$ poss\`ede une base de points fixes sous  
   $\bar{F}$, qui y est donc bijectif. Autrement dit, $F$ est un automorphisme semi-lin\'eaire du 
   $W(k_C)$-module libre de type fini $H^1_{\rm dR}(\breve Y)^{\rm sep}$ et $F-1$ est surjectif,
et donc (puisque $k_C$ est alg\'ebriquement clos) la fl\`eche naturelle 
$$W(k_C)\otimes_{\Z_p} (H^1_{\rm dR}(\breve Y)^{\rm sep})^{F=1}\to H^1_{\rm dR}(\breve Y)^{\rm sep}$$ 
est un isomorphisme. 
  Puisque $\acris^{\varphi=1}=\Z_p$ et $\varphi=pF$, cela montre que la fl\`eche naturelle 
$(H^1_{\rm dR}(\breve Y)^{\rm sep})^{\varphi=p}\to 
(\acris\otimes_{W(k_C)} H^1_{\rm dR}(\breve Y)^{\rm sep})^{\varphi=p}$ 
est un isomorphisme. 
\end{proof}

Si $M$ est un $\O_{\breve C}$-module libre de rang fini,
muni d'un frobenius $\varphi$, le th\'eor\`eme
de Dieudonn\'e-Manin fournit une d\'ecomposition $M[\frac{1}{p}]=\oplus_{r\in\Q_+}M[\frac{1}{p}]^{[r]}$
par les pentes de~$\varphi$.  On d\'efinit $M^{[r]}$ comme 
$M\cap M[\frac{1}{p}]^{[r]}$.

On suppose que $\breve Y$ est obtenu en retirant \`a une courbe propre~$\breve X$
les tubes $]Q_0[,\dots,]Q_r[$ de points $Q_0,\dots,Q_r$ de la fibre sp\'eciale.

\begin{prop}\label{basic9.11}
L'application r\'esidu induit une suite exacte de $\varphi$-modules
$$0\to H^1_{\rm dR}(\breve X)^{[1]}\to H^1_{\rm dR}(\breve Y)^{\rm sep}
\to {\rm Ker}\big(\dbar:\O_{\breve C}^{\partial^{\rm ad}Y}\to\O_{\breve C}\big)(-1)\to 0.$$
\end{prop}
\begin{proof}
Comme la pente de $\varphi$ sur $H^2_{\rm dR}(\breve X)$ est $1$, 
et que $\varphi$ est divisible
par $p$ sur $\Omega^1(\breve X)$, on a $H^1_{\rm dR}(\breve X)^{[1]}\perp \Omega^1(\breve X)$,
et comme $\Omega^1(\breve X)$
est maximalement isotrope pour le cup-produit 
on en d\'eduit que $H^1_{\rm dR}(\breve X)^{[1]}\subset \Omega^1(\breve X)$.

Compte-tenu de ce qui pr\'ec\`ede, de la prop.\,\ref{sansk0}, et de ce que
tout est complet pour la topologie $p$-adique et sans torsion, on est ramen\'e \`a
prouver l'exactitude de la suite
$$0\to k_C\otimes_{{\bf F}_p}\Omega^1(\bar X)^{C=1}
\to k_C\otimes_{{\bf F}_p}\Omega^1(\bar Y)^{C=1}\to
{\rm Ker}\big(\dbar:k_C^{\partial^{\rm ad}Y}\to k_C\big)\to 0$$ 
Maintenant, on a $\Omega^1(\bar Y)^{C=1}=\Omega^1(\bar Y^\times)^{C=1}$, o\`u
$\Omega^1(\bar Y^\times)$ d\'esigne les formes diff\'erentielles ayant des p\^oles
simples en les points de $\bar X\moins \bar Y$. 
Donc 
(en utilisant, comme pr\'ec\'edemment,
que $C$ est semi-lin\'eaire et $k_C$ alg\'ebriquement clos) on se ram\`ene \`a prouver
l'exactitude de la suite
$$0\to \Omega^1(\bar X)
\to \Omega^1(\bar Y^\times)\to
{\rm Ker}\big(\dbar:k_C^{\partial^{\rm ad}Y}\to k_C\big)\to 0$$ 
Celle-ci d\'ecoule de ce que $H^1(\bar X,\Omega^1_{\bar X})\cong k_C$
(si $\bar X= U\cup V$ et $\omega\in\Omega^1(U\cap V)$ l'image de
$\omega$ dans $k_C$ est $\sum_{x\in V}{\rm Res}_x\omega=-\sum_{x\in U}{\rm Res}_x\omega$).
\end{proof}

\subsubsection{D\'evissage de la cohomologie \'etale}
     Rappelons que l'on dispose d'un r\'egulateur 
  $\delta_Y:{\rm Symb}_p(Y^{\rm gen})\to H^1_{\rm syn}(Y,1)$ ainsi que d'une projection $H^1_{\rm syn}(Y,1)\to (\acris\otimes _{W(k_C)}H^1_{\rm dR}(\breve Y)^{\rm sep})^{\varphi=p}\cong (H^1_{\rm dR}(\breve Y)^{\rm sep})^{\varphi=p}$. 
On en d\'eduit une application 
  $$R: {\rm Symb}_p(Y^{\rm gen})\to (H^1_{\rm dR}(\breve Y)^{\rm sep})^{\varphi=p}.$$
 Concr\`etement, en suivant les diverses identifications et en utilisant les notations des paragraphes pr\'ec\'edents, on a (o\`u $\theta_0:\acris\to \O_{\breve C}$ est le morphisme usuel):
 $$R((u_n)_{n})=\lim_{n\to \infty} \tfrac{d \theta_0(\tilde{u}_n)}{\theta_0(\tilde{u}_n)}.$$
  L'application $R$ s'annule sur $\Z_p\wotimes \O(Y)^{\dual\dual}$ (lemme \ref{short10}). Elle induit donc une application 
  $$R: \frac{{\rm Symb}_p(Y^{\rm gen})}{\Z_p\wotimes \O(Y)^{\dual\dual}}
\to  (H^1_{\rm dR}(\breve Y)^{\rm sep})^{\varphi=p}$$
  
  \begin{prop}\label{shortetale1}
   L'application 
$$R: \frac{{\rm Symb}_p(Y^{\rm gen})}{\Z_p\wotimes \O(Y)^{\dual\dual}}
\to  (H^1_{\rm dR}(\breve Y)^{\rm sep})^{\varphi=p}$$
 ci-dessus est un isomorphisme.
  \end{prop}
  \begin{proof} Les deux termes sont des $\Z_p$-modules libres de type fini: il suffit d'invoquer la 
  prop.\,\ref{short21} pour le premier et la discussion ci-dessus pour le second. Il suffit donc de montrer que la r\'eduction modulo 
  $p$ de $R$ est un isomorphisme. En identifiant ${\rm Symb}_p(Y^{\rm gen})$ et $ H^1_{\eet}(Y^{\rm gen},\Z_p(1))$, et en utilisant la prop. \ref{short21} et la rem.\,\ref{BS2.1}, ainsi que l'isomorphisme ${\rm Pic}(Y^{\rm gen})\cong {\rm Pic}(\bar Y)$ (rem.\,\ref{Gerritzen}) on obtient une suite exacte 
  $$0\to \O(\bar Y)^{\dual}/ (\O(\bar Y)^{\dual})^p\to 
\big(\frac{ {\rm Symb}_p(Y^{\rm gen})}{\Z_p\wotimes \O(Y)^{\dual\dual}}\big)/p
\to {\rm Pic}(\bar Y)[p]\to 0.$$
  
  D'autre part, le cor.\,\ref{sansk} exhibe une suite exacte analogue, avec 
$(H^1_{\rm dR}(\breve Y)^{\rm sep})^{\varphi=p}/p$ \`a la place de 
$\big(\frac{{\rm Symb}_p(Y^{\rm gen})}{\Z_p\wotimes \O(Y)^{\dual\dual}}\big)/p$. 
Il suffit de montrer que l'application $R$ est compatible, modulo $p$, avec ces deux suites exactes (les applications induites sur les termes extr\^emes \'etant l'identit\'e). En suivant les diverses 
  identifications et les d\'efinitions, on se ram\`ene \`a d\'emontrer l'\'enonc\'e suivant: si
  $u_1\in C(Y^{\rm gen})^*$ et $a\in \breve C(\breve Y)^*$ sont tels que 
  $p\mid {\rm Div}(u_1)$ et $u_1/a\in \O(Y^{\rm gen})^{\dual\dual}$, alors 
  $\frac{d\theta_0(\tilde{u}_1)}{\theta_0(\tilde{u}_1)}\equiv \frac{da}{a}\pmod p.$
Or ceci est \'evident car $\theta_0(\tilde u_1/a)\equiv 1$ mod~$p$.
  \end{proof}

\begin{coro}\label{shortetale}
On dispose  d'une suite exacte naturelle:
  $$0\to  \Z_p\wotimes\O(Y)^{\dual\dual}\to H^1_{\eet}(Y^{\rm gen},\Z_p(1))\to (\acris\otimes H^1_{\rm dR}(\breve Y)^{\rm sep})^{\varphi=p}\to 0$$
  et $(\acris\otimes H^1_{\rm dR}(\breve Y)^{\rm sep})^{\varphi=p}$ est un $\Z_p$-module libre de type fini. 
\end{coro}

\begin{rema}\label{vpj}
Soit $H^1_{\rm dR}(\breve Y)^{\rm sep}_0\subset H^1_{\rm dR}(\breve Y)^{\rm sep}$
le noyau de l'application r\'esidu. 
Il r\'esulte de la prop.\,\ref{shortetale1}, du (ii) de la prop.\,\ref{short21}
et de la prop.\,\ref{basic9.11}
que l'on a un isomorphisme
$$(\bcris^+\otimes H^1_{\rm dR}(\breve X)^{[1]})\cong V_pJ(k_C).$$
Si on remplace le sous-espace $H^1_{\rm dR}(\breve X)^{[1]}$
par l'espace tout entier, on obtient (en sp\'ecialisant le th.\,\ref{AF1} au cas de
bonne r\'eduction):
$$(\bcris^+\otimes H^1_{\rm dR}(\breve X))^{\varphi=p}\cong V_pJ(\O_C/p).$$
\end{rema}

\Subsection{Comparaison syntomique-\'etale pour les petits shorts}\label{BAS8}
On dit qu'un $\O_{\breve C}$-short $Y$ est {\it petit} si $\O(Y)$ est \'etale au-dessus 
de $\O_{\breve C}\langle T\rangle$ (i.e.~si $Y$ est \'etale au-dessus d'une boule ferm\'ee). 
\subsubsection{L'op\'erateur $\psi$}
Soit $Y$ un petit short sur $\O_{\breve C}$.
Par d\'efinition, 
$\O(Y)$ est \'etale au-dessus de $\O_{\breve C}\langle T\rangle$, et
on munit de $\O(Y)$ du frobenius $\varphi$ v\'erifiant $\varphi(T)=T^p$.

Maintenant, $dT$ est une base de $\Omega^1(Y)$, et donc
$T$ est une $p$-base de $\O(Y)$ sur $\varphi(\O(Y))$:
tout \'el\'ement $x$ de $\O(Y)$ peut s'\'ecrire, de mani\`ere unique,
sous la forme $x=\sum_{i=0}^{p-1}\varphi(x_i)T^i$.
On d\'efinit $\psi:\O(Y)\to\O(Y)$ par
$$\psi\big(\sum\nolimits_{i=0}^{p-1}\varphi(x_i)T^i\big):=x_0$$
L'op\'erateur $\psi$ est un inverse \`a gauche de $\varphi$: on a $\psi\circ\varphi={\rm id}$.
On d\'efinit $\psi$ sur $\Omega^1(Y)$ par la formule
$$\psi(f\,\tfrac{dT}{T})=\psi(f)\,\tfrac{dT}{T}$$
%
On a $\psi\circ\varphi=p$ sur $\Omega^1(Y)$ et
$\psi\circ d=p\, d\circ\psi$.

\begin{lemm}\label{basic5}
$d$ induit un isomorphisme
$d:\O(Y)^{\psi=0}\overset{\sim}{\to}\Omega^1(Y)^{\psi=0}$.
\end{lemm}
\begin{proof}
Il suffit de prouver le r\'esultat modulo~$p$.
Un \'el\'ement de $(\O(Y)/p)^{\psi=0}$ (resp.~$(\Omega^1(Y)/p)^{\psi=0}$) s'\'ecrit, de
mani\`ere unique, sous la forme 
$x=\sum_{i=1}^{p-1}\varphi(x_i)\,T^i$ (resp.~$\omega=\sum_{i=1}^{p-1}\varphi(x_i)\,T^i\,\frac{dT}{T}$,
et on a $dx=\big(\sum_{i=1}^{p-1}i\varphi(x_i)T^i\big)\frac{dT}{T}$.
On en d\'eduit le r\'esultat.
%
\end{proof}

\Subsubsection{Structure des $\O(Y)$-modules munis d'un $\psi$}

\begin{lemm}\label{basic6}
Soit $M$ un $\O(Y)$-module de type fini, muni d'un op\'erateur $\psi$
v\'erifiant $\psi(\varphi(a)x)=a\psi(x)$ pour tous $a\in\O(Y)$ et $x\in M$.
Il existe un unique sous-module $M^\sharp$ de $M$,
de rang fini sur $W(k_C)$, stable par $\psi$ et sur lequel $\psi$ est
bijectif.  De plus, si $x\in M$ et $k\geq 1$, alors $\psi^n(x)\in M^\sharp+p^k
M$, pour tout $n\geq n(x,k)$.
\end{lemm}
\begin{proof}
Commen\c{c}ons par prouver le m\^eme r\'esultat modulo~$p$.
Choisissons une famille g\'en\'eratrice $m_1,\dots,m_s$ de $M$ sur
$\O(Y)$ et \'ecrivons
$\O(Y)$ comme un quotient de $\O_{\breve C}\langle X_1,\dots,X_r\rangle$.
D\'efinissons le degr\'e de $x\in\O(Y)/p$ comme le minimum
des degr\'es des \'el\'ements 
de $k_C[X_1,\dots,X_r]$ ayant pour image $x$, et
le degr\'e $\deg x$ de $x\in M/p$
comme le minimum sur toutes les \'ecritures possibles $x=\sum_ix_im_i$
du maximum des degr\'es des $x_i$.

Soit $x\in M/p$ et soit
$\sum_{i,j}a_{i}X^im_j$ une \'ecriture de
$x$ de degr\'e $\deg x$.  En \'ecrivant $i$ 
sous la forme $i_0+pi'$,
avec $i_0\in\{0,\dots,p-1\}^{[1,r]}$,
cela permet d'ecrire $x$ sous la forme
$$x=\sum_{i_0,j}X^{i_0} x_{i_0,j}^p m_j,$$
avec $x_{i_0,j}\in\O(Y)/p$ de degr\'e~$\leq\frac{1}{p}\deg x$.
Alors $$\psi(x)=\sum_{i_0}
\psi(X^{i_0}m_j) x_{i_0,j}.$$
Soit $N$ le maximum des degr\'es des
$\psi(X^{i_0}m_j)$.
La formule ci-dessus fournit la majoration
$\deg(\psi(x))\leq N+\frac{1}{p}\deg(x)$.
On en d\'eduit que $\deg(\psi^n(x))\leq \frac{pN}{p-1}$, si $n$
est assez grand.
L'ensemble $M_N$ des $x\in M/p$ tels que $\deg x\leq \frac{pN}{p-1}$
est un $k_C$-espace de dimension finie, qui est stable par $\psi$
d'apr\`es ce qui pr\'ec\`ede.
On note $\overline M^\sharp$ l'intersection des $\psi^n(M_N)$, pour $n\geq 1$
(cette intersection se stabilise pour un $n\leq\dim_{k_C}M_N$, i.e. il
existe $n_0$ tel que $\overline M^\sharp=\psi^{n_0}(M_N)$).
Alors $\overline M^\sharp$ est un $k_C$-espace de dimension finie, stable par $\psi$,
sur lequel $\psi$ est surjectif et donc bijectif, et c'est l'unique
$k_C$-espace de dimension finie avec ces propri\'et\'es.
De plus, si $x\in M$, il existe $n(x)\in\N$ tel que
$\psi^n(x)\in \overline M^\sharp$, pour tout $n\geq n(x)$.

Notons que $\psi$ est $\varphi^{-1}$-semi-lin\'eaire.  Il existe donc, d'apr\`es
le th\'eor\`eme de Dieudonn\'e-Manin, une
base $(e_1,\dots,e_t)$ de $\overline M^\sharp$ v\'erifiant
$\psi(e_i)=e_i$ pour tout $i$.

Prouvons maintenant, par r\'ecurrence sur $k$, l'existence
et l'unicit\'e de $M_k^\sharp\subset M_k=M/p^k$ satisfaisant les exigences du lemme
(d'apr\`es ce qui pr\'ec\`ede $M_1^\sharp=\overline M^\sharp$), 
et la surjectivit\'e de $M_k^\sharp\to M_1^\sharp$.
Commen\c{c}ons par prouver qu'il existe $e_{i,k+1}\in M_{k+1}$
ayant pour image $e_{i,k}$ dans $M_k$, et v\'erifiant $\psi(e_{i,k+1})=
e_{i,k+1}$.  Pour cela, choisissons un rel\`evement arbitraire
$x$ de $e_{i,k}$ dans $M_{k+1}$.  Alors $\psi(x)-x\in
p^kM_{k+1}\cong M_1$.
Il existe donc $n_0$ tel que $\psi^n(\psi(x)-x)\in p^k M_1^\sharp$ et,
quitte \`a remplacer $x$ par $\psi^{n_0}(x)$, on peut supposer
que $\psi(x)-x\in p^kM_1^\sharp$.  Maintenant, $\psi-1$ est surjectif sur
$p^kM_1^\sharp$ (car $\psi$ est $\varphi^{-1}$-semi-lin\'eaire); il existe
donc $y\in p^kM_1^\sharp$ tel que $\psi(y)-y=\psi(x)-x$, et
on peut prendre $e_{i,k+1}=x-y$.

Soit $x\in M_{k+1}$.
D'apr\`es l'hypoth\`ese de r\'ecurrence, il existe
$n_0$ tel que l'image de $\psi^{n_0}(x)$ dans $M_k$
appartienne \`a $M_k^\sharp$.  On peut donc \'ecrire
$\psi^{n_0}(x)=\sum_ia_ie_{i,k+1}+py$, avec
$y\in p^kM_{k+1}(\cong M_1)$.
Il existe $n_1$ tel que $\psi^{n_1}(y)\in p^kM_1^\sharp$
et donc $\psi^{n_1}(y)=\sum_ip^kb_ie_{i,k+1}$.
Il s'ensuit que $\psi^{n_0+n_1}(x)=\sum_i(a_i^{p^{-n_1}}+p^kb_i)e_{i,k+1}$,
et donc que $M_{k+1}^\sharp=\oplus_{i=1}^r(W(k)/p^{k+1})e_{i,k+1}$ a toutes
les propri\'et\'es voulues.

On en d\'eduit le r\'esultat en passant \`a la limite projective.
\end{proof}

\begin{lemm}\label{basic7}
Si $x\in M$, il existe $c_0(x)\in M^\sharp$, unique, tel
que $\psi^n(x-c_0(x))\to 0$ quand $n\to+\infty$.
\end{lemm}
\begin{proof}
L'unicit\'e r\'esulte de ce que $\psi$ est bijectif sur $M^\sharp$,
et donc $\psi^n(x)\to 0$ implique $x=0$, si $x\in M^\sharp$.

Passons \`a l'existence.
Soit $x\in M$.  Alors l'image de $\psi^n(x)$
dans $ M_k=M/p^kM$ appartient \`a $M^\sharp_k$, si $n\geq n(x,k)$.
Maintenant, $\psi$ est bijectif sur $M^\sharp_k$ et, si on note
$\psi_k^{-n}$ l'inverse de $\psi^n$ sur $M^\sharp_k$, alors $\psi_k^{-n}(\psi^n(x))$
ne d\'epend pas du choix de $n\geq n(x,k)$.
On note $c_{0,k}(x)\in M^\sharp_k$ cette quantit\'e;
on a $c_{0,k+1}(x)=c_{0,k}(x)$ dans $M^\sharp_k$, et donc les
$c_{0,k}(x)$ d\'efinissent un \'el\'ement $c_0(x)$ de $M^\sharp$.
Par construction, $\psi^n(x)=\psi^n(c_0(x))$ modulo~$p^k$, si $n\geq n(x,k)$.

Ceci permet de conclure.
\end{proof}

\begin{lemm}\label{basic8}
Si $M$ est muni d'un op\'erateur $\varphi_M$ tel que $\psi\circ\varphi_M=1$,
tout \'el\'ement $x$ de $ M$ peut s'\'ecrire, de mani\`ere unique,
sous la forme $x=c_0(x)+\sum_{i\geq 1}\varphi_M^{i-1}(c_{i}(x))$, avec $c_0(x)\in M^\sharp$,
$c_i(x)\in  M^{\psi=0}$ et $c_i(x)\to 0$ quand $i\to +\infty$.
\end{lemm}
\begin{proof}
Si $c_0+\sum_{i\geq 1}\varphi_M^{i-1}(c_i)=0$, alors $\psi^n(c_0)+\sum_{i\geq n}
\varphi_M^{i-n}(c_i)=0$, et donc $\psi^n(c_0)\to 0$.  Il s'ensuit
que $c_0=0$.  En appliquant $\varphi_M^n\psi^n$ \`a la s\'erie,
on d\'emontre alors par r\'ecurrence que $\varphi_M^{i-1}(c_i)=0$
et donc $c_i=0$, pour tout $i$.  On en d\'eduit l'unicit\'e.

Pour
l'existence, posons
$c_i(x)=(1-\varphi_M\psi)(\psi^{i-1}(x-c_0(x)))$, si $i\geq 1$.
Comme $\psi\circ(1-\varphi_M\psi)=\psi-\psi=0$, on a
$\psi(c_i(x))=0$, et comme $\psi^{i-1}(x-c_0(x))\to 0$, on
a $c_i(x)\to 0$.  L'identit\'e
$x=c_0(x)+\sum_{i\geq 1}\varphi_M^{i-1}((1-\varphi_M\psi)(\psi^{i-1}(x-c_0(x))))$
est alors imm\'ediate.
\end{proof}

\subsubsection{Application au calcul de la cohomologie de de Rham}

Les groupes $H^1_{\rm dR}(Y)$ et $H^1_{\rm dR}(Y)^{\rm sep}$ sont munis d'actions de $\varphi$ et $\psi$,
et on a $\psi\circ\varphi=p$.

On peut appliquer le lemme~\ref{basic8} \`a $M=\O(Y)$ (et $\varphi_M=\varphi$)
et $M=\Omega^1(Y)$ (et $\varphi_M=p^{-1}\varphi$).
On obtient des d\'ecompositions
\begin{equation}\label{basic9.1}
\O(Y)=M_0^\sharp\oplus\big(\widehat\oplus_{i\geq 0}\varphi^{i}(\O(Y)^{\psi=0})\big),\quad
\Omega^1(Y)=M_1^\sharp\oplus\big(\widehat\oplus_{i\geq 0}p^{-i}\varphi^{i}(\Omega^1(Y)^{\psi=0})\big).
\end{equation}
Par ailleurs, comme $\psi^n\circ d=p^n\,d\circ\psi^n$ et que $\psi$ est surjectif
sur $M_0^\sharp$, on a $d=0$ sur $M_0^\sharp$, et donc $M_0^\sharp=\O_{\breve C}$.
On en d\'eduit, en utilisant le lemme~\ref{basic8},
le r\'esultat suivant.
\begin{prop}\label{basic9}
L'application naturelle $H^1_{\rm dR}(Y)\to H^1_{\rm dR}(Y)^{\rm sep}$ induit
un isomorphisme $M_1^\sharp\cong H^1_{\rm dR}(Y)^{\rm sep}$ qui est $\psi$ et $\varphi$-\'equivariant,
et on a une d\'ecomposition $\psi$ et $\varphi$-\'equivariante
$$H^1_{\rm dR}(Y)\cong H^1_{\rm dR}(Y)^{\rm sep}\oplus H^1_{\rm dR}(Y)_{\rm tors}.$$
De plus, 
$$H^1_{\rm dR}(Y)^{\rm sep}=\O_{\breve C}\otimes_{\Z_p}(H^1_{\rm dR}(Y)^{\rm sep})^{\psi=1}.$$
\end{prop}
\begin{proof}
$M_1^\sharp$ 
est de rang fini sur $\O_{\breve C}$ et stable par $\psi$ qui agit
bijectivement et $\varphi^{-1}$-semi-lin\'eairement.  Il r\'esulte du th\'eor\`eme
de Dieudonn\'e-Manin que $M_1^\sharp$ admet une base $e_1,\dots,e_d$ sur $\O_{\breve C}$,
telle que $\psi(e_i)=e_i$ pour tout $i$.  Mais alors $(\varphi-p)\cdot e_i$ est tu\'e
par $\psi$ puisque $\psi\circ\varphi=p$, et donc $(\varphi-p)\cdot e_i=0$
dans $H^1_{\rm dR}(Y)$ (cela d\'ecoule du lemme~\ref{basic5}).
Autrement dit, $\varphi(e_i)=pe_i$ dans $H^1_{\rm dR}(Y)$, ce qui prouve
que $M_1^\sharp$ est stable par $\varphi$ et permet de conclure. 
\end{proof}

\subsubsection{Injectivit\'e du r\'egulateur syntomique}
Soit $\tilde R$ une $\ainf$-alg\`ebre s\'epar\'ee et compl\`ete pour la topologie
$(p,\tilde p)$-adique, et telle que $\tilde R/p$ soit int\`egre et int\'egralement close.
On suppose $\tilde R$ munie d'un frobenius $\varphi$ relevant $x\mapsto x^p$ sur
$\tilde R/p$ (on a donc $\varphi(x)-x^p\in p\tilde R$, si $x\in\tilde R$) et co\"{\i}ncidant
avec le frobenius usuel sur $\ainf$.

Soit $\tilde R^{\dual\dual}=\varinjlim_{r>0}1+(p,\tilde p^r)\tilde R\subset\tilde R^\dual$.

\begin{lemm}\label{Aff5}
Soit 
$x\in \tilde R^{\dual\dual}$. 

{\rm (i)} Si $x^p=\varphi(u)$, avec $u\in\tilde R$, alors
il existe $u'\in\tilde R^{\dual\dual}$ tel que $x=\varphi(u')$ et $u=(u')^p$.

{\rm (ii)} Si $\varphi(x)/x^p=v^{p^N}$, avec $v\in\tilde R$,
alors il existe $u\in\tilde R^{\dual\dual}$ tel que $x=u^{p^N}$.
\end{lemm}
\begin{proof}
Nous allons commencer par montrer que 
si $x\equiv 1+[\beta] b_1\pmod p$, avec $b_1\in \tilde R$ et $\beta\in {\goth m}_{C^{\flat}}$, alors on peut trouver 
$u\equiv 1$ mod~$[\beta^{1/p}]$ et $b_2\in \tilde R$ tels que $x/\varphi(u)\equiv 1+[\beta^2]b_2$ mod~$p$. 
Ecrivons $x=1+pa_1+[\beta]b_1$ avec $a_1,b_1\in\tilde R$. On a alors $x^p=1+p(pa_1+[\beta]b_1)+[\beta]^pb_1^p$ mod~$p(p,[\beta])^2$
et $\varphi(1+[\beta]b_1)=1+[\beta]^p(b_1^p+pb'_1)$, avec $b'_1\in\tilde R$,
et donc
$x^p/\varphi(1+[\beta]b_1)=1+py$, avec $y=pa_1+[\beta]b_1$ mod $(p,[\beta])^2$.
De plus, comme $x^p=\varphi(u)$, on en d\'eduit que $y$ est de la forme $\varphi(v)$,
avec $v\in\tilde R$ (\`a priori $v\in \tilde R[1/p]$, mais comme $\tilde R/p$ est r\'eduit, on voit que l'on a $v\in\tilde R$).  Il en r\'esulte qu'il existe $z\in \tilde R$ tel que $[\beta]b_1+[\beta]^2z$
soit une puissance $p$-i\`eme
modulo $p$. Puisque $\tilde R/p$ est int\'egralement clos, il s'ensuit que $b_1+[\beta]z$ est une puissance 
$p$-i\`eme modulo $p$ et donc il existe $b_1'\in \tilde R$ tel que $b_1+[\beta]z\equiv \varphi(b_1')$ mod~$p$. Mais alors 
$[\beta]b_1\equiv \varphi([\beta^{1/p}]b_1')$ mod~$(p,[\beta]^2)$ et donc $x/\varphi(1+[\beta^{1/p}]b_1')=1+pa_2+[\beta^2]b_2$ pour certains $a_2,b_2\in \tilde R$, ce qui permet de conclure.

En r\'eit\'erant ce proc\'ed\'e, et en passant \`a la limite, on voit
que l'on peut \'ecrire $x=\varphi(u_0)(1+pa)$ pour des $u_0,a\in \tilde R$. 
Comme $(1+pa)^p=\varphi(u/u_0)$, on a $p^2(a+p(\frac{p-1}{2}a^2+\cdots))=\varphi((u/u_0)-1)$,
et comme $\tilde R/p$ est int\`egre, cela implique que $\varphi((u/u_0)-1)=p^2\varphi(w)$,
et donc\footnote{Si $p=2$, on obtient $a+a^2=(a'_1)^2$ mod~$2$, et donc $a=(a+a'_1)^2$ mod~$2$.}
 que $a=(a_1')^p$ modulo $p$. 
Cela permet d'\'ecrire $1+pa$ sous la forme
$(1+p\varphi(a_1'))(1+p^2a_1)$ et, par r\'ecurrence et passage \`a la limite,
$1+pa$ sous la forme $\varphi(\prod(1+p^ka'_k))$. Ceci prouve le (i).

Passons au (ii).
On a $\varphi(x)=(xv^{p^{N_1}})^p$, et il r\'esulte du (i) qu'il existe
$u_0$ tel que $x=u_0^p$.  Mais alors $(\varphi(u_0)/u_0^p)^p=v^{p^N}$ ce qui, si $p\neq 2$,
implique $\varphi(u_0)/u_0^p=v^{p^{N-1}}$ et une r\'ecurrence imm\'ediate fournit le r\'esultat
(si $p=2$, le m\^eme raisonnement s'applique en rempla\c{c}ant \'eventuellement $u_0$ par $-u_0$).
\end{proof}

L'injectivit\'e de 
$\delta_Y:{\rm Symb}_p(Y^{\rm gen})\to H^1_{\rm syn}(Y,1)$ est une cons\'equence
du r\'esultat plus pr\'ecis suivant.
\begin{lemm}\label{short9}
Si $\frac{d\tilde u_1}{\tilde u_1}$ 
et $\frac{1}{p}\log\frac{\varphi(\tilde u_1)}{\tilde u_1^p}$
sont divisibles par $p$, alors $u_1$ est une puissance $p$-i\`eme.
\end{lemm}
\begin{proof}
On peut \'ecrire $u_1=ab$, avec $a\in{\rm Fr}(\O(\breve Y))$ et $b\in\O(Y)^{\dual\dual}$.
(On regarde $u_1$ sur la fibre sp\'eciale (apr\`es avoir divis\'e par $p^r$ pour obtenir
un \'el\'ement de valuation~$0$) et on prend pour $a$ un rel\`evement dans ${\rm Fr}(\O(\breve Y))$.)
La condition $\frac{d\tilde u_1}{\tilde u_1}$ divisible par $p$ implique que $u_1$
est une puissance $p$-i\`eme sur la fibre sp\'eciale, et donc peut se relever
en une puissance $p$-i\`eme, ce qui permet de supposer que $a=1$, et donc que $u_1\in\O(Y)^{\dual\dual}$.
Le lemme~\ref{Aff5} permet alors d'en d\'eduire que $\tilde u_1$ est une puissance $p$-i\`eme
(car $\frac{\varphi(\tilde u_1)}{\tilde u_1^p}=
\big(\exp(\frac{1}{p}\log\frac{\varphi(\tilde u_1)}{\tilde u_1^p})\big)^p$ en est une).
On en d\'eduit que $u_1$ est une puissance $p$-i\`eme, ce que l'on voulait.
\end{proof}

\subsubsection{Bijectivit\'e du r\'egulateur syntomique}
On note $H^1_{\rm dR}(\widetilde Y)^{\rm sep}$ le quotient de 
$H^1_{\rm dR}(\widetilde Y)$ par l'adh\'erence de son sous-groupe de torsion.
Il r\'esulte de la description de $H^1_{\rm dR}(\widetilde Y)$
donn\'ee dans la prop.~\ref{basic9} que l'on a
$$H^1_{\rm dR}(\widetilde Y)^{\rm sep}=M_1^\sharp\otimes_{\Z_p} \acris.$$

\begin{prop}\label{basic29}
Si $p>2$, on a une suite exacte {\rm (}si $p=2$, la suite est presque exacte{\rm)}
$$0\to \Z_p\wotimes\O(Y)^{\dual\dual}\to H^1_{\rm syn}(Y,1)
\to (H^1_{\rm dR}(\widetilde Y)^{\rm sep})^{\varphi=p}\to 0.$$

\end{prop}
\begin{proof}
Il r\'esulte du lemme~\ref{short10} que $\O(Y)^{\dual\dual}$ s'envoie sur $0$
dans $(H^1_{\rm dR}(\widetilde Y)^{\rm sep})^{\varphi=p}$, et donc
que $\Z_p\wotimes\O(Y)^{\dual\dual}$ aussi; la suite est donc un complexe.

$\bullet$ L'exactitude \`a droite r\'esulte de la surjectivit\'e
de $H^1_{\rm syn}(Y,1)
\to H^1_{\rm dR}(\widetilde Y)^{\varphi=p}$ (cf.~suite exacte~\ref{BAS18.0})
et de ce que $H^1_{\rm dR}(\widetilde Y)^{\rm sep}$ s'identifie
\`a un sous-module de $H^1_{\rm dR}(\widetilde Y)$ stable par~$\varphi$
(prop.~\ref{basic9}), et donc $(H^1_{\rm dR}(\widetilde Y)^{\rm sep})^{\varphi=p}
\subset H^1_{\rm dR}(\widetilde Y)^{\varphi=p}$.

$\bullet$ L'exactitude \`a gauche, i.e.~l'injectivit\'e
de $\Z_p\wotimes \O(Y)^{\dual\dual}\to H^1_{\rm syn}(Y,1)$, r\'esulte
du lemme~\ref{short9} puisque
${\rm Symb}_p(Y^{\rm gen})$ contient $\Z_p\wotimes \O(Y)^{\dual\dual}$.

$\bullet$ Il reste \`a prouver l'exactitude au milieu.
Soit $(e_j)_{j\in J}$ une base orthonormale de $\O(\breve Y)^{\psi=0}$ sur $\O_{\breve C}$.
Soit $x\in H^1_{\rm dR}(\widetilde Y)_{\rm tors}^{\varphi=p}$ et soit $\hat x\in
\Omega^1(\widetilde Y)$ ayant pour image $x$. 
Il r\'esulte des d\'ecompositions~(\ref{basic9.1})
que l'on peut \'ecrire $\hat x$,
de mani\`ere unique, sous la forme $\hat x=\sum_{i\geq 1}\sum_{j\in J}b_{i,j}d(p^{-i}\varphi^i(e_j))$,
avec $b_{i,j}\in\acris$ et $b_{i,j}\to 0$ quand $(i,j)\to\infty$, et la condition
$(\varphi-p)\hat x=0$ se traduit par l'existence de $a_{i,j}\in\acris$, 
$a_{i,j}\to 0$ quand $(i,j)\to\infty$, 
tels que $b_{1,j}=a_{1,j}$ et $b_{i,j}=\varphi(b_{i-1,j})+p^{i-1}a_{i,j}$ si $i\geq 2$.
On a alors $$b_{i,j}=\varphi^{i-1}(a_{1,j})+p\varphi^{i-2}(a_{2,j})+\cdots+p^{i-1}a_{i,j}.$$
On d\'eduit la prop.\,\ref{basic29} du lemme~\ref{gabr} ci-dessous
en raisonnant comme pour la preuve de la prop.~\ref{boule1} (i.e.~en utilisant
la d\'ecomposition $\acris=\O_{\breve C}\oplus{\rm Ker}\,\theta_0$, puis 
le fait que $\acris=\ainf[\frac{\tilde p^k}{k!},\  k\geq 0]^\wedge$ et enfin, le fait que
tout \'el\'ement de $\ainf$ a un d\'eveloppement de Teichm\"uller).
\end{proof}
\begin{lemm}\label{gabr}
Si $e\in \O(\breve Y)$ et $v_{C^\flat}(a)>0$,
alors $\Sigma(a,e)=\sum_{i\geq 0}[a^{p^i}]d(p^{-i}\varphi^i(e))$
peut s'\'ecrire sous la forme $\frac{dv}{v}$ avec $v\in\O(\widetilde Y)^\dual$.
\end{lemm} 
\begin{proof} 
D'apr\`es la rem.\,2.13 de \cite{BS}, 
il existe une suite d'applications  $\delta_n:\O(\widetilde Y)\to \O(\widetilde Y)$, pour $n\geq 0$, avec $\delta_0(x)=x$,
telles que $$\varphi^n(x)=\delta_0(x)^{p^n}+p\delta_1(x)^{p^{n-1}}+\cdots+p^n\delta_n(x).$$
On en d\'eduit que 
$$\sum_{i\geq 0}p^{-i}[a^{p^i}]\varphi^i(e)=\sum_{n\geq 0}\ell([a^{p^n}]\delta_n(e)),
\quad{\text{avec $\ell(x)=\sum_{k\geq 0}p^{-k}x^{p^k}$.}}$$
Il s'ensuit que 
$$\exp\big(\sum_{i\geq 0}p^{-i}[a^{p^i}]\varphi^i(e)\big)=
\prod_{n\geq 0}\exp_{\rm AH}([a^{p^n}]\delta_n(e)),$$
o\`u $\exp_{\rm AH}$ est l'exponentielle d'Artin-Hasse (qui appartient \`a $1+X\Z_p[[X]]$);
le produit converge dans $\O(\widetilde Y)^{\dual\dual}$ et si $v$ est le r\'esultat,
on a $\frac{dv}{v}=\Sigma(a,e)$, ce que l'on voulait.
\end{proof}

\begin{theo}\label{short11}
$\delta_Y:{\rm Symb}_p(Y^{\rm gen})\to H^1_{\rm syn}(Y,1)$ est un isomorphisme si $p>2$
{\rm(}si $p=2$, c'est presque un isomorphisme{\rm)}.
\end{theo}
\begin{proof}
Cela r\'esulte, via le lemme des 5, du diagramme commutatif suivant dont
les lignes sont exactes:
%
$$\xymatrix@R=.4cm@C=.5cm{
0\ar[r]&\Z_p\wotimes\O(Y)^{\dual\dual}\ar@{=}[d]\ar[r]& {\rm Symb}_p(Y^{\rm gen})
\ar[d]\ar[r] & (\acris\otimes H^1_{\rm dR}(\breve Y)^{\rm sep})^{\varphi=p}\ar[d]\ar[r]&0\\
0\ar[r]&\Z_p\wotimes\O(Y)^{\dual\dual}\ar[r]& H^1_{\rm syn}(Y,1)\ar[r]
&(\acris\otimes H^1_{\rm dR}(\breve Y)^{\rm sep})^{\varphi=p}\ar[r]&0
}$$
(L'exactitude de la ligne du haut est la prop.\,\ref{shortetale1} et celle
de la ligne du bas est la prop.\,\ref{basic29}.)
%
%
\end{proof}

\section{Cohomologie des courbes quasi-compactes}\label{BAS4}
Dans ce chapitre, on calcule la cohomologie d'une courbe quasi-compacte
\`a partir d'une triangulation: une telle triangulation fournit un patron de la
courbe et on exprime toutes les cohomologies en termes des cohomologies
des termes de ce patron: en particulier, les symboles (prop.\,\ref{basic14}),
la cohomologie \'etale $\ell$-adique (th.\,\ref{ladique}
plus formule {\og de Picard-Lefschetz\fg} de la rem.\,\ref{PL}), la cohomologie de de Rham (cor.\,\ref{basic16.2})
et sa s\'epar\'ee (th.\,\ref{Aff1}), la cohomologie de Hyodo-Kato (th.\,\ref{basic21} et
rem.\,\ref{BAS12.4}).  Cela permet de relier les symboles 
$p$-adiques, la cohomologie syntomique et la cohomologie
\'etale $p$-adique (th.\,\ref{basic35} et cor.\,\ref{basic35.1}).
Enfin, on exprime la cohomologie syntomique en termes du complexe de de Rham 
(rem.\,\ref{basic24} et th.\,\ref{basic26}) et on en d\'eduit le th.\,\ref{intro1.1}
(cf.~th.\,\ref{basic40}).

\subsection{Notations}
Soit $Y$ une
courbe quasi-compacte sur $C$ (i.e.~un affino\"{\i}de ou une
courbe propre).  Soit $S$ une triangulation de $Y$, et soit
$Y_S$ le $\O_C$-mod\`ele semi-stable associ\'e.
On suppose $S$
suffisamment fine pour
que:

\quad $\bullet$ la fibre sp\'eciale $Y_S^{\rm sp}$ ait au moins deux composantes
irr\'eductibles

\quad $\bullet$ ces composantes irr\'eductibles
soient lisses et deux d'entre elles s'intersectent en au plus un point,

\quad $\bullet$ les shorts $Y_s$, pour $s\in S$, soient\footnote{
Cette derni\`ere condition n'intervient que dans le \S\,\ref{BAS14} o\`u l'on
utilise les r\'esultats du \S\,\ref{BAS8}. Si on part d'une triangulation $S$
v\'erifiant les deux premiers points, on en fabrique une v\'erifiant le dernier de la
mani\`ere suivante: si $s\in S$, on peut d\'ecomposer le short $Y_s$ en un petit
short et un nombre fini de boules ouvertes, tubes de points de la fibre sp\'eciale;
on peut d\'ecomposer chacune de ces boules en une boule ferm\'ee et une couronne, et
il suffit de rajouter \`a $S$ les points de Gauss de ces boules ferm\'ees pour
obtenir une triangulation ayant les propri\'et\'es voulues.}
 petits.

\smallskip
Soit $(\Gamma,(Y_i)_{i\in I},(\iota_{i,j})_{(i,j)\in I_{2,c}})$ un patron
de $Y$ associ\'e \`a $S$ (cf.~\S\,\ref{Pconstr10}); cela inclut:

\quad $\bullet$ un graphe bipartite marqu\'e $\Gamma=(I,I_2,\mu)$, avec
$I=S\sqcup A_c$ et $I_{2,c}=\{(a,s),\ a\in A_c,\ s\in S(a)\}$,
$\mu(a)\in\Q_+^\dual$ si $a\in A_c$,

\quad $\bullet$ des shorts $Y_s$ pour $s\in S$, des jambes $Y_a$ pour $a\in A_c$ (avec $Y_a$ de
longueur $\mu(a)$), la fibre sp\'eciale $Y_s^{\rm sp}$ de $Y_s$ \'etant propre (et lisse) par convention
(avec certains points marqu\'es de multiplicit\'e $0^+$ si $s\in\partial Y$),

\quad $\bullet$ si $a\in A_c$, un param\`etre local $T_{a,s_1}$ 
de $Y_a$ 
(ce qui fournit une orientation de $Y_a$
et donc fixe une origine $s_1$ et un bout $s_2$)
admettant un prolongement \`a des ouverts de $Y_{s_1}$ et $Y_{s_2}$,
et $T_{a,s_2}=p^{\mu(a)}/T_{a,s_1}$,

\quad $\bullet$ des cercles fant\^omes $Y_{a,s}$, pour $a\in A(s)$, avec 
$\O(Y_{a,s})=\O_C[[T_{a,s},T_{a,s}^{-1}\rangle$,

\vskip.1cm
A ces donn\'ees, on rajoute deux entiers $N_1(S)$ et $N(S)$ d\'efinis par\footnote{Si $Y_S$ est d\'efini
sur $K$, d'indice de ramification absolu $e$, alors $N_1(S)\leq N(s)\leq\lceil\frac{\log e}{\log p}\rceil$.}:

\qquad $\bullet$ $N_1(S)$ est le plus petit entier $N$ tel que $p^N\mu(a)\geq 1$, pour tout $a\in A_c$.

\qquad $\bullet$
 $N(S)$ est le plus petit entier $N\geq N_1(S)$ tel que $T_{a,s}\in\O(\breve Y_s)+p^{1/p^N}\O(Y_s)$,
pour tout $(a,s)\in I_{2,c}$.

\smallskip
On rappelle que $\Sigma(Y)$ est l'ensemble des noeuds de $Y$; on a $\Sigma(Y)\subset S$.

\Subsection{Cohomologie \'etale}
\subsubsection{Localisation des symboles}
Soit $\ell$ un nombre premier.
On a d\'efini les groupes ${\rm Symb}_\ell(Z)$ pour $Z=Y,Y^{\rm gen}_i,Y^{\rm gen}_{i,j}$
(cf.~\no\ref{BAS6} pour $Y$ et $Y^{\rm gen}_s$ et \no\ref{BAS6.6} pour $Y^{\rm gen}_a$ et~$Y^{\rm gen}_{a,s}$).
Les restrictions induisent des applications naturelles
$${\rm Symb}_\ell(Y)\to {\rm Symb}_\ell(Y^{\rm gen}_i)\to {\rm Symb}_\ell(Y^{\rm gen}_{i,j}).$$

Posons 
\begin{align*}
A_{\ell,\infty}(Y_S)=&\ 
\genfrac{\{}{\}}{0pt}{}{((u_{i,n},v_{i,n}\in C(Y_i)^\dual)_{i\in I,n\in\N}, 
(g_{i,j,n}\in C(Y_{i,j})^\dual)_{(i,j)\in I_{2,c}, n\in\N},}
{u_{i,n+1}=u_{i,n}v_{i,n}^{\ell^n},\ 
u_{i,n}=g_{i,j,n}^{\ell^n}u_{j,n},\ g_{i,j,n+1}^\ell=g_{i,j,n}\frac{v_{i,n}}{v_{j,n}}},\\
A_{\ell,\infty}(Y_S)^{\rm triv}=&\ 
\{(a_{i,n}^{\ell^n},a_{i,n+1}^\ell/a_{i,n})_{i\in I,n\in\N},(a_{i,n}/a_{j,n})_{(i,j)\in I_{2,c},n\in\N}\}
\subset A_{\ell,\infty}(Y_S)
\end{align*}
On pose alors
$${\rm Symb}_\ell(Y_S)=A_{\ell,\infty}(Y_S)/A_{\ell,\infty}(Y_S)^{\rm triv}.$$ 
\begin{lemm}\label{basic14.1}
On a un isomorphisme naturel
$${\rm Symb}_\ell(Y_S)\cong {\rm Symb}_\ell(Y).$$
\end{lemm}
\begin{proof}
Compactifions $Y$ en recollant des disques $Y_a$ le long des cercles fant\^omes
$Y_{a,s}$, pour $a\in A\moins A_c$ (auquel cas $S(a)$ n'a qu'un \'el\'ement $s=s(a)$),
via $T_{s,a}$.
On note $X$ la courbe ainsi obtenue (si $Y$ est propre, alors $X=Y$, bien s\^ur); 
c'est l'analytifi\'ee d'une courbe alg\'ebrique propre
d'apr\`es GAGA rigide.

Posons $g_{i,j,n}=1$ si $(i,j)\in I_2\moins I_{2,c}$.
Les $g_{i,j,n}$, pour $(i,j)\in I_2$, d\'efinissent un fibr\'e en droites ${\cal F}$
sur $X$ vu comme espace adique, et par le principe GaGa,
cela fournit\footnote{${\cal F}$ admet des sections
globales $e_i$ sur $Y_i$, pour $i\in I'$, et on a $e_a=g_{a,s,n}e_s$ sur~$Y_{a,s}$. Si $U_{a,s}$
est un ouvert de $Y_s$ sur lequel $T_{a,s}$ est holomorphe, il r\'esulte du lemme~\ref{basic32}
que l'on peut factoriser $g_{a,s,n}$ sous la forme $u_{a,s}v_{a,s}$, avec $u_{a,s}\in\O(U_{a,s})^\dual$
et $v_{a,s}\in\O(Y_a)^\dual$. Il s'ensuit que ${\cal F}$ admet une section globale
$e_{a,s}$ sur $V_{a,s}=U_{a,s}\sqcup Y_a$ (recoll\'es le long de $Y_{a,s}$), o\`u
$e_{a,s}=u_{a,s}^{-1}e_s$ sur $U_{a,s}$ et $e_{a,s}=v_{a,s}e_a$ sur $Y_a$).
Les $V_{a,s}$ formant un recouvrement de $X$ par des ouverts rigides, cela d\'efinit
un fibr\'e en droites sur $X$ vu comme vari\'et\'e rigide.}
un fibr\'e en droites sur $X$ vu comme vari\'et\'e rigide
et donc, gr\^ace \`a GAGA rigide, un fibr\'e en droites sur la courbe alg\'ebrique
sous-jacente.

${\cal F}$ \'etant alg\'ebrique sur $X$,
il poss\`ede une section globale m\'eromorphe
sur $X$ tout entier et donc, a fortiori, sur $Y$.
Autrement dit, on peut trouver des $g_{i,n}\in C(Y^{\rm gen}_i)^\dual$, pour $i\in I$, tels que
$\frac{g_{i,n}}{g_{j,n}}=g_{i,j,n}$. 
Les $\frac{u_{i,n}}{g_{i,n}^{\ell^n}}$ se recollent alors pour fabriquer $u_n\in A_{\ell,n}(Y)$,
et on a
$u_{n+1}=u_n\big(\frac{v_{i,n}g_{i,n}}{g_{i,n+1}^\ell}\big)^{p^n}$ sur $Y_i$.
Or $\frac{v_{i,n}g_{i,n}}{g_{i,n+1}^\ell}=\frac{v_{j,n}g_{j,n}}{g_{j,n+1}^\ell}$ sur $Y_{i,j}$
puisque $g_{i,j,n+1}^\ell=g_{i,j,n}\frac{v_{i,n}}{v_{j,n}}$, et donc les $\frac{v_{i,n}g_{i,n}}{g_{i,n+1}^\ell}$ se recollent en $v_n\in C(Y)^\dual$.
Il s'ensuit que $(u_n,v_n)_n\in A_{\ell,\infty}(Y)$.

Les $g_{i,n}$ sont uniques \`a multiplication simultan\'ee pr\`es par $g_n\in C(Y)^\dual$.
Il s'ensuit que l'image de $(u_n)_{n\in\N}$ dans ${\rm Symb}_\ell(Y)$ ne d\'epend
pas du choix des $g_{i,n}$, ce qui fournit une
fl\`eche naturelle ${\rm Symb}_\ell(Y_S)\to {\rm Symb}_\ell(Y)$.

Cette fl\`eche est surjective car $(u_n,v_n)_{n\in\N}\in A_{\ell,\infty}(Y)$
est, bien \'evidemment, l'image de 
$(u_{i,n},v_{i,n})_{i\in I,n\in\N}, (g_{i,j,n})_{(i,j)\in I_{2,c},n\in\N}$, 
o\`u $u_{i,n}$ et $v_{i,n}$ sont les restrictions de $u_n$ et $v_n$ \`a $Y_i$
et $g_{i,j,n}=1$.

Elle est injective car, si l'image est triviale, $u_n=a_n^{\ell^n}$ et $v_n=a_{n+1}^\ell/a_n$,
et donc $u_{i,n}=(a_ng_{i,n})^{\ell^n}$, $v_{i,n}=(a_{n+1}g_{i,n+1})^\ell/(a_ng_{i,n})$ 
et $g_{i,j,n}=
(a_ng_{i,n})/(a_ng_{j,n})$, ce qui montre que notre symbole localis\'e est trivial. 

Ceci permet de conclure.
\end{proof}

\begin{prop}\label{basic14}
On a une suite exacte
$$0\to H^1(\Gamma,\Z_\ell(1))\to {\rm Symb}_\ell(Y_S)\to {\rm Ker}
\big[\prod_{i\in I}{\rm Symb}_\ell(Y^{\rm gen}_i)\to
\hskip-.2cm\prod_{(i,j)\in I_{2,c}}\hskip-.3cm{\rm Symb}_\ell(Y^{\rm gen}_{i,j}) \big]\to 0.$$
\end{prop}
\begin{proof}
Partons de symboles $(u_{i,n},v_{i,n})\in A_{\ell,\infty}(Y_i)$ se recollant sur $Y_{i,j}$.
Il existe donc $g_{i,j,n}$ tel que $u_{i,n}=g_{i,j,n}^{\ell^n}u_{j,n}$ sur $Y_{i,j}$,
et comme cela ne d\'efinit $g_{i,j,n}$ qu'\`a une racine $\ell^n$-i\`eme de l'unit\'e pr\`es,
on peut, par r\'ecurrence sur $n$, imposer que $g_{i,j,n+1}=g_{i,j,n}^\ell\frac{v_{i,n}}{v_{j,n}}$.
On en d\'eduit la surjectivit\'e \`a droite.

Pour l'exactitude au milieu, on part de $A_{\ell,\infty}(Y_S)$ tel que
$u_{i,n}=v_{i,n}=1$, pour tous $i\in I$ et $n\in\N$.  Alors les $(g_{i,j,n})_n$ d\'efinissent
un \'el\'ement $g_{i,j}$ de $\Z_\ell(1)$ et le r\'esultat appartient \`a
$A_{\ell,\infty}(Y_S)^{\rm triv}$ si et seulement si il existe des $a_i\in \Z_\ell(1)$,
pour $i\in I$, tels que $g_{i,j}=a_i-a_j$.  Le passage au quotient
nous donne, gr\^ace \`a la rem.\,\ref{change}, le $H^1(\Gamma,\Z_\ell(1))$ que l'on voulait.
\end{proof}

\subsubsection{Le cas $\ell\neq p$}
Si $\ell\neq p$, la prop.\,\ref{basic14}, combin\'ee
avec le lemme~\ref{basic14.1} et l'isomorphisme ${\rm Symb}_\ell(Y)\cong H^1_{\eet}(Y,\Z_\ell(1))$
du cor.\,\ref{basic12},
fournit une description concr\`ete de $H^1_{\eet}(Y,\Z_\ell(1))$.
Voir aussi \cite[5.2]{Duc} pour une approche un peu diff\'erente. 

\begin{theo}\label{ladique}
Si $\ell\neq p$, alors $H^1_{\eet}(Y,\Z_\ell(1))$
admet une filtration naturelle dont les quotients successifs sont:
$$H^1_{\eet}(Y,\Z_\ell(1))=\big[\xymatrix@C=.4cm{
H^1(\Gamma,\Z_\ell(1))\ar@{-}[r]& \prod_{s\in \Sigma(Y)}H^1_{\eet}(Y_s^{\rm sp},\Z_\ell(1))
\ar@{-}[r]& H^1_c(\Gamma,\Z_\ell)^\dual}\big].$$
\end{theo}
\begin{proof}
La prop.~\ref{basic14} fournit le premier cran de la filtration.
Le second cran est le noyau de l'application r\'esidu dont l'image est
$$
{\rm Ker}\big[\prod_{i\in I}{\rm Ker}\big[\dbar_i:\Z_\ell^{\partial^{\rm ad}Y_i}\to \Z_\ell^{\{i\}}\big]
\to \Z_\ell^{I_{2,c}}\big]$$
(prop.~\ref{short21} combin\'ee
avec l'isomorphisme $H^1_{\eet}(Y^{\rm sp}_s,\Z_\ell(1))\cong T_\ell {\rm Pic}(Y^{\rm sp}_s)=
T_\ell(J(Y^{\rm sp}_s))$ pour $s\in S$, et prop.~\ref{jambe11} pour $i\in A_c$).
On conclut en utilisant les rem.\,\ref{lapl1} et~\ref{change}, et en supprimant les
$s\in S\moins\Sigma(Y)$ puisque, pour un tel $s$, on a $Y_s^{\rm sp}=\piqp$
et donc $H^1_{\eet}(Y_s^{\rm sp},\Z_\ell(1))=0$.
\end{proof}

\begin{rema}\label{PL}
{\rm (i)}  
Comme $H^1_{\eet}(Y,\Q_\ell(1))=\Q_\ell\otimes_{\Z_\ell}H^1_{\eet}(Y,\Z_\ell(1))$,
la description ci-dessus de $H^1_{\eet}(Y,\Z_\ell(1))$ permet, gr\^ace \`a la rem.\,\ref{monod},
de d\'efinir, si $t\in\Z_\ell(1)$, un op\'erateur de monodromie
 $$tN:H^1_{\eet}(Y,\Q_\ell(1))\to H^1_{\eet}(Y,\Q_\ell(1)).$$ 

{\rm (ii)} Le choix de $r\mapsto p^r$ fournit une section de
la projection modulo~$H^1(\Gamma,\Q_\ell(1))$.
En effet, on peut imposer \`a $(u_{i,n},v_{i,n})_n\in A_{\ell,\infty}(Y_i)$ les
conditions suppl\'ementaires suivantes:

$\bullet$ si $i=a\in A_c$, alors $u_{i,n}=T_{a,s_1}^{k_{a,n}}$, o\`u $k_{a,n}$
a une limite $k_a\in\Z_\ell$,

$\bullet$ si $i=s\in S$, alors la restriction de $u_{s,n}$ \`a $Y_{a,s}$, pour $a\in A(s)$,
est de la forme $T_{a,s}^{k_{s,a,n}}u_{s,n}^0$, avec $u_{s,n}^0\in \O(Y_{a,s})^{\dual\dual}$.

(Pour $a\in A_c$, cela suit de la prop.~\ref{jambe11}; pour $s\in S$, cela r\'esulte de ce qu'on peut
multiplier $u_{s,n}$ par $f_{s,n}^{\ell^n}$, avec $f_{s,n}\in C(Y)^\dual$, ce qui permet
de modifier \`a loisir le coefficient dominant en un nombre fini de points.) 

Comme $x\in\O(Y_{a,s})^{\dual\dual}$ a une racine $\ell^n$-i\`eme
naturelle, \`a savoir $\sum_{k\geq 0}\binom{1/\ell^n}{k}(x-1)^k$,
il en est de m\^eme de
$\frac{u_{a,n}}{u_{s,n}}$,
si $a\in A_c$ et $s\in S(a)=\{s_1,s_2\}$: 

-- Si $s=s_1$, $\frac{u_{a,n}}{u_{s,n}}=T_{a,s}^{k_{a,n}-k_{a,s,n}}(u_{s,n}^0)^{-1}$
et
$k_{a,n}-k_{a,s,n}\in\ell^n\Z$.

-- Si $s=s_2$, $\frac{u_{a,n}}{u_{s,n}}=p^{k_{a,n}\mu(a)}
T_{a,s}^{-(k_{a,n}+k_{a,s,n})}(u_{s,n}^0)^{-1}$ et
$k_{a,n}+k_{a,s,n}\in\ell^n\Z$, 
et $p^{k_{a,n}\mu(a)}$ a comme racine $\ell^n$-i\`eme $p^{\ell^{-n}k_{a,n}\mu(a)}$.

Si on note $g_{a,s,n}$ la racine $\ell^n$-i\`eme naturelle de $\frac{u_{a,n}}{u_{s,n}}$,
on fabrique un scindage $s_{r\mapsto p^r}$ en envoyant la classe de
$(u_{i,n},v_{i,n})_{i,n}$ sur celle de $((u_{i,n},v_{i,n})_{i,n},(g_{i,j,n})_{i,j,n})$.

{\rm (iii)} Si on change $r\mapsto p^r$ en $r\mapsto \zeta(r)p^r$,
o\`u $\zeta$ est un morphisme de groupes de $\Q/\Z$ dans le groupe des racines de l'unit\'e,
cela multiplie $g_{a,s_2,n}$ par $\zeta(\ell^{-n} k_{a,n}\mu(a))$ sans modifier
$g_{a,s_1,n}$.  On en d\'eduit la formule {\og de Picard-Lefschetz\fg}
$$s_{r\mapsto \zeta(r)p^r}=s_{r\mapsto p^r}+t_\zeta N,$$
o\`u $t_\zeta=(\zeta(\ell^{-n}))_{n\in\N}\in\Z_\ell(1)$.
\end{rema}

\begin{rema}\label{PL2}
Si $Y$ est d\'efini sur $\O_K$, on peut utiliser ce qui pr\'ec\`ede pour d\'ecrire
l'action du sous-groupe d'inertie $I_K$ de $ G_K$ sur $H^1_{\eet}(Y,\Q_\ell(1))$, mais il vaut mieux
prendre les jambes $Y_a$ de la forme 
${\rm Spf}(\O_K[[Y_{a,s_1},Y_{a,s_2}]]/(Y_{a,s_1}Y_{a,s_2}-\pi^{e\mu(a)}))$,
o\`u $\pi$ est une uniformisante de $K$, $e$ l'indice de ramification absolu de $K$
(et $e\mu(a)$ est entier), et choisir un morphisme $r\mapsto \pi^r$ (de $\Z[\frac{1}{\ell}]$
dans $C^\dual$) plut\^ot
que $r\mapsto p^r$. 

Si $\sigma\in I_K$, il existe $\zeta_\ell:\Z[\frac{1}{\ell}]\to\mu_{\ell^n}$ tel
que l'on ait $\sigma(\pi^r)=\zeta_\sigma(r)\pi^r$,
et alors $t_\sigma=(\zeta_\sigma(\ell^{-n}))_n\in \Z_\ell(1)$ et $\sigma\mapsto t_\sigma$
est un $1$-cocycle sur $I_K$ \`a valeurs dans $\Z_\ell(1)$,
et on a $\sigma(c)=c+t_\sigma N(c)$ si $c\in H^1_{\eet}(Y,\Q_\ell(1))$.
Voir~\cite{Il} pour un point de vue diff\'erent.
\end{rema}

\Subsection{Cohomologie de de Rham et variantes}\label{BAS7}
On peut 
calculer les diverses cohomologies de $Y_S$ \`a la \v{C}ech, en utilisant le recouvrement
par les $Y_i$, pour $i\in I$. Ce calcul est simplifi\'e par
le fait que
les seules intersections non vides sont les $Y_{i,j}$, pour $(i,j)\in I_{2,c}$. 
Par exemple (prop.~\ref{tetrapil2}), 
les groupes $H^i_{\rm dR}(Y_S)$ de cohomologie de de Rham (logarithmique) de $Y_S$
sont les groupes
de cohomologie du complexe:
$$C_{\rm dR}^\bullet(Y_S):=
\big[\prod_{i\in I}\Omega^\bullet(Y_i)\longrightarrow \prod_{(i,j)\in I_{2,c}}\Omega^\bullet(Y_{i,j})\big].$$
Le complexe ci-dessus est le complexe naturel pour calculer la cohomologie
de de Rham mais, \`a $p^{2N(S)}$ pr\`es, on peut aussi utiliser
$$\overline C_{\rm dR}^\bullet(Y_S):=
\big[\prod_{i\in I}\Omega^\bullet(Y_i)\longrightarrow \prod_{(i,j)\in I_{2,c}}
\overline \Omega^\bullet(Y_{i,j})\big],$$
o\`u $\overline\Omega^\bullet(Y_{i,j})$ est le quotient de $\Omega^\bullet(Y_{i,j})$
d\'efini dans la discussion pr\'ec\'edant le lemme~\ref{jambe1} ci-dessous.

\subsubsection{Cohomologie de de Rham}\label{BAS9}

\smallskip
Soit $$ S_{\rm int}=S\moins \partial Y.$$

Si $s\in S$, on 
construit une 
compactification partielle $Y_s^\lozenge$
de $Y_s$ en recollant, via les $T_{a,s}$,
les disques ouverts $\{v_p(T_{a,s})>0\}$ le long des cercles fant\^omes
$Y_{a,s}$, pour $a\in A_c(s)$.
\begin{rema}
Si $s\in S_{\rm int}$, alors $Y_s^\lozenge$ est propre; si $s\in \partial Y $,
c'est un short.
\end{rema}

On note
$H^1_{\rm dR}(Y_s^\lozenge)_0$
l'ensemble des $\omega\in H^1_{\rm dR}(Y_s^\lozenge)$
tels que ${\rm Res}_a(\omega)=0$ pour tout $a\in A(s)\moins A_c(s)$.
(Si $s\in S_{\rm int}$, alors $H^1_{\rm dR}(Y_s^\lozenge)_0=H^1_{\rm dR}(Y_s^\lozenge)$,
mais si $s\in \partial Y $ et $|A(s)\moins A_c(s)|=r$,
alors $H^1_{\rm dR}(Y_s^\lozenge)_0$ est de corang~$r-1$ dans $H^1_{\rm dR}(Y_s^\lozenge)$.)

\begin{rema}
$Y_s^\lozenge$ d\'epend du choix des $T_{a,s}$, mais
$H^1_{\rm dR}(Y_s^\lozenge)$ (et donc aussi $H^1_{\rm dR}(Y_s^\lozenge)_0$)
n'en d\'epend pas
car il s'identifie \`a la cohomologie convergente de $Y^{\rm sp}_s$
si $s\in S_{\rm int}$, et \`a celle de l'ouvert compl\'ementaire
des points de multiplicit\'e $0^+$ si $s\in \partial Y $.
\end{rema}

\begin{prop}\label{basic16}
$H^1_{\rm dR}(Y_S)$ a une filtration dont les quotients
successifs sont, \`a $p^{N_1(S)}$ pr\`es,
$$H^1_{\rm dR}(Y_S)=\big[\xymatrix@C=.5cm{
H^1(\Gamma,\O_C)\ar@{-}[r]&
\prod_{s\in S} H^1_{\rm dR}(Y_s^\lozenge)_0\ar@{-}[r]&
H^1_c(\Gamma,\O_C)^\dual}\big].$$
\end{prop}
\begin{proof}
On a une suite exacte
$$0\to H^0_{\rm dR}(Y_S)\to
\prod_{i\in I}H^0_{\rm dR}(Y_i)\to
\hskip-.3cm \prod_{(i,j)\in I_{2,c}}\hskip-.3cm H^0_{\rm dR}(Y_{i,j})
\to H^1_{\rm dR}(Y_S)\to
\prod_{i\in I}H^1_{\rm dR}(Y_i)\to
\hskip-.3cm \prod_{(i,j)\in I_{2,c}}\hskip-.3cm H^1_{\rm dR}(Y_{i,j}).$$
Le sous-groupe $H^1(\Gamma,\O_C)$ correspond
au conoyau
de $\prod_{i\in I}H^0_{\rm dR}(Y_i)\to
\prod_{(i,j)\in I_{2,c}}H^0_{\rm dR}(Y_{i,j})$.

Le quotient $H^1_c(\Gamma,\O_C)^\dual$ est fourni par
l'application r\'esidu comme dans la preuve du th.\,\ref{ladique}.

Si $a\in A_c$, on a 
$$\Omega^1(Y_a)=\Omega^1(B_{a,s_1})\oplus
\Omega^1(B_{a,s_2})\oplus\O_C\tfrac{dT_{a,s_1}}{T_{a,s_1}},$$ 
o\`u $B_{a,s_1}$
est la boule ouverte avec $\O(B_{a,s_i})=\O_C[[T_{a,s_i}]]$, si $i=1,2$.
Cette d\'ecomposition induit
une d\'ecomposition
$$H_{\rm dR}^1(Y_a)=H_{\rm dR}^1(B_{a,s_1})\oplus
H_{\rm dR}^1(B_{a,s_2})\oplus\O_C\tfrac{dT_{a,s_1}}{T_{a,s_1}}.$$
Maintenant, l'image de $H_{\rm dR}^1(B_{a,s_i})$ dans
$H^1_{\rm dR}(Y_{a,s_{3-i}})$ est tu\'ee par $p^{N_1(S)}$ (cf.~lemme~\ref{jambe1}).
Il s'ensuit que le complexe
$\prod_{i\in I}H^1_{\rm dR}(Y_i)_0\to \prod_{(i,j)\in I_{2,c}}H^1_{\rm dR}(Y_{i,j})_0$
se $p^{N_1(S)}$-d\'ecompose en
$$\prod_{s\in S}\Big[H^1_{\rm dR}(Y_s)_0\oplus\prod_{a\in A_c(s)}\hskip-.2cm H^1_{\rm dR}(B_{a,s})\to
\prod_{a\in A_c(s)}\hskip-.2cm H^1_{\rm dR}(Y_{a,s})_0\Big].$$
Or le noyau de
$H^1_{\rm dR}(Y_s)_0\oplus\prod_{a\in A_c(s)}H^1_{\rm dR}(B_{a,s})\to
\prod_{a\in A_c(s)}H^1_{\rm dR}(Y_{a,s})_0$ n'est autre que
$H^1_{\rm dR}(Y_s^\lozenge)_0$ comme on le voit en consid\'erant le recouvrement
de $Y_s^\lozenge$ form\'e de $Y_s$ et des $B_{a,s}$, pour $a\in A_c(s)$.
(On aurait aussi pu utiliser la prop.\,\ref{dege1} ci-dessous.)

On en d\'eduit le r\'esultat.
\end{proof}
\begin{rema}\label{basic16.1}
On a $H^1_{\rm dR}(Y)\stackrel{\sim}{\leftarrow} C\otimes_{\O_C}H^1_{\rm dR}(Y_S)$ pour tout $S$; en particulier
$C\otimes_{\O_C}H^1_{\rm dR}(Y_S)$ ne d\'epend pas de $S$.  Cela peut se voir directement:
si on raffine $S$ en $S'$, la fl\`eche naturelle $H^1_{\rm dR}(Y_S)\to H^1_{\rm dR}(Y_{S'})$
induit un isomorphisme 
$C\otimes_{\O_C} H^1_{\rm dR}(Y_S)\overset{\sim}{\to}C\otimes_{\O_C} H^1_{\rm dR}(Y_{S'})$ car:

$\bullet$ $\Gamma^{\rm ad}(S')$ se r\'etracte sur $\Gamma^{\rm ad}(S)$ et donc a m\^eme cohomologie.

$\bullet$ Les $Y_s^\lozenge$, pour $s\in S'\moins S$, sont des $\piqp$ et donc
$H^1_{\rm dR}(Y_s^\lozenge)_0=0$ si $s\in S'\moins S$.
\end{rema}
Apr\`es avoir \'elimin\'e les $\piqp$ superflus,
on obtient le corollaire ci-dessous.
\begin{coro}\label{basic16.2}
$H^1_{\rm dR}(Y)$ a une filtration dont les quotients
successifs sont
$$H^1_{\rm dR}(Y)=\big[\xymatrix@C=.5cm{
H^1(\Gamma,C)\ar@{-}[r]&
\prod_{s\in \Sigma(Y)} H^1_{\rm dR}(Y_s^\lozenge)_0\ar@{-}[r]&
H^1_c(\Gamma,C)^\dual}\big].$$
\end{coro}

\begin{rema}
Si $Y$ est propre, tous les $Y_s^\lozenge$ sont propres et ont une cohomologie de dimension finie,
ce qui est en accord avec le fait que $H^1_{\rm dR}(Y)$ est de dimension finie.
Si $Y$ est un affino\"ide, alors $H^1_{\rm dR}(Y)$ est de dimension infinie;
c'est d\^u au fait que $\partial Y $ est non vide et que, si $s\in \partial Y $, alors
$Y_s^\lozenge$ est un $\O_C$-short et donc sa cohomologie est de dimension infinie.
\end{rema}

\subsubsection{D\'eg\'enerescence de courbes}
On note $Y_S^\infty$ la courbe obtenue en sp\'ecialisant en $(0)_{a\in A_c}$ la famille
de courbes du \no\ref{familles} dans le cas $R=\O_C[[T_a,\ a\in A_c]]$,
ce qui remplace $Y_a$
par la r\'eunion $Y^\infty_a$ de deux disques attach\'es en un point $P_a$.
Alors $Y_S^\infty$ est aussi la courbe obtenue en recollant les $Y_s^{\lozenge}$ en les $P_a$:
i.e.~on identifie le point $T_{a,s_1}=0$ de $Y_{s_1}^\lozenge$
avec le point $T_{a,s_2}=0$ de $Y_{s_2}^\lozenge$, pour tout $a\in A_c$.
\begin{prop}\label{dege1}
Les complexes $\overline C_{\rm dR}^\bullet(Y_S^\infty)$ et
$\overline C_{\rm dR}^\bullet(Y_S)$ sont naturellement $p^{N_1(S)}$-quasi-isomorphes, et la cohomologie
de de Rham {\rm(}logarithmique{\rm)} de $Y$ est donc naturellement isomorphe \`a celle de $Y^\infty$.
\end{prop}
\begin{proof}
Les groupes intervenant dans les complexes $C_{\rm dR}^\bullet(Y_S^\infty)$ 
et $C_{\rm dR}^\bullet(Y_S)$ sont naturellement isomorphes: $\O(Z^\infty)=\O(Z)$ et
$\Omega^1(Z^\infty)=\Omega^1(Z)$,
si $Z=Y_s,Y_{a,s}$ et, si $Z=Y_a$, 
$$\O(Y^\infty_a)=\O_{C}[[T_{a,s_1}T_{a,s_2}]]/(T_{a,s_1}T_{a,s_2}),\quad
\O(Y_a)=\O_C[[T_{a,s_1}T_{a,s_2}]]/(T_{a,s_1}T_{a,s_2}-p^{\mu(a)}),$$
et on dispose d'un isomorphisme 
$\O_{C}$-lin\'eaire\footnote{ 
{\it Ce n'est pas un morphisme d'anneaux} puisque $T_{a,s_1}T_{a,s_2}=0$
dans $\O(Y^\infty_a)$ et $T_{a,s_1}T_{a,s_2}=\tilde p^{\mu(a)}$
dans $\O(Y_a)$.}
qui envoie $T_{a,s_i}^k$ sur $T_{a,s_i}^k$, si $k\geq 0$;
l'isomorphisme correspondant $\Omega^1(Y^\infty_a)\overset{\sim}{\to}\Omega^1(Y_a)$
envoie $T_{a,s_i}^k\frac{dT_{a,s_i}}{T_{a,s_i}}$ sur $T_{a,s_i}^k\frac{dT_{a,s_i}}{T_{a,s_i}}$ 
si $k\geq 0$ (pour $k=0$ les deux valeurs obtenues co\"{\i}ncident puisque
$\frac{dT_{a,s_1}}{T_{a,s_1}}+\frac{dT_{a,s_2}}{T_{a,s_2}}=0$).

Le lemme~\ref{jambe1} permet de montrer que ceci d\'efinit un $p^{N_1(S)}$-quasi-isomorphisme
$\overline C_{\rm dR}^\bullet(Y_S^\infty)\to \overline C_{\rm dR}^\bullet(Y_S)$,
ce qui permet de conclure.
\end{proof}

\subsubsection{Cohomologie de de Rham s\'epar\'ee}\label{BAS10}
Le groupe $H^1_{\rm dR}(Y_S)$ peut avoir de la torsion d'exposant non born\'e
(c'est effectivement le cas si $Y$ est un affino\"{\i}de); on note $H^1_{\rm dR}(Y_S)^{\rm sep}$
son quotient par l'adh\'erence $H^1_{\rm dR}(Y_S)_{\rm tors}$
du sous-groupe de torsion: on a $H^1_{\rm dR}(Y)=C\otimes_{\O_C}
H^1_{\rm dR}(Y_S)$ et $C\otimes_{\O_C}H^1_{\rm dR}(Y_S)^{\rm sep}$ est le s\'epar\'e de
$H^1_{\rm dR}(Y)$
(i.e.~son quotient par l'adh\'erence de $0$).

\begin{exem}\label{constr21}
{\rm (i)} Si $Y$ est propre, $H^1_{\rm dR}(Y_S)$ est sans torsion
et $H^1_{\rm dR}(Y_S)^{\rm sep}=H^1_{\rm dR}(Y_S)$.

{\rm (ii)} Si $Y$ est la fibre g\'en\'erique d'un 
$\O_C$-short $Y_S$, alors $H^1_{\rm dR}(Y_S)^{\rm sep}$ est le $\O_C$-module
$M_1^\sharp$ de la prop.~\ref{basic9}; il est de rang fini, mais la torsion
de $H^1_{\rm dR}(Y_S)$ est d'exposant non born\'e.
\end{exem}

On note $\Gamma_{\rm int}$ le sous-graphe de $\Gamma$ dont les sommets sont $S_{\rm int}$
et les ar\^etes $A_{\rm int}$, i.e. l'ensemble des $a\in A$ telles que 
$S(a)\subset S_{\rm int}$; alors $\Gamma_{\rm int}$ est un graphe compact.

\begin{theo}\label{Aff1}
{\rm (i)}
Le groupe
$H^1_{\rm dR}(Y_S)^{\rm sep}$ est de rang fini
et admet une filtration naturelle dont les quotients successifs, \`a $p^{N_1(S)}$ pr\`es,
sont:
$$H^1_{\rm dR}(Y_S)^{\rm sep}=\big[\xymatrix@C=.3cm{
H^1(\Gamma_{\rm int},\O_C)\ar@{-}[r]&\prod_{s\in S}H^1_{\rm dR}(Y_s^\lozenge)_0^{\rm sep}
\ar@{-}[r]& H^1_c(\Gamma,\O_C)^\dual}\big].$$

{\rm (ii)}
Le groupe
$H^1_{\rm dR}(Y)^{\rm sep}$ est de dimension finie
et admet une filtration naturelle dont les quotients successifs
sont:
$$H^1_{\rm dR}(Y)^{\rm sep}=\big[\xymatrix@C=.3cm{
H^1(\Gamma_{\rm int},C)\ar@{-}[r]&\prod_{s\in \Sigma(Y)}(C\otimes_{\O_C}H^1_{\rm dR}(Y_s^\lozenge)_0^{\rm sep})
\ar@{-}[r]& H^1_c(\Gamma,C)^\dual}\big].$$ 
\end{theo}
\begin{proof}
Le (ii) se d\'eduit du (i) en inversant $p$ et en \'eliminant les $\piqp$ superflus
(cf.~rem.\,\ref{basic16.1}).

Prouvons le (i).
Comme le quotient $H^1_c(\Gamma,\O_C)^\dual$
de $H^1_{\rm dR}(Y_S)$ est s\'epar\'e, on a une suite
exacte
$$0\to H^1_{\rm dR}(Y)_0^{\rm sep}\to H^1_{\rm dR}(Y_S)^{\rm sep}\to
H^1_c(\Gamma,\O_C)^\dual\to 0.$$
Maintenant, si on note $M$ l'intersection
de $H^1(\Gamma,\O_C)$ et de l'adh\'erence de
$(H^1_{\rm dR}(Y_S)_0)^{\rm tors}$, 
une petite chasse au diagramme fournit la suite exacte:
$$0\to H^1(\Gamma,\O_C)/M\to H^1_{\rm dR}(Y_S)_0^{\rm sep}
\to \prod_s H^1_{\rm dR}(Y_s^\lozenge)_0^{\rm sep}\to 0.$$
Comme les $H^1_{\rm dR}(Y_s^\lozenge)_0^{\rm sep}$ sont de rang fini
(c'est clair si $Y_s^\lozenge$ est propre, et si $Y_s^\lozenge$ est un short,
cela r\'esulte de la prop.\,\ref{basic9.11}),
on en d\'eduit que $H^1_{\rm dR}(Y_S)^{\rm sep}$ est
de rang fini.

Pour terminer la preuve du th\'eor\`eme, il ne reste donc plus qu'\`a d\'eterminer $M$,
ce qui fait l'objet du lemme~\ref{basic18} ci-dessous.
\end{proof}

\begin{lemm}\label{basic18}
On a $${\rm Ker}\big[H^1(\Gamma,\O_C)\to H^1_{\rm dR}(Y_S)^{\rm sep}\big]
={\rm Ker}\big[H^1(\Gamma,\O_C)\to H^1(\Gamma_{\rm int},\O_C)\big].$$
\end{lemm}
\begin{proof}
On note $A_{\rm ext}$ l'ensemble des $a\in A$ telles
que $S(a)\subset \partial Y $ et
$A_{\rm bord}=A\moins(A_{\rm ext}\cup A_{\rm int})$ l'ensemble
des $a$ dont une extr\'emit\'es appartient \`a $\partial Y $ et l'autre
\`a $A_{\rm int}$ (on note $s_1=s_1(a)$ celle appartenant
\`a $A_{\rm int}$ et $s_2=s_2(a)$ celle appartenant
\`a $\partial Y $).

On note $Y_S'$ la courbe obtenue en rempla\c{c}ant $p^{\mu(a)}$ par $0$
si $a\in A_{\rm bord}$ dans le patron de $Y_S$.
Cela remplace les $Y_a$, pour $a\in A_{\rm bord}$,
par la r\'eunion $Y'_a$ de deux disques $B_{a,s_1}, B_{a,s_2}$ attach\'es en un point $P_a$.
On note
$Y'_{S,{\rm int}}$ la r\'eunion des $Y_s$ pour $s\in S_{\rm int}$, des $Y_a$ pour
$a\in A_{\rm int}$, et des $B_{a,s_1}$ pour $a\in A_{\rm bord}$;
c'est un mod\`ele semi-stable d'une courbe propre.

Les m\^emes arguments que pour la preuve de la prop.~\ref{dege1} montrent que les
complexes $C_{\rm dR}^\bullet(Y_S)_0$ et $C_{\rm dR}^\bullet(Y_S')_0$
sont $p^{N_1(S)}$-quasi-isomorphes,
et qu'on peut calculer $H^1_{\rm dR}(Y_S)_0$ 
(\`a $p^{N_1(S)}$ pr\`es) en utilisant $C_{\rm dR}^\bullet(Y_S')_0$.

La restriction fournit une fl\`eche naturelle
$H^1_{\rm dR}(Y_S')_0\to H^1_{\rm dR}(Y'_{S,{\rm int}})_0$.
Comme $Y'_{S,{\rm int}}$ est propre, $H^1_{\rm dR}(Y'_{S,{\rm int}})_0$ est
s\'epar\'e et la fl\`eche ci-dessus 
se factorise \`a travers $H^1_{\rm dR}(Y_S')_0^{\rm sep}$.
De plus, cette fl\`eche induit, par restriction,
la fl\`eche naturelle
$H^1(\Gamma,\O_C)\to H^1(\Gamma_{\rm int},\O_C)$, et comme cette derni\`ere fl\`eche
est surjective, 
on en d\'eduit que le membre de gauche dans l'\'enonc\'e du lemme
est inclus dans le membre de droite.

Pour prouver l'\'egalit\'e,
il s'agit donc de prouver que ${\rm Ker}\big[H^1(\Gamma,\O_C)\to H^1(\Gamma_{\rm int},\O_C)\big]$
est dans l'adh\'erence de $0$ dans $H^1_{\rm dR}(Y_S)$.  Or ce noyau est engendr\'e
par des classes de la forme $e_{a,s}=((\omega_i)_i,(g_{i,j})_{i,j})$, 
o\`u $S(a)=\{s,s'\}$ et $s\in \partial Y $, avec $\omega_i=0$ pour tout~$i$,
$g_{a,s'}=1$ et $g_{i,j}=0$ si $(i,j)\neq (a,s')$.

Comme $k_C\otimes Y_s^\lozenge$ est affine puisque $s\in \partial Y $,
on peut trouver $\phi\in\O(Y_s^\lozenge)$
telle que $\phi(P_a)=1$ et $\phi(P_{b})=0$ si $b\in A(s)\moins\{a\}$.

Soit $\alpha\in 1+p\O_C$.  Si $n\in\N$, alors $\phi_n=p^n\log(1+(\alpha^{1/p^n}-1)\phi)\in
\O(Y_s^\lozenge)$.  On peut donc consid\'erer le bord $c_n$
de $(g_{i,n})_{i\in I}$, o\`u
$g_{i,n}=0$ si $i\notin\{s\}\sqcup A(s)$
et $g_{i,n}=\phi_n$ si $i\in\{s\}\sqcup A(s)$.
Le lemme suivant montre que $(\log\alpha)e_{a,s}$ est dans l'adh\'erence de $0$
dans $H^1_{\rm dR}(Y_S)$, ce qui permet de conclure.
\end{proof}
\begin{lemm}\label{basic18.1}
Il existe $N\in\N$ tel que, si $n\geq N$,
alors $c_n-(\log\alpha)e_{a,s}$ est divisible par $p^{n-N}$.
\end{lemm}
\begin{proof}
$dg_{i,n}$ est divisible par $p^n$ et,
sur $Y_{i,j}$, on a $g_{i,n}-g_{j,n}=0$ sauf si $(i,j)=(b,t)$, avec $b\in A(s)$
et $t$ est l'autre extr\'emit\'e de $b$, o\`u $g_{t,n}=0$, mais $g_{b,n}=\phi_n\neq 0$: 

\quad $\bullet$ Si $b=a$, comme $\phi(P_a)=1$, on a 
$v_{Y_{b,t}}(\frac{1+(\alpha^{1/p^n}-1)\phi}{\alpha^{1/p^n}}-1)\geq \mu(a)$, et il existe
$N(a)$ (ne d\'ependant que $\mu(a)$) tel que
$v_{Y_{b,t}}(\phi_n-\log\alpha)\geq n-N(a)$.

\quad $\bullet$ Si $b\in A(s)\moins\{ a\}$, comme $\phi(P_b)=0$, 
on a $v_{Y_{b,t}}((1+(\alpha^{1/p^n}-1)\phi)-1)\geq \mu(b)$, et donc
$v_{Y_{b,t}}(\phi_n)\geq n-N(b)$.

Ceci permet de conclure.
\end{proof}

\Subsubsection{Cohomologie cristalline}\label{BAS11}
$\bullet$ Si $s\in S$, on choisit un mod\`ele $\breve Y_s$ de $Y_s$ sur $\O_{\breve C}$
et un frobenius $\varphi$ sur $\O(\breve Y_s)$.
On pose $\O(\widetilde Y_s): =\acris\wotimes_{\O_{\breve C}}\O(\breve Y_s)$;
on a bien \'evidemment 
$$\O(\breve Y_s)=\O_{\breve C}\wotimes_{\acris}\O(\widetilde Y_s)
\quad{\rm et}\quad
\O(Y_s)=\O_{C}\wotimes_{\acris}\O(\widetilde Y_s).$$
Soit $r(S)=p^{-N(S)}$, de telle sorte que
$T_{a,s}\in \O(\breve Y_s)+p^{r(S)}\O(Y_s)$, pour tout $(a,s)\in I_{2,c}$.
On peut donc \'ecrire $T_{a,s}$, sur $Y_s$, sous la forme
$$T_{a,s}=T_{a,s,0}+p^{r(S)}T'_{a,s},\quad{\text{avec $T_{a,s,0}\in\O(\breve Y_s)$
et $T'_{a,s}\in \O(Y_s)$.}}$$
On choisit un rel\`evement $\tilde T'_{a,s}$ de $T'_{a,s}$ dans $\O(\widetilde Y_s)$,
et on pose
$$\tilde T_{a,s}=T_{a,s,0}+\tilde p^{r(S)}\tilde T'_{a,s}\in \O(\widetilde Y_s).$$

$\bullet$ Si $a\in A_c$, 
on pose 
$$\O(\widetilde Y_a):=\acris[[T_{a,s_1}, T_{a,s_2}]]/(T_{a,s_1}T_{a,s_2}-\tilde p^{\mu(a)}).$$
On a 
$$
\O(Y_a)=\O_{C}\wotimes_{\acris}\O(\widetilde Y_a).$$

$\bullet$ Si $s\in S(a)$, 
on pose $\O(\widetilde Y_{a,s}): =\acris[[T_{a,s},T_{a,s}^{-1}\rangle$.
Si $h\geq 0$,
on note $\O(\widetilde Y_{a,s})^{{\rm PD}_h}$ le compl\'et\'e $p$-adique
de l'enveloppe \`a puissances divis\'ees logarithmiques partielles, de niveau $h$,
du noyau de $\O(\widetilde Y_{a})\wotimes_{\acris}\O(\widetilde Y_s)\to
\O(Y_{a,s})$: 
si $U_{a,s}=\frac{T_{a,s}\otimes1}{1\otimes \tilde T_{a,s}}$, on a
un isomorphisme
$$\O(\widetilde Y_{a,s})^{{\rm PD}_h}
\cong\O(\widetilde Y_{a,s})[\tfrac{(U_{a,s}-1)^k}{[k/p^h]!},\,k\in\N]^{p-\wedge}.$$
Si $h\geq N(S)$, on peut remplacer $U_{a,s}$ par $U'_{a,s}=\frac{T_{a,s}\otimes1}{1\otimes T_{a,s,0}}$
dans l'isomorphisme ci-dessus car $\tilde p^{r(S)}$ 
a des puissances divis\'ees partielles de niveau $h$.

\vskip.1cm
$\bullet$ Si $Z=Y_a,Y_s,Y_{a,s}$, on note 
$\Omega^j(\widetilde Z)$  le $\acris$-module $\Omega^j_{\O(\widetilde Z)/\acris}$,
et $\Omega^j(\widetilde Y_{a,s})^{\rm PD}$  
le $\acris$-module $\Omega^j_{\O(\widetilde Y_{a,s})^{\rm PD}/\acris}$.
On a bien s\^ur 
$\Omega^j(\widetilde Z)=0$ si $j\geq 2$ (resp.~$j\geq 3$) et $Z=Y_a,Y_s$
(resp.~$Z=Y_{a,s}$).
Le lemme de Poincar\'e nous donne le r\'esultat suivant:
\begin{lemm}\label{basic19}
Si $h\geq 0$, le complexe de de Rham 
$\big [\O(\widetilde Y_{a,s})\to \Omega^1(\widetilde Y_{a,s})\big]$
est $p^h$-quasi-isomorphe au complexe
$\big[\O(\widetilde Y_{a,s})^{{\rm PD}_h}\to 
\Omega^1(\widetilde Y_{a,s})^{{\rm PD}_h}_{d=0}\big]$.
\end{lemm}

Notons que l'on peut munir $\O(\widetilde Y_s)$ et $\O(\widetilde Y_a)$
d'endomorphismes de Frobenius ($\O(\widetilde Y_s)$ par extension des
scalaires \`a partir de $\O(\breve Y_s)$ qui est lisse sur $\O_{\breve C}$,
et $\O(\widetilde Y_a)$, en envoyant $T_i$ sur $T_i^p$); cela munit
aussi $\O(\widetilde Y_{a,s})^{{\rm PD}_h}$ du frobenius produit tensoriel.

On d\'efinit les $H^i_{\rm dR}(\widetilde Y_S)$, pour $i\leq 2$, comme les groupes
de cohomologie du complexe\footnote{On \'ecrit ${\rm PD}$ pour ${\rm PD}_0$.} 
$$C_{\rm dR}^\bullet(\widetilde Y_S):=
\big[\xymatrix@C=.6cm{
\prod\limits_{i\in I}\O(\widetilde Y_i)
\ar[r]& \prod\limits_{i\in I}\Omega^1(\widetilde Y_i)
\oplus\prod\limits_{(i,j)\in I_{2,c}}\O(\widetilde Y_{i,j})^{\rm PD}\ar[r]&
\prod\limits_{(i,j)\in I_{2,c}}\Omega^1(\widetilde Y_{i,j})^{\rm PD}_{d=0}}\big],$$
la premi\`ere fl\`eche \'etant 
$((f_i)_i)\mapsto ((df_i)_i,(f_i\otimes 1-1\otimes f_j)_{i,j})$,
et la seconde \'etant
$((\omega_i)_i,(f_{i,j})_{i,j})\mapsto (\omega_i\otimes 1-1\otimes\omega_j-df_{i,j})_{i,j}$.
Notons que $\varphi$ commute aux fl\`eches ci-dessus et donc
les $H^i_{\rm dR}(\widetilde Y_S)$ sont munis d'une action de $\varphi$.
\begin{rema}
(i) 
Les techniques cristallines \`a base de lemme de Poincar\'e permettent de montrer
que les $H^i_{\rm dR}(\widetilde Y_S)$, ainsi que l'action de $\varphi$,
ne d\'ependent pas des choix faits: nous laissons au lecteur le soin de v\'erifier que
les $H^i_{\rm dR}(\widetilde Y_S)$ sont les groupes
de cohomologie cristalline logarithmique absolue (i.e.~sur $\Z_p$) de
$(\O_C/p)\otimes Y_S$.

(ii) Si on remplace ${\rm PD}$ par ${\rm PD}_h$ dans la d\'efinition ci-dessus, on 
obtient un complexe $p^h$-quasi-isomorphe d'apr\`es le lemme~\ref{basic19}.

(iii) On pourrait remplacer $\acris$ par $\ainf$ dans les d\'efinitions des $\O(\widetilde Z)$
et utiliser le cas $R=\ainf$ du \no\ref{familles} pour produire un sch\'ema formel
$\widetilde Y_S$ ayant $Y_S$ et $\breve Y_S$ comme sp\'ecialisations.  Malheureusement,
on ne peut pas munir $\widetilde Y_S$ d'un frobenius global, et son existence
n'aide pas pour munir la cohomologie d'une action de $\varphi$.
\end{rema} 
Le complexe
$C_{\rm dR}^\bullet(\widetilde Y_S)$ est le cylindre
$$\big[\prod_{i\in I}\Omega^\bullet(\widetilde Y_i)\to
\prod_{(i,j)\in I_{2,c}}\Omega^{\tau\leq 1}(\widetilde Y_{i,j})\big].$$
Pour faire les calculs, on peut remplacer $C_{\rm dR}^\bullet(\widetilde Y_S)$
par le complexe quasi-isomorphe
$$\overline C_{\rm dR}^\bullet(\widetilde Y_S):=
\big[\prod_{i\in I}\Omega^\bullet(\widetilde Y_i)\to
\prod_{(i,j)\in I_{2,c}}\overline\Omega^{\tau\leq 1}(\widetilde Y_{i,j})\big],$$
o\`u $\overline\Omega^{\tau\leq 1}(\widetilde Y_{i,j})$ est le complexe
d\'efini pour le lemme~\ref{jambe2}.

\subsubsection{Cohomologie de Hyodo-Kato}\label{BAS12}
On d\'efinit les $H^i_{\rm dR}(\breve Y_S)$, pour $i\leq 2$, comme les groupes
de cohomologie du complexe $C_{\rm dR}^\bullet(\breve Y_S)=
\O_{\breve C}\otimes_{\acris}C_{\rm dR}^{\bullet}(\widetilde Y_S)$.
Comme $\theta_0:\acris\to\O_{\breve C}$ commute \`a $\varphi$,
le complexe $C_{\rm dR}^\bullet(\breve Y_S)$ est muni d'une action
de $\varphi$ et donc les $H^i_{\rm dR}(\breve Y_S)$ aussi.

On cherche \`a exprimer les $H^i_{\rm dR}(\widetilde Y_S)$
en termes des $H^i_{\rm dR}(\breve Y_S)$. Pour ce faire, on remplace ${\rm PD}$
par ${\rm PD}_h$ dans la d\'efinition de $C_{\rm dR}^{\bullet}(\widetilde Y_S)$,
ce qui produit un complexe $p^{N(S)}$-isomorphe, si $h= N(S)$.
 
$\bullet$ Si $i\in I$ ou si $(i,j)\in I_{2,c}$, on pose
$$\O(\breve Y_i)=\O_{\breve C}\wotimes_{\acris}\O(\widetilde Y_i),\quad
\O(\breve Y_{i,j})^{{\rm PD}_h}=\O_{\breve C}\wotimes_{\acris}\O(\widetilde Y_{i,j})^{{\rm PD}_h}$$
(Il n'y a pas de conflit de notations si $i\in S$.)
On dispose de
plongements naturels
$$\iota_i:\O(\breve Y_i)\to\O(\widetilde Y_i),
\quad
\iota_{i,j}:\O(\breve Y_{i,j})^{{\rm PD}_h}\to \O(\widetilde Y_{i,j})^{{\rm PD}_h}$$
\hskip.6cm $\bullet$ si $i\in S$, 
ce plongement est l'inclusion naturelle
et est donc un morphisme d'anneaux qui commute \`a $\varphi$.

\hskip.2cm $\bullet$ Si $(i,j)=(a,s)\in I_{2,c}$, 
ce plongement est l'inclusion 
$$\O(\breve{Y}_{a,s})[\tfrac{(U'_{a,s}-1)^k}{[k/p^h]!},\ k\in\N]^{p-\wedge}\hookrightarrow
\O(\widetilde{Y}_{a,s})[\tfrac{(U'_{a,s}-1)^k}{[k/p^h]!},\ k\in\N]^{p-\wedge}$$
(on a $\O(\breve{Y}_{a,s})=\O_{\breve C}[[T_{a,s},T_{a,s}^{-1}\rangle$);
c'est un morphisme d'anneaux qui commute \`a~$\varphi$.

\hskip.2cm $\bullet$ Si $i=a\in A$, on a
$$\O(\breve Y_a)=\O_{\breve C}[[T_{a,s_1}, T_{a,s_2}]]/(T_{a,s_1}T_{a,s_2}),\quad
\O(\widetilde Y_a)=\acris[[T_{a,s_1}, T_{a,s_2}]]/(T_{a,s_1}T_{a,s_2}-\tilde p^{\mu(a)}),$$
et $\iota_a$ est l'application $\O_{\breve C}$-lin\'eaire qui envoie $T_{a,s_i}^k$ sur $T_{a,s_i}^k$.
{\it Ce n'est pas un morphisme d'anneaux} puisque $T_{a,s_1}T_{a,s_2}=0$
dans $\O(\breve Y_a)$ et $T_{a,s_1}T_{a,s_2}=\tilde p^{\mu(a)}$
dans $\O(\widetilde Y_a)$, mais $\iota_a$ commute \`a $\varphi$.

\begin{lemm}\label{BAS12.1}
La
fl\`eche 
$\iota:\overline C_{\rm dR}^\bullet(\breve Y_S)\to
\overline C_{\rm dR}^\bullet(\widetilde Y_S)$, 
induite par les $\iota_s,\iota_a,\iota_{a,s}$, 
est un $p^{N(S)}$-morphisme de complexes qui commute \`a $\varphi$.  
\end{lemm}
\begin{proof}
C'est une cons\'equence du lemme~\ref{jambe2} (et du fait que les $\iota_s,\iota_a,\iota_{a,s}$
commutent \`a $\varphi$).
\end{proof}
On en d\'eduit le r\'esultat suivant.
\begin{coro}\label{BAS12.2}
On a des $p^{N(S)}$-isomorphismes 
{\rm (le premier de chaque ligne commute \`a $\varphi$):}
\begin{align*}
\acris\wotimes_{\O_{\breve C}} \overline C_{\rm dR}^\bullet(\breve Y_S)\cong
\overline C_{\rm dR}^\bullet(\widetilde Y_S)
\quad&{\rm et}\quad
\O_C\wotimes_{\O_{\breve C}} \overline C_{\rm dR}^\bullet(\breve Y_S)\cong
\overline C_{\rm dR}^\bullet(Y_S),\\
\acris\otimes_{\O_{\breve C}}H^1_{\rm dR}(\breve Y_S)^{\rm sep}\cong
H^1_{\rm dR}(\widetilde Y_S)^{\rm sep}
\quad&{\rm et}\quad
\O_C\otimes_{\O_{\breve C}}H^1_{\rm dR}(\breve Y_S)^{\rm sep}\cong
H^1_{\rm dR}(Y_S)^{\rm sep}.
\end{align*}
\end{coro}
On d\'efinit les groupes de {\it cohomologie de Hyodo-Kato}
$H^i_{\rm HK}(Y)$ par:
$$H^i_{\rm HK}(Y):=\breve C\otimes_{\O_{\breve C}}H^i_{\rm dR}(\breve Y_S)
=\breve C\otimes_{\acris}H^i_{\rm dR}(\widetilde Y_S).$$
(Ces groupes ne d\'ependent pas de $S$).  Par construction, ce sont
des $\breve C$-espaces et ils sont munis d'une action semi-lin\'eaire
de $\varphi$; on 
les munit aussi d'un op\'erateur de monodromie $N$ 
v\'erifiant $N\varphi=p\varphi N$ (cf.~rem.\,\ref{BAS12.4} ci-dessous).
Le cor.\,\ref{BAS12.2} a pour cons\'equence imm\'ediate l'\'enonc\'e suivant
dans lequel on a pos\'e $H^1_{\rm dR}(\widetilde Y)=\Q_p\otimes_{\Z_p}H^1_{\rm dR}(\widetilde Y_S)$
(le r\'esultat ne d\'epend pas de $S$).
\begin{theo}\label{basic21}
{\rm (i)} On a des isomorphismes 
{\rm (le premier commute \`a $\varphi$):}
$$\bcris^+\otimes_{\breve C}H^1_{\rm HK}(Y)^{\rm sep}\cong
H^1_{\rm dR}(\widetilde Y)^{\rm sep}
\quad{\rm et}\quad
\iota_{\rm HK}:C\otimes_{\breve C}H^1_{\rm HK}(Y)^{\rm sep}\cong
H^1_{\rm dR}(Y)^{\rm sep}.$$

{\rm (ii)}
$H^1_{\rm dR}(\widetilde Y)^{\rm sep}$ 
est un $\bcris^+$-module libre de rang fini
et on a des identifications
$$C\otimes_{\bcris^+}H^1_{\rm dR}(\widetilde Y)^{\rm sep}=H^1_{\rm dR}(Y)^{\rm sep}
\quad{\rm et}\quad
\breve C\otimes_{\bcris^+}H^1_{\rm dR}(\widetilde Y)^{\rm sep}=
H^1_{\rm HK}(Y)^{\rm sep},$$
la seconde \'etant $\varphi$-\'equivariante. 
\end{theo}

\begin{rema}\label{BAS12.4}
{\rm (o)} $\iota_{\rm HK}$ est {\it l'isomorphisme de Hyodo-Kato}.

{\rm (i)} 
En reprenant les arguments des preuves de la prop.\,\ref{basic16} et du th.\,\ref{Aff1},
on obtient les r\'esultats suivants, si $Z=\widetilde Y,\breve Y$, et si $\Lambda_Z=\acris$ 
ou~$\O_{\breve C}$ suivant que $Z=\widetilde Y$ ou~$Z=\breve Y$:
le groupe $H^1_{\rm dR}(Z_S)$ a une filtration stable par $\varphi$ dont les quotients
successifs sont, \`a $p^{N(S)}$ pr\`es,
$$H^1_{\rm dR}(Z_S)=\big[\xymatrix@C=.5cm{
H^1(\Gamma,\Lambda_Z)\ar@{-}[r]&
\prod_{s\in S} H^1_{\rm dR}(Z_s^\lozenge)_0\ar@{-}[r]&
H^1_c(\Gamma,\Lambda_Z)^\dual(-1)}\big].$$
Le (-1) signifie que l'action naturelle de $\varphi$
est multipli\'ee par $p$ (car
$\varphi(\frac{dT_{a,s}}{T_{a,s}})=
p\frac{dT_{a,s}}{T_{a,s}})$.

{\rm (ii)}
En utilisant l'isomorphisme $C\otimes_{\breve C}H^1_{\rm HK}(Y)^{\rm sep}\cong H^1_{\rm dR}(Y)^{\rm sep}$,
le th.\,\ref{Aff1} et la prop.\,\ref{basic9.11}, on prouve que
le groupe
$H^1_{\rm HK}(Y)^{\rm sep}$ 
admet une filtration naturelle stable par $\varphi$ dont les quotients successifs
sont\footnote{Rappelons que $M^{[1]}$ d\'esigne la partie de pente 
$1$ d'un $\breve C$-espace $M$
de dimension finie muni d'une action semi-lin\'eaire de $\varphi$.}:
$$H^1_{\rm HK}(Y)^{\rm sep}=\big[\xymatrix@C=.3cm{
H^1(\Gamma_{\rm int},\breve C)\ar@{-}[r]
&\big(\hskip-.35cm \prod\limits_{s\in \Sigma(Y)_{\rm int}}
\hskip-.35cm H^1_{\rm cris}(Y_s^{\rm sp})\big)
\oplus
\big(\hskip-.15cm\prod\limits_{s\in \partial Y}
\hskip-.15cm H^1_{\rm cris}(Y_s^{\rm sp})^{[1]}\big)
\ar@{-}[r]& H^1_c(\Gamma,\breve C)^\dual(-1)}\big].$$ 

{\rm (iii)}
L'op\'erateur $N_\mu:H^1_c(\Gamma,\Lambda_Z)^\dual\to
H^1(\Gamma,\Lambda_Z)$ du \no\ref{TTT19} induit des op\'erateurs
(de monodromie) v\'erifiant $N\circ N=0$ et $N\varphi=p\varphi N$, et compatibles
entre eux (par extension des scalaires)
$$N:H^1_{\rm dR}(\widetilde Y)\to H^1_{\rm dR}(\widetilde Y),
\quad
N:H^1_{\rm HK}(Y)^{\rm sep}\to H^1_{\rm HK}(Y)^{\rm sep}.$$
\end{rema}

\Subsection{Cohomologie syntomique et cohomologie \'etale $p$-adique}\label{BAS14}
\subsubsection{Cohomologie syntomique}\label{BAS15}
Si $i\in I$, on note
$F^1\O(\widetilde Y_i)$ le noyau du morphisme surjectif
 $\O(\widetilde Y_i)\to\O(Y_i)$,
et si $(i,j)\in I_{2,c}$, on note
$F^1\O(\widetilde Y_{i,j})^{\rm PD}$ le noyau 
de $\O(\widetilde Y_{i,j})^{\rm PD}\to\O(Y_{i,j})$.

On note ${\rm Syn}(Y_S,1)$ le complexe total associ\'e au complexe double
$$\xymatrix@C=.6cm@R=.6cm{
\prod_{i\in I}F^1\O(\widetilde Y_i)
\ar[r]\ar@<.2cm>[d]^-{1-\frac{\varphi}{p}}&
\prod_{i\in I}\Omega^1(\widetilde Y_i)
\oplus\prod_{(i,j)\in I_{2,c}}F^1\O(\widetilde Y_{i,j})^{\rm PD}\ar[r]
\ar@<-1cm>[d]^-{1-\frac{\varphi}{p}}\ar@<1.6cm>[d]^-{1-\frac{\varphi}{p}}&
\prod_{(i,j)\in I_{2,c}}\Omega^1(\widetilde Y_{i,j})^{\rm PD}_{d=0}\ar@<.3cm>[d]^-{1-\frac{\varphi}{p}}\\
\prod_{i\in I}\O(\widetilde Y_i)
\ar[r]& \prod_{i\in I}\Omega^1(\widetilde Y_i)
\oplus\prod_{(i,j)\in I_{2,c}}\O(\widetilde Y_{i,j})^{\rm PD}\ar[r]&
\prod_{(i,j)\in I_{2,c}}\Omega^1(\widetilde Y_{i,j})^{\rm PD}_{d=0}}$$
Les $H^i_{\rm syn}(Y_S,1)$ sont les groupes de cohomologie du complexe
${\rm Syn}(Y_S,1)$. 

\subsubsection{Comparaison syntomique-\'etale}\label{BAS17}
Si $u=(u_n)_n\in {\rm Symb}_p(Y)$, et si $i\in I$,
on choisit des rel\`evements $\tilde u_{i,n}$ et $(u_n)_{|Y_i}$ comme pr\'ec\'edemment,
ce qui permet d'associer 
\`a $u$ un $1$-cocycle $((x_i,y_i)_i,(z_{i,j})_{i,j})$ de
${\rm Syn}(Y_S,1)$, 
en posant 
\begin{align*}
(x_i,y_i)=\delta_{Y_i}(u_{|Y_i})=&\ \big(\lim_{n\to\infty}\tfrac{d\tilde u_{i,n}}{\tilde u_{i,n}},
\lim_{n\to\infty}\tfrac{1}{p}\log \tfrac{\varphi(\tilde u_{i,n})}{\tilde u_{i,n}^p}\big),\\
z_{i,j}=&\ \lim_{n\to\infty} \log\tfrac{\tilde u_{i,n}\otimes 1}{1\otimes \tilde u_{j,n}}
\end{align*}
(on a $z_{i,j}\in F^1\O(\widetilde Y_{i,j})$ car $u_{i,n}$ et $u_{j,n}$ co\"{\i}ncident sur
$Y_{i,j}$).
L'image $\delta_{\tilde p}(u)$ de ce $1$-cocycle dans $H^1_{\rm syn}(Y_S,1)$ ne d\'epend que de $u$,
ce qui fournit
une application $$\delta_{\tilde p}:{\rm Symb}_p(Y)\to H^1_{\rm syn}(Y_S,1).$$
\begin{rema}\label{basic36}
L'application 
$$\delta_{\tilde p}:{\rm Symb}_p(Y)\to H^1_{\rm dR}(\widetilde Y_S)^{\varphi=p}$$
d\'epend
du choix de $r\mapsto p^r$.  En effet, si on modifie $r\mapsto p^r$, cela change $\tilde p$
en $\tilde p[\epsilon]$, avec $\epsilon=(1,\epsilon_1,\dots)\in\O_{C^\flat}$,
et cela change $T_{a,s_2}$ en $[\epsilon^{\mu(a)}]T_{a,s_2}$ sur $\widetilde Y_a$, et donc  
aussi $z_{a,s_2}$ en $z_{a,s_2}+
{\rm Res}_a(\frac{du}{u})\mu(a)\log[\epsilon]$. 
\end{rema}
\begin{theo}\label{basic35}
L'application $\delta_{\tilde p}:{\rm Symb}_p(Y)\to H^1_{\rm syn}(Y_S,1)$
d\'efinie ci-dessus est un isomorphisme si $p>2$ {\rm(}si $p=2$, c'est presque un isomorphisme{\rm)}. 
\end{theo}
\begin{proof}
On a un diagramme commutatif dans lequel les suites horizontales
sont exactes:
$$\xymatrix@C=.5cm@R=.5cm{
0\ar[r]& H^1(\Gamma,\Z_p(1))\ar[d]^-{\wr}\ar[r] & {\rm Symb}_p(Y)
\ar[r]\ar[d] & \prod\limits_{i\in I}{\rm Symb}_p(Y_i)\ar[r]
\ar[d]^-{\wr}& \prod\limits_{(i,j)\in I_{2,c}} {\rm Symb}_p(Y_{i,j})\ar[d]^-{\wr}\\
0\ar[r]& H^1(\Gamma,\Z_pt)\ar[r] & H^1_{{\rm syn}}(Y_S,1)
\ar[r] & \prod\limits_{i\in I}H^1_{{\rm syn}}(Y_i,1)
\ar[r] & \prod\limits_{(i,j)\in I_{2,c}} H^1_{{\rm syn}}(Y_{i,j},1)}
$$
Les isomorphismes verticaux sont \'etablis dans la prop.~\ref{basic29} pour les $Y_i$ avec $i\in S$, 
et dans la prop.~\ref{basic34} pour les $Y_i$, avec $i\in A$, et les $Y_{i,j}$.
Le r\'esultat s'en d\'eduit.
\end{proof}

En combinant le r\'esultat pr\'ec\'edent avec celui du cor.\,\ref{basic12},
on obtient:
\begin{coro}\label{basic35.1} 
On a des isomorphismes naturels si $p>2$ {\rm(}si $p=2$, la fl\`eche de droite
est un isomorphisme, celle de gauche est presque un isomorphisme{\rm)}
$$H^1_{\rm syn}(Y_S,1)\overset{\sim}{\longleftarrow}{\rm Symb}_p(Y)\overset{\sim}{\longrightarrow}H^1_{\eet}(Y,\Z_p(1)).$$
\end{coro}

\subsubsection{Cohomologie syntomique et complexe de de Rham}
On note ${\rm HK}(Y_S,1)$ le complexe ${\rm Syn}(Y_S,1)$ avec $F^1\O(\widetilde Z)$ remplac\'e par $\O(\widetilde Z)$
et $\O(\widetilde Z)$ par $\O(\widetilde Z)'=\O(\widetilde Z)+\frac{\varphi}{p}\O(\widetilde Z)$, 
si $Z=Y_i,Y_{i,j}$. Le quotient
${\rm HK}(Y_S,1)/{\rm Syn}(Y_S,1)$ est $p$-isomorphe au complexe
$$\big[\prod_{i\in I}\O(Y_i)\to\prod_{(i,j)\in I_{2,c}}\O(Y_{i,j})\big]$$
qui calcule la cohomologie du faisceau $\O$.

On note ${\rm HK}_S^i$ les groupes de cohomologie du complexe ${\rm HK}(Y_S,1)$.
On a une suite $p^2$-exacte longue
$$0\to H^0_{\rm syn}(Y_S,1)\to {\rm HK}^0_S\to \O(Y_S)\to 
H^1_{\rm syn}(Y_S,1)\to {\rm HK}^1_S\to H^1(Y_S,\O)\to 
H^2_{\rm syn}(Y_S,1)\to {\rm HK}^2_S\to 0.$$

\begin{prop}\label{basic24.5}
On a une suite $p^2$-exacte
$$0\to\O(Y_S)/\O_C\to H^1_{\rm syn}(Y_S,1)\to H^1_{\rm dR}(\widetilde Y_S)^{\varphi=p}
\to H^1(Y_S,\O)$$
\end{prop}
\begin{proof}
On a ${\rm HK}_S^0=H^0_{\rm dR}(\widetilde Y_S)^{\varphi=p}=\acris^{\varphi=p}$;
on en d\'eduit que 
$$H^0_{\rm syn}(Y_S,1)=(F^1\acris)^{\varphi=p}=\Z_p t
\quad{\rm et}\quad
{\rm HK}_S^0/H^0_{\rm syn}(Y_S,1)\cong\O_C.$$
Par ailleurs,
${\rm HK}(Y_S,1)$ est le cylindre $[\xymatrix{C_{\rm dR}(\widetilde Y_S)\ar[r]^{1-\frac{\varphi}{p}}&
C_{\rm dR}(\widetilde Y_S)'}]$, o\`u $C_{\rm dR}(\widetilde Y_S)'$ est obtenu
\`a partir de $C_{\rm dR}(\widetilde Y_S)$ en rempla\c{c}ant $\O(\widetilde Z)$ par $\O(\widetilde Z)'$
si $Z=Y_i,Y_{i,j}$.
et comme $1-\frac{\varphi}{p}$ est surjectif de $H^0(C_{\rm dR}(\widetilde Y_S))$ sur 
$H^0(C_{\rm dR}(\widetilde Y_S)')$,
on a ${\rm HK}^1_S=H^1_{\rm dR}(\widetilde Y_S)^{\varphi=p}$.
On en d\'eduit le r\'esultat.
\end{proof}

\begin{rema}\label{basic24}
{\rm (i)} Si $Y$ est propre, alors $\O(Y_S)=\O_C$ et la suite
devient
$$0\to H^1_{\rm syn}(Y_S,1)\to H^1_{\rm dR}(\widetilde Y_S)^{\varphi=p}
\to H^1(Y_S,\O)\to
H^2_{\rm syn}(Y_S,1)\to\cdots$$

{\rm (ii)} Si $Y$ est un affino\"{\i}de, alors $H^1(Y_S,\O)$ est tu\'e par une puissance de $p$,
et la torsion de $H^1_{\rm dR}(\widetilde Y_S)^{\varphi=p}$ est d'exposant non born\'e.

{\rm (iii)} Une classe de $H^1_{\rm syn}(Y_S,1)$ est repr\'esent\'ee par
une collection $((\omega_i)_i,(f_{i,j})_{i,j},(g_i)_i)$ satisfaisant
un certain nombre de relations dont $\omega_i\otimes 1-1\otimes\omega_j=df_{i,j}$.
Cette relation, modulo $F^1$, devient $\omega_i-\omega_j=0$,
ce qui fournit une fl\`eche $H^1_{\rm syn}(Y_S,1)\to \Omega^1(Y_S)$.
De m\^eme, modulo $F^1$, le complexe $C_{\rm dR}(\widetilde Y_S)$ devient
$C_{\rm dR}(Y_S)$, ce qui fournit une fl\`eche $H^1_{\rm dR}(\widetilde Y_S)\to H^1_{\rm dR}(Y_S)$
et un diagramme commutatif \`a lignes exactes
$$\xymatrix@R=.6cm@C=.6cm{0\ar[r]&\O(Y_S)/\O_C\ar[r]\ar@{=}[d]& H^1_{\rm syn}(Y_S,1)\ar[r]\ar[d]
& H^1_{\rm dR}(\widetilde Y_S)^{\varphi=p}
\ar[r]\ar[d]& H^1(Y_S,\O)\ar@{=}[d]\\
0\ar[r]&\O(Y_S)/\O_C\ar[r]& \Omega^1(Y_S)\ar[r]& H^1_{\rm dR}(Y_S)
\ar[r]& H^1(Y_S,\O)\ar[r]&0}$$

\end{rema}

\begin{rema}\label{basic25}
{\rm (i)} La fl\`eche $\O(Y_S)/\O_C\to H^1_{\rm syn}(Y_S,1)$ admet la description suivante\footnote{Elle
n'est d\'efinie que sur le sous-groupe $\O'(Y_S)$ des $g$ tels que $(1-\frac{\varphi}{p})\tilde g_i\in\tilde R_i$,
pour tout $i$. Le quotient $\O(Y_S)/\O'(Y_S)$ est tu\'e par $p$.}.
Si $i\in I$ notons simplement
$R_i$,$\tilde R_i$ les anneaux $\O(Y_i)$, $\O(\widetilde Y_i)$,
et $R_{i,j}$, $\tilde R_{i,j}$ 
les anneaux $\O(Y_{i,j})$, $\O(\widetilde Y_{i,j})$, si $(i,j)\in I_{2,c}$.
Soit alors $g\in\O(Y_S)$. Si $i\in I$, choisissons un
rel\`evement $\tilde g_i$
de la restriction $g_i\in R_i$ de $g$ \`a $Y_i$.
La famille $\tilde g=((d\tilde g_i)_{i},((1-\frac{\varphi}{p})\tilde g_i)_{i},
(\tilde g_i\otimes 1-1\otimes\tilde g_j)_{i,j})$
est un $1$-cocycle du complexe  ${\rm Syn}(Y_S,1)$,
dont la classe de cohomologie ne d\'epend que de $g$,
puisque, si $\tilde g'$ est un autre choix, alors
$\tilde g'-\tilde g$ est le bord de $(\tilde g'_i-\tilde g_i)_i$
et $\tilde g'_i-\tilde g_i\in F^1\tilde R_i$ par construction.

{\rm (ii)}  La recette ci-dessus fournit aussi
une application naturelle $\O(Y)^{\dual\dual}\to
H^1_{\rm syn}(Y_S,1)$.
Si $u\in\O(Y)^{\dual\dual}$, on choisit un rel\`evement $\tilde u_i\in\tilde R_i$
de la restriction $u_i\in R_i$ de $u$ \`a $U_i$.
Comme $v_p(u_i-1)>0$, la s\'erie d\'efinissant $\log \tilde u_i$ converge dans
$\tilde R_i[\frac{1}{p}]$; notons $\tilde g_i$ sa somme.
On a $(1-\frac{\varphi}{p})\tilde g_i=
\frac{1}{p}\log(\frac{\varphi(\tilde u_i)}{\tilde u_i^p})\in \tilde R_i$
car $\frac{\varphi(\tilde u_i)}{\tilde u_i^p}\in 1+p\tilde R_i$.
On a aussi $d\tilde g_i=\frac{d\tilde u_i}{\tilde u_i}\in \tilde\Omega^1_i$,
et $\tilde g_i\otimes 1-1\otimes\tilde g_j=
\log\frac{\tilde u_i\otimes1}{1\otimes\tilde u_j}\in F^1\tilde R_{i,j}$
puisque $\frac{\tilde u_i\otimes1}{1\otimes\tilde u_j}\in 1+F^1\tilde R_{i,j}$.
La famille $\Delta(\tilde u):=((d\tilde g_i)_{i},((1-\frac{\varphi}{p})\tilde g_i)_{i},
(\tilde g_i\otimes 1-1\otimes\tilde g_j)_{i,j})$
est un $1$-cocycle du complexe double ${\rm Syn}(Y_S,1)$,
dont la classe de cohomologie ne d\'epend que de $u$
car, si $\tilde u'_i$ est
un autre rel\`evement de $u$, alors $\log\frac{\tilde u'_i}{\tilde u_i}\in F^1\tilde R_i$
puisque $\frac{\tilde u'_i}{\tilde u_i}\in 1+F^1\tilde R_{i}$.

L'application ci-dessus se prolonge,
par continuit\'e, en une  
application naturelle $\Z_p\wotimes\O(Y)^{\dual\dual}\to
H^1_{\rm syn}(Y_S,1)$.  Pour le voir, on \'ecrit un \'el\'ement
$u$ de $\Z_p\wotimes\O(Y)^{\dual\dual}$ sous la forme $\prod_{n\geq 0}u_n^{p^n}$
et on envoie $u$ sur la classe de $\sum_{n\geq 0}p^n\Delta(\tilde u_n)$ avec
des notations \'evidentes.  Il faut v\'erifier que ceci ne d\'epend
pas de l'\'ecriture et du choix des $\tilde u_n$; cela revient \`a prouver que,
si $\prod_{n\geq 0}u_n^{p^n}=1$, alors $\sum_{n\geq 0}p^n\Delta(\tilde u_n)$
est un cobord.  Dans ce cas, $u_0$ est une puissance $p$-i\`eme dans
$\Z_p\wotimes\O(Y)^{\dual\dual}$ et donc aussi dans $\O(Y)^{\dual\dual}$,
et on peut choisir la racine $p$-i\`eme $v_0$ telle que $v_0\prod_{n\geq 1}u_n^{p^{n-1}}=1$.
On choisit un rel\`evement $\tilde v_0$ de $v_0$ et alors 
$\Delta(\tilde u_0)=p\Delta(\tilde v_0)+dx_0$ (o\`u $d$ est la diff\'erentielle
de ${\rm Syn}(Y_S,1)$, i.e.~$\Delta(\tilde u_0)-p\Delta(\tilde v_0)$ est
un cobord). De m\^eme, $v_0u_1=v_1^p$ avec
$v_1\prod_{n\geq 2}u_n^{p^{n-2}}=1$, et $\Delta(\tilde v_0\tilde u_1)=p\Delta(\tilde v_1)+dx_1$,
etc. Mais alors
\begin{align*}
\Delta(\tilde u_0)+p\Delta(\tilde u_1)+p^2\Delta(\tilde u_2)+\cdots&=
dx_0+p\Delta(\tilde v_0\tilde u_1)+p^2\Delta(\tilde u_2)+\cdots\\
&= dx_0+pdx_1+p^2\Delta(\tilde v_1\tilde u_2)+p^3\Delta(\tilde u_3)+\cdots\\
&= dx_0+pdx_1+p^2x_2+\cdots=d(x_0+px_1+p^2x_2+\cdots)
\end{align*}
Autrement dit, $\sum_{n\geq 0}p^n\Delta(\tilde u_n)$
est un cobord, ce que l'on voulait.

{\rm (iii)}  Les deux applications ci-dessus se correspondent via
l'exponentielle: si $g\in\O(Y_S)$, alors $\exp(2pg)\in\O(Y)^{\dual\dual}$ et le diagramme
suivant est commutatif:
$$\xymatrix@R=.6cm@C=1.5cm{
\O(Y_S)\ar[d]\ar[r]^-{g\mapsto\exp(2pg)}&\O(Y)^{\dual\dual}\ar[d]\\
H^1_{\rm syn}(Y_S,1)\ar[r]^-{x\mapsto 2px}&H^1_{\rm syn}(Y_S,1)}$$

{\rm (iv)}  La compos\'ee de
$\O(Y)^{\dual\dual}\to
H^1_{\rm syn}(Y_S,1)\to H^1_{\rm dR}(\widetilde Y_S)^{\varphi=p}$ fournit
des classes de torsion:
si $N\in\N$ est tel que $p^N\tilde g_i\in R_i$ pour tout $i$, alors
$p^Nc$ est le bord de $(p^N\tilde g_i)_{i\in I}$,
et la classe de $f$ dans ${\rm HK}_S^1=H^1_{\rm dR}(\widetilde Y_S)^{\varphi=p}$ 
est tu\'ee par $p^N$.
On en d\'eduit que l'application compos\'ee
$$\Z_p\wotimes\O(Y)^{\dual\dual}\to H^1_{\rm syn}(Y_S,1)
\to (H^1_{\rm dR}(\widetilde Y_S)^{\rm sep})^{\varphi=p}$$
est identiquement nulle.
\end{rema}

\subsubsection{Monodromie}\label{BAS22}
Pour rem\'edier au probl\`eme soulev\'e dans la rem.~\ref{basic36}, on va 
\'etendre les scalaires \`a $\ast=\acris[{\rm Log}\,\tilde p]^{\rm PD}$,
o\`u ${\rm Log}\,\tilde p$ est transcendant sur $\acris$, 
pour faire dispara\^{\i}tre la {\og monodromie
autour de $p$\fg}.
On munit
$\ast$ d'une action de $\varphi$ en posant $\varphi({\rm Log}\,\tilde p)=
p\,{\rm Log}\,\tilde p$ et de la d\'erivation $N=\frac{d}{d\,{\rm Log}\,\tilde p}$.
On a $N\varphi=p\varphi N$ et
une suite exacte $0\to \acris\to \ast\overset{N}\to\ast\to 0$.
Notons que $\ast$ contient $\log(\tilde p[\epsilon])={\rm Log}\,\tilde p+\log[\epsilon]$
puisque $\log[\epsilon]\in\acris$, et comme $\varphi(\log[\epsilon])=p\log[\epsilon]$,
ni $\ast$, ni $\varphi$ ni $N$ ne d\'ependent
du choix de $\tilde p$ (ou, ce qui revient au m\^eme, de $r\mapsto p^r$).

Sur $\ast\wotimes_{\acris}H^1_{\rm dR}(\widetilde Y_S)$, on dispose du frobenius 
$\varphi=\varphi\otimes\varphi$ et 
de la monodromie $N=N\otimes 1+1\otimes N$, 
et comme $N^2=0$ sur $H^1_{\rm dR}(\widetilde Y_S)$,
l'application $$x\mapsto \iota_{\tilde p}(x)=1\otimes x-({\rm Log}\,\tilde p)\otimes Nx$$
d\'efinit un isomorphisme 
$$\iota_{\tilde p}:H^1_{\rm dR}(\widetilde Y_S)\overset{\sim}{\to}
(\ast\wotimes_{\acris}H^1_{\rm dR}(\widetilde Y_S))^{N=0}$$
qui commute \`a l'action de $\varphi$ (et d\'epend de $\tilde p$).

Soit $\alpha:H^1_{\rm syn}(Y_S,1)\to H^1_{\rm dR}(\widetilde Y_S)^{\varphi=p}$
la surjection naturelle (cf.~(ii) de la rem.~\ref{basic24}).

\begin{prop}\label{basic37}
L'application 
$$\iota_{\tilde p}\circ\alpha\circ\delta_{\tilde p}:{\rm Symb}_p(Y)\to
(\ast\wotimes_{\acris}H^1_{\rm dR}(\widetilde Y_S))^{N=0}$$
et la suite exacte
$$0\to \O(Y_S)/\O_C\to {\rm Symb}_p(Y)\to
(\ast\wotimes_{\acris}H^1_{\rm dR}(\widetilde Y_S))^{N=0,\varphi=p}\to 0$$
ne d\'ependent pas de $r\mapsto p^r$.
\end{prop}
\begin{proof}
Comme on l'a vu dans la rem.~\ref{basic36}, changer $\tilde p$ en $\epsilon\tilde p$
change les $z_{a,s_2}$ en 
$z_{a,s_2}+
{\rm Res}_a(\frac{du}{u})\mu(a)\log[\epsilon]$.
On en d\'eduit la formule
$$
\alpha\circ\delta_{\tilde p[\epsilon]}(u)=\alpha\circ\delta_{\tilde p}(u)+
N(\alpha\circ\delta_{\tilde p}(u))\log[\epsilon].$$
Maintenant 
$$\iota_{\tilde p[\epsilon]}(x)=1\otimes x-(\log\tilde p+\log[\epsilon])\otimes Nx=
\iota_{\tilde p}(x)-\log[\epsilon]\otimes Nx.$$
On en d\'eduit l'identit\'e
$$\iota_{\tilde p[\epsilon]}(\alpha\circ\delta_{\tilde p[\epsilon]}(u))=
\iota_{\tilde p}(\alpha\circ\delta_{\tilde p}(u))$$
qui permet de conclure.
\end{proof}

\subsubsection{Affino\"{\i}des}\label{BAS16}

\begin{theo}\label{basic26}
Si $Y$ est un affino\"{\i}de,
on a une suite $p^2$-exacte\footnote{Rappelons que $H^1(Y_S,\O)$ est tu\'e par $p^N$, pour $N$ assez grand.}
$$0\to \Z_p\wotimes\O(Y)^{\dual\dual}\to H^1_{\rm syn}(Y_S,1)
\to (H^1_{\rm dR}(\widetilde Y_S)^{\rm sep})^{\varphi=p}\to H^1(Y_S,\O).$$
\end{theo}
\begin{proof}
Il r\'esulte du (iv) de la rem.~\ref{basic25}
que l'on a un complexe
$\Z_p\wotimes\O(Y)^{\dual\dual}\to H^1_{\rm syn}(Y_S,1)
\to (H^1_{\rm dR}(\widetilde Y_S)^{\rm sep})^{\varphi=p}$.

$\bullet$  L'exactitude \`a gauche, i.e.~l'injectivit\'e
de $\Z_p\wotimes \O(Y)^{\dual\dual}\to H^1_{\rm syn}(Y_S,1)$, est une
cons\'equence du th\'eor\`eme de comparaison symboles-syntomique (th.~\ref{basic35})
et de la suite exacte de Kummer en cohomologie \'etale.

$\bullet$ Pour prouver l'exactitude au milieu, 
partons d'un $1$-cocycle $c=((\omega_i)_{i},(f_i)_{i},(g_{i,j})_{i,j})$
dans l'adh\'erence de la torsion.
Maintenant, $c$ \'etant dans l'adh\'erence de la torsion, 
il en est de m\^eme de sa restriction \`a $Y_i$,
et il existe $\tilde v_i\in\Z_p\wotimes \tilde R_i^{\dual\dual}$ tel que
$\omega_i=\frac{d\tilde v_i}{\tilde v_i}$ et $f_i=(1-\frac{\varphi}{p})\log\tilde v_i$
(cf.~prop.\,\ref{basic29} si $i\in S$, et~prop.\,\ref{basic34} (plus lemme~\ref{basic13}) si
$i\in A_c$; en fait $\tilde v_i\in \tilde R_i^{\dual\dual}$, si $i\in A_c$,
car $\tilde R_i^{\dual\dual}$ est complet pour la topologie $p$-adique).

Alors $\log \frac{\tilde v_i\otimes 1}{1\otimes \tilde v_j}
=g_{i,j}\in F^1\tilde R_{i,j}$.  On en d\'eduit, si $v_i=\theta(\tilde v_i)$, que
$v_i=v_j$ (resp. $v_i^2=v_j^2$ si $p=2$) sur $Y_{i,j}$ car $\log v_i=\log v_j$ (ce qui implique
que le quotient est une racine de l'unit\'e d'ordre $p^N$)
et $\log \frac{\tilde v_i\otimes 1}{1\otimes \tilde v_j}$ est entier, ce qui implique
que cette racine de l'unit\'e a un rel\`evement dans $\acris$ dont le logarithme est entier,
et donc est \'egale \`a $1$ (ou \`a $-1$ si $p=2$).

Les $v_i$ (resp.~les $v_i^2$) se recollent donc pour donner naissance \`a
$v\in \Z_p\wotimes\O(Y)^{\dual\dual}$.
Mais alors $c-\delta_{\tilde p}(v)$ est dans l'adh\'erence de la torsion
et dans $H^1(\Gamma,\acris^{\varphi=p})$.  Une petite modification de
la preuve du lemme~\ref{basic18.1} permet alors de prouver
que $c-\delta_{\tilde p}(v)$ est dans l'image de $\Z_p\wotimes\O(Y)^{\dual\dual}$,
ce qui permet de conclure.

$\bullet$ Comme $\Z_p\wotimes\O(Y)^{\dual\dual}$ se surjecte sur 
$(H^1_{\rm dR}(\widetilde Y_S)_{\rm tors})^{\varphi=p}$, et comme cette fl\`eche se factorise
\`a travers $H^1_{\rm syn}(Y_S,1)$ dont l'image dans $H^1(Y_S,\O)$ est nulle, cela induit, par 
passage au quotient, la fl\`eche $(H^1_{\rm dR}(\widetilde Y_S)^{\rm sep})^{\varphi=p}\to H^1(Y_S,\O)$ 
du th\'eor\`eme.
L'exactitude en $(H^1_{\rm dR}(\widetilde Y_S)^{\rm sep})^{\varphi=p}$ est alors
une cons\'equence de ce qui pr\'ec\`ede et de la prop.\,\ref{basic24.5}.

Ceci permet de conclure.
\end{proof}
\begin{rema}\label{basic27}
Il r\'esulte du (iii) de la rem.~\ref{basic24} que
la suite exacte du th.~\ref{basic26} se compl\`ete en le diagramme
commutatif suivant 
$$\xymatrix@R=.5cm@C=.6cm{
0\ar[r]& \Z_p\wotimes\O(Y)^{\dual\dual}\ar[r]\ar@<-.3cm>@{=}[d]& H^1_{\rm syn}(Y_S,1)
\ar[r]\ar[d]& (H^1_{\rm dR}(\widetilde Y_S)^{\rm sep})^{\varphi=p}\ar[r]\ar[d]& 0\\
0\ar[r]& \Z_p\wotimes\O(Y)^{\dual\dual}\ar[r]^-{\rm dlog}& \Omega^1(Y_S)
\ar[r]& H^1_{\rm dR}(Y_S)^{\rm sep}\ar[r]& 0}$$
dans lequel la ligne du haut est $p^{N(S)}$-exacte et celle du bas est un complexe
(exact \`a gauche et \`a droite mais pas au milieu).
\end{rema}

\subsubsection{Le th\'eor\`eme de comparaison pour les courbes quasi-compactes}\label{BAS23}
En modifiant le diagramme du (iii) de la rem.~\ref{basic24} gr\^ace \`a l'isomorphisme
du (i) du th.\,\ref{basic21}
et \`a la suite exacte du th.\,\ref{basic26},
on obtient le r\'esultat suivant, \`a la surjectivit\'e pr\`es
de $(\bst^+\otimes H^1_{\rm HK}(Y)^{\rm sep})^{N=0,\varphi=p}
\to H^1(Y,\O)$ dans le cas compact (dans le cas non compact, on a $H^1(Y,\O)=0$
et cette surjectivit\'e est triviale; voir le (i) de la rem.\,\ref{basic41} pour le cas compact).

\begin{theo} \label{basic40}
Si $Y$ est une courbe quasi-compacte, on a un diagramme commutatif dont la premi\`ere
ligne est exacte et la seconde est un complexe:
$$\xymatrix@C=.5cm@R=.5cm{
0\ar[r]& \Q_p\wotimes\O(Y)^{\dual\dual}\ar[r]\ar@<-.3cm>@{=}[d]& H^1_{\rm syn}(Y,1)
\ar[r]\ar[d]& (\bst^+\otimes _{\breve{C}}H^1_{\rm HK}(Y)^{\rm sep})^{N=0,\varphi=p}
\ar[r]\ar[d]&H^1(Y,\O)\ar@{=}[d]\ar[r]& 0\\
0\ar[r]& \Q_p\wotimes\O(Y)^{\dual\dual}\ar[r]^-{\rm dlog}& \Omega^1(Y)
\ar[r]& H^1_{\rm dR}(Y)^{\rm sep}\ar[r]&H^1(Y,\O)\ar[r]& 0}$$
\end{theo}

\begin{rema}\label{basic41}
{\rm (i)} Si $Y$ est propre, alors
$$\Q_p\wotimes \O(Y)^{\dual\dual}=0,\quad H^1_{\rm HK}(Y)^{\rm sep}=H^1_{\rm HK}(Y),\quad
\dim_{\Q_p}H^1_{\proet}(Y,\Q_p(1))<\infty,$$
et la suite du haut devient
$$0\to H^1_{\proet}(Y,\Q_p(1))
\to (\bst^+\otimes _{\breve{C}}H^1_{\rm HK}(Y))^{N=0,\varphi=p}\to
H^1(Y,\O)\to 0.$$  
L'exactitude \`a droite peut se prouver en utilisant un argument de Dimensions d'Espaces
de Banach: $(\bst^+\otimes H^1_{\rm HK}(Y)^{\rm sep})^{N=0,\varphi=p}$
et $H^1(Y,\O)$ sont les $C$-points d'Espaces de Banach ${\mathbb X}_{\rm st}$ et ${\mathbb X}_{\rm dR}$
de Dimension finie\footnote{
L'espace ${\mathbb X}_{\rm st}$ est le foncteur 
$\Lambda\mapsto (\Bst^+(\Lambda)\otimes_{\breve C} H^1_{\rm HK}(Y))^{N=0,\varphi=p}$,
l'espace ${\mathbb X}_{\rm dR}$ est le foncteur
$\Lambda\mapsto \Lambda\otimes_C H^1(Y,\O)$ et $f=(f_\Lambda)_\Lambda$,
avec $f_\Lambda=\theta_\Lambda\otimes\iota'_{\rm HK}$,
o\`u $\theta_\Lambda:\Bst^+(\Lambda)\to\Lambda$ est l'application habituelle, et
$\iota'_{\rm HK}$ est l'application de Hyodo-Kato
compos\'ee avec la surjection $H^1_{\rm dR}(Y)\to H^1(Y,\O)$.}
 et la fl\`eche ci-dessus s'\'etend en un morphisme $f:{\mathbb X}_{\rm st}\to
{\mathbb X}_{\rm dR}$. Si $Y$ est de genre $g$, alors ${\mathbb X}_{\rm dR}$ est l'espace trivial
${\mathbb V}^g$, de Dimension $(g,0)$, tandis que ${\mathbb X}_{\rm st}$ est de dimension
$\sum_{r_i\leq 1}(1-r_i,1)$, o\`u les $r_i$ sont les pentes de $\varphi$ sur $H^1_{\rm HK}(Y)$ avec
multiplicit\'e;
comme $0\leq r_i\leq 1$ et $r_i\mapsto 1-r_i$ est une bijection de l'ensemble des
$r_i$ car le cup-produit induit un accouplement parfait, commutant \`a $\varphi$,
de $H^1_{\rm HK}(Y)\times H^1_{\rm HK}(Y)$ dans $\breve C(-1)=H^2_{\rm HK}(Y)$,
on a ${\rm Dim}({\mathbb X}_{\rm st})=(g,2g)$.  On conclut en utilisant \cite[cor.\,5.17 (ii)]{CN}
ou bien en utilisant le fait que $\dim_{\Q_p}H^1_{\proet}(Y,\Q_p(1))=2g$ et la formule
$${\rm Dim}({\rm Im}\,f)={\rm Dim}({\mathbb X}_{\rm st})-{\rm Dim}({\rm Ker}\,f)=
(g,2g)-(0,2g)=(g,0)={\rm Dim}({\mathbb X}_{\rm dR}).$$
\hskip.4cm {\rm (ii)} si $Y$ n'est pas propre
{\rm (i.e.~$Y$ est un affino\"{\i}de)}, alors
$H^1(Y,\O)=0$, et la suite devient
$$0\to \Q_p\wotimes \O(Y)^{\dual\dual}\to H^1_{\proet}(Y,\Q_p(1))
\to (\bst^+\otimes _{\breve{C}}H^1_{\rm HK}(Y)^{\rm sep})^{N=0,\varphi=p}\to0.$$
En particulier, $H^1_{\proet}(Y,\Q_p(1))$ n'est pas de dimension finie
sur $\Q_p$ puisqu'il contient $\Q_p\otimes (\O(Y)^{\dual\dual}/1+{\goth m}_C)$ qui est un $C$-espace
de dimension infinie (isomorphe \`a $\O(Y)/C$ par le logarithme).
\end{rema}

\section{Affino\"{\i}des surconvergents}\label{BAS24}
Dans ce chapitre, on regarde ce qui se passe quand on transforme un affino\"{\i}de 
en courbe sans bord en le rendant surconvergent: il est bien connu que cela simplifie nettement
la structure de la cohomologie de de Rham, et la cohomologie pro\'etale $p$-adique prend aussi une forme 
plus sympathique (th.\,\ref{basic47}, la cohomologie pro\'etale $\ell$-adique, pour $\ell\neq p$, 
quant \`a elle ne change pas). On donne aussi deux interpr\'etations du
groupe $(\bst^+\otimes_{\breve C}H^1_{\rm HK}(X))^{N=0,\varphi=p}$, pour $X$ courbe propre,
en termes du rev\^etement universel du groupe $p$-divisible de la jacobienne de $X$,
l'une (th.\,\ref{basic49}) via la description de la cohomologie pro\'etale $p$-adique,
l'autre (th.\,\ref{inte1}) via l'int\'egration $p$-adique.
\Subsection{Cohomologie d'un affino\"{i}de surconvergent}
\subsubsection{Structure surconvergente sur un affino\"{i}de}\label{SSS11}
Soit $Y$ un affino\"{\i}de lisse de dimension $1$ sur $C$.
On fixe une structure surconvergente $Y^\dagger$ sur $Y$: on plonge $Y$ dans l'int\'erieur
d'une courbe quasi-compacte $X$ (par exemple en recollant des boules ouvertes le long des cercles
fant\^omes \`a la fronti\`ere de $Y$ pour obtenir (cf.~prop.\,\ref{Pconstr9}) une courbe propre).
On choisit une triangulation $S_X$ de $X$ contenant une triangulation $S_Y$ de $Y$,
et donc on a une inclusion de squelettes $\Gamma(S_Y)\subset\Gamma(S_X)$.
Si $\delta>0$ est assez petit, alors $\{r\in \Gamma(S_X),\ d(r,\Gamma(S_Y))\leq\delta\}$
est le squelette d'une triangulation d'un sous-affino\"{\i}de $Y_\delta$ de $X$,
et $Y_\delta$ est contenu dans l'int\'erieur de $Y_{\delta'}$ si $\delta<\delta'$;
de plus $\cap_{\delta>0}Y_\delta=Y$.
Alors $Y^\dagger$ est la {\og limite projective\fg} des $Y_{\delta}$, pour $\delta>0$, 
i.e. 
$$\O(Y^\dagger)=\varinjlim\nolimits_{\delta>0} \O(Y_{\delta})$$ 
et les $Y_\delta$, pour $\delta>0$, forment une {\it pr\'esentation} de $Y$. 
Le sous-anneau $\O(Y^\dagger)$ de $\O(Y)$
({\it des fonctions surconvergentes} sur $Y$) d\'epend du choix de $X$ (ou, ce qui revient
au m\^eme, de la pr\'esentation par les $Y_\delta$), mais deux pr\'esentations
diff\'erentes donnent des $\O(Y^\dagger)$ isomorphes (non canoniquement).
En pratique, $Y$ est souvent un sous-affino\"{\i}de d'un $X$ comme ci-dessus et
on dispose alors d'une pr\'esentation naturelle, mais dans le cas contraire tous
les objets consid\'er\'es d\'ependent du choix d'une pr\'esentation.

Si $\delta$ est assez petit, alors $Y_{\delta}\moins Y$ est contitu\'e de couronnes, une par cercle fant\^ome
\`a la fronti\`ere de $Y$, ouvertes
du c\^ot\'e de $Y$ et ferm\'ees de l'autre c\^ot\'e: si $X$ est obtenu
en recollant des boules ouvertes $D_i^-$, on peut identifier $D_i^-$ \`a la boule unit\'e
modulo le choix
d'un param\`etre local $z_i$, et alors $Y_\delta\moins Y$ est la r\'eunion des
couronnes $\{z_i\in D_i^-,\ 0<v_p(z_i)\leq\delta\}$.

\subsubsection{Cohomologie de de Rham}
On d\'efinit les $H^i_{\rm dR}(Y^\dagger)$ comme les groupes
de cohomologie du complexe $\O(Y^\dagger)\to\Omega^1(Y^\dagger)$.
On a
$$H^0_{\rm dR}(Y^\dagger)=C
\quad{\rm et}\quad
H^1_{\rm dR}(Y^\dagger)=\varinjlim\nolimits_{\delta>0}H^1_{\rm dR}(Y_\delta).$$

\begin{prop}\label{basic43}
L'application naturelle $H^1_{\rm dR}(Y^\dagger)\to H^1_{\rm dR}(Y)^{\rm sep}$
est surjective.
\end{prop}
\begin{proof}
Fixons $\delta>0$.  Si $\omega\in \Omega^1(Y)$, on peut \'ecrire
$\omega$ comme la limite d'une suite
$(\omega_n)_{n\geq 0}$, avec $\omega_n\in\Omega^1(Y_{\delta})$,
la suite convergeant dans $\Omega^1(Y)$ mais, en g\'en\'eral,
pas dans $\Omega^1(Y_{\delta})$.

Soit $W=H^1_{\rm dR}(Y)^{\rm sep}$, et soit ${\rm pr}:\Omega^1(Y)\to W$
la projection naturelle.  Alors $W$ est de dimension finie;
choisissons en une base $e_1,\dots,e_d$, et munissons $W$ de la valuation
$v_W(x_1e_1+\cdots+x_de_d)=\inf_iv_p(x_i)$. Choisissons $\omega_i\in\Omega^1(Y)$
ayant pour image $e_i$, et \'ecrivons 
$\omega_i$ comme la limite d'une suite 
 $(\omega_{i,n})_{n\geq 0}$ comme ci-dessus.
Alors ${\rm pr}(\omega_{i,n})$ tend vers $e_i$, et il existe $n_i$ tel que
$f_i={\rm pr}(\omega_{i,n_i})$ v\'erifie $v_W(e_i-f_i)>0$.
Les $f_i$ forment donc une base de $W$, qui est incluse
dans l'image de
la fl\`eche $H^1_{\rm dR}(Y^\dagger)\to H^1_{\rm dR}(Y)^{\rm sep}$.

Ceci permet de conclure.
\end{proof}

Si $\beta>\alpha>0$, on choisit une triangulation de $Y_{\beta}$ adapt\'ee
\`a $Y_{\alpha}$ et $Y$ (on peut supposer, ce que nous ferons,
que $S_\beta\moins S_\alpha=\partial Y_\beta$).  Cela nous fournit des mod\`eles sur $\O_C$,
semi-stables, $Y_S\subset Y_{S_\alpha}\subset Y_{S_\beta}$ 
de $Y\subset Y_{\alpha}\subset Y_{\beta}$.
On note $\Gamma\subset\Gamma_\alpha\subset\Gamma_\beta$ les graphes
correspondants: ces graphes ont la m\^eme cohomologie car on passe de l'un
\`a l'autre en changeant juste les longueurs des ar\^etes issues des
points de $\partial^{\rm ad}Y$ (et en rajoutant des points sur ces ar\^etes).

\begin{lemm}\label{basic44}
L'application naturelle $H^1_{\rm dR}(Y_{\beta})\to H^1_{\rm dR}(Y_{\alpha})$
se factorise \`a travers $H^1_{\rm dR}(Y_{\beta})^{\rm sep}$.
\end{lemm} 
\begin{proof}
Si $\delta=\alpha,\beta$, notons $H^1_{\rm dR}(Y_\delta)_0$ le noyau de l'application 
$\omega\mapsto {\rm Res}(\omega)$, i.e.~${\rm Ker}[H^1_{\rm dR}(Y_\delta)\to
H^1_c(\Gamma_\delta,C)^\dual]$. Comme $H^1_c(\Gamma,C)^\dual$, qui est
isomorphe \`a $H^1_{\rm dR}(Y_\delta)/H^1_{\rm dR}(Y_\delta)_0$, est un quotient
de $H^1_{\rm dR}(Y_\delta)^{\rm sep}$, il suffit de prouver l'\'enonc\'e
pour $H^1_{\rm dR}(Y_{\beta})_0\to H^1_{\rm dR}(Y_{\alpha})_0$.

La prop.~\ref{basic16} nous donne (apr\`es avoir invers\'e $p$)
un diagramme commutatif, dont les lignes sont $p^N$-exactes (avec $N$
d\'ependant de $\alpha$ et $\beta$)
$$\xymatrix@R=.5cm@C=.6cm{
0\ar[r]&H^1(\Gamma_\beta,C)\ar[r]\ar@{=}[d]&H^1_{\rm dR}(Y_\beta)_0\ar[r]\ar[d]&
\prod_{s\in S_\beta}H^1_{\rm dR}(Y_{\beta,s}^\lozenge)_0\ar[r]\ar[d]&0\\
0\ar[r]&H^1(\Gamma_\alpha,C)\ar[r]&H^1_{\rm dR}(Y_\alpha)_0\ar[r]&
\prod_{s\in S_\alpha}H^1_{\rm dR}(Y_{\alpha,s}^\lozenge)_0\ar[r]&0\\
}$$
La fl\`eche de droite est le
produit sur $s\in S_\beta$ des fl\`eches naturelles:

$\bullet$ si $s\in S_\alpha\moins\partial Y_\alpha$, alors $Y_{\beta,s}^\lozenge= Y_{\alpha,s}^\lozenge$
et la fl\`eche correspondante est l'identit\'e,

$\bullet$ si $s\in \partial Y_\alpha$, 
alors $Y_{\beta,s}^\lozenge$ est une compactification de $Y_{\alpha,s}^\lozenge$
obtenue en recollant des boules ouvertes le long des cercles fant\^omes appartenant \`a $Y_s$
\`a la fronti\`ere de $Y_\alpha$ et la fl\`eche est l'injection naturelle,

$\bullet$ si $s\in\partial Y_\beta$, la fl\`eche correspondante est identiquement nulle.

On conclut en utilisant le th.~\ref{Aff1}, 
qui fournit une suite exacte
$$0\to H^1(\Gamma_{\beta,{\rm int}},C)\to H^1_{\rm dR}(Y_\beta)_0^{\rm sep}\to 
\prod_{s\in S_\beta}H^1_{\rm dR}(Y_{\beta,s}^\lozenge)^{\rm sep}_0\to 0,$$
le fait que $\Gamma_{\delta,{\rm int}}$ et $\Gamma_\delta$ ont m\^eme cohomologie,
si $\delta=\alpha,\beta$, et les \'egalit\'es
$H^1_{\rm dR}(Y_{\beta,s}^\lozenge)^{\rm sep}_0=H^1_{\rm dR}(Y_{\beta,s}^\lozenge)_0
=H^1_{\rm dR}(Y_{\beta,s}^\lozenge)$,
si $s\in S_\beta\moins\partial S_\beta$ (car $Y_{\beta,s}^\lozenge$ est propre), 
et $H^1_{\rm dR}(Y_{\beta,s}^\lozenge)^{\rm sep}_0=0$ si $s\in\partial S_\beta$
(car $Y_{\beta,s}$ est un disque).
\end{proof}

\begin{coro}\label{basic45}
{\rm (i)}
Si $\delta>0$, l'application naturelle $H^1_{\rm dR}(Y_{\delta})\to H^1_{\rm dR}(Y^\dagger)$
induit un isomorphisme
$H^1_{\rm dR}(Y_{\delta})^{\rm sep}\overset{\sim}{\to} H^1_{\rm dR}(Y^\dagger)$.

{\rm (ii)} 
Le groupe
$H^1_{\rm dR}(Y^\dagger)$ est de rang fini
et admet une filtration naturelle dont les quotients successifs
sont:
$$H^1_{\rm dR}(Y^\dagger)=\big[\xymatrix@C=.3cm{
H^1(\Gamma,C)\ar@{-}[r]&\prod_{s\in\Sigma(Y)}H^1_{\rm dR}(Y_s^\lozenge)
\ar@{-}[r]& H^1_c(\Gamma,C)^\dual}\big].$$
\end{coro}
\begin{proof}
D'apr\`es le lemme~\ref{basic44}, on a
$H^1_{\rm dR}(Y^\dagger)=\varinjlim_{\delta>0}H^1_{\rm dR}(Y_{\delta})^{\rm sep}$.
Or on d\'eduit du th.~\ref{Aff1} que
$H^1_{\rm dR}(Y_{\beta})^{\rm sep}\to H^1_{\rm dR}(Y_{\alpha})^{\rm sep}$
est un isomorphisme si $\beta>\alpha$ (la surjectivit\'e est imm\'ediate
sur la description du th\'eor\`eme, et l'injectivit\'e r\'esulte
de ce que $H^1_{\rm dR}(Y_s^\lozenge)^{\rm sep}_0=0$ si $s\in\partial Y_\delta$
car $s$ est un cercle et $Y_s^{\lozenge}$ est un disque).

Cela prouve le (i), et le (ii) s'en d\'eduit en utilisant de nouveau le th.~\ref{Aff1},
le fait que $\Gamma_{\delta,{\rm int}}$ et $\Gamma_\delta$ ont m\^eme cohomologie
si $\delta>0$, et le fait que $H^1_{\rm dR}(\piqp)=0$ pour passer de $S$ \`a $\Sigma(Y)$.
\end{proof}

\subsubsection{Cohomologie de Hyodo-Kato}\label{basic46}
Si $\beta>\alpha$,
on a un morphisme de complexes $C_{\rm dR}^\bullet(\widetilde Y_\beta)\to
C_{\rm dR}^\bullet(\widetilde Y_\alpha)$
(comme ci-dessus, ce morphisme est l'identit\'e sur toutes les composantes
$Y_a$, $Y_s$, $Y_{a,s}$, pour $a\in A_\alpha$ et $s\in S_\alpha$,
et est l'application nulle sur les autres).

On en d\'eduit une application naturelle $H^1_{\rm HK}(Y_{\beta})\to H^1_{\rm HK}(Y_{\alpha})$
commutant \`a $\varphi$ et $N$ et qui se factorise \`a travers $H^1_{\rm HK}(Y_{\beta})^{\rm sep}$,
et on d\'efinit $H^1_{\rm HK}(Y^\dagger)$ par 
$$H^1_{\rm HK}(Y^\dagger):=\varinjlim\nolimits_{\delta>0}
H^1_{\rm HK}(Y_{\delta})=\varinjlim\nolimits_{\delta>0}
H^1_{\rm HK}(Y_{\delta})^{\rm sep}.$$
Les isomorphismes $C\otimes_{\breve C}H^1_{\rm HK}(Y_{\delta})^{\rm sep}\cong 
H^1_{\rm dR}(Y_{\delta})^{\rm sep}\cong H^1_{\rm dR}(Y^\dagger)$
montrent que l'application naturelle
$H^1_{\rm HK}(Y_{\delta})^{\rm sep}\to H^1_{\rm HK}(Y^\dagger)$
est un isomorphisme pour tout $\delta>0$.
En reprenant les arguments de la preuve du cor.\,\ref{basic45},
on d\'eduit de la rem.~\ref{BAS12.4} le r\'esultat suivant:

\begin{prop}\label{basic45.1}
Le groupe
$H^1_{\rm HK}(Y^\dagger)$ est de rang fini sur $\breve C$
et admet une filtration naturelle dont les quotients successifs
sont:
$$H^1_{\rm HK}(Y^\dagger)=\big[\xymatrix@C=.3cm{
H^1(\Gamma,{\breve C})\ar@{-}[r]&\prod_{s\in\Sigma(Y)}H^1_{\rm cris}(Y^{\rm sp}_s)
\ar@{-}[r]& H^1_c(\Gamma,{\breve C})^\dual}(-1)\big].$$
\end{prop}
\begin{rema}\label{basic45.2}
En passant \`a la limite dans le (i) du th.\,\ref{basic21},
on en d\'eduit un isomorphisme
$$\iota_{\rm HK}:C\otimes_{\breve C}H^1_{\rm HK}(Y^\dagger)\cong H^1_{\rm dR}(Y^\dagger).$$
\end{rema}

\subsubsection{Cohomologie pro\'etale $p$-adique}
On d\'efinit $H^1_{\proet}(Y^\dagger,\Q_p(1))$ par:
$$H^1_{\proet}(Y^\dagger,\Q_p(1)):=
\varinjlim\nolimits_{\delta>0}H^1_{\proet}(Y_{\delta},\Q_p(1)).$$
(Notons que $H^1_{\proet}(Y_{\delta},\Q_p(1))=H^1_{\eet}(Y_{\delta},\Q_p(1))$ car
$Y_\delta$ est quasi-compact.)

En passant \`a la limite quand $\delta\to 0$ 
dans le diagramme du th.~\ref{basic40},
et en utilisant l'isomorphisme $H^1_{\rm syn}(Y_\delta,1)\cong H^1_{\eet}(Y_{\delta},\Q_p(1))$
du cor.\,\ref{basic35.1},
on obtient le r\'esultat suivant.
\begin{theo}\label{}\label{basic47}
Si $Y^\dagger$ est un affino\"{\i}de surconvergent,
on a le diagramme commutatif fonctoriel d'espaces vectoriels topologiques
suivant:
$$
\xymatrix@R=.6cm@C=.6cm{
0\ar[r] & \O(Y^\dagger)/C\ar[r]^-{\exp}\ar@{=}[d] & H^1_{\proet}(Y^\dagger,\Q_p(1))\ar[d] \ar[r] &
(\bst^+\otimes_{\breve C} H^1_{\rm HK}(Y^\dagger))^{N=0,\varphi=p}\ar[r]\ar[d]^{\theta\otimes\iota_{\rm HK}} & 0\\
0\ar[r]& \O(Y^\dagger)/C \ar[r]^-d & \Omega^1(Y^\dagger) \ar[r]^-{\pi_{\rm dR}} & H^1_{\rm dR}(Y^\dagger)\ar[r] & 0
}
$$
dans lequel les lignes sont exactes et toutes les fl\`eches sont d'image ferm\'ee.
\end{theo}
\begin{proof}
Comme $\Q_p\wotimes C^\dual=0$, on peut supposer que toutes les fonctions qui
interviennent prennent la valeur $1$ en $Q$, o\`u $Q\in Y$ est fix\'e.
Sous cette hypoth\`ese, si
$0<\delta'<\delta$, et
si $g\in \O(Y_\delta)^{\dual\dual}$, alors
$g\in 1+p^{\delta-\delta'}\O(Y_{\delta'})$ d'apr\`es la prop.\,\ref{metr1}
(car $g-1$ s'annule sur $Y_\delta$, en $Q$).
Il en r\'esulte que, si $(f_n)_{n\in\N}$ est une suite d'\'el\'ements
de $\O(Y_\delta)^{\dual\dual}$, alors $\prod_{n\in\N}f_n^{p^n}$ converge dans $\O(Y_{\delta'})^\dual$.
On en d\'eduit que $$\varinjlim \Q_p\widehat\otimes (\O(Y_\delta)^{\dual\dual}/C^\dual)=
\Q_p\otimes (\O(Y^\dagger)^{\dual\dual}/C^\dual),$$
ce qui fournit le diagramme ci-dessus avec $\Q_p\otimes (\O(Y^\dagger)^{\dual\dual}/C^\dual)$
au lieu de $\O(Y^\dagger)/C$.
On conclut en utilisant le fait que
$\log\,f$ converge dans $\O(Y^\dagger)$ si
$f\in \O(Y^\dagger)^{\dual\dual}$, et que $\log$ induit un isomorphisme
$$\Q_p\otimes (\O(Y^\dagger)^{\dual\dual}/C^\dual)\cong \O(Y^\dagger)/C.$$
Les fl\`eches horizontales sont d'image ferm\'ee 
car $H^1_{\rm dR}(Y^\dagger)$ est s\'epar\'e, et donc $H^1_{\rm HK}(Y^\dagger)$
aussi.  La fl\`eche verticale de droite est d'image ferm\'ee
car c'est la trace sur les $C$-points d'un morphisme d'Espaces de Banach
de Dimension finie~\cite{BC} (on peut aussi utiliser \cite[Lemma 3.3]{Ki} au lieu
de la th\'eorie des Espaces de Banach
de Dimension finie).
Cela implique que la fl\`eche verticale du milieu est d'image ferm\'ee
(le conoyau s'injecte dans un espace s\'epar\'e).
\end{proof}
\begin{rema}\label{ELL11}Si $\ell\neq p$, 
il r\'esulte du th.\,\ref{ladique} que
l'application naturelle
$H^1_{\proet}(Y_\delta,\Q_\ell(1))\to H^1_{\proet}(Y,\Q_\ell(1))$ est un isomorphisme
pour tout $\delta>0$.  Il s'ensuit que l'on a un isomorphisme
$$H^1_{\proet}(Y^\dagger,\Q_\ell(1))\overset{\sim}{\longrightarrow} H^1_{\proet}(Y,\Q_\ell(1)),$$
et donc que rendre un affino\"{\i}de surconvergent ne change pas sa cohomologie
pro\'etale $\ell$-adique, si $\ell\neq p$; c'est tr\`es loin d'\^etre le cas si $\ell=p$
comme le montre une comparaison des th.\,\ref{basic47} et~\ref{basic40} (ou plut\^ot du
(ii) de la rem.\,\ref{basic41}).
\end{rema}

\subsubsection{Le rev\^etement universel du groupe $p$-divisible de la jacobienne}\label{SSS4.1} 
Soit $X$ une courbe propre d\'efinie sur un sous-corps complet $K$ de $C$; 
notons $J$ sa jacobienne; alors $J(C)$ est naturellement un groupe de Lie sur $C$.
On note $\hJ$ l'ensemble des suites $x=(x_n)_{n\in\Z}$ d'\'el\'ements
de $J(C)$ telles que $x_n=p\cdot x_{n+1}$, pour tout $n\in\Z$, et
$x_n\to 0$ quand $n\to -\infty$.  Alors $\hJ$ est un $\Q_p$-espace vectoriel
muni d'une action continue de $ G_K$ et
contenant $V_p(J)=\Q_p\otimes_{\Z_p}T_p(J)$ (qui n'est autre que le sous-$\Q_p$-espace vectoriel des $x=(x_n)_{n\in\Z}$,
avec $x_n=0$ si $n\ll 0$).
L'application logarithme $$\log_J:J({C})\to {C}\otimes_K{\rm Lie}(J)$$ induit une
application $$\log_J:\hJ\to {C}\otimes_K{\rm Lie}(J)$$ (envoyant
 $(x_n)$ to $p^{n}\log_Jx_n$, pour n'importe quel choix de $n$),
 et donne naissance \`a la suite exacte
$$0\to V_p(J)\to \hJ\to {C}\otimes_K{\rm Lie}(J)\to 0$$
de $ G_K$-modules.
\begin{rema}\label{basic48}
Supposons que $X$ est obtenu en compactifiant un affino\"{\i}de $Y$ par recollement de boules
ouvertes $D_i^-$ le long des cercles fant\^omes \`a la fronti\`ere de $Y$.
Notons $\iota:X\to J$
l'injection envoyant $P\in X$ sur la classe de $P-P_0$, o\`u $P_0$ est fix\'e.
Si $\delta>0$, soient $Y_\delta=Y\moins\big(\cup_i D(P_i,\delta)^-\big)$
et $H_\delta$ le sous-groupe de $J$ engendr\'e par les
$\iota(D(P_i,\delta)^-)$. On a une suite exacte
$$0\to \Q_p\widehat\otimes (\O(Y_\delta)^{\dual\dual}/C^\dual)\to H^1_{\eet}({Y_\delta},\Q_p(1))_0\to V_p(J/H_\delta)\to 0,$$
pour tout $\delta>0$.
Maintenant, les boules ferm\'ees $D(P_i,\delta)$ se plongent dans $X$,
ce qui implique que l'image de $H_\delta$ 
par $\log_J$ est un sous-groupe ouvert born\'e de $C\otimes_K{\rm Lie}(J)$
et que le sous-groupe de torsion de $H_\delta$ est fini; on en d\'eduit que
$T_p(J/H_\delta)$ est un r\'eseau de $\hJ$ et donc
que  $V_p(J/H_\delta)=\hJ$, pour tout $\delta>0$, et donc
$\varinjlim_\delta V_p(J/H_\delta)=\hJ$.

En passant \`a la limite dans la suite ci-dessus, on en d\'eduit une suite exacte
fonctorielle
$$\xymatrix{
C\ar[r]& \O(Y^\dagger)\ar[r]^-{\exp}& H^1_{\eet}(Y^\dagger,\Q_p(1))_0\ar[r]& \widehat J\ar[r]& 0.}$$
\end{rema}

En comparant les suites exactes de la rem.~\ref{basic48} et du th.~\ref{basic47},
on en d\'eduit le r\'esultat suivant qui admet une interpr\'etation en termes
d'int\'egration $p$-adique comme nous le verrons au \no\ref{SSS5} (cf.~th.\,\ref{inte1}).

\begin{theo}\label{basic49}
On a un isomorphisme naturel 
$$\iota_{\rm st}:\hJ\cong (\bst^+\otimes_{\breve C} H^1_{\rm HK}(X))^{N=0,\varphi=p}.$$
\end{theo}

\begin{rema}\label{basic50}
(i)  Si, au lieu de $J$, on consid\`ere un groupe $p$-divisible ${\cal G}$ sur $\O_C$,
et qu'on d\'efinit le rev\^etement universel $\widehat{\cal G}$ de mani\`ere
analogue, alors on a un isomorphisme naturel (\cite[ch. 4]{FF} ou \cite[sec. 5.1]{SW})
$$\widehat{\cal G}\cong(\bcris^+\otimes_{\breve C}D({\cal G}))^{\varphi=p},$$
o\`u $D({\cal G})$ est le module de Dieudonn\'e (covariant) de ${\cal G}$. 
On pourrait en d\'eduire
le th.\,\ref{basic49} dans le cas o\`u $X$ a bonne r\'eduction.

(ii) $H^1_{\rm HK}(X)$ est contravariant en $X$, 
mais est aussi isomorphe \`a $H^1_{\rm HK}(X)^\dual(-1)$,
et c'est sous cette forme qu'il appara\^{\i}t dans le th.\,\ref{inte1} qui
est la g\'en\'eralisation naturelle de l'\'enonc\'e pour les groupes $p$-divisibles.
\end{rema}

\begin{rema}\label{basic50.1}
Soient $X$ propre et $Y$ un affino\"{\i}de munis d'une triangulation $S$ assez fine, et donc d'un patron.
La filtration sur $H^1_{\rm HK}(X)$ ou $H^1_{\rm HK}(Y)^{\rm sep}$ induit une filtration sur
$\hJ$. On note $H^1_{\rm HK}(X)_0$ et $H^1_{\rm HK}(Y)^{\rm sep}_0$ les noyaux
des fl\`eches vers $H^1_c(\Gamma,\breve C)^\dual(-1)$; l'op\'erateur $N$ est identiquement nul
sur ces sous-groupes (ce qui explique que l'on peut utiliser $\bcris$ au lieu de $\bst$ dans les \'enonc\'es
ci-dessous).

{\rm (i)} Si $X$ est propre,
notons ${\cal J}$ le mod\`ele de N\'eron de $J$, ${\cal J}^0$
la composante connexe de $0$, et $J^{\rm sp}$ la fibre sp\'eciale de
${\cal J}^0$ (c'est une vari\'et\'e semi-ab\'elienne, extension de $\prod_{s\in S}J(Y_s^{\rm sp})$
par $H^1(\Gamma,\Z)\otimes {\bf G}_m$, qui s'identifie \`a ${\rm Pic}(X_S^{\rm sp})$).
Le groupe $\hJ$ vit dans une suite exacte
$$0\to V_p({\cal J}^0(\O_C/p))\to \hJ\to V_p(\pi_0({\cal J}))\to 0,$$
o\`u $\pi_0({\cal J})$ est 
le groupe des composantes connexes (isomorphe \`a $(\Q/\Z)\otimes H^1(\Gamma,\Z)^\dual$).
On a de plus des isomorphismes
\begin{align*}
V_p(J^{\rm sp}(\O_C/p))\cong V_p({\cal J}^0(\O_C/p))\cong (\bcris^+\otimes _{\breve{C}}H^1_{\rm HK}(X)_0)^{\varphi=p},\\
V_p(\pi_0({\cal J}))\cong H^1(\Gamma,\Q_p)^\dual
=(\bcris^+\otimes_{\breve C}H^1(\Gamma,\breve C)^\dual(-1))^{\varphi=p}
\end{align*}

{\rm (ii)} Si $Y$ est un affino\"{\i}de, reprenons les notations du (ii) de la rem.\,\ref{BAS12.4} d\'ecrivant
le groupe $H^1_{\rm HK}(Y)^{\rm sep}$. Rappelons que $Y_S^{\rm sp}$ est, par d\'efinition, la compactifi\'ee
de la fibre sp\'eciale classique de $Y_S$; cela fait que
$H^1_{\rm HK}(Y)_0^{\rm sep}$ est un quotient de
$H^1_{\rm cris}(Y_S^{\rm sp})_0$, sous-groupe du groupe de cohomologie log-cristalline des \'el\'ements
dont les r\'esidus sont nuls en tous les points avec une structure logarithmique.
On note $Y^{\rm sp}_{S,{\rm int}}$ la r\'eunion des $Y_s$, pour $s\in S_{\rm int}$.
Soient $J^{\rm sp}={\rm Pic}(Y_S^{\rm sp})$ et 
$J^{\rm sp}_{\rm int}={\rm Pic}(Y_{S,{\rm int}}^{\rm sp})$.
Alors $J^{\rm sp}$ et $J^{\rm sp}_{\rm int}$ sont des vari\'et\'es semi-ab\'eliennes,
et on dispose d'une fl\`eche naturelle $J^{\rm sp}\to J^{\rm sp}_{\rm int}$;
on note $J^{\rm sp}_{\rm ext}$ le noyau de ce morphisme.
On a, comme ci-dessus, un isomorphisme 
$$(\bcris^+\otimes_{\breve C}H^1_{\rm cris}(Y_S^{\rm sp})_0)^{\varphi=p}\cong
V_p(J^{\rm sp}(\O_C/p))$$ 
et un diagramme commutatif
\`a lignes horizontales exactes:
$$
\xymatrix@R=.5cm@C=.4cm{
0\ar[r]&V_p(J^{\rm sp}_{\rm ext}(\O_C/p))\ar[r]\ar[d]
&(\bcris^+\otimes_{\breve C}H^1_{\rm cris}(Y_S^{\rm sp})_0)^{\varphi=p}
\ar[r]\ar[d]&V_p(J^{\rm sp}_{\rm int}(\O_C/p))\ar[r]\ar@{=}[d]&0\\
0\ar[r]&V_p(J^{\rm sp}_{\rm ext}(k_C))\ar[r]
&(\bcris^+\otimes_{\breve C}H^1_{\rm HK}(Y)_0^{\rm sep})^{\varphi=p}
\ar[r]&V_p(J^{\rm sp}_{\rm int}(\O_C/p))\ar[r]&0}
$$
\end{rema}

\Subsection{Comparaison avec l'int\'egration $p$-adique}\label{BAS25}
Dans ce paragraphe, $C=\C_p$ et $K$ est une extension finie\footnote{Cette restriction
est due au fait que l'int\'egration $p$-adique~\cite{Cn85,Cz-inte} est d\'evelopp\'ee dans ce cadre.
Mais les m\'ethodes de~\cite{Cz-inte} permettent d'\'etendre cette int\'egration sur
une vari\'et\'e ab\'elienne $A$,
dans le cas g\'en\'eral, au sous-groupe de $A(C)$ des $x$ tels que $N!\,x$ tend vers $0$
quand $N\to\infty$ (si $C=\C_p$, tout $x$ v\'erifie cette condition). Cela suffirait pour
\'etendre ce qui suit au cas o\`u $v_p$ est discr\`ete sur $K$ et $C$ est le compl\'et\'e
de sa cl\^oture alg\'ebrique.}
de $\Q_p$,
$X$ est une courbe propre et lisse d\'efinie sur $K$
et $J$ est sa jacobienne.
On fixe $P_0\in X(C)$ et on note $\iota:X\to J$ le plongement
envoyant $P$ sur la classe du diviseur $P-P_0$.

\subsubsection{L'accouplement de p\'eriodes}\label{SSS5}
L'int\'egration $p$-adique induit des accouplements de p\'eriodes, $ G_K$-\'equivariants:
$$\langle\ ,\ \rangle_{\bdr}: H^1_{\rm dR}(X)\otimes_{\Q_p} \hJ\to \bdr^+,
\quad \langle\ ,\ \rangle_{{\C_p}}: H^1_{\rm dR}(X)\otimes_{\Q_p} \hJ\to {\C_p},$$
et on a $\langle\ ,\ \rangle_{{\C_p}}=\theta\circ\langle\ ,\ \rangle_{\bdr}$.

Passer par l'extension universelle $\tJ$ de $J$
permet d'en donner une d\'efinition particuli\`erement compacte
(cf.~\cite[prop.\,B.2.2]{Cz-inte}).
On dispose d'applications naturelles $ G_K$-\'equivariantes (Lemme~12 ou \S\,B.2 de~\cite{Cz-inte}):
$$\iota_{\C_p}:\hJ\to \tJ({\C_p}),\quad \iota_{\bdr}:\hJ\to \tJ(\bdr^+)$$
d\'efinies de la mani\`ere suivante: si $x=(x_n)_{n\in\Z}\in\hJ$, on choisit
une suite born\'ee $(\hat x_n)_{n\in\Z}$ de rel\`evements des $x_n$ dans
$\tJ({\C_p})$ (resp.~$\tJ(\bdr^+$)), et on envoie $x$ sur la limite
de $p^n\cdot \hat x_n$, quand $n\to +\infty$, la multiplication par $p^n$
\'etant celle sur $\tJ$.
On a, bien \'evidemment,
$$\iota_{\C_p}=\theta\circ\iota_{\bdr}.$$
Maintenant,
si $$\log_{\tJ}:\tJ\to H^1_{\rm dR}(J)^\dual=H^1_{\rm dR}(X)^\dual$$ est le logarithme de $\tJ$ \`a valeurs
dans son alg\`ebre de Lie, et si $\eta\in H^1_{\rm dR}(X)$ et $x\in\hJ$, alors
$$\langle\eta ,x \rangle_{\bdr}=\langle\log_{\tJ}\circ\iota_{\bdr}(x),\eta \rangle_{\rm dR}
\quad{\rm et}\quad
\langle\eta ,x \rangle_{{\C_p}}=\langle\log_{\tJ}\circ\iota_{{\C_p}}(x),\eta \rangle_{\rm dR}.$$
De plus, $\log_{\tJ}\circ\iota_{\bdr}:\hJ\to \bdr^+\otimes_K H^1_{\rm dR}(X)^\dual$
est injective (c'est une cons\'equence de la non d\'eg\'en\'erescence de l'accouplement
de p\'eriodes, cf.~\cite[th.\,II.3.5]{Cz-inte} par exemple).

Notons que $\bst^+\otimes_{\breve\C_p} H^1_{\rm HK}(X_{\C_p})^\dual$ 
s'injecte dans $\bdr^+\otimes_KH^1_{\rm dR}(X)^\dual$.
\begin{theo}\label{inte1}
{\rm (i)}
L'application $\log_{\tJ}\circ \iota_{\bdr}:\hJ\to \bdr^+\otimes_KH^1_{\rm dR}(X)^\dual$
se factorise \`a travers $(\bst^+\otimes_{\breve\C_p}H^1_{\rm HK}(X)^\dual)^{N=0,\varphi=1}$.

{\rm (ii)} Si on identifie $H^1_{\rm dR}(X)^\dual$ \`a $H^1_{\rm dR}(X)$ gr\^ace au cup-produit
\`a valeurs dans $H^2_{\rm dR}(X)=K$, alors
$$(\bst^+\otimes_{\breve\C_p}H^1_{\rm HK}(X)^\dual)^{N=0,\varphi=1}
=(\bst^+\otimes_{\breve\C_p} H^1_{\rm HK}(X))^{N=0,\varphi=p}
\quad{\rm et}\quad
\log_{\tJ}\circ \iota_{\bdr}=\iota_{\rm st}.$$
\end{theo}
\begin{rema}\label{basic51}
L'inclusion de $\log_{\tJ}\circ \iota_{\bdr}(\hJ)$
dans $(\bst^+\otimes_{\breve\C_p}H^1_{\rm HK}(X)^\dual)^{N=0,\varphi=1}$
traduit le fait que
la restriction de l'accouplement de p\'eriodes $\langle\ ,\ \rangle_{\bdr}$
\`a $H^1_{\rm HK}(X)$ est \`a valeurs dans $\bst^+$ et commute aux actions de $\varphi$ et $N$ sur
$H^1_{\rm HK}(X)$ en plus de celle de $ G_K$ sur $\hJ$
(cf.~\cite{Cz-periodes,CI} pour des r\'esultats dans cette direction).
\end{rema}
\begin{proof}
Soit $Y$ un affino\"{\i}de de $X$, compl\'ementaire d'un nombre fini
de boules ouvertes.  On a un diagramme (le $0$ en indice
indique le sous-groupe des classes sans r\'esidus le long des cercles fant\^omes
\`a la fronti\`ere des boules enlev\'ees):
$$\xymatrix@R=.5cm@C=1.2cm{H^1_{\proet}(Y_{\C_p}^\dagger,\Q_p(1))_0\ar[d]\ar[r]&\hJ\ar[r]^-{\log_{\tJ}\circ \iota_{\bdr}}&\bdr^+\otimes_K H^1_{\rm dR}(X)^\dual\ar@{=}[d]\\
(\bst^+\otimes_{\breve\C_p} H^1_{\rm HK}(Y_{\C_p}^\dagger)_0)^{N=0,\varphi=p}\ar[r]^-{1\otimes\iota_{\rm HK}}&
\bdr^+\otimes_K H^1_{\rm dR}(Y^\dagger)_0 &\bdr^+\otimes_K H^1_{\rm dR}(X)\ar[l]_-{\sim}}$$
dans lequel les deux fl\`eches partant de $H^1_{\proet}(Y_{\C_p}^\dagger,\Q_p(1))_0$
proviennent de la suite de Kummer (pour l'horizontale)
et de la comparaison avec la cohomologie syntomique (pour la verticale),
et ont m\^eme noyau. Il s'agit
de prouver que le diagramme commute.

Il suffit donc
de prouver que, si $v\in\hJ$, et si $(f_n)_{n\in\N}$ est un rel\`evement
de $v$ dans ${\rm Symb}_p(Y_\delta)$, alors
$\log_{\tJ}(\iota_{\bdr}(v))=\big(\pi_{\rm dR}(\frac{d\tilde f_n}{\tilde f_n})\big)_{n\in\N})$.
On va prouver cet \'enonc\'e apr\`es application de $\theta$
(i.e. en rempla\c{c}ant $\iota_{\rm dR}$ par $\iota_{\C_p}$, ce qui \'evite
de choisir des $\tilde f_n$),
et indiquer les modifications \`a faire pour prouver
le r\'esultat pour $\iota_{\rm dR}$.
Pour prouver l'\'enonc\'e avec $\iota_{\C_p}$,
on va utiliser une fonction de Green du diviseur~$\Theta$ pour
construire un rel\`evement
de $v$ dans ${\rm Symb}_p(Y_\delta)$, avec $\delta>0$ fix\'e.

\subsubsection{Formes diff\'erentielles de troisi\`eme esp\`ece}\label{SSS22}
Rappelons qu'une forme diff\'erentielle m\'eromorphe sur $X$ est dite:

$\bullet$ {\it de premi\`ere esp\`ece} si elle est holomorphe,

$\bullet$ {\it de seconde esp\`ece} si les r\'esidus en ses p\^oles sont tous nuls,

$\bullet$ {\it de troisi\`eme esp\`ece} si elle n'a que des p\^oles simples et si les r\'esidus
en ces p\^oles sont des entiers.

Si $\omega$ est de troisi\`eme esp\`ece, on note ${\rm Div}(\omega)$ son diviseur:
${\rm Div}(\omega)=\sum_x{\rm Res}_x\omega \cdot (x)$; il est de degr\'e $0$.  Si $\omega=\frac{df}{f}$,
alors ${\rm Div}(\omega)={\rm Div}(f)$.  Alors $\tJ$ est le quotient du groupe des
formes de troisi\`eme esp\`ece par celui des $\frac{df}{f}$;
l'application naturelle $\pi_J:\widetilde J\to J$ est induite par $\omega\mapsto{\rm Div}(\omega)$,
son noyau est l'espace $H^0(X,\Omega^1)$ des formes de premi\`ere esp\`ece.  Enfin,
le quotient de l'espace des formes de seconde esp\`ece par celui des $df$
est $H^1_{\rm dR}(X)$.

\smallskip
Choisissons une base $\omega_1,\dots,\omega_g$ de $\Omega^1_{\rm inv}(J)$, et notons $\partial_1,\dots,\partial_g$
la base duale de l'espace des formes diff\'erentielles invariantes sur $J$ (on a donc
$df=\sum_{i=1}^g\partial_if\,\omega_i$), et $\lambda_i$ le logarithme de $J$ solution
de l'\'equation diff\'erentielle $d\lambda_i=\omega_i$.

Soit $\Theta$ le diviseur de $J$, image de $X^{g-1}$ par $(Q_1,\dots,Q_{g-1})\mapsto
\ominus \iota(Q_1)\ominus \cdots\ominus \iota(Q_{g-1})$.  
Si $u\in J$ est g\'en\'eral, l'intersection
$X_u$ de $\iota(X)$ et\footnote{On note $\oplus$ l'addition sur $J$.}
 $u\oplus\Theta$ est constitu\'ee des $g$ points $Q_{u,1},\dots,Q_{u,g}$
solutions de l'\'equation $\iota(Q_{u,1})\oplus\cdots\oplus \iota(Q_{u,g})=u$.

Soit $G$ une fonction de Green de $\Theta$ (cf.~\no2 de \cite[\S\,II.2]{Cz-inte}).
Si $u$ est g\'en\'eral,
$$\eta_{u,i}=\iota^\dual\big(d(\partial_iG(x\ominus u))\big)$$
est une forme diff\'erentielle de seconde esp\`ece sur $X$, holomorphe en dehors
de p\^oles doubles en les points de $X_u$, dont l'image $\eta_i$ dans $H^1_{\rm dR}(X)$
ne d\'epend pas de $u$. De plus, $\iota^\dual\omega_1,\dots,\iota^\dual\omega_g,\eta_1,\dots,\eta_g$
forment une base de $H^1_{\rm dR}(X)$, ce qui nous fournit un scindage de la filtration de Hodge
(cf.~\cite[prop.\,II.2.4]{Cz-inte}).

Si $u$ est g\'en\'eral, soit
$$\beta_u=\iota^\dual( dG(x\ominus u)).$$
Si $u$ et $v$ sont g\'en\'eraux, alors $\beta_u-\beta_v$ est une forme diff\'erentielle
de troisi\`eme esp\`ece dont le diviseur ${\rm Div}(\beta_u-\beta_v)$ est $X_u-X_v$.
De plus, les $\beta_u-\beta_v$ engendrent un suppl\'ementaire
de $H^0(X,\Omega^1)$ dans l'espace des formes diff\'erentielles de troisi\`eme esp\`ece~\cite[prop.\,II.2.9]{Cz-inte}.
Toute forme $\beta$ de troisi\`eme esp\`ece peut donc s'\'ecrire sous la forme
$\beta=\beta^0+\sum_u n_u\beta_u$, o\`u $\beta^0\in H^0(X,\Omega^1)$ est uniquement d\'etermin\'ee,
$\sum_un_n=0$, et les $u$ sont des points g\'en\'eraux de $J$ tels que
${\rm Div}(\beta)=\sum n_u X_u$.
On a alors (cf.~\cite[th.\,II.2.11]{Cz-inte}):
$$\log_{\tJ}(\beta)=\beta^0+\sum_u n_u\sum_{i=1}^g\lambda_i(u)\eta_i=\beta^0+\sum_{i=1}^g\lambda_i({\rm Div}(\beta))\eta_i.$$
Remarquons que $G$ n'est unique qu'\`a addition pr\`es d'un polyn\^ome de degr\'e~$\leq 2$ en les
$\lambda_i$.  Il s'ensuit que, {\it si $\beta$ est donn\'ee, on peut choisir $G$ de telle sorte
que $\beta^0=0$}.

\subsubsection{Fin de la preuve du th.~\ref{inte1}}\label{BAS26}
Soit donc $v=(v_n)_{n\in\Z}\in\hJ$.
Le groupe $\hJ$ est un $\Q_p$-espace vectoriel; on pose $\tilde v_n=\iota_{\C_p}(p^{-n}v)$, si $n\in\Z$.
Alors $(\tilde v_n)_{n\in\N}$ est une suite born\'ee de points de $\tJ$, et on a
$\pi_J(\tilde v_n)=v_n$ pour tout $n$.

Soit $u\in J$ g\'en\'eral, tel que $X_u\subset \cup_i D(P_i,1)^-$,
et tel que $u\oplus v_n$ soit g\'en\'eral pour tout~$n$.
Comme $v_n\to 0$ quand $n\to -\infty$, on a $X_{u\oplus v_{-n}}\subset \cup_i D(P_i,1)^-$, si $n\geq n_0$.
Dans la suite, on suppose que l'on peut prendre $n_0=0$, et que $\lambda_i(v_0)$ est
suffisamment petit pour que le reste du
d\'eveloppement de Taylor de $G$ en $x\ominus u$
$$G(x\ominus u\ominus p^n\cdot v_0)-G(x\ominus u)+\sum_{i=1}^g\partial_iG(x\ominus u)\lambda_i(v_n)$$ soit divisible
par $p^{2n+2}$ sur $Y_\delta$, pour tout $n\in\N$
(c'est possible
 en rempla\c{c}ant $v$ par $p^{N}v$,
ce qui ne fait que tout multiplier par $p^{N}$,
gr\^ace \`a~\cite[Lemme~II.2.6]{Cz-inte} et au fait que $\lambda_i(v_n)=p^n\lambda_i(v_0)$.)

On note $\beta_n$ la forme diff\'erentielle de troisi\`eme esp\`ece, de diviseur $X_{u\oplus v_n}-X_u$,
dont l'image dans $\tJ$ est $\tilde v_n$.
On a alors
$$\beta_n=\beta_{u\oplus v_n}-\beta_u+\beta_n^0,\quad {\text{avec $\beta_n^0\in H^0(X,\Omega^1),$}}$$
et donc
$$\log_{\tJ}(\iota_{\C_p}(v))=p^n\log_{\tJ}(\beta_n)=p^n\big(\beta_n^0+\sum_{i=1}^g\lambda_i(v_n)\eta_i\big),
\quad{\text{pour tout $n\in\Z$.}}$$
Comme $p^n\lambda_i(v_n)$ ne d\'epend pas de $n$, il en est de m\^eme de $p^n\beta_n^0$
et, quitte \`a modifier $G$ par un polyn\^ome de degr\'e~$\leq 2$ en les $\lambda_i$,
{\it on peut supposer que $\beta_n^0=0$ pour tout $n$}.

Si $n\geq 0$, le diviseur $p^{2n}X_{u\oplus v_n}-X_{u\oplus v_{-n}}-(p^{2n}-1)X_u$ est principal,
puisque $p^{2n}\cdot(u\oplus v_n)\ominus(u\oplus v_{-n})\ominus(p^{2n}-1)\cdot u=0$.
C'est donc le diviseur d'une fonction $F_n$, rationnelle sur $X$.
\begin{lemm}\label{inte11}
{\rm (i)} Il existe $f_n\in\O(Y_\delta)-\{0\}$ telle que $f_n^{p^n}=F_n$.

{\rm (ii)} $f_n\in A_{p,n}(Y_\delta)$, et $(f_n)_{n\in\N}$ d\'efinit un \'el\'ement de
$H^1_{\eet}({Y_\delta},\Z_p(1))$ dont l'image dans $\widehat J$ est $v$.
\end{lemm}
\begin{proof}
On a
\begin{align*}
p^{2n}X_{u\oplus v_n}&-X_{u\oplus v_{-n}}-(p^{2n}-1)X_u\\
=&\ p^n(p^{n}X_{u\oplus v_n}-X_{u\oplus v_{0}}-(p^{n}-1)X_u)+
(p^{n}X_{u\oplus v_0}-X_{u\oplus v_{-n}}-(p^{n}-1)X_u)
\end{align*}
Comme $p^{n}X_{u\oplus v_n}-X_{u\oplus v_{0}}-(p^{n}-1)X_u$ est principal, pour prouver le (i),
il suffit de prouver que la fonction $F'_n$, de diviseur
$p^{n}X_{u\oplus v_0}-X_{u\oplus v_{-n}}-(p^{n}-1)X_u$, est une puissance $p^n$-i\`eme sur ${Y_\delta}$.
Or on a
$$\log F'_n = p^n \big(G(x\ominus u\ominus v_0)-G(x\ominus u)\big)-\big(G(x\ominus u\ominus p^n\cdot v_0)-G(x\ominus u)\big),$$
et notre hypoth\`ese sur le d\'eveloppement de Taylor de $G$ en $x\ominus u$
implique que
$\log F'_n$ est divisible par $p^{n+1}$, ce qui permet de d\'emontrer le (i).
De plus, le diviseur de $f_n$ sur ${Y_\delta}$ est $p^nX_{u\oplus v_n}$, ce qui prouve que $f_n\in A_{p,n}(Y_\delta)$.

Le diviseur de $F_{n+1}/F_n^p$ est
$$p^{2n+1}\big(pX_{u\oplus v_{n+1}}-X_{u\oplus p\cdot v_{n+1}}-(p-1)X_u\big)
-(X_{u\oplus p^{n+1}v_0}-X_u)+p(X_{u\oplus p^n\cdot v_0}-X_u).$$
(On a utilis\'e les formules $p\cdot v_{n+1}=v_n$, $v_{-n-1}=p^{n+1}\cdot v_0$ et $v_{-n}=p^n\cdot v_0$.)
Maintenant, $pX_{u\oplus v_{n+1}}-X_{u\oplus p\cdot v_{n+1}}-(p-1)X_u$ est principal
et $$G(x\ominus(u\oplus p^{n+1}v_0))-G(x\ominus u)-p\big(G(x\ominus(u\oplus p^{n}v_0))-G(x\ominus u)\big)$$
est divisible par $p^{2n+3}$ sur ${Y_\delta}$ pour les m\^emes raisons que ci-dessus.
On en d\'eduit que $F_{n+1}/F_n^p$ est une puissance $p^{2n+1}$-i\`eme sur ${Y_\delta}$, et donc que
$f_{n+1}/f_n$ est une puissance $p^n$-i\`eme.
Il s'ensuit que $(f_n)_{n\in\N}$ d\'efinit un \'el\'ement de
$H^1_{\eet}({Y_\delta},\Z_p(1))$.

Enfin, $p^{-n}{\rm Div}(f_n)=X_{u\oplus v_n}$ a pour image $u\oplus v_n$ dans ${\rm Pic}({Y_\delta})$,
mais comme $u$ est dans le sous-groupe $H$ de $J$ tel que ${\rm Pic}({Y_\delta})=J/H$, cette image
est aussi $v_n$, ce qui permet de conclure.
\end{proof}

Revenons \`a la d\'emonstration du th\'eor\`eme.
On a $\frac{df_n}{f_n}=p^n\beta_n-p^{-n}\beta_{-n}$, et comme
$$p^{-n}\beta_{-n}=p^{-n}\big(\iota^\dual dG(x\ominus u\ominus p^n\cdot v_0)-\iota^\dual dG(x\ominus u)\big)=
-\sum_{i=1}^g\lambda_i(v_0)\eta_{u,i}\quad {\rm mod}~p^n,$$
 on voit que
$\frac{df_n}{f_n}$ tend vers $\sum_{i=1}^g\lambda_i(v_0)\eta_{u,i}$, et donc:
$${\rm dlog}((f_n)_{n\in\N})=\sum_{i=1}^g\lambda_i(v_0)\eta_{u,i}
\quad{\rm et}\quad
\pi_{\rm dR}\circ {\rm dlog}((f_n)_{n\in\N})=\sum_{i=1}^g\lambda_i(v_0)\eta_{i}=
\log_{\tJ}(\iota_{\C_p}(v)).$$
On en d\'eduit la commutativit\'e du diagramme dans le cas de $\iota_{\C_p}$.

Pour traiter le cas de $\iota_{\bdr}$, on pose $\tilde v_{n,{\rm dR}}=\iota_{\bdr}(p^{-n}v)$
et $v_{n,{\rm dR}}=\pi_J(\tilde v_{n,{\rm dR}})$
On a donc $\theta(\tilde v_{n,{\rm dR}})=\tilde v_n$ et $\theta(v_{n,{\rm dR}})=v_n$.  
Ensuite, on d\'efinit
$\beta_{n,{\rm dR}}$ comme \'etant $\beta_{u\oplus v_{n,{\rm dR}}}-\beta_u$ en ayant choisi
$G$ pour que l'image de $\beta_{n,{\rm dR}}$ dans $\tilde J(\bdr^+)$ soit $\tilde v_{n,{\rm dR}}$.
Alors
$$\log_{\tJ}(\iota_{\bdr}(v))=p^n\sum_{i=1}^g\lambda_i(v_{n,{\rm dR}})\eta_i,
\quad{\text{pour tout $n\in\Z$.}}$$
On d\'efinit $F_{n,{\rm dR}}$ par ${\rm Div}(F_{n,{\rm dR}})=
p^{2n}X_{u\oplus v_{n,{\rm dR}}}-X_{u\oplus v_{-n,{\rm dR}}}-(p^{2n}-1)X_u$,
et on prouve en reprenant la d\'emonstration du lemme~\ref{inte11} 
que $F_{n,{\rm dR}}=f_{n,{\rm dR}}^{p^n}$
(alors $f_{n,{\rm dR}}$ est un rel\`evement de $f_n$).
Enfin, le m\^eme calcul que ci-dessus montre que $\frac{df_{n,{\rm dR}}}{f_{n,{\rm dR}}}$
tend vers $\sum_{i=1}^g\lambda_i(v_{0,{\rm dR}})\eta_{u,i}$ et donc que son image
dans $H^1_{\rm dR}(Y^\dagger_{\bdr})$ est
$\sum_{i=1}^g\lambda_i(v_{0,{\rm dR}})\eta_{i}=\log_{\tJ}(\iota_{\bdr}(v))$,
ce que l'on voulait.
\end{proof}

\section{Cohomologie des courbes non propres, sans bord}\label{CI67}
Soit $Y$ une courbe non propre, sans bord.
Le th.~\ref{paire00} ci-dessous d\'ecrit la cohomologie pro\'etale $p$-adique
en termes du complexe de de Rham et de la cohomologie de Hyodo-Kato
et, dans le cas o\`u $C=\C_p$, le th.~\ref{ella1} compare la cohomologie
de Hyodo-Kato et la cohomologie pro\'etale $\ell$-adique pour $\ell\neq p$.
Dans le \S\,\ref{affin16}, on explique un autre point de vue, utilisant la g\'eom\'etrie
rigide comme dans~\cite{CI}, pour d\'evoiler les {\og structures cach\'ees sur les courbes $p$-adiques\fg}.

\Subsection{Cohomologie pro\'etale $p$-adique}\label{CI68}
\begin{theo}\label{paire00}
Si $Y$ est une courbe non propre, sans bord,
on a le diagramme commutatif fonctoriel
 de fr\'echets suivant:
$$
\xymatrix@R=.6cm@C=.5cm{
0\ar[r] & \O(Y)/C\ar[r]\ar@{=}[d] & H^1_{\proet}(Y,\Q_p(1))\ar[d] \ar[r] &
(\bst^+\wotimes_{\breve C} H^1_{\rm HK}(Y))^{N=0,\varphi=p}\ar[r]\ar[d]^{\theta\otimes\iota_{\rm HK}} & 0\\
0\ar[r]& \O(Y)/C \ar[r]^-d & \Omega^1(Y) \ar[r] & H^1_{\rm dR}(Y)\ar[r] & 0
}
$$
dans lequel les lignes sont exactes et les fl\`eches verticales sont d'image ferm\'ee.
\end{theo}
\begin{proof}
On peut \'ecrire $Y$ comme, au choix, une r\'eunion croissante stricte
d'affino\"{\i}des $Y_n$ ou d'affino\"{\i}des surconvergents $Y_n^\dagger$.
Le th\'eor\`eme se d\'eduit donc, par passage \`a la limite projective
du th.~\ref{basic40} ou du th.~\ref{basic47} (on a ${\rm R}^1\lim\O(Y_n)=0$ 
d'apr\`es le th\'eor\`eme de Kiehl). 

L'image de $\gamma:(\bst^+\wotimes H^1_{\rm HK}(Y))^{N=0,\varphi=p}\to
H^1_{\rm dR}(Y)$ est ferm\'ee car $\gamma$ est la limite projective 
des $\gamma_n:(\bst^+\wotimes H^1_{\rm HK}(Y_n^\dagger))^{N=0,\varphi=p}\to
H^1_{\rm dR}(Y_n^\dagger)$ dont les images sont ferm\'ees comme nous l'avons d\'ej\`a vu.
Il en r\'esulte que 
$H^1_{\proet}(Y,\Q_p(1))\to \Omega^1(Y)$
est d'image ferm\'ee.
\end{proof}

\begin{rema}\label{paire000}
(i) En passant \`a la limite dans le (ii) du th.\,\ref{Aff1} ou dans le cor.\,\ref{basic45},
on prouve que $H^1_{\rm dR}(Y)$ admet une filtration 
dont les quotients successifs sont:
$$H^1_{\rm dR}(Y)=\big[\xymatrix@C=.3cm{
H^1(\Gamma,C)\ar@{-}[r]&\prod_{s\in\Sigma}H^1_{\rm dR}(Y^\lozenge_s)
\ar@{-}[r]& H^1_c(\Gamma,C)^\dual}\big],$$
o\`u $\Gamma=\Gamma^{\rm an}(Y)$ et $\Sigma=\Sigma(Y)$.

(ii)
En passant \`a la limite dans le (ii) de la rem.\,\ref{BAS12.4} ou dans la prop.\,\ref{basic45.1},
on prouve que $H^1_{\rm HK}(Y)$ admet une filtration stable par $\varphi$
dont les quotients successifs sont:
$$H^1_{\rm HK}(Y)=\big[\xymatrix@C=.3cm{
H^1(\Gamma,{\breve C})\ar@{-}[r]&\prod_{s\in\Sigma}H^1_{\rm cris}(Y^{\rm sp}_s)
\ar@{-}[r]& H^1_c(\Gamma,{\breve C})^\dual}(-1)\big].$$
L'op\'erateur de monodromie s'obtient par passage \`a la limite et est aussi
celui obtenu via la rem.\,\ref{monod}.

(iii)
En passant \`a la limite dans le (i) du th.\,\ref{basic21},
on en d\'eduit un isomorphisme
$$\iota_{\rm HK}:C\otimes_{\breve C}H^1_{\rm HK}(Y)\cong H^1_{\rm dR}(Y).$$
\end{rema}

\Subsection{Fibre sp\'eciale d'une courbe sans bord}\label{bTTT25}
Soit $Y$ une courbe sans bord.
On va d\'efinir ce qui correspond {\og \`a la fibre sp\'eciale du mod\`ele semi-stable minimal de $Y$\fg}.
Comme une courbe analytique n'a pas forc\'ement de mod\`ele semi-stable minimal, la d\'efinition
qui suit est un peu ad hoc, mais a la vertu de fournir un objet sur lequel ${\rm Aut}_CY$
agit (et m\^eme ${\rm Aut}_KY$ si $Y$ est d\'efinie sur $K$).  L'id\'ee est de partir de la
fibre sp\'eciale de n'importe quel mod\`ele semi-stable de $Y$ et de contracter tous les $\piqp$
contractibles.  Ce faisant, on peut tomber sur un point, ce qui demande
de modifier un peu la notion de courbe marqu\'ee du \no~\ref{TTT26}.
\subsubsection{Courbes marqu\'ees}\label{bTTT26}
Soit $X$ une courbe sur $k_C$ (i.e. une r\'eunion d\'enombrable, localement finie,
de courbes irr\'eductibles, ou bien un point ou un {\og point double\fg}).  
Si $X$ n'est pas un point ou un point double,
un {\it point marqu\'e} sur $X$ est un couple
$(P,\mu(P))$, o\`u
$P\in X(k_C)$, et
{\it la multiplicit\'e $\mu(P)$}
de $P$ est un \'el\'ement de\footnote{Comme nos courbes sont sans bord, leur
fibre sp\'eciale n'a pas de points de multiplicit\'e $0^+$.} $\R_+^\dual\sqcup\{\infty\}$.

\smallskip
Une {\it courbe marqu\'ee} $(X,A)$ est une courbe $X$ munie d'un ensemble $A$ de points marqu\'es.
Une courbe marqu\'ee $(X,A)$ est {\it semi-stable} si:

$\bullet$ $X$ est \`a singularit\'es nodales,

$\bullet$ les composantes irr\'eductibles de $X$ sont propres et ne comportent qu'un
nombre fini de points marqu\'es,

$\bullet$ les points singuliers de $X$ sont marqu\'es et leur multiplicit\'e 
appartient \`a~$\Q_+^\dual$. 

Elle est {\it stable} si elle est semi-stable et si, de plus, aucune composante
connexe de $X$ n'est un $\piqp$ avec $0$, $1$ ou $2$ points marqu\'es.

\subsubsection{Contractions de $\piqp$}
A partir d'une courbe marqu\'ee $(X,A)$ semi-stable, on fabrique une courbe marqu\'ee stable,
en contractant tous les $\piqp$ ayant $0$, $1$ ou $2$ points marqu\'es (compt\'es
avec multiplicit\'e):

$\bullet$  Si un $\piqp$ n'a pas
de point marqu\'e ou a un unique point marqu\'e qui est lisse, c'est que la courbe est r\'eduite
 \`a ce $\piqp$, et quand on le contracte on obtient un point 

$\bullet$ Si un $\piqp$ a un unique point marqu\'e qui est singulier,
quand on contracte le $\piqp$ 
le point marqu\'e devient un point
lisse (que l'on d\'emarque) si ce point est le point d'intersection avec une autre composante
irr\'eductible ou bien le $\piqp$ devient un point double si ce $\piqp$ a de l'autointersection.

$\bullet$  Si un $\piqp$
a deux points marqu\'es $P_1$ et $P_2$, 
on marque le point $Q$ obtenu
en contractant le $\piqp$ en posant\footnote{Avec la convention $a+\infty=\infty$.}
$\mu(Q)=\mu(P_1)+\mu(P_2)$.
\begin{rema}\label{gam5}
{\rm (i)}  Les r\`egles pr\'ec\'edentes sont dict\'ees par ce qui se passe
si on a deux triangulations $S$ et $S'$ d'une courbe analytique $Y$, telles que
$S'=S\sqcup\{s_0\}$.  On passe de la fibre sp\'eciale de $Y_{S'}$
\`a celle de $Y_S$ en contractant le $\piqp$ correspondant \`a $s_0$ et il y a deux cas:
soit $s_0$ est de valence $1$ sur $\Gamma(S')$ et on obtient $\Gamma(S)$ \`a partir de $\Gamma(S')$
en retirant $s_0$ et l'ar\^ete qui part de $s_0$, soit $s_0$ est de valence $2$,
et on obtient $\Gamma(S)$ \`a partir de $\Gamma(S')$ en retirant $s_0$ et en fusionnant
les deux ar\^etes partant de $s_0$ en une seule (dont la longueur est la somme des
longueurs des deux ar\^etes).  Il suffit alors de traduire ce qui pr\'ec\`ede
en utilisant le fait que
$\Gamma(S)$ et $\Gamma(S')$ sont les graphes duaux des fibres sp\'eciales
de $Y_S$ et $Y_{S'}$ pour obtenir les r\`egles ci-dessus.

{\rm (ii)}
Pour trianguler une couronne correspondant \`a une ar\^ete non relativement compacte
de $\Gamma^{\rm an}(Y)$, il faut une infinit\'e de sommets et dans l'op\'eration ci-dessus
il peut y avoir une infinit\'e d'op\'erations \`a faire,
ce qui demande de passer \`a la limite dans certaines multiplicit\'es.

{\rm (iii)}
Une fois tous les $\piqp$ contract\'es,
on obtient une courbe stable ou un point (double, par convention, si le $\piqp$
que l'on contracte a un point singulier).
\end{rema}

\subsubsection{Fibre sp\'eciale}\label{bTTT28}
Si $Y$ est une courbe sans bord, on d\'efinit {\it la fibre sp\'eciale $Y^{\rm sp}$ de $Y$}
comme \'etant la courbe stable obtenue \`a partir de la fibre sp\'eciale
du mod\`ele semi-stable associ\'e \`a n'importe quelle triangulation $S$ de $Y$.
(Le r\'esultat ne d\'epend pas de la triangulation car, deux triangulations $S_1$, $S_2$
\'etant donn\'ees, on peut trouver une triangulation $S_3$ plus fine que $S_1$ et $S_2$,
et les courbes stables obtenues \`a partir de $S_1$ et $S_2$ sont isomorphes \`a celle
obtenue \`a partir de $S_3$.)

Les composantes irr\'eductibles de $Y^{\rm sp}$ sont les $Y_s^{\rm sp}$, pour $s\in \Sigma(Y)$.
Si $\Sigma(Y)=\emptyset$, on est dans un des cas suivants:

$\bullet$ $Y$ n'est pas une courbe de Tate et $Y^{\rm sp}$ est un point.

$\bullet$ $Y$ est une courbe de Tate et $Y^{\rm sp}$ est un point double.

\subsubsection{Fonctorialit\'e de la fibre sp\'eciale}\label{bTTT30}
Soit $f:X\to Y$ un morphisme de courbes analytiques (non constant).
Si $S$ et $T$ sont des triangulations de $X$ et $Y$ telles que $S\supset f^{-1}(T)$,
et si $X_S$ et $Y_T$ sont les mod\`eles semi-stables de $X$ et $Y$
associ\'es \`a $S$ et $T$, alors $f$ se prolonge, de mani\`ere unique, en
$f:X_S\to Y_T$; on note $f_0:X_S^{\rm sp}\to Y_T^{\rm sp}$
le morphisme induit sur les fibres sp\'eciales.

\begin{theo}\label{gam7}
Si $f:X\to Y$ est une morphisme de courbes analytiques, il
existe un unique morphisme de courbes $f^{\rm sp}:X^{\rm sp}\to Y^{\rm sp}$ tel
que, pour toutes triangulations $S$ et $T$ de $X$ et $Y$ telles que $S\supset f^{-1}(T)$,
le diagramme suivant soit commutatif:
$$\xymatrix@R=.6cm{X_S^{\rm sp}\ar[d]\ar[r]^-{f_0}&Y_T^{\rm sp}\ar[d]\\
X^{\rm sp}\ar[r]^-{f^{\rm sp}}&Y^{\rm sp}}$$
{\rm(Les fl\`eches verticales sont celles obtenues en contractant tous les $\piqp$ contractibles.)}
\end{theo}
\begin{proof}
Cela r\'esulte de ce que l'image d'un $\piqp$ priv\'e d'un nombre~$N$ de points
est un point ou bien un $\piqp$ priv\'e d'un nombre~$\leq N$ de points:
ceci implique que l'image d'un $\piqp$ contractible sur $X_S^{\rm sp}$ est
un point ou bien est contractible sur $Y_T^{\rm sp}$.
Il s'ensuit que, si $S$ et $T$ sont fix\'es, il existe $f^{\rm sp}_{S,T}$ unique faisant
commuter le diagramme.

Pour montrer que $f^{\rm sp}_{S,T}$ ne d\'epend pas de $S$ et $T$, il suffit,
si $S_1,S_2$ et $T_1,T_2$ sont des triangulations de $X$ et $Y$, de consid\'erer
une triangulation $T_3$ de $X$, plus fine que $T_1$ et $T_2$, et une triangulation
$S_3$ plus fine que $S_1,S_2$ et $f^{-1}(T_3)$, et de constater que
$f^{\rm sp}_{S_1,T_1}=f^{\rm sp}_{S_3,T_3}=f^{\rm sp}_{S_2,T_2}$.
\end{proof}

En appliquant le th.~\ref{gam7} \`a $Y=X$, on en d\'eduit le r\'esultat suivant.
\begin{coro}\label{gam8}
Le groupe ${\rm Aut}_C(Y)$ agit naturellement sur $Y^{\rm sp}$.
\end{coro}

\subsubsection{Les groupes ${\rm W}_Y$ et ${\rm WD}_Y$}\label{TTT32}
Soient $K$ une extension finie de $\Q_p$,
$q=|k_K|$ et $C=\C_p$. 
Si $Y$ est d\'efinie sur $K$, alors $ G_K$ agit sur ${\rm Aut}_CY$ par conjugaison, et le groupe
${\rm Aut}_KY$, qui agit naturellement sur $Y^{\rm sp}$,
est un produit: ${\rm Aut}_KY= G_K\times ({\rm Aut}_CY)^{ G_K}$.

Le frobenius g\'eom\'etrique $\varphi$ agit, lui aussi,
sur $Y^{\rm sp}$ comme sur toute vari\'et\'e de
caract\'eristique $p$.
Il en r\'esulte que
le sous-semi-groupe ${\rm Aut}_K(Y)\times\varphi^\N$ 
de ${\rm Aut}_K(Y)\times\varphi^\Z$ agit sur $Y^{\rm sp}$.
Il agit donc aussi sur le corps $k_C$ des constantes, et on note ${\rm W}_Y^+$
le sous-semi-groupe de ${\rm Aut}_K(Y)\times\varphi^\N$
des \'el\'ements agissant trivialement sur $k_C$, et ${\rm W}_Y$ le sous-groupe
de ${\rm Aut}_K(Y)\times\varphi^\Z$ engendr\'e par ${\rm W}_Y^+$.

\begin{rema}
{\rm (i)}
Le {\it groupe de Weil} ${\rm W}_K$ de $K$ est le sous-groupe
de $ G_K\times\varphi^\Z$ des $(g,\varphi^n)$ agissant trivialement
sur $k_C=\overline{\bf F}_p$. La projection $ G_K\times\varphi^\Z\to\varphi^\Z\overset{\sim}\to\Z$
d\'efinit une fonction degr\'e $\deg:{\rm W}_K\to f\Z$, o\`u $q=p^f$,
 et la projection $ G_K\times\varphi^\Z\to G_K$
fournit une injection $\iota:{\rm W}_K\hookrightarrow  G_K$: 
un \'el\'ement $g$ de $ G_K$ est de la forme $\iota(\gamma)$, avec $\gamma\in {\rm W}_K$
si et seulement si il existe $n\in\Z$ tel que $g$ agisse
par $x\mapsto x^{p^{n}}$ sur $\overline{\bf F}_p$, et alors $\deg\gamma=-n$.

{\rm (ii)}
Par construction, on dispose d'applications naturelles
$${\rm deg}:{\rm W}_Y\to\varphi^\Z\overset{\sim}\to\Z
\quad{\rm et}\quad
\iota:{\rm W}_Y\to {\rm Aut}_K(Y),$$
dont les restrictions \`a ${\rm W}_K\subset{\rm W}_Y$ sont les applications ci-dessus.
On a ${\rm W}_Y={\rm W}_K\times ({\rm Aut}_CY)^{ G_K}$, et comme
${\rm W}_K$ est dense dans $ G_K$, on en d\'eduit que
${\rm W}_Y$ {\it est dense dans} ${\rm Aut}_KY$.
\end{rema}
 
On note ${\rm WD}_Y$ le groupe alg\'ebrique produit semi-direct de ${\bf G}_a$
et ${\rm W}_Y$, le produit \'etant donn\'e par
 $(\lambda_1,g_1)(\lambda_2,g_2)=(\lambda_1+p^{\deg g_1}\lambda_2,g_1g_2)$, si $\lambda_1,\lambda_2\in {\bf G}_a$
et $g_1,g_2\in {\rm W}_Y$.
Le sous-groupe de ${\rm WD}_Y$ engendr\'e par ${\rm W}_K$ et ${\bf G}_a$ n'est autre que le groupe de Weil-Deligne
${\rm WD}_K$ de $K$.

Si $L$ est un corps, une {\it $L$-repr\'esentation de ${\rm WD}_Y$} est une repr\'esentation du groupe
de ses $L$-points, et est \'equivalente \`a la donn\'ee d'une $L$-repr\'esentation
de ${\rm W}_Y$ et d'un op\'erateur $N$ tel que $Ng=p^{\deg g}gN$.
Une telle repr\'esentation est dite {\it lisse} si tout \'el\'ement est fix\'e par un sous-groupe
ouvert du groupe d'inertie $I_K$ (naturellement un sous-groupe de ${\rm W}_K\subset {\rm W}_Y$).

Si $L_1$ et $L_2$ sont des corps et si $V_1$ et $V_2$ sont des repr\'esentations
lisses de ${\rm WD}_Y$ sur $L_1$ et $L_2$, on dit que {\it $V_1$ et $V_2$ sont similaires}
si elle deviennent isomorphes apr\`es extension des scalaires \`a un corps contenant $L_1$ et $L_2$.

\subsubsection{Lissit\'e de l'action de ${\rm Aut}_K(Y)$}
On suppose $\Sigma(Y)\neq\emptyset$.
\begin{theo}\label{HDR7}
Si $g\in{\rm Aut}_C(Y)$ agit trivialement sur $Y^{\rm sp}$, alors
$g$ agit trivialement sur $H^1_{\rm dR}(Y)$.
\end{theo}
\begin{proof}
On utilise la description de la rem.\,\ref{paire000}.
Comme $\Sigma(Y)\neq\emptyset$, le graphe $\Gamma=\Gamma^{\rm an}(X)$ est le
graphe dual de $Y^{\rm sp}$; il est donc fixe par $g$, puisque $g$ agit trivialement sur $Y^{\rm sp}$.
Comme $ g $ fixe $\Gamma$, il agit trivialement sur
$H^1(\Gamma,C)$ et sur $H^1_c(\Gamma,C)^\dual$,
et il suffit donc de prouver qu'il agit trivialement sur
$H^1_{\rm dR}(Y^\lozenge_s)$, si $s\in \Sigma$.
Autrement dit, on est ramen\'e au cas o\`u $Y$ est propre et a bonne r\'eduction.

On peut recouvrir $Y^{\rm sp}$ par des ouverts $U_i$ \'etales au-dessus de la droite affine
(et donc munis de $z_i\in \O(U_i)$ tel que $dz_i$ ne s'annule pas sur $U_i$),
et utiliser le recouvrement de $Y$ par les tubes $]U_i[$ des $U_i$ pour calculer
$H^1_{\rm dR}(Y)$.  Cela permet de se ramener au cas d'un affino\"{\i}de $X$ muni
de $z\in\O^+(X)$ tel que $\partial=\frac{d}{dz}$ soit une d\'erivation de $\O^+(X)$
(pour globaliser, on a aussi besoin de savoir que $H^1(Y,C)=0$, ce qui suit de ce que
toutes les intersections des $]U_i[$ sont non vides et donc que le nerf du recouvrement
est un simplexe).  L'hypoth\`ese selon laquelle $g$ fixe la fibre sp\'eciale
implique qu'il existe $r>0$ tel que $g^\dual z-z\in p^r\O^+(X)$.
Il s'ensuit que, si 
$\omega\in\Omega^1(X)$, alors $g^\dual\omega-\omega$
est la diff\'erentielle de 
$\sum_{n\geq 1}\frac{1}{n!}\partial^{n-1}(\frac{\omega}{dz})(g^\dual z-z)^n$,
et la s\'erie converge dans $\O(X)$ car, si $p^N\frac{\omega}{dz}\in \O^+(X)$, alors
$p^N\frac{1}{n!}\partial^{n-1}(\frac{\omega}{dz})\in\frac{1}{n}\O^+(X)$ (pour prouver ceci,
il suffit de v\'erifier que c'est vrai en restriction \`a toute boule r\'esiduelle, i.e.~au tube $]P[$
d'un point $P$ de la fibre sp\'eciale, et cela suit de ce que $z-z(P_0)$ est un param\`etre local de
$]P[$ pour tout choix de $P_0\in]P[$).
Ceci prouve que $g$ agit trivialement sur $H^1_{\rm dR}(X)$ et permet de conclure.
\end{proof}
\begin{rema}
Si $\Sigma(Y)=\emptyset$ et si $Y$ n'est pas une courbe de Tate, $H^1_{\rm dR}(Y)=0$
et le r\'esultat est encore vrai. Si $Y$ est une courbe de Tate, il faut convenir que la multiplication
par $-1$ n'agit pas trivialement sur le point double $Y^{\rm sp}$ si on veut que le r\'esultat
reste valable.
\end{rema}

\Subsection{Cohomologie (pro)\'etale $\ell$-adique}\label{TTT34}

Soit $Y$ une courbe sans bord 
et soit $\Gamma=\Gamma^{\rm an}(Y)$.
On peut \'ecrire $Y$ comme la r\'eunion croissante
d'affino\"{\i}des $Y_n$.
Alors $H^1_{\proet}(Y,\Q_\ell(1))=\varprojlim_n H^1_{\eet}(Y_n,\Q_\ell(1))$ (on a
$H^1_{\proet}(Y_n,\Q_\ell(1))=H^1_{\eet}(Y_n,\Q_\ell(1))$ puisque $Y_n$ est un affino\"{\i}de);
la limite ne d\'epend pas du choix des $Y_n$ car deux tels syst\`emes sont cofinaux.
Comme les $H^1_{\eet}(Y_n,\Q_\ell(1))$ sont des banachs,
$H^1_{\proet}(Y,\Q_\ell(1))$ est naturellement un fr\'echet.

\begin{theo}\label{ELL10}
Si $\ell\neq p$,
le groupe $H^1_{\rm proet}(Y,\Q_\ell(1))$ admet une filtration naturelle
dont les quotients successifs sont donn\'es par:
$$H^1_{\proet}(Y,\Q_\ell(1))=
\big[\xymatrix@C=.4cm{H^1(\Gamma,\Q_\ell)(1)\ar@{-}[r]
&\prod_{s\in\Sigma(Y)}H^1_{\eet}(Y_s,\Q_\ell(1))_0\ar@{-}[r] & H^1_c(\Gamma,\Q_\ell)^\dual}\big].$$
\end{theo}
\begin{proof}
C'est un r\'esultat classique~\cite[5.2.5]{Duc} ou~\cite{Duc1}; il se d\'eduit du th.\,\ref{ladique}
par passage \`a la limite.
\end{proof}

\Subsubsection{De la fibre sp\'eciale \`a la fibre g\'en\'erique}\label{TTT37}
On suppose que $Y$ est d\'efinie sur une extension finie $K$ de $\Q_p$,
et donc $H^1_{\proet}(Y,\Q_\ell(1))$ est muni d'une action de ${\rm Aut}_KY$.

\subsubsubsection{L'action de ${\rm WD}_Y$ sur $H^1_{\proet}(Y^{\rm sp},\Q_\ell(1)),$ $\ell\neq p$}\label{TTT38}
La fibre sp\'eciale $Y^{\rm sp}$ de $Y$ est munie d'une action de ${\rm W}_Y$
et donc sa cohomologie (log)pro\'etale $\ell$-adique aussi.
En tant que $\Q_\ell$-espace vectoriel, on a
$$H^1_{\proet}(Y^{\rm sp},\Q_\ell(1))=
H^1(\Gamma,\Q_\ell)(1)\oplus
\prod_{s\in\Sigma(Y)}H^1_{\eet}(Y^{\rm sp}_s,\Q_\ell(1))\oplus H^1_c(\Gamma,\Q_\ell)^\dual,$$
cette d\'ecomposition \'etant la d\'ecomposition par les poids de frobenius.
L'action de ${\rm W}_Y$ respecte chacun des trois espaces ci-dessus:
${\rm W}_Y$ agit sur $Y^{\rm sp}$ et donc aussi sur $\Gamma$ et, par suite
sur $H^1_c(\Gamma,\Q_\ell)^\dual$ et $H^1(\Gamma,\Q_\ell)$ (cette action est d'ailleurs
obtenue par extension des scalaires d'une action sur 
$H^1_c(\Gamma,\Q)^\dual$ et $H^1(\Gamma,\Q)$).
L'action sur $\prod_{s\in\Sigma(Y)}H^1_{\eet}(Y^{\rm sp}_s,\Q_\ell(1))$ est un produit
d'induites: si $S$ est un syst\`eme de repr\'esentants de $\Sigma(Y)$ modulo l'action
de ${\rm W}_Y$, et si ${\rm W}_s$ est le stabilisateur de $s\in S$ dans ${\rm W}_Y$, alors
$$\prod_{s\in\Sigma(Y)}H^1_{\eet}(Y^{\rm sp}_s,\Q_\ell(1))=
\prod_{s\in S}\big({\rm Ind}_{{\rm W}_s}^{{\rm W}_Y}H^1_{\eet}(Y^{\rm sp}_s,\Q_\ell(1))\big).$$ 
On peut naturellement enrichir cette action de ${\rm W}_Y$
en une action de ${\rm WD}_Y$ en d\'efinissant $N$ par la recette habituelle (rem.\,\ref{monod}).

\subsubsubsection{La recette de Fontaine}\label{TTT38.1}
La description ci-dessus, combin\'ee avec le th.~\ref{ELL10}, montre, en utilisant le fait que 
$$H^1_{\eet}(Y^{\rm sp}_s,\Q_\ell(1))
\cong H^1_{\eet}(Y_s,\Q_\ell(1))_0,$$
que $H^1_{\proet}(Y,\Q_\ell(1))\cong H^1_{\proet}(Y^{\rm sp},\Q_\ell(1))$
en tant que $\Q_\ell$-espaces.
Mais on a bien mieux: {\it les actions de ${\rm Aut}_KY$ et ${\rm WD}_Y$
sur $H^1_{\proet}(Y,\Q_\ell(1))$ et $H^1_{\proet}(Y^{\rm sp},\Q_\ell(1))$
se d\'eduisent l'une de l'autre} (th.~\ref{WD1} ci-dessous).
Rappelons que,
si $V$ est une $\Q_\ell$-repr\'esentation de ${\rm Aut}_KY$, on 
peut lui associer une repr\'esentation
${\rm WD}(V)$
de ${\rm WD}_Y$ par la recette suivante (de Fontaine~\cite{Fo-sst}).

$\bullet$ On choisit des syst\`emes $(\zeta_{\ell^n})_{n\in\N}$ et $(p^{1/\ell^n})_{n\in\N}$
de racines $\ell^n$-i\`emes de $1$ et $p$. 

$\bullet$ On note:

\quad $\diamond$ $\chi_\ell:G_K\to \Z_\ell^\dual$ 
le {\it caract\`ere cyclotomique}:
$g(\zeta_{\ell^n})=\zeta_{\ell^n}^{\chi_\ell(g)}$, si $n\in\N$, 

\quad $\diamond$ $c_\ell:G_K\to \Z_\ell$ le {\it cocycle
de Kummer} associ\'e \`a $p$: $g(p^{1/\ell^n})/p^{1/\ell^n}=\zeta_{\ell^n}^{c_\ell(g)}$.

$\bullet$ On fait agir ${\rm Aut}_KY$ sur $\Q_\ell[u]$ \`a travers son quotient $G_K$ qui agit
par $$g(u)=\chi_\ell^{-1}(g)(u+c_\ell(g)),$$ et on munit $\Q_\ell[u]$ de l'action
triviale\footnote{L'action de $\varphi$ est triviale sur la cohomologie \'etale $\ell$-adique
car c'est l'identit\'e sur les espaces topologiques.  Il n'en sera pas de m\^eme
pour la cohomologie de Hyodo-Kato.} de $\varphi$ et de la d\'erivation $N=\frac{d}{du}$.
On a $N\circ g=\chi_\ell(g)^{-1}g\circ N$, si $g\in {\rm Aut}_KY$ (et donc $N\circ g=p^{\deg g}g\circ N$, si $g\in{\rm W}_Y$).

$\bullet$
On pose alors
$${\rm WD}(V)=\varinjlim\nolimits_{[L:K]<\infty}(\Q_\ell[u]\otimes V)^{I_L}.$$
C'est une repr\'esentation de ${\rm Aut}_KY$ (agissant diagonalement)
munie d'actions de $\varphi$ (triviale) et $N$ via leurs actions sur le premier facteur.
En restreignant l'action de ${\rm Aut}_KY\times\varphi^\Z$ \`a ${\rm W}_Y$
cela fournit une $\Q_\ell$-repr\'esentation de ${\rm W}_Y$
(car $N\circ g=p^{\deg g}g\circ N$ sur $\Q_\ell[u]$, si $g\in{\rm W}_Y$),
qui est lisse par construction. 

Si $V$ est une limite projective $\varprojlim_nV_n$ de repr\'esentations de dimension finie 
$V_n$, on pose ${\rm WD}(V)=\varprojlim_n{\rm WD}(V_n)$.

\begin{rema}
Dans les deux cas, on retrouve $V$ (ou plut\^ot sa restriction au sous-groupe dense ${\rm W}_Y$ de ${\rm Aut}_KY$) 
\`a partir de ${\rm WD}(V)$ par la formule
$$V=(\Q_\ell[u]\wotimes{\rm WD}(V))^{N=0};$$
  il est donc \'equivalent de se donner $V$ ou ${\rm WD}(V)$.
\end{rema}
\begin{theo}\label{WD1}
Si $\ell\neq p$, on a une identification de ${\rm WD}_Y$-repr\'esentations:
$${\rm WD}(H^1_{\proet}(Y,\Q_\ell(1)))=H^1_{\proet}(Y^{\rm sp},\Q_\ell(1)).$$
\end{theo}
\begin{proof}
Cela suit, par passage \`a la limite, 
de ce que $H^1_{\proet}(Y_s,\Q_\ell(1))_0=H^1_{\eet}(Y_s^{\rm sp},\Q_\ell(1))$ 
(on est dans le cas de bonne r\'eduction) en tant que repr\'esentation de $W_s$, et de la
formule {\og de Picard-Lefschetz\fg} (rem.\,\ref{PL2}).
\end{proof}

\subsubsection{Comparaison entre les cohomologies \'etales $p$-adique et $\ell$-adique}
On suppose encore que $Y$ est d\'efinie sur une extension finie $K$ de $\Q_p$.
\begin{theo}\label{ella1}
Les repr\'esentations de ${\rm WD}_Y$ sur 
$H^1_{\proet}(Y^{\rm sp},\Q_\ell(1))$, pour $\ell\hskip.5mm{\neq}\hskip.5mm p$, 
et sur $H^1_{\rm HK}(Y)(1)$,
sont similaires.
\end{theo}
\begin{proof}
La filtration de la rem.\,\ref{paire000} pour $H^1_{\rm HK}(Y)$ admet un scindage
naturel par les poids comme $H^1_{\proet}(Y^{\rm sp},\Q_\ell)$, et ce scindage
est stable par ${\rm W}_Y$.  Il suffit donc de prouver que chacun des termes
fournit des repr\'esentations similaires de ${\rm W}_Y$ et que les op\'erateurs
de monodromie sont les m\^emes.

Sur les parties de la cohomologie ne d\'ependant que du graphe,
c'est clair puisqu'elles sont obtenues, par extension des scalaires, \`a partir d'une
m\^eme $\Q$-repr\'esentation. En particulier, les deux op\'erateurs de monodromie
sont les m\^emes.

Il reste \`a v\'erifier que l'action de ${\rm W}_Y$ est la m\^eme des deux c\^ot\'es
sur le reste.  Pour cela, choisissons
un syst\`eme $(Y_s^{\rm sp})_{s\in S}$ de repr\'esentants de
composantes irr\'eductibles de $Y^{\rm sp}$ modulo l'action de
${\rm W}_Y$ (notons que ${\rm W}_Y$ agit sur les composantes puisque $\varphi$ les laisse stables).  
Si $s\in S$, on note $W_s$ le stabilisateur de $Y^{\rm sp}_s$.
Alors l'action de $W_s$ sur $H^1_{\eet}(Y_s^{\rm sp},\Q_\ell)$ et $H^1_{\rm cris}(Y_s^{\rm sp})$
se factorise \`a travers ${\rm End}(J_s)$, o\`u $J_s$ est la jacobienne de
$Y_s^{\rm sp}$, et on conclut en utilisant le fait
que $H^1_{\eet}(Y_s^{\rm sp},\Q_\ell)$ et $H^1_{\rm cris}(Y_s^{\rm sp})$ sont
similaires en tant que repr\'esentations de ${\rm End}(J_s)$.
\end{proof}

\Subsection{Une d\'ecomposition de la cohomologie de de Rham}\label{affin16}
Dans tout le reste de ce chapitre, on suppose que $C=\C_p$.
Cela permet d'utiliser l'int\'egration sur les courbes (\cite{Cn85,CdS88}) qui repose
sur le fait qu'une puissance de frobenius devient lin\'eaire, ou~\cite{Cz-inte}
qui repose sur le fait que $J(C)/H$ est de torsion si $J$ est la jacobienne
d'une courbe et $H$ est un sous-groupe ouvert de $J(C)$).

Le but de ce \S\ est de faire le lien avec le point de vue de~\cite{CI} pour la d\'efinition
des groupes de cohomologie de Hyodo-Kato.
Nous allons d\'efinir
des sous-groupes $H^1_{\rm dR}(Y)_{\rm int}$,
$H^1_{\rm an}(Y,\C_p)$, $H^1_{\rm dR}(Y)_{{\rm log},{\cal L}}$
de $H^1_{\rm dR}(Y)$ (les deux premiers ne d\'ependent de rien mais le troisi\`eme
d\'epend du choix d'une branche du logarithme ou, ce qui revient au m\^eme,
de ${\cal L}=\log p\in \C_p$) et de prouver le r\'esultat suivant.

\begin{theo}\label{sympat7}
Si $Y$ est une courbe analytique connexe sans bord,
on a une d\'ecomposition fonctorielle
\begin{center}
$H^1_{\rm dR}(Y)=H^1_{\rm dR}(Y)_{\rm int}\bigoplus H^1_{\rm an}(Y,\C_p)\bigoplus
H^1_{\rm dR}(Y)_{{\rm log},{\cal L}},$
\end{center}
et des isomorphismes fonctoriels
\begin{align*}
H^1_{\rm an}(Y,\C_p)= H^1(\Gamma^{\rm an}({Y}),\C_p)
\quad &
H^1_{\rm dR}(Y)_{{\rm log},{\cal L}}=H^1_c(\Gamma^{\rm an}({Y}),\C_p)^\dual\\
H^1_{\rm dR}(Y)_{\rm int}\cong &\prod_{s\in \Sigma({Y})}H^1_{\rm rig}(Y^{\rm sp}_s/\O_{\C_p}).
\end{align*}
\end{theo}

\begin{rema}
(i)
On peut \'ecrire un affino\"{\i}de surconvergent comme une limite
projective de {\og wide opens\fg} pour lesquels le th.~\ref{sympat7} s'applique.
On en d\'eduit un scindage (d\'ependant de ${\cal L}$) de la filtration
du (ii) du cor.\,\ref{basic45}.

(ii) La d\'ependance de ce scindage par rapport au choix de ${\cal L}$ fait intervenir
l'op\'erateur de monodromie $N$ sur $H^1_{\rm dR}(Y)$ (cf.~prop.\,\ref{HDR3}).
\end{rema}

\subsubsection{D\'ecoupage associ\'e \`a une pseudo-triangulation}\label{TTT23}
Soient $S$ une pseudo-triangulation fine de $Y$, et $A$ l'ensembles des
ar\^etes de $\Gamma=\Gamma(S)$.  On fixe une orientation de $\Gamma$.
A partir de $S$, on fabrique un d\'ecoupage de $Y$ en couronnes analytiques (fibres g\'en\'eriques rigides
de jambes) et affino\"{\i}des avec bonne r\'eduction (fibres g\'en\'eriques rigides de shorts).
On pourrait reconstituer $Y$ en recollant ces morceaux le long de {\it cercles fant\^omes analytiques}
-- encore appel\'ees {\it couronnes d'\'epaisseur nulle}, l'anneau des fonctions analytiques
sur un tel objet est l'anneau de Robba $\varinjlim_{r>0}(\varprojlim_{s\in ]0,r]}\O(\{s\leq v_p(z)\leq r\}))$ --
mais on va plut\^ot recouvrir $Y$ par des 
{\it pantalons}\footnote{Un {\og basic wide open\fg} dans la terminologie de Coleman, i.e.~le compl\'ementaire d'un nombre fini de boules ferm\'ees dans une courbe propre ayant bonne r\'eduction.} 
se recollant le long de couronnes ouvertes.

Si $a\in A$, notons $Y_a$ la couronne (i.e.~la fibre g\'en\'erique d'une jambe) correspondant \`a $a$
et, si $s\in S$, notons $Y_s$ l'affino\"ide correspondant
\`a $s$,
et $Z_s$ le sous-espace analytique correspondant \`a l'\'etoile de sommet $s$;
on a donc $Z_s\moins Y_s=\sqcup_{a\in A(s)}Y_a$. 
Alors $Z_s$ est un pantalon, et $Z_s=Y_s\sqcup(\sqcup_{a\in A(s)}Y_a)$
est son d\'ecoupage en (fibre g\'en\'eriques rigides de) short et jambes.
Les $Z_s$ forment un recouvrement de $Y$, et 
comme $S$ est suppos\'ee fine: 

$\bullet$ $Z_{s_1}\cap Z_{s_2}$ est vide ou est une
couronne ouverte $Y_a$, avec $a\in A_c$, si $s_1\neq s_2$,

$\bullet$ $Z_{s_1}\cap Z_{s_2}\cap Z_{s_3}=\emptyset$ si $s_1,s_2,s_3$ sont deux \`a deux distincts.

On peut utiliser
ce recouvrement pour calculer la cohomologie de de Rham de $Y$ \`a la \v{C}ech.
Un $1$-cocycle est une collection 
$$\big((\omega_s)_{s\in S},(f_a)_{a\in A_c}\big),\quad
{\text{avec $\omega_s\in\Omega^1(Z_s)$, $f_a\in\O(Y_a)$, et
 $df_a=\omega_{s_2(a)}-\omega_{s_1(a)}$,}}$$
pour tous $s\in S$ et $a\in A_c$.
Un $1$-cobord est un $1$-cocycle de la forme
$$\big((dF_s)_{s\in S},(F_{s_2(a)}-F_{s_1(a)})_{a\in A_c}\big),
\quad{\text{avec $F_s\in\O(Z_s)$,}}$$ pour tout $s\in S$,
et $H^1_{\rm dR}(Y)$ est le quotient du groupe
$Z^1_{\rm dR}(S)$ des $1$-cocycles par celui $B^1_{\rm dR}(S)$ des $1$-cobords.

\smallskip
Cette description de $H^1_{\rm dR}(Y)$ permet d'introduire
un certain nombre de sous-groupes naturels.

\subsubsection{R\'esidus}\label{affin19}
Si $D$ est un disque ouvert, $H^1_{\rm dR}(D)=0$, et si $Y$ est une couronne ouverte
g\'en\'eralis\'ee $\{\alpha<v_p(z)<\beta\}$, 
alors $H^1_{\rm dR}(Y)$ est un $\C_p$-espace de dimension~$1$, engendr\'e
par la classe de $\frac{dz}{z}$. 
En particulier, 
si $a\in A$, $H^1_{\rm dR}(Y_a)$ est un $\C_p$-espace vectoriel de dimension~$1$
muni d'une base naturelle (au signe pr\`es; le signe est d\'etermin\'e par l'orientation de $a$).
Si $\omega\in\Omega^1(Y_a)$, on d\'efinit le {\it r\'esidu} ${\rm Res}_a\omega$ de $\omega$
comme la coordonn\'ee de l'image de $\omega$ dans cette base
de $H^1_{\rm dR}(Y_a)$.

Si $s\in S$, et si $a\in A(s)$, on d\'efinit le {\it r\'esidu} ${\rm Res}(\omega,a)$ de $\omega\in \Omega^1(Z_s)$
en $a$
comme celui de sa restriction \`a $Y_a$.  La fonction ${\rm Res}(\omega):A(s)\to \C_p$ ainsi
d\'efinie est dans le noyau de $\dbar_s:\C_p^{A(s)}\to \C_p^{\{s\}}$.
Plus pr\'ecis\'ement, si on compactifie $Z_s$ en une courbe propre $\overline Z_s$
en recollant des boules ouvertes $D_a$ le long des couronnes $Y_a$, pour $a\in A(s)$,
alors $Z_s$ et les $D_a$ forment un recouvrement de $\overline Z_s$, ce
qui fournit le diagramme commutatif, \`a lignes exactes, suivant:
$$\xymatrix@R=.6cm@C=.5cm{
\Omega^1(Z_s)
\bigoplus\big(\bigoplus\limits_{a\in A(s)} \Omega^1(D_a)\big)\ar[d]\ar[r]
&\bigoplus\limits_{a\in A(s)}\Omega^1(Y_a)\ar[d]^{\rm Res}\ar[r]& H^1(\overline Z_s,\Omega^1)\ar[d]^{\wr}\ar[r]&0\\
H^1_{\rm dR}(Z_s)\ar[r]^-{\rm Res}
&\C_p^{A(s)}\ar[r]^{\dbar_s}& \C_p^{\{s\}}\ar[r]&0}$$

\subsubsection{Formes localement log-exactes}\label{affin21}
 Une forme diff\'erentielle $\omega\in\Omega^1(Z_s)$ est {\it log-exacte}
si on peut l'\'ecrire sous la forme $df_0+\sum_{i=1}^r\lambda_i\frac{df_i}{f_i}$,
avec $f_0\in\O(Z_s)$ et $f_i\in\O(Z_s)^\dual$, si $1\leq i\leq r$.
Elle est {\it $\Q$-log-exacte} si elle est log-exacte et si on peut prendre
les $\lambda_i$ rationnels dans la d\'ecomposition ci-dessus (ce qui \'equivaut \`a
ce que ${\rm Res}_a\omega\in\Q$, pour tout $a\in A(s)$).
On note $\Omega^1_{\rm log}(Z_s)$ le sous-espace de $\Omega^1(Z_s)$ des
formes log-exactes, et $H^1_{\rm dR}(Z_s)_{\rm log}$
l'image de $\Omega^1_{\rm log}(Z_s)$ dans $H^1_{\rm dR}(Z_s)$.
On a alors un isomorphisme 
$$H^1_{\rm dR}(Z_s)_{\rm log}\cong {\rm Ker}\big[\dbar_s:\C_p^{A(s)}\to \C_p^{\{s\}}\big]$$
d\'ecoulant du
lemme.~\ref{paire12} ci-dessous.
\begin{lemm}\label{paire12}
Si $\phi\in
{\rm Ker}\big(\dbar_s:\Z^{A(s)}\to\Z^{\{s\}}\big)$, il existe $M\in\N$ et
$f\in\O(Z_s)^\dual$ tels que ${\rm Res}(\frac{df}{f})=M\phi$.
\end{lemm}
\begin{proof}
On peut compactifier $Z_s$ en une courbe propre $Z^\lozenge_s$
en recollant des boules ouvertes $D_a$ le long des $Y_a$ via le choix
d'un param\`etre
local $T_a$ de $Y_a$, de telle sorte que $D_a=\{v_p(T_a)>0\}$ et $Y_a=\{\mu(a)>v_p(T_a)>0\}$;
on note $P_a$ le centre de $D_a$.

Soit $J$ la jacobienne de $Z^\lozenge_s$.  Choisissons $a_0\in A$, et notons $P_{a_0}$
le centre du disque $D_{a_0}$ et $\iota:Z^\lozenge_s\to J$ l'application
$P\mapsto (P-P_{a_0})$.  Le sous groupe $H$ de $J(\C_p)$ engendr\'e
par $\iota (D'_{a_0})$, o\`u $D'_{a_0}=D_{a_0}\moins Y_{a_0}$, est ouvert;
on en d\'eduit, car $J(\overline{\bf F}_p)$ est de torsion, que
$J(\C_p)/H$ est de torsion.  Il existe donc $M\in\N$ et des $Q_j\in D'_{a_0}$
tels que $M\big(\sum_{a\in A}\phi(a)\iota(P_a)\big)+\sum_j\pm\iota(Q_j)=0$.
Mais alors $M\big(\sum_{a\in A}\phi(a)(P_a-P_{a_0})\big)+\sum_j\pm(Q_j-P_{a_0})$
est un diviseur principal et la fonction $f\in\C_p(Z^\lozenge_s)^\dual$
dont c'est le diviseur a les propri\'et\'es voulues.
\end{proof}

Un $1$-cocycle $\big((\omega_s)_{s\in S},(f_a)_{a\in A_c}\big)\in Z^1_{{\rm dR}}(S)$ 
est {\it localement log-exact}
si $\omega_s\in \Omega^1_{\rm log}(Z_s)$ pour tout $s\in S$.
Il est {\it globalement log-exact} s'il existe $f\in\O(Y)^\dual$ tel que $\omega_s=\frac{df}{f}$
pour tout $s\in S$, et $f_a=0$ pour tout $a\in A_c$.
On note $H^1_{\rm dR}(Y)_{\rm log}$ l'image dans
$H^1_{\rm dR}(Y)$ des $1$-cocycles localement log-exacts.

\subsubsection{Cohomologie analytique}\label{affin22}
${\rm Coker}\big[\partial:\C_p^S\to \C_p^{A_c}\big]$ s'identifie 
naturellement \`a un sous-espace
de $H^1_{\rm dR}(Y)_{\rm log}$: si $\phi\in \C_p^{A_c}$, on peut lui associer
le $1$-cocycle $((\omega_s)_{s\in S},(f_a)_{a\in A_c})$, avec $\omega_s=0$
et $f_a=\phi(a)$, qui est trivialement localement log-exact; ce cocycle est un cobord
si et seulement si $\phi$ est dans l'image de $\partial$.
Plus conceptuellement, si on note $H_{\rm an}^1(Y,\C_p)$ le groupe de cohomologie analytique,
on a
$$
{\rm Coker}\big[\partial:\C_p^S\to \C_p^{A_c}\big]= 
H^1(\Gamma,\C_p)=H_{\rm an}^1(Y,\C_p).$$

Si $\omega\in H^1_{\rm dR}(Y)$, on note ${\rm Res}(\omega)\in \C_p^A$
la fonction $a\mapsto {\rm Res}(\omega,a)$, o\`u ${\rm Res}(\omega,a)$
est le r\'esidu de la restriction de $\omega$ \`a $Y_a$.
On a 
$${\rm Res}(\omega)\in{\rm Ker}\big[\dbar:\C_p^A\to \C_p^S\big]=H_c^1(\Gamma,\C_p)^\dual,$$
et ${\rm Res}$ induit une suite exacte
$$0\to H^1(\Gamma,\C_p) \to H^1_{\rm dR}(Y)_{\rm log}\to H_c^1(\Gamma,\C_p)^\dual .$$

\subsubsection{$\log_{\cal L}$-cobords}\label{affin23}
Choisissons une branche $\log_{\cal L}$ du logarithme ($\log_{\cal L}$ est la branche du logarithme
d\'efinie par $\log_{\cal L}p={\cal L}$).
Un $1$-cocycle localement log-exact est un {\it $\log_{\cal L}$-cobord} s'il existe $F_s$, pour $s\in S$,
 de la forme
$F_s=f_{s,0}+\sum_{i=1}^r\lambda_i\log_{\cal L}f_{s,i}$, 
avec $f_{s,0}\in \O(Z_s)$, $f_{s,i}\in\O(Z_s)^\dual$ si $1\leq i\leq r$,
telles que l'on ait $\omega_s=dF_s$ et $f_a=F_{s_2(a)}-F_{s_1(a)}$ pour tous
$s\in S$ et $a\in A_c$.  
On note
$H^1_{\rm dR}(Y)_{{\rm log},{\cal L}}$ le sous-espace de
$H^1_{\rm dR}(Y)_{\rm log}$ engendr\'e par les $\log_{\cal L}$-cobords.

\begin{prop}\label{HDR1}
${\rm Res}$ induit un isomorphisme
$${\rm Res}:H^1_{\rm dR}(Y)_{{\rm log},{\cal L}}\overset{\sim}{\to} {\rm Ker}\,\dbar=H^1_c(\Gamma,\C_p)^\dual.$$
\end{prop}
\begin{proof}
Il r\'esulte du lemme~\ref{paire12} que l'on peut trouver 
$\omega_s$ de la forme $\sum_i\lambda_i\frac{df_i}{f_i}$, avec
$f_i\in \O(Z_s)^\dual$ dont le r\'esidu est $\phi_{|A(s)}$.
On obtient un $\log_{\cal L}$-cobord dont la restriction du r\'esidu \`a $A\moins A_c$
est $\phi$ en faisant la somme
des $(\omega_s,(F_{s,a})_{a\in A_c(s)})$ o\`u, si $a\in A_c(s)$, $F_{s,a}$ est la restriction \`a
$Y_a$ de $\pm\sum_i\lambda_i\log_{\cal L}f_i$
suivant que $s_2(a)=s$
ou $s_1(a)=s$: cela marche car $F_{s_2(a),a}-F_{s_1(a),a}\in\O(Y_a)$
bien que ni $F_{s_2(a),a}$ ni $F_{s_1(a),a}$ ne soient holomorphes sur $Y_a$ a priori.
\end{proof}

\Subsubsection{Formule {\og de Picard-Lefschetz\fg}}
\begin{rema}\label{HDR2}
{\rm (i)} La prop.~\ref{HDR1} prouve que 
$H^1_{\rm dR}(Y)_{\rm log}\to H_c^1(\Gamma,\C_p)^\dual$
est surjective et
fournit un scindage 
(d\'ependant de ${\cal L}$)
de la suite exacte 
$$0\to H^1(\Gamma,\C_p) \to H^1_{\rm dR}(Y)_{\rm log}\to H_c^1(\Gamma,\C_p)^\dual\to 0 .$$

{\rm (ii)}
L'application $\alpha_{\cal L}:H^1_{\rm dR}(Y)_{\rm log}\to H^1(\Gamma,\C_p)$
qui se d\'eduit de ce scindage est la suivante. 
Soit $((\omega_s)_{s\in S},(f_a)_{a\in A_c})$ un $1$-cocycle localement log-exact.
Il existe une primitive
$F_s$ de $\omega_s$ sur $Z_s$ de la forme $f_0+\sum_i\lambda_i\log_{\cal L}f_i$, avec $f_i\in\O(Z_s)^\dual$
(bien d\'efinie \`a addition d'une constante pr\`es).
Si $a\in A_c$, alors $c_a=f_a-(F_{s_2(a)}-F_{s_1(a)})$ est une constante, et
$(c_a)_{a\in A_c}$ d\'efinit un \'el\'ement de $H^1(\Gamma,\C_p)={\rm Coker}\,\partial$,
ce qui fournit une application 
$$\alpha_{\cal L}:H^1_{{\rm dR}}(Y)_{\rm log}\to H^1(\Gamma,\C_p)$$
qui est l'identit\'e sur $H^1(\Gamma,\C_p)$ de mani\`ere \'evidente.

\end{rema}

\begin{prop}\label{HDR3}
{\rm (i)} Si ${\cal L}_1,{\cal L}_2\in \C_p$, alors
$$\alpha_{{\cal L}_2}(\omega)=
\alpha_{{\cal L}_1}(\omega)-({\cal L}_2-{\cal L}_1)\,N_\mu({{\rm Res}}(\omega)).$$

{\rm (ii)}
Si $f\in\O(Y)^\dual$, alors $N_\mu({\rm Res}(\frac{df}{f}))=0$. 
\end{prop}
\begin{proof}
Soit $a\in A_c$.
On peut trouver des fonctions $x_{a,i}$, pour $i=1,2$, m\'eromorphes sur $Y_{s_i(a)}$, holomorphes
sur $Y_a$, telles que $v_{Z_{s_i(a)}}(x_{a,i})=0$, et $x_{a,1}x_{a,2}=\alpha_a$, avec
$v_p(\alpha_a)=\mu(a)$.

Si ${{\rm Res}}(\omega,a)=\kappa$,
alors 
$$F_{s_1(a)}=\kappa \log_{\cal L}x_{a,1}+G_1\quad{\rm et}\quad
F_{s_2(a)}=-\kappa \log_{\cal L}x_{a_2}+G_2$$ sur~$Y_a$, o\`u $G_1,G_2\in\O(Y_a)$,
et donc $F_{s_2(a)}-F_{s_1(a)}=-\kappa \log_{\cal L}\alpha_a+G_2-G_1$,
puisque $x_{a,1}x_{a,2}=\alpha_a$.  La formule
$\log_{{\cal L}_2}\alpha_a=\log_{{\cal L}_1}\alpha_a
+({\cal L}_2-{\cal L}_1)\mu(a)$
permet de prouver le (i).

Si $((\omega_s)_{s\in S},(f_a)_{a\in A_c})=\frac{df}{f}$ est globalement log-exacte,
on peut poser $F_s=\log_{\cal L}f$, pour tout $s$, ce qui prouve que
$$\alpha_{\cal L}(\tfrac{df}{f})=0,$$
pour tout ${\cal L}$,
et permet de d\'eduire le (ii) du (i).
\end{proof}

\begin{rema}\label{HDR4}

  (i) $H^1_{\rm an}(Y,\C_p)$ a une $\Q$-structure naturelle fournie
par $H^1_{\rm an}(Y,\Q)$ qui co\"{\i}ncide avec la $\Q$-structure naturelle
sur ${\rm Ker}\,\partial$. De m\^eme, $H^1_{\rm dR}(Y)_{{\rm log},{\cal L}}$
a une $\Q$-structure naturelle fournie par les
$\log_{\cal L}$-cobords qui sont localement $\Q$-log-exacts,
et cette $\Q$-structure co\"{\i}ncide avec la $\Q$-structure naturelle
sur ${\rm Ker}\,\dbar$.
La formule d\'efinissant $N_\mu$ montre que $N_\mu$ respecte les $\Q$-structures.

 (ii) Pour une formule du m\^eme genre, dans le cas propre, cf. Mieda~\cite{Mied}. 
\end{rema}

\subsubsection{Cohomologie de de Rham int\'erieure}\label{affin20}
On d\'efinit la {\it cohomologie de de Rham int\'erieure}
de $Z_s$ comme le sous-groupe
$$
 H^1_{{\rm dR}}(Z_s)_{\rm int}={\rm Ker}\big[
H^1_{\rm dR}(Z_s)\to \oplus_{a\in A(s)} H^1_{\rm dR}(Y_a)\big].
$$
Ce groupe s'interpr\`ete aussi comme la cohomologie rigide de $Y^{\rm sp}_s$:
on note $Y_s^{{\rm sp},\times}$ la courbe $Y_s^{\rm sp}$ munie de la structure logarithmique induite par $Y^{\rm sp}$
et juste $Y_s^{\rm sp}$ la m\^eme courbe avec la structure logarithmique triviale. On a alors~\cite{GK}
des isomorphismes
$$H^1_{\rig}(Y_s^{{\rm sp},\times}/\O_{\C_p}^{\times})\cong {\C_p}\otimes_{K_0}
H^1_{\rig}(Y_s^{{\rm sp},\times}/\O_{K_0}^{\times}),
\quad
H^1_{\rig}(Y_s^{\rm sp}/\O_{\C_p})\cong {\C_p}\otimes_{K_0}H^1_{\rig}(Y_s^{\rm sp}/\O_{K_0}),$$
et
le diagramme commutatif suivant dans lequel les lignes sont exactes
(la premi\`ere ligne est juste la suite de Gysin):
$$
\xymatrix@C=.4cm@R=.6cm{
0\ar[r] &  H^1_{\rig}(Y_s^{\rm sp}/\O_{{\C_p}})\ar[r]\ar[d]^{\wr} & H^1_{\rig}(Y_s^{{\rm sp},\times}/\O_{\C_p}^{\times})\ar[r]\ar[d]^{\wr} 
& \oplus_{a\in A(s)}H^0_{\rig}(P_a/\O_{\C_p})\ar[r]\ar[d]^{\wr} & H^2_{\rig}(Y_s^{\rm sp}/\O_{\C_p})\ar[r]\ar[d]^{\wr} & 0\\
0\ar[r]& H^1_{{\rm dR}}(Z_s)_{\rm int}\ar[r]& H^1_{\rm dR}(Z_s)\ar[r]^-{{\rm Res}}
& {\C_p}^{A(s)}\ar[r]^-{\dbar_s}& {\C_p}^{\{s\}}\ar[r]& 0}
$$
(cf.~\cite{GK} pour le second isomorphisme vertical, le premier est fourni par la commutativit\'e du diagramme).

On a une d\'ecomposition:
$$H^1_{\rm dR}(Z_s)=H^1_{\rm dR}(Z_s)_{\rm int}\oplus H^1_{\rm dR}(Z_s)_{\rm log},$$
cons\'equence directe de l'exactitude de la seconde ligne du diagramme et de l'isomorphisme
$$H^1_{\rm dR}(Z_s)_{\rm log}\cong {\rm Ker}\big[\dbar_s:{\C_p}^{A(s)}\to {\C_p}^{\{s\}}\big].$$
\begin{lemm}\label{12.2}
Si $Y$ est un pantalon, et si $f\in\O(Y)^\dual$
v\'erifie ${\rm Res}\big(\frac{df}{f}\big)=0$, alors $\frac{df}{f}$ est exacte:
il existe $g\in\O(Y)$ tel que $dg=\frac{df}{f}$.
\end{lemm}
\begin{proof}
D\'ecoupons $Y$ en un short $Y_s$ et des jambes $Y_a$, pour $a\in A$.
Comme ${\rm Res}\big(\frac{df}{f}\big)=0$, cela signifie que $v_p(f)$ est constante
sur chacune des couronnes, et donc le minimum est atteint sur $Y_s$, et donc $v_p(f)$
est constante. Comme on peut diviser $f$ par une constante, on peut supposer
qu'il existe $x_0\in Y_s(C)$ tel que $f(x_0)=1$, 
et alors $v_Z(f)=0$ sur tout sous-affino\"{\i}de~$Z$ de~$Y$.  

Si $0<\delta<\inf_a\mu(a)$, et si $Y_a$ est la couronne $0<v_p(z_a)<\mu(a)$
(recoll\'ee \`a $Y_s$ le long du cercle fant\^ome en $v_p(z_a)=0$), on 
note $Y_\delta$ l'affino\"{\i}de r\'eunion de $Y_s$ et des couronnes
$0<v_p(a)\leq \mu(a)-\delta$.  Alors $d(Y_\delta,\partial Y)\geq \delta$,
et donc $v_{Y_\delta}(f-1)\geq\delta$ (prop.\,\ref{metr1}) 
puisque $x_0\in Y_\delta$ et $f(x_0)=1$.
Il s'ensuit que la s\'erie d\'efinissant $\log f$ converge dans $\O(Y_\delta)$ et, ceci
\'etant vrai pour tout $\delta$, que $\log f\in\O(Y)$.  On a alors $\frac{df}{f}=d(\log f)$,
ce qui permet de conclure.
\end{proof}

La th\'eorie de l'int\'egration de Coleman fournit un plongement naturel
de la cohomologie int\'erieure des $Z_s$ dans $H^1_{\rm dR}(Y)$.
Soit $s_{0}\in S$, et soit $\omega\in\Omega^1(Z_{s_0})$.
Une branche $\log_{\cal L}$ du logarithme \'etant fix\'ee,
l'int\'egration de Coleman permet de d\'efinir une primitive
$F_\omega$ de $\omega$ sur $Z_{s_0}$, localement analytique (pour la topologie $p$-adique,
pas pour la rigide), uniquement d\'etermin\'ee \`a addition pr\`es d'une constante.
De plus cette int\'egration est fonctorielle, et si $\omega=df_0+\sum_{i=1}^r\lambda_i\frac{df_i}{f_i}$
est log-exacte, alors 
$F_\omega=f_0+\sum_{i=1}^r\lambda_i\log_{\cal L} f_i$ (\`a constante pr\`es).
En particulier, il r\'esulte du lemme~\ref{12.2} que 
$\omega$ d\'efinit une classe de cohomologie int\'erieure
si et seulement si $F_\omega$ est holomorphe sur $Y_a$, pour tout $a\in A(s_0)$
(dans ce cas $F_\omega$ ne d\'epend pas du choix de ${\cal L}$, et $\omega_s$
est exacte si et seulement si $F_\omega$ est holomorphe sur $Y_s$).

Ceci permet, si $\omega\in H^1_{\rm dR}(Z_{s_0})_{\rm int}$,
 de d\'efinir la classe de $\omega$ dans $H^1_{\rm dR}(Y)$ comme la
classe du
cocycle $\big((\omega_s),(f_a)\big)$, avec $\omega_{s_0}=\omega$,
$\omega_s=0$ si $s\neq s_0$, $f_a=0$ si $a\notin A(s_0)$, et $f_a=\pm F_\omega$ 
suivant que $s_2(a)=s_0$
ou $s_1(a)=s_0$, si $a\in A(s_0)$.
Si on change $F_\omega$, cela modifie le cocycle pr\'ec\'edent par un cobord,
et donc
la formule ci-dessus d\'efinit une injection naturelle de 
$H^1_{{\rm dR}}(Z_s)_{\rm int}$ dans
$H^1_{\rm dR}(Y)$, pour tout $s\in S$.

\begin{lemm}\label{HDR5}
L'injection ci-dessus de $H^1_{{\rm dR}}(Z_s)_{\rm int}$ dans
$H^1_{\rm dR}(Y)$, pour tout $s\in S$, s'\'etend en une injection continue
de $\prod_{s\in S}H^1_{{\rm dR}}(Z_s)_{\rm int}$ dans $H^1_{\rm dR}(Y)$.
\end{lemm}
\begin{proof}
Si $(\omega_s)_{s\in S}$ est une collection de formes diff\'erentielles dont
les classes sont int\'erieures, la s\'erie des cocycles associ\'es converge
car elle ne fait intervenir, localement, que des sommes finies, et on obtient
ainsi une injection $\prod_{s\in S}H^1_{{\rm dR}}(Z_s)_{\rm int}\to H^1_{{\rm dR}}(Y)$
car les $Y_s$ sont disjoints, et que l'exactitude de $\omega_s$ sur $Y_s$
\'equivaut \`a l'exactitude de $\omega_s$.
\end{proof}

On d\'efinit la {\it cohomologie de de Rham int\'erieure} de $Y$
comme le sous-groupe 
$\prod_{s\in S}H^1_{{\rm dR}}(Z_s)_{\rm int}$
de $H^1_{\rm dR}(Y)$.

\begin{prop}\label{HDR6}
{\rm (i)} $H^1_{{\rm dR}}(Y)_{\rm int}$ ne d\'epend pas du choix de $S$:
$$H^1_{{\rm dR}}(Y)_{\rm int}=
\prod_{s\in \Sigma(Y)}H^1_{{\rm dR}}(Z_s)_{\rm int}.$$

{\rm (ii)}
On a une d\'ecomposition naturelle (d\'ependant de ${\cal L}$)
$$H^1_{\rm dR}(Y)=H^1_{\rm dR}(Y)_{{\rm log},{\cal L}}\oplus
H^1_{\rm an}(Y,{\C_p})\oplus H^1_{{\rm dR}}(Y)_{\rm int}.$$
\end{prop}
\begin{proof}
Le (i) suit de ce que $Y^{\rm sp}_s$ est de genre $0$
si $s\in S\moins \Sigma(Y)$.
Pour prouver le (ii),
on utilise l'isomorphisme 
$H^1_{\rm dR}(Y)_{{\rm log},{\cal L}}\cong {\rm Ker}\,\dbar$
pour tuer les r\'esidus, puis les classes int\'erieures pour tuer les $\omega_s$
et obtenir un \'el\'ement de $H^1_{\rm an}(Y,{\C_p})$.
\end{proof}

\subsubsection{Fonctorialit\'e}\label{affin24}
Soit $u:Y'\to Y$ un morphisme de courbes analytiques sans bord, et soient $S$ et $S'$ des triangulations
de $Y$ et $Y'$.
Si $\omega=\big((\omega_s),(f_a)\big)\in Z^1_{\rm dR}(S)$, on d\'efinit
$u^*\omega=\big((\omega_{s'}),(f_{a'})\big)\in Z^1_{\rm dR}(S')$,
en posant 
$$\omega_{s'}=\begin{cases}u^\dual\omega_s
&{\text{si $u(s')\in Z_s\moins\big(\sqcup_{a\in A_c(s)}Y_a\big)$,}}\\
u^\dual\omega_{s_1(a)} &{\text{si $u({s'})\in Y_a$ et $a\in A_c$.}}\end{cases}$$
$$f_{a'}=\begin{cases}-u^\dual f_a &{\text{si $u(Y_{a'})\subset Y_a$, 
$u(s_1(a'))={s_2(a)}$
et $u({s_2(a')})\neq {s_2(a)}$,}}\\
u^\dual f_a &{\text{si $u(Y_{a'})\subset Y_a$, $u({s_2(a')})={s_2(a)}$
et $u({s_1(a')})\neq {s_2(a)}$,}}\\
0&{\text{dans les autres cas.}}\end{cases}$$
Alors $u^*\omega$ est un cobord si $\omega$ en est un et 
$u^\dual:H^1_{\rm dR}(Y)\to H^1_{\rm dR}(Y')$ est l'application induite.

Il est alors imm\'ediat que:
$$u^\dual(H^1_{\rm an}(Y,\Q))\subset H^1_{\rm an}(Y',\Q)
\quad{\rm et}\quad
u^\dual(H^1_{\rm dR}(Y)_{{\rm log},{\cal L}})\subset H^1_{\rm dR}(Y')_{{\rm log},{\cal L}},$$
et la fonctorialit\'e de l'int\'egration de Coleman fournit l'inclusion
$$u^\dual(H^1_{\rm dR}(Z)_{{\rm int}})\subset H^1_{\rm dR}(Z')_{{\rm int}}.$$
On en d\'eduit la fonctorialit\'e la d\'ecomposition
$$H^1_{\rm dR}(Y)_{\rm log}=H^1_{\rm dR}(Y)_{\rm int}
\oplus H^1_{\rm an}(Y,{\C_p})\oplus H^1_{\rm dR}(Y)_{{\rm log},{\cal L}}.$$

\subsubsection{Cohomologie de Hyodo-Kato}\label{TTT50}
Ce r\'esultat permet de d\'ecrire directement l'isomorphisme de Hyodo-Kato: on d\'efinit
$H^1_{\rm HK}(Y)$ comme
le $\breve\C_p$-espace vectoriel
\begin{center}
$\big(\prod_{s\in \Sigma}H^1_{\rm rig}(Y^{\rm sp}_s/\O_{\breve\C_p})\big)
\bigoplus H^1(\Gamma,\breve\C_p)
\bigoplus H^1_c(\Gamma,\breve\C_p)^\dual(-1).$
\end{center}
que l'on munit:

$\bullet$ du frobenius naturel $\varphi$ sur chacun des facteurs,

$\bullet$ de l'op\'erateur de monodromie $N$ de la rem.\,\ref{monod},

$\bullet$ de
l'isomorphisme $\iota_{{\rm HK},{\cal L}}:K\otimes H^1_{\rm HK}(Y)\cong H^1_{\rm dR}(Y)$
(qui d\'epend du choix de~${\cal L}$), somme directe des isomorphismes
\begin{align*}
K\otimes H^1(\Gamma({\cal Y}),\breve\C_p)\cong H^1_{\rm an}(Y,K),&
\quad K\otimes H^1_c(\Gamma({\cal Y}),\breve\C_p)^\dual\cong H^1_{\rm dR}(Y)_{{\rm log},{\cal L}},\\
K\otimes H^1_{\rm rig}( {\cal Y}_s/\O_{\breve\C_p})& \cong H^1_{\rm dR}(Z_s)_{\rm int}.
\end{align*}

\begin{rema}\label{sympat8.1}
{\rm (i)} L'isomorphisme $\iota_{\rm HK}$ du (iii) de la rem.\,\ref{paire000} correspond
\`a $\iota_{{\rm HK},{\cal L}}$ pour ${\cal L}=0$.

{\rm (ii)} Les fonctorialit\'es de la d\'ecomposition de la cohomologie de de Rham et
de la cohomologie rigide impliquent que 
$Y\mapsto (H^1_{\rm HK}(Y),H^1_{\rm dR}(Y),\iota_{{\rm HK},{\cal L}})$
est fonctoriel.

{\rm (iii)} 
La d\'efinition de
$H^1_{\rm HK}(Y)$ en fournit une d\'ecomposition naturelle qui n'est autre
que la d\'ecomposition
par les poids de frobenius: $H^1_{\rm rig}( {\cal Y}_s/\O_{\breve\C_p})$ est le sous-espace de poids $1$,
$H^1(\Gamma({\cal Y}),K_0)$ celui de poids~$0$ et $H^1_c(\Gamma({\cal Y}),K_0)^\dual(-1)$ celui
de poids $2$.
\end{rema}

\begin{appendix}
\section{Plaidoyer pour un peu de mod\'eration}\label{appen1}
Dans cet appendice, on \'etudie les pathologies de la cohomologie \'etale
de la boule unit\'e ouverte, et on propose une piste pour les supprimer.

Soit $\ocirc{B}$ la boule unit\'e ouverte (i.e.~$\O(\ocirc{B})=\O_C[[T]]$).
On a d\'efini au \no\ref{BAS19.2} le complexe ${\rm Syn}(\ocirc{B},1)$ et calcul\'e
ses groupes de cohomologie $H^i_{\rm syn}(\ocirc{B},1)$ (prop.~\ref{basic30}). 
On va s'int\'eresser au lien entre les $H^i_{\rm syn}(\ocirc{B},1)$ et la cohohomologie
\'etale (\`a coefficients dans $\Z_p(1)$) de {\og la\fg} fibre 
g\'en\'erique\footnote{Comme nous l'expliquons au \S\,\ref{appen21}, il y a plusieurs objets
diff\'erents qui peuvent pr\'etendre \^etre {\og la\fg} fibre g\'en\'erique de $\ocirc{B}$. Dans le texte
principal, nous avons juste d\'efini $\O(\ocirc{B}^{\rm gen})$ sans sp\'ecifier dans quelle cat\'egorie
$\ocirc{B}^{\rm gen}$ vivait.}
$\ocirc{B}^{\rm an}$ de $\ocirc{B}$ (c'est
l'espace rigide associ\'e: si $r_n$ est une suite de rationnels
v\'erifiant $r_n>r_{n+1}$ et $\lim r_n=0$, alors $\ocirc B^{\rm an}$ est la r\'eunion croissante
des $B_{r_n}^{\rm an}$, o\`u $B_r^{\rm an}$ est la fibre g\'en\'erique
de la boule ferm\'ee $B_r=\{v_p(T)\geq r\}$, i.e.~$\O(B_r)=\O_C\langle \frac{T}{p^r}\rangle$).  
Notons que les r\'esultats qui suivent s'\'etendent verbatim aux couronnes ouvertes:
pour toutes ces histoires, une couronne ouverte se comporte
comme la r\'eunion de deux boules ouvertes attach\'ees en un point.

Si $r\in\Q$, on pose $\O(\widetilde B_r)=\acris\langle \frac{T}{\tilde p^r}\rangle$.
On a donc $\Omega^1(\widetilde B_r)=\O(\widetilde B_r)\frac{dT}{\tilde p^r}$.
On dispose du complexe
$${\rm Syn}(B_r,1):=
\xymatrix@C=1.8cm{F^1\O(\widetilde B_r) \ar[r]^-{(d,1-\frac{\varphi}{p})}&
\Omega^1(\widetilde B_r) \oplus \O(\widetilde B_r)
\ar[r]^-{(1-\frac{\varphi}{p})-d}& \Omega^1(\widetilde B_r)},$$
o\`u $\varphi(T)=T^p$.
Ce complexe calcule la cohomologie \'etale de la fibre g\'en\'erique de $B_r$: 
c'est clair pour $H^0$; pour $H^1$, c'est un cas particulier du cor.\,\ref{basic35.1} (ou du th.\,\ref{short11},
via le le cor.\,\ref{basic12}); pour $H^2$, cela d\'ecoule de ce que $H^2_{\rm syn}(B_r,1)=0$ (prop.\,\ref{boule1})
et de ce que
$H^2(B_r^{\rm an},\Z_p(1))=0$, d'apr\`es Berkovich~\cite[cor. 6.1.3]{Berk}.

\subsection{Groupe de Picard}\label{appen2}
Comme $C\langle\frac{T}{p^r}\rangle$ est, pour tout $r\in\Q$, un anneau principal, 
on a $${\rm Pic}(\ocirc{B}^{\rm an})=\frac{\{(f_n)_{n\in\N},\ f_n\in\O(B_{r_n})^\dual\}} 
{\{(g_{n+1}g_n^{-1})_{n\in\N},\ g_n\in\O(B_{r_n})^\dual\}}$$
De plus, on peut remplacer $\O(B_{r_n})^\dual$ par son sous-groupe des $f$ valant $1$ en $0$,
et tous les $f_n$, $g_n$, etc. qui vont intervenir v\'erifient cette propri\'et\'e.

\subsubsection{Le th\'eor\`eme de Lazard}\label{appen3}
Rappelons le r\'esultat fondamental de Lazard~\cite{Laz} concernant le groupe ${\rm Pic}(\ocirc{B}^{\rm an})$.
\begin{prop} \label{appen4}
{\rm (Lazard)} On a la dichotomie suivante:

$\bullet$ Si $C$ est sph\'eriquement complet, alors ${\rm Pic}(\ocirc{B}^{\rm an})=0$.

$\bullet$ Si $C$ n'est pas sph\'eriquement complet, alors ${\rm Pic}(\ocirc{B}^{\rm an})\neq 0$.
\end{prop}
Nous allons avoir besoin de pr\'eciser ce qui se passe dans le cas non sph\'eriquement complet.
Les r\'esultats suivants sont des variations sur~\cite[\S\,V, prop.\,6]{Laz}.

\subsubsection{Construction d'\'el\'ements non divisibles}\label{appen5}
Comme $C$ n'est pas sph\'eriquement complet, il existe des suites
$(D_n)_{n\in\N}$ strictement d\'ecroissantes de boules ferm\'ees, 
dont la valuation $-r_n$ tend vers~$0$,
et telles que $\cap_n D_n=\emptyset$.  On choisit une telle suite et $x_n\in D_n$ pour tout $n$, 
puis $a_n\in D_n\moins D_{n+1}$ v\'erifiant
$v_p(a_n-x_{n+1})>-r_n$; on a alors 
$D_n=B(a_n,-r_n)$
et $$ -r_{n+1}>v_p(a_{n+1}-a_n)> -r_n>v_p(a_n-a_{n-1})\cdots, \quad
\lim_{n\to\infty}r_n=0,\quad \cap_n B(a_n,-r_n)=\emptyset$$
Soit $u_n=1-(a_n-a_{n-1})T$. Comme $v_p(a_n-a_{n-1})>-r_n$, on a $u_n\in\O(B_{r_n})^\dual$.
\begin{prop}\label{appen6}
La classe de $(u_n)_{n\in\N}$ dans ${\rm Pic}(\ocirc{B}^{\rm an})$ n'est pas divisible par $p$ et, plus g\'en\'eralement,
celle de $(u_n^{p^k})_{n\in\N}$ n'est pas divisible par $p^{k+1}$.
\end{prop}
\begin{proof}
Suppons que cette classe est divisible par $p$.  Il existe alors $v_n,g_n\in \O(B_{r_n})^\dual$, tels
que $u_n=\frac{v_{n+1}}{v_n}g_n^p$, pour tout $n\in\N$.
Quitte \`a diviser $g_n,v_n$ par des constantes, on peut supposer que 
$$g_n=1+\alpha_n T+\cdots,\quad v_n=1+\nu_nT+\cdots,$$
et l'hypoth\`ese d'inversibilit\'e implique
$$v_p(\alpha_n)>-r_n,\quad v_p(\nu_n)>-r_n.$$
Soit $w_n=(1-(a_1-a_0)T)\cdots(1-(a_{n-1}-a_{n-2})T)$.  On a $u_n=\frac{w_{n+1}}{w_n}$
sur $B_{r_n}$.  On en d\'eduit que
$$\frac{w_{n+1}}{v_{n+1}}=\frac{w_n}{v_n}g_n^p.$$
Si $\frac{w_n}{v_n}=1+\beta_n T+\cdots$, la relation pr\'ec\'edente implique
$$\beta_{n+1}=\beta_n+p\alpha_n,\quad{\text{et donc $\beta_n-\beta_0\in p^{1-r_0}\O_C$, pour tout $n$.}}$$
Maintenant, on a $\beta_n=a_0-a_n-\nu_n$, et donc
$$a_0-\beta_0=a_n+\nu_n+(\beta_n-\beta_0)\in B(a_n,-r_n),\quad{\text{pour tout $n$.}}$$
On en d\'eduit que $a_0-\beta_0\in \cap_n B(a_n,-r_n)$, ce qui conduit \`a une contradiction
puisque cette intersection est vide par hypoth\`ese.  Cela prouve que cette classe n'est pas
divisible par $p$.

Le cas $k$ g\'en\'eral se traite de la m\^eme mani\`ere: il suffit d'\'elever tout \`a la puissance~$p^k$,
ce qui multiplie les coefficients de $T$ par $p^k$.
\end{proof}

\subsubsection{Le sous-groupe de torsion du groupe de Picard}\label{appen23}
Le but de ce num\'ero est de prouver le r\'esultat suivant.
\begin{prop}\label{appen24}
Le groupe ${\rm Pic}(\ocirc{B})$ est sans torsion.
\end{prop}
\begin{proof}
Il suffit de prouver que ${\rm Pic}(\ocirc{B})$ n'a pas de $p$-torsion.
Si $(u_n)_{n\in\N}\in\prod_{n\in\N}\O(B_{r_n})^\dual$, on note
$[(u_n)]$ sa classe dans ${\rm Pic}(\ocirc{B})$.
On pose $T_n=T/p^{r_n}$, et donc $T_{n+k}=p^{r_n-r_{n+k}}T_n$.
Alors $\O(B_{r_n})=\O_C\langle T_n\rangle$.

Nous aurons besoin du lemme suivant.
\begin{lemm}\label{appen25}
Soit $(u_n)_{n\in\N}\in\prod_{n\in\N}\O(B_{r_n})^\dual$.
S'il existe $\delta>0$ tel que $u_n\in 1+p^\delta T_n\O(B_{r_n})$ pour tout $n$
assez grand, alors $[(u_n)]=0$.
\end{lemm}
\begin{proof}
On peut \'ecrire $u_n=\prod_{k\geq 1} u_{n,k}$, avec $u_{n,k}=1+p^\delta a_{n,k}T_n^k$
et $a_{n,k}\in\O_C$ tend vers~$0$ quand $k\to\infty$ ($n$ \'etant fix\'e): l'existence
des $a_{n,k}\in\O_C$ est imm\'ediate et pour prouver que $a_{n,k}\to 0$, il suffit
de d\'evelopper, de regarder modulo~$p^{2\delta}$, $p^{3\delta}$, etc., et d'utiliser le
fait que $u_n$ est un polyn\^ome modulo~$p^{2\delta}$, $p^{3\delta}$, etc.

On choisit $N_k$ tel que $\delta-kr_n\geq 0$ si $n\geq N_k$.
On d\'efinit $v_{n,k}$ par $v_{n,k}=1$ si $n=N_k$ et
$v_{n+1,k}=u_{n,k}v_{n,k}$, si $n\in\N$.
De mani\`ere explicite, on a
$$v_{n,k}=\begin{cases}
\prod_{i=N_k}^{n-1}(1+p^\delta a_{i,k}p^{-kr_i}T^k) &{\text{si $n\geq N_k$,}}\\
\prod_{i=n}^{N_k-1}(1+p^\delta a_{i,k}p^{k(r_n-r_i)}T_n^k)^{-1} &{\text{si $n\leq N_k-1$.}}
\end{cases}$$
Comme $a_{n,k}\to 0$ quand $k\to\infty$, on voit que, si $0<\delta'<\delta$,
alors $v_{n,k}\in 1+p^{\delta'}T_n\O(B_{r_n})$,
pour tout $k$ (on a m\^eme $v_{n,k}\in 1+p^\delta T\O_C[[T]]$ si $N_k\leq n$), 
et que $v_{n,k}-1\to 0$ dans $p^{\delta'}T_n\O(B_{r_n})$ quand $k\to +\infty$.
Il s'ensuit que $v_n=\prod_{k\geq 1}v_{n,k}\in\O(B_{r_n})^\dual$, et comme
on a $u_n=\prod u_{n,k}=\prod(v_{n+1,k}/v_{n,k})=v_{n+1}/v_n$,
cela permet de conclure.
\end{proof}

Revenons \`a la preuve de la prop.\,\ref{appen24}. 
On suppose que $[(u_n^p)]=0$ et on veut en d\'eduire que $[(u_n)]=0$.

\'Ecrivons $u_n=1+\sum_{k\geq 1}a_{n,k}T_n^k$ et posons
$u_n^{(p)}=1+\sum_{k\geq 1}a_{n,k}^pT_n^{pk}$.
Alors $u_n^{(p)}/u_n^p\in 1+pT_n\O(B_{r_n})$, pour tout $n$, et il r\'esulte
du lemme~\ref{appen25} que $[u_n^{(p)}]=0$ puisque $[(u_n^p)]=0$.
On peut donc \'ecrire $u_n^{(p)}=w_{n+1}/w_n$ pour tout $n$, avec $w_n\in\O(B_{r_n})^\dual$.
De plus, comme les seules puissances de $T$ intervenant dans $u_n^{(p)}$ sont
les $T^{pk}$, on peut \'eliminer les puissances de $T$ premi\`eres \`a $p$ dans les $w_n$
une par une (on commence par $T$ en remarquant que le coefficient $a_1$ de $T$ dans $w_n$ 
ne d\'epend pas de $n$, ce qui permet de diviser $w_n$ par $1+a_1 T$ pour tout $n$, et on recommence
avec $T^2$, etc.). On a alors $w_n=1+\sum_{k\geq 1}w_{n,k}T_n^{pk}$.

Soit $w'_{n,k}=(w_{n,k})^{1/p}$ (pour un choix de racine $p$-i\`eme),
et soit $w'_n=1+\sum_{k\geq 1}w'_{n,k}T_n^{k}\in 1+T_n\O(B_{r_n})$.
Alors $(u_n(w'_{n+1}/w'_n)^{-1})^p\in 1+pT_n\O(B_{r_n})$ car $(w'_n)^p/w_n\in 1+pT_n\O(B_{r_n})$
et $u_n^p/u_n^{(p)} \in 1+pT_n\O(B_{r_n})$.
Il en r\'esulte que $u_n(w'_{n+1}/w'_n)^{-1}\in 1+p^{1/p}T_n\O(B_{r_n})$, et le lemme~\ref{appen25}
fournit une suite de $v'_n\in\O(B_{r_n})^\dual$ tels que
$u_n(w'_{n+1}/w'_n)^{-1}=v'_{n+1}/v'_n$.  Si on pose alors $v_n=v'_nw'_n$,
on a $v_n\in \O(B_{r_n})^\dual$ et $u_n=v_{n+1}/v_n$ pour tout $n$.
On en d\'eduit que $[(u_n)]=0$, ce que l'on voulait.
\end{proof}

\subsection{Cohomologie \'etale}\label{appen13}
On a une suite exacte 
$$0\to \Z/p^n\otimes \O(\ocirc{B}^{\rm an})^\dual\to H^1_{\eet}(\ocirc{B}^{\rm an},\Z/p^n(1))\to {\rm Pic}(\ocirc{B}^{\rm an})[p^n]\to 0.$$
Comme $\O(\ocirc{B}^{\rm an})^\dual=C^\dual\times(1+\O_C[[T]])$ et que $\Z/p^n\otimes C^\dual=0$,
on d\'eduit des prop.~\ref{appen4} et~\ref{appen24} le r\'esultat suivant 
(le cas de $\Z_p$ se d\'eduit du cas de $\Z/p^n$ en passant \`a la limite
puisque $1+\O_C[[T]]$ est s\'epar\'e et complet pour la topologie $p$-adique).
\begin{coro}\label{appen14}
On a des isomorphismes:
\begin{align*}
&H^1_{\eet}(\ocirc{B}^{\rm an},\Z/p^n(1))\cong
\Z/p^n\otimes (1+\O_C[[T]])
,\quad{\text{si $n\geq 1$,}}\\
&H^1_{\eet}(\ocirc{B}^{\rm an},\Z_p(1))\cong 1+T\O_C[[T]].
\end{align*}
\end{coro}

Comme $H^2_{\eet}(\ocirc{B}^{\rm an},\Z_p(1))$ est le compl\'et\'e $p$-adique de ${\rm Pic}(\ocirc{B}^{\rm an})$, on d\'eduit
de la prop.~\ref{appen6}
le r\'esultat suivant.
\begin{theo}\label{appen15}
On a la dichotomie suivante:

$\bullet$ Si $C$ est sph\'eriquement complet, $H^2_{\eet}(\ocirc{B}^{\rm an},\Z_p(1))=0$.

$\bullet$ Si $C$ n'est pas sph\'eriquement complet, $H^2_{\eet}(\ocirc{B}^{\rm an},\Z_p(1))\neq 0$
et n'a pas de torsion.
\end{theo}

\begin{rema}\label{appen16}
Supposons $C$ non sph\'eriquement complet.

(i) En utilisant les exemples de la prop.~\ref{appen6}, on peut montrer
que $H^2_{\eet}(\ocirc{B}^{\rm an},\Z_p(1))$ est, non seulement non nul, mais {\og \'enorme\fg}.

(ii) Comme $H^2_{\eet}(\ocirc{B}^{\rm an},\Q_p(1))\neq 0$, cela fournit un exemple de non injectivit\'e
de l'application naturelle de la cohomologie \'etale vers la pro\'etale
puisque 
$H^2_{\proet}(\ocirc{B}^{\rm an},\Q_p(1))=0$ (cf.~\cite{CN2,CDN2}).
\end{rema}
\Subsection{Cohomologie syntomique et cohomologie du groupe fondamental}\label{appen17}

On remarque, en comparant les r\'esultat 
du \S\,\ref{appen13} avec ceux de la prop.\,\ref{basic30},
 que ${\rm Syn}(\ocirc{B},1)$ calcule la cohomologie \'etale de $\ocirc{B}^{\rm an}$ si $C$ est sph\'eriquement
complet mais pas si $C$ n'est pas sph\'eriquement complet.
Une explication est la suivante:
dans tous les cas,  le complexe total associ\'e au complexe double
$$\xymatrix@R=.5cm{
\prod_{n\in\N}F^1\O(\widetilde B_{r_n}) \ar[r]^-{(d,1-\frac{\varphi}{p})}\ar[d]&
\prod_{n\in\N}\Omega^1(\widetilde B_{r_n}) \oplus \prod_{n\in\N}\O(\widetilde B_{r_n})
\ar[r]^-{(1-\frac{\varphi}{p})-d}\ar@<-1cm>[d]\ar@<1cm>[d]& 
\prod_{n\in\N}\Omega^1(\widetilde B_{r_n})\ar[d]\\
\prod_{n\in\N}F^1\O(\widetilde B_{r_n}) \ar[r]^-{(d,1-\frac{\varphi}{p})}&
\prod_{n\in\N}\Omega^1(\widetilde B_{r_n}) \oplus \prod_{n\in\N}\O(\widetilde B_{r_n})
\ar[r]^-{(1-\frac{\varphi}{p})-d}& \prod_{n\in\N}\Omega^1(\widetilde B_{r_n})
}$$
dans lequel les fl\`eches verticales sont $(x_n)_{n\in\N}\mapsto(x_n-x_{n+1})_{n\in\N}$,
calcule la cohomologie \'etale de $\ocirc{B}^{\rm an}$ puisque les $B_{r_n}^{\rm an}$ forment un recouvrement
croissant de $\ocirc{B}^{\rm an}$.

$\bullet$ Si $C$ 
est sph\'eriquement complet, on a ${\rm R}^1\lim_{n}p^{-s_n}\O_C=0$,
pour toute suite strictement d\'ecroissante $(s_n)_{n\in\N}$ de nombres rationnels.
On en d\'eduit des r\'esultats analogues pour $\acris$, $F^1\acris$, etc.
La nullit\'e des ${\rm R}^1\lim$ ci-dessus implique que les fl\`eches
verticales sont surjectives; on peut donc
 remplacer le double complexe
par le complexe des noyaux des fl\`eches verticales qui n'est autre que
${\rm Syn}(\ocirc{B},1)$. 

$\bullet$ Si $C$ 
n'est pas sph\'eriquement complet, les fl\`eches verticales ne sont pas surjectives,
et 
on ne peut pas remplacer le double complexe par ${\rm Syn}(\ocirc{B},1)$.

\smallskip
Il y a donc une dichotomie un peu d\'esagr\'eable mais,
comme nous allons le voir,
${\rm Syn}(\ocirc{B},1)$ calcule dans tous les cas 
la cohomologie continue du groupe fondamental de $\ocirc{B}^{\rm an}$.

Soit $R=\O_C[[T]]$, et soient $\overline R$ la c\^oture int\'egrale de $R$ dans l'extension
profinie \'etale maximale de $R[\frac{1}{p}]$ et $G_R={\rm Aut}(\overline R/R)$.
\begin{prop}\label{appen20}
Le complexe ${\rm Syn}(\ocirc{B},1)$ calcule la cohomologie continue de $G_R$
\`a valeurs dans $\Z_p$ {\rm (ou $\Z_p(1)$, puisque $C$ est alg\'ebriquement clos)}.
\end{prop}
\begin{proof}
On note $R_\infty$ la sous-$R$-alg\`ebre $R[(1+T)^{p^{-\infty}}]$
de $\overline R$, et $\widehat R_\infty$ son compl\'et\'e pour la topologie
$(p,T)$-adique.  Alors $\Gamma={\rm Aut}(R_\infty/R)\cong\Z_p$ (on choisit
un g\'en\'erateur topologique $\gamma$ de $\Gamma$), et $\widehat R_\infty$
est perfecto\"{\i}de.  On en d\'eduit que la cohomologie de $G_R$ est calcul\'ee
par le complexe 
$$\xymatrix@C=1.8cm{\widetilde R\ar[r]^-{(1-\gamma,1-\varphi)}&
\widetilde R\oplus\widetilde R\ar[r]^-{(1-\varphi)-(1-\gamma)}&\widetilde R}$$
o\`u $\A_C=W(C^\flat)$, $\A^+_C=W(\O_{C^\flat})$ et $\widetilde R=
\A_C\otimes_{\A^+_C}\A_C^+[[T]]$, les actions de $\gamma$ et $\varphi$
sur $\widetilde R$ \'etant donn\'ees
par $\gamma(T)=(1+\pi)T+\pi$, o\`u $\pi=[\epsilon]-1$ et
$\varphi(T)=(1+T)^p-1$.
(Pour aboutir \`a cette expression au lieu de $\widetilde R_\infty=W(\widehat R_\infty[\frac{1}{p}]^\flat)$,
il faut utiliser l'inverse \`a gauche $\psi:\widetilde R\to \widetilde R$ de $\varphi$
donn\'e par $\psi(\sum_{i=0}^{p-1}(1+T)^i\varphi(x_i))=x_0$, et le fait que
$\gamma-1$ est inversible sur $\widetilde R^{\psi=0}$; cela permet de d\'ecompl\'eter
comme d'habitude.)

Ensuite, on v\'erifie que le complexe ci-dessus est quasi-isomorphe \`a
$$\xymatrix@C=1.5cm{\frac{1}{\varphi^{-1}(\pi)}\widetilde R^+\ar[r]^-{(1-\gamma,1-\varphi)}&
\widetilde R^+\oplus\frac{1}{\pi}\widetilde R^+\ar[r]^-{(1-\varphi)-(1-\gamma)}&\widetilde R^+},$$
o\`u $\widetilde R=\A_C^+[[T]]$ (c'est comme d'habitude, utiliser le complexe avec $\psi$
au lieu de $\varphi$ et le fait que $\widetilde R^{\psi=1}=(\widetilde R^+)^{\psi=1}$
et le fait qu'une solution $y\in \widetilde R^{\psi=0}$ de $(\gamma-1)y=x$ 
appartient \`a $(\frac{1}{\pi}\widetilde R^+)^{\psi=0}$ si $x\in (\widetilde R^+)^{\psi=0}$).
On plonge $\frac{1}{\varphi^{-1}(\pi)}\widetilde R^+$ et $\frac{1}{\pi}\widetilde R^+$
dans $F^1\widetilde R^+$ et $\widetilde R^+$ par $x\mapsto \pi x$, puis on plonge 
$F^1\widetilde R^+$ et $\widetilde R^+$ dans $F^1\O(\widetilde B)$ et $\O(\widetilde B)$
(il faut v\'erifier que ceci induit un quasi-isomorphisme).
Enfin, on remarque que $\frac{\gamma-1}{\pi}=\partial$ modulo $\pi$,
ce qui permet de passer du complexe obtenu \`a ${\rm Syn}(B,1)$ comme dans~\cite{CN}, \cite{SG}.
\end{proof}

\subsection{La fibre g\'en\'erique adoque}\label{appen21}
Il r\'esulte de la prop.~\ref{appen20} et du \S\,\ref{appen17} que, 
si $C$ est sph\'eriquement complet, la cohomologie
\'etale de $\ocirc{B}^{\rm an}$ est celle de son groupe fondamental (i.e.~$\ocirc{B}^{\rm an}$ est un $K(\pi,1)$), alors
que si $C$ n'est pas sph\'eriquement complet, ce n'est pas du tout le cas.
On peut s'extasier devant la richesse du monde $p$-adique ou consid\'erer que c'est
une pathologie dont on aimerait se d\'ebarrasser.

On peut penser
que le probl\`eme vient de ce que 
l'on n'a pas pris la bonne fibre g\'en\'erique et que, pour restaurer un
peu d'harmonie, il faudrait la remplacer par {\it la fibre g\'en\'erique adoque}
dont nous allons esquisser une d\'efinition possible.

\vskip.2cm
Soit $Y$ un affino\"{\i}de sur $C$.  
On retrouve la topologie rigide sur $Y$ en prenant comme base d'ouverts
les images inverses (i.e.~les tubes) des ouverts de mod\`eles de $Y$ sur $\O_C$.
Si $U$ est un ouvert d'un tel mod\`ele ${\cal Y}$, on a $\O_Y(]U[)=\O_{\cal Y}(U)[\frac{1}{p}]$.

On peut enrichir cette structure, en d\'efinissant $Y^{\rm ado}$ comme l'ensemble
des points de l'espace adique associ\'e \`a $Y$ se sp\'ecialisant en un point de
${\cal Y}^{\rm ado}$ pour un choix de mod\`ele assez fin de $Y$ (on n'obtient ainsi
que des valuations de rang $1$ ou $2$ induisant $v_p$ sur $C$), et en prenant
comme base d'ouverts de $Y^{\rm ado}$ les images inverses (i.e.~les tubes) 
des ouverts des sch\'emas adoques associ\'es aux mod\`eles de $Y$ sur $\O_C$.
Si $U$ est un ouvert adoque d'un tel mod\`ele ${\cal Y}$, on pose 
$\O_{Y^{\rm ado}}(]U[)=\O_{{\cal Y}^{\rm ado}}(U)[\frac{1}{p}]$;
{\it la fibre
g\'en\'erique de $U$} est l'espace annel\'e image inverse de $U$ dans $Y^{\rm ado}$.

Une {\it vari\'et\'e adoque} est alors un espace annel\'e, localement isomorphe
\`a un ouvert de $Y^{\rm ado}$, o\`u $Y$ est un affino\"{\i}de sur $C$.

Si $Y$ est un sch\'ema adoque obtenue en recollant les $U_i$
le long des $U_{i,j}$, sa fibre g\'en\'erique $Y^{\rm gen}$
est la vari\'et\'e adoque obtenue en recollant les fibres g\'en\'eriques
des $U_i$ le long des des fibres g\'en\'eriques des $U_{i,j}$.
En particulier, si $Y$ est un affino\"{\i}de, $Y^{\rm ado}$ 
est la fibre g\'en\'erique de n'importe lequel de ses mod\`eles.

\begin{rema}
{\rm (i)}
Soit $\ocirc{B}^{\rm ado}$ la boule unit\'e ouverte vue comme vari\'et\'e adoque.
Alors $\O(\ocirc{B}^{\rm ado})=\O_C[[T]][\frac{1}{p}]$.
Si on consid\`ere \`a la place la boule unit\'e rigide $\ocirc{B}^{\rm an}$,
alors $\O(\ocirc{B}^{\rm an})={\cal R}^+$, anneau des $\sum_{n\geq 0}a_nT^n$, avec $a_n\in C$
et $v_p(a_n)+nr\to+\infty$ quand $n\to +\infty$, pour tout $r>0$.
Remarquons que $\ocirc{B}^{\rm an}$ peut aussi \^etre consid\'er\'ee comme vari\'et\'e adoque,
r\'eunion des boules ferm\'ees $B_{r_n}$, et on a une injection
$\ocirc{B}^{\rm an}\hookrightarrow \ocirc{B}^{\rm ado}$ (correspondant \`a l'injection naturelle
$\O(\ocirc{B}^{\rm ado})\to \O(\ocirc{B}^{\rm an})$), mais cette injection n'est pas un isomorphisme: le
cercle fant\^ome \`a la fronti\`ere de $\ocirc{B}^{\rm ado}$ n'appartient pas \`a $\ocirc{B}^{\rm an}$
car il n'appartient \`a aucune des $B_{r_n}$.

{\rm (ii)} 
$\ocirc{B}^{\rm ado}$ est quasi-compacte contrairement \`a $\ocirc{B}^{\rm an}$ car
un voisinage du cercle fant\^ome \`a la fronti\`ere
 doit contenir un domaine rationnel non trivial, et on peut donc extraire
de tout recouvrement ouvert un recouvrement fini.  Cela signifie que l'on n'a pas de probl\`eme
de ${\rm R}^1\lim$, et donc que
$${\rm Pic}(\ocirc{B}^{\rm ado})=0.$$
On en d\'eduit que ${\rm Syn}(\ocirc{B},1)$ calcule la cohomologie \'etale
de $\ocirc{B}^{\rm ado}$ ind\'ependamment du fait que $C$ soit ou ne soit pas sph\'eriquement complet.

{\rm (iii)} Ce r\'esultat laisse esp\'erer que les (fibres g\'en\'eriques d') affines adoques
soient des $K(\pi,1)$ comme le sont les affino\"{\i}des d'apr\`es Scholze~\cite[th. 1.2]{RAV}.
Si tel n'est pas le cas, on pourrait envisager de changer la d\'efinition de la cohomologie pro\'etale
pour l'imposer...
\end{rema}

\end{appendix}

\end{document}